\documentclass[nosignatures]{CSReport}

\usepackage{color}
\definecolor{violet}{RGB}{150,0,200}
\definecolor{orange}{RGB}{230,125,0}

\usepackage{subfig}
\usepackage{bm}
\usepackage{multirow}
\usepackage{graphicx}
\usepackage{lscape}
\usepackage{placeins}
\usepackage{epstopdf}
\usepackage{ifthen}
\usepackage[square,numbers,sort&compress]{natbib}
\usepackage{longtable}
\usepackage[super]{nth}
\usepackage[utf8]{inputenc}
\usepackage[english]{babel}
\usepackage[final]{pdfpages}
\usepackage[super]{nth}
\usepackage{hyperref}
\usepackage{booktabs}
\usepackage{setspace}
\usepackage{mathtools}
\usepackage{tabularx}
\usepackage[flushleft]{threeparttable} % http://ctan.org/pkg/threeparttable
\usepackage{rotating}
\usepackage{bibentry}
\usepackage{soul}
\usepackage{hyperref}

\makeindex

\CSDepartment{}
\CSFaculty{Electrical and Electronic Engineering}
\CSName{Martin Peter Neuenhofen}
\CSNameCover{\Large{\CSname}}
\CSNameTitle{Martin Peter Neuenhofen}
\CSPosition{PhD Student}
\CSDate{\today}
\CSTitle{Quadratic Integral Penalty Methods for Numerical Trajectory Optimization}
\CSSubTitle{Dissertation Thesis}
\CSKeyWords{Numerical Methods, Applied Mathematics, Optimal Control}
\CSCoverPicture{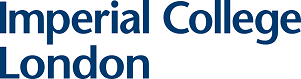}

\setThesisInfo
\usetikzlibrary{patterns,shapes,arrows}
\usetikzlibrary{positioning}
\tikzstyle{block} = [draw, rectangle, minimum height=3em, minimum width=3em]
\tikzstyle{gain} = [draw, isosceles triangle, minimum height=3em]
\tikzstyle{sum} = [draw, circle, node distance=1cm]
\tikzstyle{input} = [coordinate]
\tikzstyle{output} = [coordinate]
\tikzstyle{pinstyle} = [pin edge={to-,thin,black}]
\tikzstyle{module} = [draw, rectangle, minimum height=3em, minimum width=3em]
\tikzstyle{channel} = [draw, ellipse, minimum height=3em, minimum width=3em]
\tikzstyle{decision} = [diamond, draw, fill=blue!20,
    text width=7em, text badly centered, inner sep=0pt]
\tikzstyle{decision1} = [diamond, draw, fill=blue!20,
    text width=7em, text badly centered, inner sep=0pt, node distance = 3cm]
\tikzstyle{block} = [rectangle, draw, fill=blue!20,
    text width=13em, text centered, rounded corners, minimum height=3em]
\tikzstyle{blockwhite} = [rectangle, draw, fill=white!20,
    text width=10em, text centered, rounded corners, minimum height=3em]
\tikzstyle{block2} = [rectangle, draw, fill=blue!20,
    text width=5em, text centered, rounded corners, minimum height=3em, node distance = 4cm]
\tikzstyle{block3} = [rectangle, draw, fill=red!20,
    text width=5em, text centered, rounded corners, minimum height=3em, node distance = 4cm]
\tikzstyle{blockc1} = [rectangle, draw, fill=blue!20,
    text width=13em, text centered, rounded corners, minimum height=3em, node distance = 2.2cm]
\tikzstyle{blockc2} = [rectangle, draw, fill=blue!20,
    text width=13em, text centered, rounded corners, minimum height=3em, node distance = 2.2cm]
\tikzstyle{blockl2} = [rectangle, draw, fill=blue!20,
    text width=5em, text centered, rounded corners, minimum height=3em, node distance = 8cm]
\tikzstyle{blockl4} = [rectangle, draw, fill=blue!20,
    text width=5em, text centered, rounded corners, minimum height=3em, node distance = 5cm]
\tikzstyle{blockc3} = [rectangle, draw, fill=red!20,
    text width=5em, text centered, rounded corners, minimum height=3em, node distance = 2.2cm]
\tikzstyle{line} = [draw, very thick, color=black!50, -latex']
\tikzstyle{cloud} = [draw, ellipse,fill=red!20, node distance=2.5cm, minimum height=2em]

\usepackage{nomencl}
\makenomenclature

\usepackage[symbol]{footmisc}

\makeatletter
\renewcommand*\env@matrix[1][\arraystretch]{%
  \edef\arraystretch{#1}%
  \hskip -\arraycolsep
  \let\@ifnextchar\new@ifnextchar
  \array{*\c@MaxMatrixCols c}}
\makeatother

\RequirePackage{ifthen}
\renewcommand{\nomgroup}[1]{%
\ifthenelse{\equal{#1}{A}}{\item[\textbf{\\Abbreviations}]}{%
\ifthenelse{\equal{#1}{R}}{\item[\textbf{\\Symbols (Roman)}]}{%
\ifthenelse{\equal{#1}{G}}{\item[\textbf{\\Symbols (Greek)}]}{
\ifthenelse{\equal{#1}{D}}{\item[\textbf{\\Subscripts}]}{
\ifthenelse{\equal{#1}{E}}{\item[\textbf{\\Superscripts}]}{}}}}}}

\DeclareUnicodeCharacter{22C6}{$^\prime$}
\allowdisplaybreaks
\usepackage{array,multirow}

\newlength\colA
\newlength\colB
\newlength\colC
\newlength\colD
\newlength\colE

\usepackage{algorithm}
\usepackage{algpseudocode}

\theoremstyle{remark}
\newtheorem{thm}{Theorem}[section]
\newtheorem{cor}[thm]{Corollary}
\newtheorem{comment}[thm]{Comment}
\newtheorem{remark}[thm]{Remark}
\newtheorem{example}[thm]{Example}
\newtheorem{lem}[thm]{Lemma}
\newtheorem{theorem}[thm]{Theorem}

\newtheorem{prop}[thm]{Proposition}
\newtheorem{defn}[thm]{Definition}

\theoremstyle{definition}
\usepackage{amsmath}

\newcommand\dunderline[3][-1pt]{{%
		\sbox0{#3}%
		\ooalign{\copy0\cr\rule[\dimexpr#1-#2\relax]{\wd0}{#2}}}}

\usepackage[tracking=true]{microtype}

\usepackage{mdframed}
\usepackage{footnote}
\newenvironment{mdframedwithfoot}
{   
	\savenotes
	\begin{mdframed}[nobreak=true]
		\stepcounter{footnote}
		
	}
	{
	\end{mdframed}
	\spewnotes
}

\usepackage{scalerel,stackengine}
\newcommand\pig[1]{\scalerel*[5.5pt]{\Big#1}{%
		\ensurestackMath{\addstackgap[1.5pt]{\big#1}}}}

\usepackage{adjustbox}
\usepackage{csquotes}

\usepackage{balance}

\usepackage{placeins}

\usepackage{tabularx}

\usepackage{lmodern,textcomp}

\usepackage{enumerate}% http://ctan.org/pkg/enumerate

\usepackage{xspace} 
\newcommand\largeparbreak{\par\bigskip}

\newcommand*\tageq{\refstepcounter{equation}\tag{\theequation}}
\newcommand{\nc}{{n_c}\xspace}
\newcommand{\nb}{{n_b}\xspace}
\newcommand{\nx}{{n_x}\xspace}

\newcommand{\inv}{^{-1}\xspace}

\newcommand{\blambda}{\boldsymbol{\lambda}\xspace}

\newcommand{\bxi}{{\boldsymbol{\xi}}\xspace}
\newcommand{\bs}{{\textbf{s}}\xspace}
\renewcommand{\t}{^{\textsf{T}}\xspace}

\newcommand{\hby}{\hat{\textbf{y}}\xspace}
\newcommand{\hbz}{\hat{\textbf{z}}\xspace}

\newcommand{\away}[1]{}

\newcommand{\R}{{\mathbb{R}}\xspace}
\newcommand{\Q}{{\mathbb{Q}}\xspace}

\newcommand{\N}{\mathbb{N}\xspace}

\newcommand{\cV}{\mathcal{V}\xspace}
\newcommand{\cY}{\mathcal{Y}\xspace}
\newcommand{\cI}{\mathcal{I}\xspace}
\newcommand{\cW}{\mathcal{W}\xspace}

\newcommand{\cA}{\mathcal{A}\xspace}
\newcommand{\cZ}{\mathcal{Z}\xspace}
\newcommand{\cB}{\mathcal{B}\xspace}

\newcommand{\cL}{\mathcal{L}\xspace}
\newcommand{\cX}{\mathcal{X}\xspace}
\newcommand{\cP}{\mathcal{P}\xspace}
\newcommand{\tcP}{\widetilde{\mathcal{P}}\xspace}
\newcommand{\cQ}{\mathcal{Q}\xspace}

\newcommand{\cJ}{\mathcal{J}\xspace}
\newcommand{\cU}{\mathcal{U}\xspace}

\newcommand{\cC}{\mathcal{C}\xspace}

\newcommand{\cT}{\mathcal{T}\xspace}
\newcommand{\cO}{\mathcal{O}\xspace}

\newcommand{\bA}{\textbf{A}\xspace}
\newcommand{\bn}{\textbf{n}\xspace}
\newcommand{\bJ}{\textbf{J}\xspace}

\newcommand{\bB}{\textbf{B}\xspace}
\newcommand{\bG}{\textbf{G}\xspace}

\newcommand{\tbH}{\tilde{\textbf{H}}\xspace}

\newcommand{\bZ}{\textbf{Z}\xspace}
\newcommand{\bD}{\textbf{D}\xspace}
\newcommand{\bS}{\textbf{S}\xspace}

\newcommand{\bF}{\textbf{F}\xspace}
\newcommand{\bx}{{\textbf{x}}\xspace}

\newcommand{\by}{{\textbf{y}}\xspace}
\newcommand{\bb}{{\textbf{b}}\xspace}
\newcommand{\br}{{\textbf{r}}\xspace}
\newcommand{\bc}{{\textbf{c}}\xspace}
\newcommand{\bg}{{\textbf{g}}\xspace}

\newcommand{\bq}{{\textbf{\textit{q}}}\xspace}

\newcommand{\tbc}{\tilde{\textbf{c}}\xspace}

\newcommand{\hbx}{\hat{\textbf{x}}\xspace}

\newcommand{\bz}{\textbf{z}\xspace}
\newcommand{\be}{\textbf{1}\xspace}
\newcommand{\bei}[1]{{\textbf{e}}\xspace}
\newcommand{\bw}{\textbf{w}\xspace}
\newcommand{\bM}{\textbf{M}\xspace}
\newcommand{\bN}{\textbf{N}\xspace}
\newcommand{\bC}{\textbf{C}\xspace}

\newcommand{\bW}{\textbf{W}\xspace}

\newcommand{\bH}{\textbf{H}\xspace}
\newcommand{\bI}{\textbf{I}\xspace}

\newcommand{\bO}{\textbf{0}\xspace}
\newcommand{\tbx}{\tilde{\textbf{x}}\xspace}

\newcommand{\bu}{\textbf{u}\xspace}
\newcommand{\bv}{\textbf{v}\xspace}

\newcommand{\opdiag}{\text{diag}\xspace}

\newcommand{\tol}{{\textsf{tol}}\xspace}

\newcommand{\commentout}[1]{}

\newcommand{\dhyh}{\dot{\hat{y}}_h\xspace}
\newcommand{\hyh}{\hat{y}_h\xspace}
\newcommand{\huh}{\hat{u}_h\xspace}

\newcommand{\bbeta}{\boldsymbol{\beta}\xspace}
\newcommand{\tbbeta}{\tilde{\boldsymbol{\beta}}\xspace}
\newcommand{\bj}{\boldsymbol{j}\xspace}

\newcommand{\phiL}{\phi_{\texttt{L}}\xspace}
\newcommand{\phiR}{\phi_{\texttt{R}}\xspace}
\newcommand{\psiL}{\psi_{\texttt{L}}\xspace}
\newcommand{\psiR}{\psi_{\texttt{R}}\xspace}
\newcommand{\bbL}{\bb_{\texttt{L}}\xspace}
\newcommand{\bbR}{\bb_{\texttt{R}}\xspace}
\newcommand{\bzL}{\bz_{\texttt{L}}\xspace}
\newcommand{\bzR}{\bz_{\texttt{R}}\xspace}
\newcommand{\hbzL}{\hbz_{\texttt{L}}\xspace}
\newcommand{\hbzR}{\hbz_{\texttt{R}}\xspace}
\newcommand{\gL}{g_{\texttt{L}}\xspace}
\newcommand{\gR}{g_{\texttt{R}}\xspace}

{
\newcommand{\yL}{y_{\texttt{L}}\xspace}
\newcommand{\yR}{y_{\texttt{R}}\xspace}
\newcommand{\uL}{u_{\texttt{L}}\xspace}
\newcommand{\uR}{u_{\texttt{R}}\xspace}

\renewcommand{\bf}{\textbf{f}\xspace}

\newcommand{\Iref}{I_\text{ref}\xspace}
\newcommand{\Cquad}{C_\ell\xspace}
\newcommand{\Cbox}{C_\text{box}\xspace}
\newcommand{\Cobj}{C_\text{obj}\xspace}
\newcommand{\Clam}{C_\lambda\xspace}
\newcommand{\Tcl}[1]{\cT^{\text{CGL}}_{#1}\xspace}
\newcommand{\Tclref}[1]{\cT^{\text{CGL}}_{\text{ref},#1}\xspace}

\newcommand{\AssumpI}{\hyperref[assumption:A1]{(A.1)}\xspace}
\newcommand{\AssumpII}{\hyperref[assumption:A2]{(A.2)}\xspace}
\newcommand{\AssumpIII}{\hyperref[assumption:A3]{(A.3)}\xspace}

\newcommand{\tmpyO}{\textcolor{blue}{\texttt{y}_{\texttt{0}}}\xspace}
\newcommand{\tmpyT}{\textcolor{blue}{\texttt{y}_{\texttt{T}}}\xspace}
\newcommand{\tmpyt}{\textcolor{blue}{\texttt{y}_{\texttt{t}}}\xspace}
\newcommand{\tmput}{\textcolor{blue}{\texttt{u}_{\texttt{t}}}\xspace}

\newcommand{\tforall}{\widetilde{\forall}\xspace}

\newcommand{\abbTab}{Table}
\newcommand{\abbtab}{Table}

\newcommand{\abbalg}{Algorithm}
\newcommand{\abbthm}{Theorem}

\newcommand{\pval}{\omega\xspace} 	% penalty parameter
\newcommand{\pmval}{\rho\xspace} 	% mild penalty parameter
\newcommand{\wquad}{\alpha\xspace} 	% quadrature weight
\newcommand{\bdual}{\blambda} 						% dual symbol
\newcommand{\hbA}{\hat{\boldsymbol{A}}\xspace} 		% aux A

\newcommand{\eqnKKT}{(KKT)\xspace}
\newcommand{\eqnKKTa}{(KKT1)\xspace}
\newcommand{\eqnKKTb}{(KKT2)\xspace}

\newcommand{\tbM}{\tilde{\textbf{M}}\xspace}
\newcommand{\tbJ}{\tilde{\textbf{J}}\xspace}
\newcommand{\tbg}{\tilde{\textbf{g}}\xspace}

\newcommand{\infexpr}{{\operatornamewithlimits{min}_{x_h \in \cX_{h,p}}\|x^\star_{\omega,\tau}-x_h\|_\cX}\xspace}

\sloppy

\begin{document}
\bibliographystyle{abbrvnat}
%========================== Front matter ======================================

\frontmatter
\maketitle

\begin{onehalfspacing}
	
\chapter{Declaration of Originality}
\label{ch:originality}

This work is -- except for introduction and conclusion -- a concatenation of other texts that the PhD student has written. The thesis consists of the following articles. They are listed in the order of appearance in the thesis:
\begin{enumerate}
	\item Martin P. Neuenhofen and Yuanbo Nie and Eric C. Kerrigan. \emph{Direct Quadrature Penalty Methods: Everywhere between Direct Integral Penalty and Direct Collocation Methods}. Planned for publication.
	\item Martin P. Neuenhofen and Eric C. Kerrigan. \emph{A direct method for solving integral penalty transcription of optimal control problems.} Proceedings of the IEEE Conference on Decision and Control 2020.
	\item Martin P. Neuenhofen and Eric C. Kerrigan. \emph{Dynamic Optimization with Convergence Guarantees.} arXiv:1810.04059, 2018.
	\item Martin P. Neuenhofen and Eric C. Kerrigan. \emph{An integral penalty-barrier direct transcription method for optimal control.} Proceedings of the IEEE Conference on Decision and Control 2020.
\end{enumerate}

The majority of the thesis is based on the first item. The second and third item make Chapters~\ref{chap:MALM}--\ref{chap:PBF_SICON}. The fourth item is peer-reviewed published work that summarizes the method presented in the third item.

Yuanbo Nie has contributed to the numerical experiments in Section~\ref{sec:NumExp}: We formulated the expressions of these problems so that they fit our format. He helped with generating interpolations from the interpolation data for the computed reference solutions in ICLOCS-II~\cite{iclocs2} that we used for validation. He also helped with the abstract and we had many valuable discussions on the design of figures and the line of presentation. This does not include the actual implementation of the experiments, of the solver, and the generation of data and figures.

In all items, Eric Kerrigan contributed to the theoretical analysis in the form of verification and modification of proofs, advice on the line of presentation, suitable formulations, and suggestions of references.

The independent contributions of the PhD student in all above listed items are the development and computational analysis, including the proof of convergence and of convergence rates, for all presented methods; further, the creation of figures and tables, design of examples, structure and organization of presentation, and most of the writing.

\largeparbreak

The following statement is provided by Imperial College London and confirmed by the student:

The work presented hereafter is based on research carried out by the author at the Imperial College London and it is all the author’s own work under supervisors' supervision, except where otherwise acknowledged. Reuse of author’s own published works during the PhD degree program are also acknowledged according to publishers' guidelines.

In reference to IEEE copyrighted material which is used with permission in this thesis, the IEEE does not endorse any of Imperial College London’s products or services. Internal or personal use of this material is permitted. If interested in reprinting/republishing IEEE copyrighted material for advertising or promotional purposes or for creating new collective works for resale or redistribution, please go to \url{http://www.ieee.org/
	publications_standards/publications/rights/rights_link.html} to learn how to obtain a License from RightsLink.

\vspace{10mm}

\textit{Martin Neuenhofen}

London, March 2022

\chapter{Copyright Declaration}
\label{ch:copyright}

Students at Imperial College London must publish their PhD thesis under one of four possible licenses. The following states the license of this thesis and is a unique text defined by the selected license:

The copyright of this thesis rests with the author. Unless otherwise indicated, its contents are licensed under a Creative Commons Attribution-Non Commercial-No Derivatives 4.0 International Licence (CC BY-NC-ND). Link: \url{https://creativecommons.org/licenses/by-nc-nd/4.0/}

Under this licence, you may copy and redistribute the material in any medium or format on the condition that; you credit the author, do not use it for commercial purposes and do not distribute modified versions of the work.

When reusing or sharing this work, ensure you make the licence terms clear to others by naming the licence and linking to the licence text.

Please seek permission from the copyright holder for uses of this work that are not included in this licence or permitted under UK Copyright Law.

The below way of citation is suggested for this work:
\vspace{5mm}
\newline
\noindent
Martin Peter Neuenhofen. \emph{Quadratic Integral Penalty Methods for Numerical Trajectory Optimization}. PhD Thesis. Imperial College London. 2022.

\chapter{Abstract}
\label{ch: Abstract}

This thesis presents new mathematical algorithms for the numerical solution of a mathematical problem class called \emph{dynamic optimization problems}. These are mathematical optimization problems, i.e., problems in which numbers are sought that minimize an expression subject to obeying equality and inequality constraints. Dynamic optimization problems are distinct from non-dynamic problems in that the sought numbers may vary over one independent variable. This independent variable can be thought of as, e.g., time.

This thesis presents three methods, with emphasis on algorithms, convergence analysis, and computational demonstrations. The first method is a direct transcription method that is based on an integral quadratic penalty term. The purpose of this method is to avoid numerical artifacts such as ringing or erroneous/spurious solutions that may arise in direct collocation methods. The second method is a modified augmented Lagrangian method that leverages ideas from augmented Lagrangian methods for the solution of optimization problems with large quadratic penalty terms, such as they arise from the prior direct transcription method. Lastly, we present a direct transcription method with integral quadratic penalties and integral logarithmic barriers. All methods are motivated with applications and examples, analyzed with complete proofs for their convergence, and practically verified with numerical experiments.

\chapter{Acknowledgement}
\label{ch: Acknowledgement}

First and foremost, I would like to thank my supervisor, Prof.~Eric Kerrigan, for the continuous support throughout this research program. The first time I got introduced to him was during my employment at Mercedes AMG Petronas Formula One. During my work there I came to notice that the methods we used to solve optimization problems did often fail to converge. I thus developed numerical methods and convergence analyses to fix these issues. When I searched for a PhD supervisor, I sent my drafts to Eric. After careful reading, he offered me supervision of my PhD thesis. He is a morally, fair, and very professional person. He never broke any agreements. He believed in my work from the start and gave me all possible freedom to do research on all kinds of methods related to optimal control. He was always interested and at most of the time had very good advice. I very much enjoyed working with him.

Special thanks go to Yuanbo Nie, who has introduced me to all the meta-subjects of PhD studies at Imperial. There are a lot of frustrating and unnecessary procedures. Yuanbo helped me a lot with the research, with the motivation, the work, the procedures, the LaTeX annoyances, the review frustrations, deep discussions, and also with practical help in uncountable instances. His wife and he offered me to stay at their apartment for three weeks and showed me the city when I came to visit London for the second time.

I would like to thank Chen Greif for everything he offered to me and for his huge financial support. I was introduced to Chen as a reviewer of a manuscript on next-generation Krylov methods. He was very interested and contributed eventually, and offered me a scholarship at UBC. As head of department, Chen was very busy at the time when I joined his group. Nonetheless, he was always very nice, supportive, went every extra mile that he could, and with Mstab we published a great piece of research.

Last but definitely not least, I thank my family. We build things together, share advice, go places, celebrate life, and enjoy every day.

Thanks to all of you!

\vspace{10mm}

\textit{Martin Neuenhofen}

M\"onchengladbach, March 2022

\chapter{Motivation of this Thesis and Overview of Presented Methods}
As mentioned in the abstract, we present three numerical methods. All methods are either to directly transcribe dynamic optimization problems or to solve finite-dimensional nonlinear optimization problems that arise from the transcription.

Today, the state-of-the-art method for direct transcription is direct collocation. The motivation for the development and analysis of other methods stems from the fact that there are problems for which direct collocation fails to converge. Challenges are problems with singular arcs, with high-index differential-algebraic constraints, and problems with consistently over-determined constraints. This thesis is concerned with the development of numerical methods that can solve these problems in a black-box fashion reliably and efficiently to high accuracy.

In the following we give an abstract for each method. In Section~\ref{sec:Scope} we explain how the thesis is structured and where each method is described in detail.

\paragraph{Method 1: Direct Quadrature Penalty Method}
We present a quadrature penalty method as a competitive alternative to the popular method of direct collocation for solving optimal control problems with ordinary differential equations and differential-algebraic equations numerically.

There is consensus that direct collocation methods do not converge for some problems of practical interest. However, there is lack of practical means to assess the reasons behind these failures. 

In this work, we aim to provide full transparency and accessibility to the numerical reasons behind the convergence of a generalization of direct collocation methods: quadrature penalty methods. A wealth of illustrated case studies, numerical experiments, and convergence theory demonstrates the practical benefits of quadrature penalty methods in terms of efficiency and robustness when compared to direct collocation.

To make the quadratic penalty method readily available for a wide audience, including readers who are new to optimal control, we include a full background to the origins, motivations, developments, and reasons behind the design of today’s optimal control methods. To improve accessibility, we refrained from sophisticated mathematics whenever it was possible to explain the same concept in a familiar but elaborate way. Familiarity with the following methods is sufficient: i) explicit Euler, ii) gradient-descent for local optimization, iii) Newton method for systems of equations, and iv) Gaussian quadrature.

\paragraph{Method 2: Modified Augmented Lagrangian Method}
We present a numerical method for the minimization of constrained optimization problems where the objective is augmented with large quadratic penalties of inconsistent equality constraints. Such objectives arise from quadratic integral penalty methods for the direct transcription of optimal control problems.

The Augmented Lagrangian Method (ALM) has a number of advantages over the Quadratic Penalty Method (QPM). However, if the equality constraints are inconsistent, then ALM might not converge to a point that minimizes the bias of the objective and penalty term. Therefore,  we present a modification of ALM that fits our purpose.

We prove convergence of the modified method and bound its local convergence rate by that of the unmodified method. Numerical experiments demonstrate that the modified ALM can minimize certain quadratic penalty-augmented functions faster than QPM, whereas the unmodified ALM converges to a minimizer of a significantly different problem.

\paragraph{Method 3: Penalty-Barrier Method with Quadratic Penalties and Logarithmic Barriers}
We present a novel direct transcription method to solve optimization problems subject to nonlinear differential and inequality constraints.

We prove convergence of our numerical method under reasonably mild assumptions: boundedness and Lipschitz-continuity of the problem-defining functions.
We do not require uniqueness, differentiability or constraint qualifications to hold and we avoid the use of Lagrange multipliers. Our approach differs fundamentally from well-known methods based on collocation; we follow a penalty-barrier approach, where we compute integral quadratic penalties on the equality path constraints and point constraints, and integral log-barriers on the inequality path constraints.

The resulting penalty-barrier functional can be minimized numerically using finite elements and penalty-barrier interior-point nonlinear programming solvers. Order of convergence results are derived, even if components of the solution are discontinuous.

\end{onehalfspacing}

\toc

%========================== Main matter ======================================
%\listoffigures 
%\listoftables

%Here the nomenclature is generated
\cleardoublepage
\markboth{Nomenclature}{Nomenclature}
\printnomenclature[1in]

\mainmatter

\begin{onehalfspacing}

\part{Introduction to Optimal Control}
\label{part:Intro}

\chapter{Introduction to Optimal Control and Direct Transcription}\label{sec:Intro}
In this first chapter, the concept of optimal control is introduced with three example applications. The first example considers the acceleration of an electric car. We will demonstrate how this optimal control problem can be solved by combining two familiar numerical methods. After laying out the general concepts, we describe in broad terms the contributions of this thesis. The subsequent Section~\ref{sec:Formalities} will introduce formal mathematical definitions.

\section{From Optimization and Integration to Optimal Control}\label{sec:Intro:OCP}
In the following two subsections, we recall two methods that are fundamental to applied mathematics. They will be combined later.

\subsection{Explicit Euler Method} 
The explicit Euler method is attributed to Euler in 1768 \cite{Hairer1}. This method solves initial value problems (IVPs) numerically. These take on the form: Given are initial values $y_\text{ini}$, a final time $T$, and a flux-function $f$. Find a function~$y$ of $t$ that solves
\begin{align}
y(0)=y_\text{ini}\,,\qquad
\dot{y}(t)=f\big(\,y(t),t\,\big)\quad \tforall t \in [0,T]\,.\label{eqn:IVP}
\end{align}
\begin{figure}
	\centering
	\includegraphics[width=0.85\linewidth]{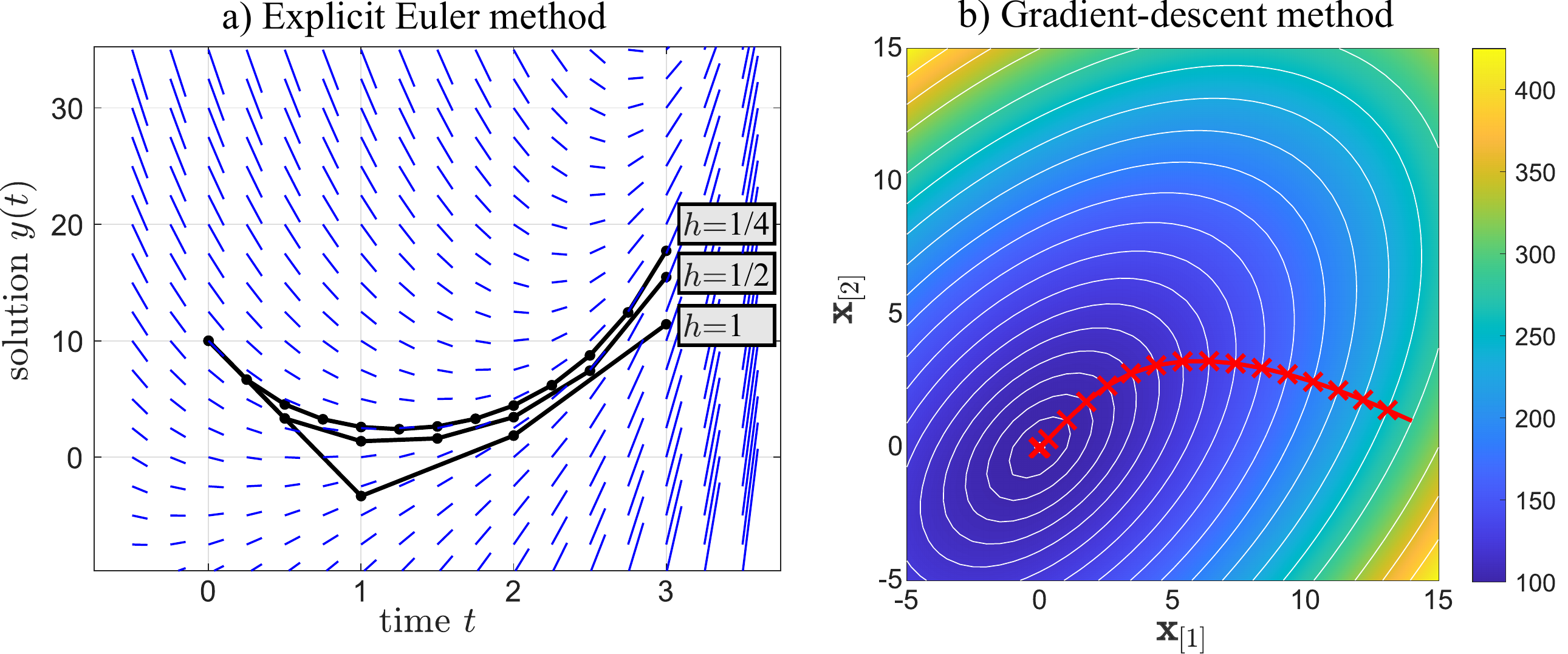}
	\caption{Examples of the explicit Euler method and the gradient-descent method. a) The Explicit Euler method constructs piecewise linear approximations \cite[p.~41]{Hairer1} that follow the slope of the flux-function. Blue dashes indicate the slope value of the flux-function $f$. b) The gradient-descent method proceeds repetitively along the gradient direction. A color map and white contours indicate function values of the objective $\bf$.}
	\label{fig:exampleeuleroptim}
\end{figure}

\begin{remark}
	We use the symbol $\forall$ to abbreviate ``for every'' and $\tforall$ to abbreviate ``for \emph{almost} every''. Other notations like ``almost all (a.a.)'' and ``almost every (a.e.)'' are commonly used for the same purpose in the context of differential equations in optimal control problems \cite{Maurer,MR4046772,Hager2000}.
	
	For example, the indicator function
	\begin{align*}
	\Xi_\Q(t) := \begin{cases}
	1 &\text{when } t \in \Q\\
	0 &\text{otherwise}
	\end{cases}
	\end{align*}
	is zero $\tforall t \in \R$. Notice hence that
	\begin{align}
	\int_\R \Xi_\Q(t)\,\mathrm{d}t=0\,. \label{eqn:LebesgueIntegral}
	\end{align}
	The reason for using $\tforall$ in~\eqref{eqn:IVP} is to accommodate cases where $\dot{y}$ may be undefined at a countable infinite number of points (e.g., due to edges in $y$).
\end{remark}

Figure~\ref{fig:exampleeuleroptim}~(a) shows an example of the explicit Euler method for the initial value problem 
\begin{align}
y_\text{ini}=10\,,\qquad T=3\,,\qquad f(t,y)=4.7129\cdot \sinh(\sqrt 2 \cdot t)	- 3.3349 \cdot \cosh(\sqrt 2 \cdot t) - y	\label{eqn:exampleIvp}
\end{align}
for various step-sizes $h \in \R_{>0}$.

\subsection{Gradient-Descent Method}
The gradient-descent method is a numerical method attributed to Cauchy in 1847 \cite{cauchy1847}. It solves optimization problems of the form: Given is an objective-function $\bf$. Find a local solution $\bx^\star$ to
\begin{align}
\operatornamewithlimits{min}_{\bx \in \R^n} \quad \bf(\bx)\,. \label{eqn:Min}
\end{align}

\begin{remark}
	We chose $f$ and $\bf$, which are unrelated, to keep standard notation of each sub-discipline and not opt for esoteric symbols. We write $\bx_{[i]}$ for the $i^{\text{th}}$ component of a vector $\bx$ henceforth.
\end{remark}

Figure~\ref{fig:exampleeuleroptim}~(b) shows an example of the gradient-descent method for minimizing
\begin{align}
\bf(\bx) = 0.5 \cdot \big( 200 + \bx_{[1]}^2 + \bx_{[2]}^2 + (\bx_{[2]}-\bx_{[1]})^2 \big) \label{eqn:examplefMin}
\end{align}
from the guess $\bx_\text{guess}=(14,\ 1) \in \R^2$.
A color-map visualizes values of $\bf$ at coordinates $(\bx_{[1]},\bx_{[2]})$. White lines illustrate contour lines of $\bf$. The method converges to the minimizer $\bx^\star = (0,\,0)$. We denote minimizers to optimization problems with a star superscript henceforth.

\subsection{Integration + Optimization = Optimal Control}
{Optimal Control Problems} are problems that combine the tasks of integrating functions and optimizing objectives. However, the optimization is not about finding optimal values of numbers. Instead, optimal shapes of functions must be found. Section~\ref{sec:exampleOCP} gives an example of what is meant by this.

Neither the explicit Euler method nor the gradient-descent method alone can optimize functions: The explicit Euler method can only \emph{integrate} functions, while the gradient-descent method can only optimize \emph{numbers}. 

Figure~\ref{fig:integrationoptimizationcombine} depicts the tasks of integration and optimization in a two-dimensional diagram. Traversal along the diagonal arrow requires one to \emph{integrate and optimize} in a combined manner. For all practical purposes, optimal control problems are solved numerically with \emph{direct transcription}. Direct transcription is a numerical method, just like explicit Euler method and gradient-descent method.

In the following subsections, we present and discuss an example of an optimal control problem: We formulate mathematically the application of an accelerating car as an optimal control problem. We then show how it can be solved with a numerical method. In-between, we give a brief subsection on terminology, notation, and symbols.

\begin{figure}
	\centering
	\includegraphics[width=0.65\linewidth]{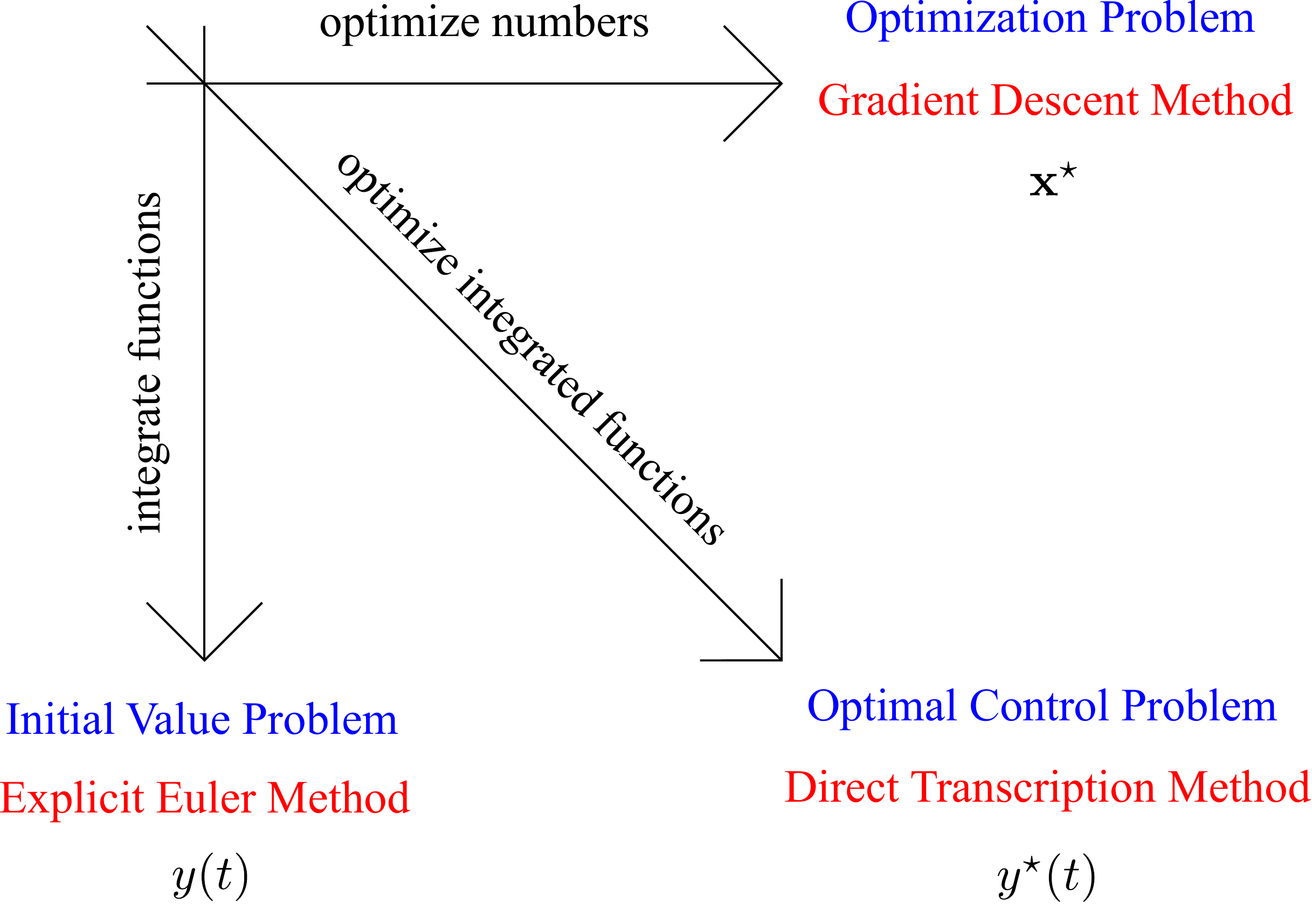}
	\caption{Organigram of tasks involved in three problems types. Problem types are labeled in blue. Tasks are written in black. Examples of direct numerical methods are given in red. Optimal control problems combine the tasks of integration and optimization.}
	\label{fig:integrationoptimizationcombine}
\end{figure}

\subsection{First Example of an Optimal Control Problem: Accelerating Car}\label{sec:exampleOCP}

\begin{figure}
	\centering
	\includegraphics[width=0.6\linewidth]{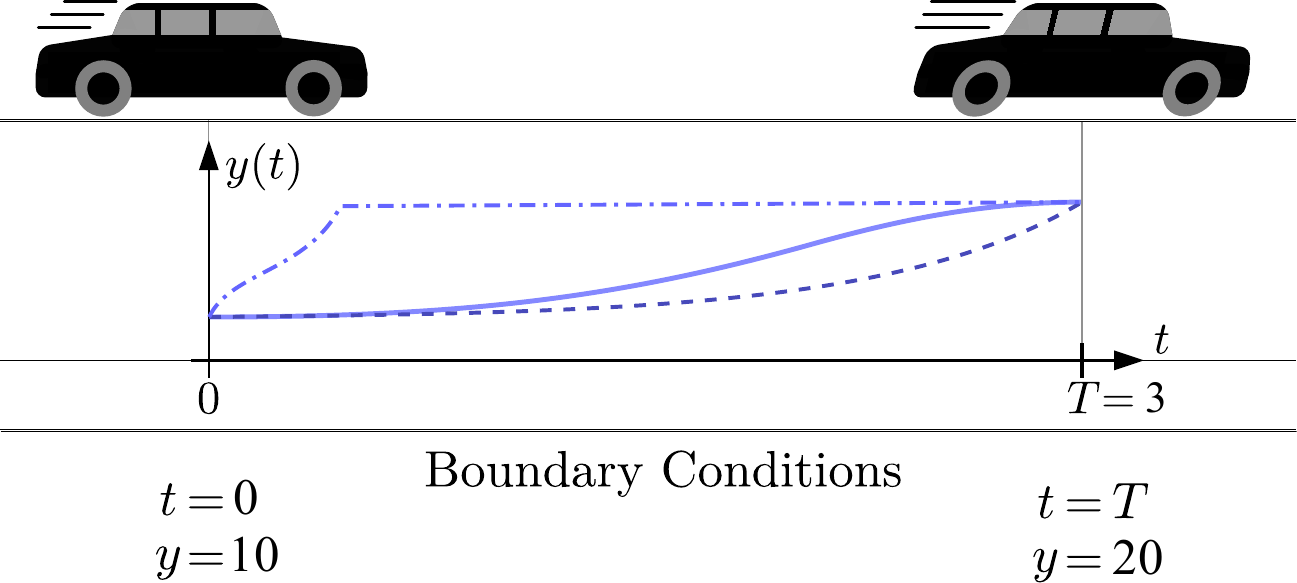}
	\caption{Example of an optimal control problem: An electric vehicle should increase its velocity from an initial value to a final value within a prescribed time-interval $[0,T]$.}
	\label{fig:speedchangeexample}
\end{figure}

Figure~\ref{fig:speedchangeexample} illustrates an example of an optimal control problem. Suppose $y(t)$ is the velocity of a vehicle at time $t$. The car's acceleration at time $t$ is modeled to depend on the throttle input $u(t)$ and the present velocity $y(t)$, e.g. due to friction:
\begin{align}
\dot{y}(t)=u(t)-y(t) \quad \tforall\ t \in [0,T]\,. \label{eqn:exampleODE} 
\end{align}

We wish to increase the car's velocity from an initial value $y(0)=10$~m/s to a final value $y(T)=20$~m/s within a time limit of $T=3$~s. We wish to accelerate the vehicle such that the total wear and tear on energy storage, motors, and bearings is minimized. We model the wear as
\begin{align*}
\int_0^T y^2(t) + u^2(t)\,\mathrm{d}t\,. \tageq\label{eqn:exampleIntegral}
\end{align*}
There are many possible ways to transition from the initial velocity to the final velocity. Figure~\ref{fig:speedchangeexample} illustrates three possible different ways in blue. Which way of transition leads to the least amount of wear and tear?

This question can be stated mathematically as
\begin{equation}
\label{eqn:exampleOCP}
\begin{aligned}
&\quad\operatornamewithlimits{min}_{y,u}	&\int_0^3 &y^2(t) + u^2(t)\,\mathrm{d}t\\
&\quad\text{subject to}\quad\quad  				& y(0)&= 10\,,\\
& 						 						& y(T)&= 20\,,\\
& 									& \dot{y}(t)&=u(t)-y(t) \quad \tforall t \in [0,3]
\end{aligned}
\end{equation}
In this problem, we search functions for $y$ and $u$ such that $y$ passes through the initial and final values and satisfies the differential equation. The desired solution should minimize the wear from model~\eqref{eqn:exampleIntegral}.

\subsection{Second Example: Maximizing Net Worth}
We introduce a model from finance. This model is phrased into a mathematical problem. The problem is solved and gives an insight into finance.

\subsubsection{Finance Model}
Figure~\ref{fig:automaticmining} presents a flow-chart from finance. Possession like solar panels or bitcoin farms produce revenue $r$ proportional to their value $y_{[1]}$. This revenue is split into two purposes: A fraction $u \in [0,1]$ is re-invested to buy more possession, whereas the rest is debited into savings $y_{[2]}$. Net worth is the account balance of savings plus the resale price of the whole possession. At an initial time $t=0$, we start with little possession and no savings: $y_{[1]}(0)=1$, $y_{[2]}(0)=0$. We seek to maximize net worth at the final time $t=10$.

\begin{figure}
	\centering
	\includegraphics[width=0.7\linewidth]{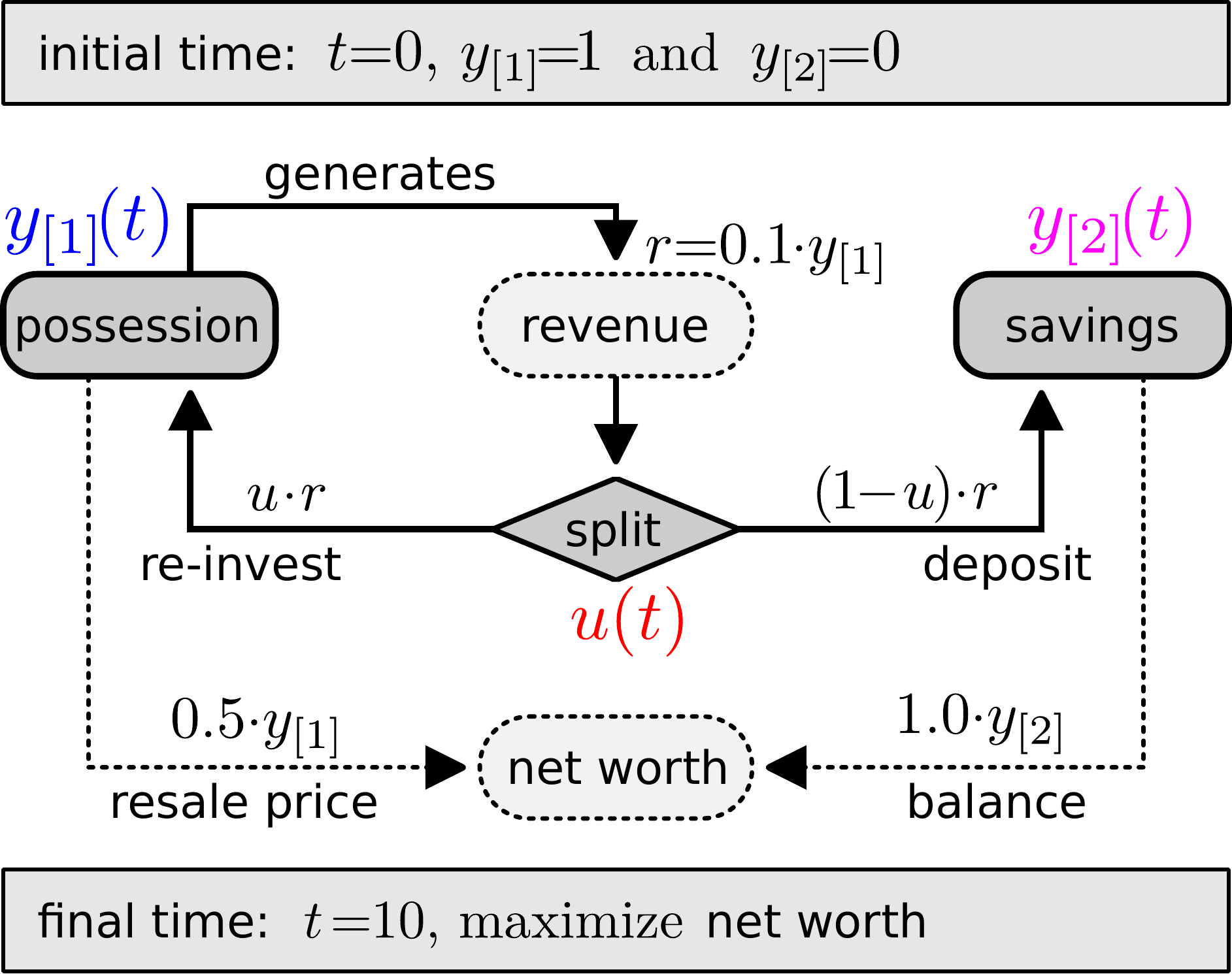}
	\caption{Flow cart of a re-investment model in finance. Revenue is split into re-investment and deposit.}
	\label{fig:automaticmining}
\end{figure}

\subsubsection{Mathematical Formulation}
How should we design the split $u(t)$ of revenue at each time $t$ in order to really become as rich as possible at $t=10$? Of course, initially we probably want to re-invest a lot whereas in the end it would not pay off anymore. But how do we arrive at an accurate answer?

The optimal profile of $u(t)$ can be determined accurately by solving the following optimization problem:
\begin{equation}
\label{eqn:Mining}
\left\lbrace
\begin{aligned}
&\operatornamewithlimits{max}_{y,u} & 0.5\cdot y_{[1]}&(10) + y_{[2]}(10) &\\
& \text{subject to} & y_{[1]}(0)&=1\,, &&\\
&  					& y_{[2]}(0)&=0\,, &&\\
& 					& \dot{y_{[1]}}(t)&=0.1\cdot y_{[1]}(t) \cdot u(t)       		 		&\forall\ &t \in [0,10]\,,\\
& 					& \dot{y_{[2]}}(t)&=0.1\cdot y_{[1]}(t) \cdot \big(1-u(t)\big)\quad 	&\forall\ &t \in [0,10]\,,\\
& 					& 0 \leq u(t) &\leq 1 									 		&\forall\ &t \in [0,10]\,.
\end{aligned}
\right\rbrace
\end{equation}
Likewise, this can be posed as a minimization problem by minimizing the negative net worth.

\subsubsection{Computational Solution}
The solution to problem \eqref{eqn:Mining} are functions $y_{[1]},y_{[2]},u$ of time on the time-interval $[0,10]$. Figure~\ref{fig:miningcontrol} presents the optimal solution: Final net worth is maximized when the revenue is fully re-invested into possession until $t=5$ and fully deposited into savings after $t=5$. The optimal net worth is attained as $\approx 1.6$\,.

\begin{figure}
	\centering
	\includegraphics[width=1\linewidth]{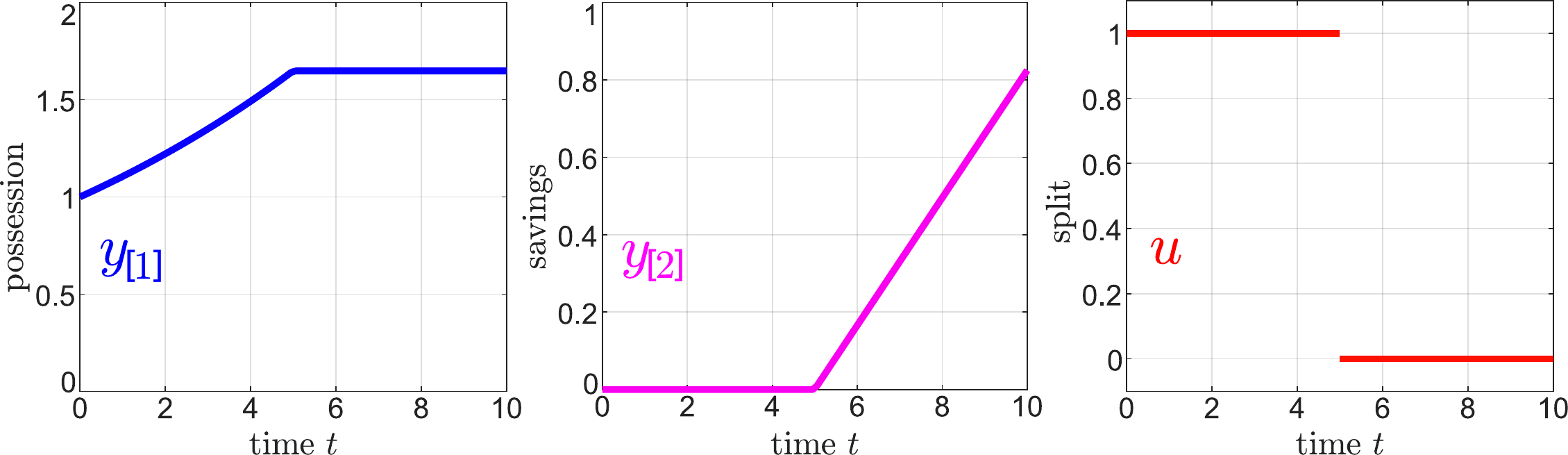}
	\caption{Optimal solution to problem~\eqref{eqn:Mining}.}
	\label{fig:miningcontrol}
\end{figure}

\subsection{Third Example: Economic Operation of a Commute Train}\label{sec:CommuteTrainExample}
Suppose the following dynamic model for a commute train:
\begin{align*}
\dot{x}(t)&=v(t)\,,&
\dot{v}(t)&=a(t)\,,&
\dot{c}(t)&= 1 \cdot v(t) + 0.01\cdot v^2(t) + 100 \cdot \big(\max\lbrace\,0\,,\,a(t)\,\rbrace\big)^2\,.
\end{align*}
In these differential equations, $x(t)$ is the position, $v(t)$ the velocity, $a(t)$ the acceleration, and $c(t)$ the operational cost of the commute train. All physical units are SI units. Cost is in US Dollars. Cost is generated whenever the train moves and/or accelerates. For the purpose of this example, deceleration generates no cost, hence the max-expression in the last equation.

\begin{figure}
	\centering
	\includegraphics[width=0.8\linewidth]{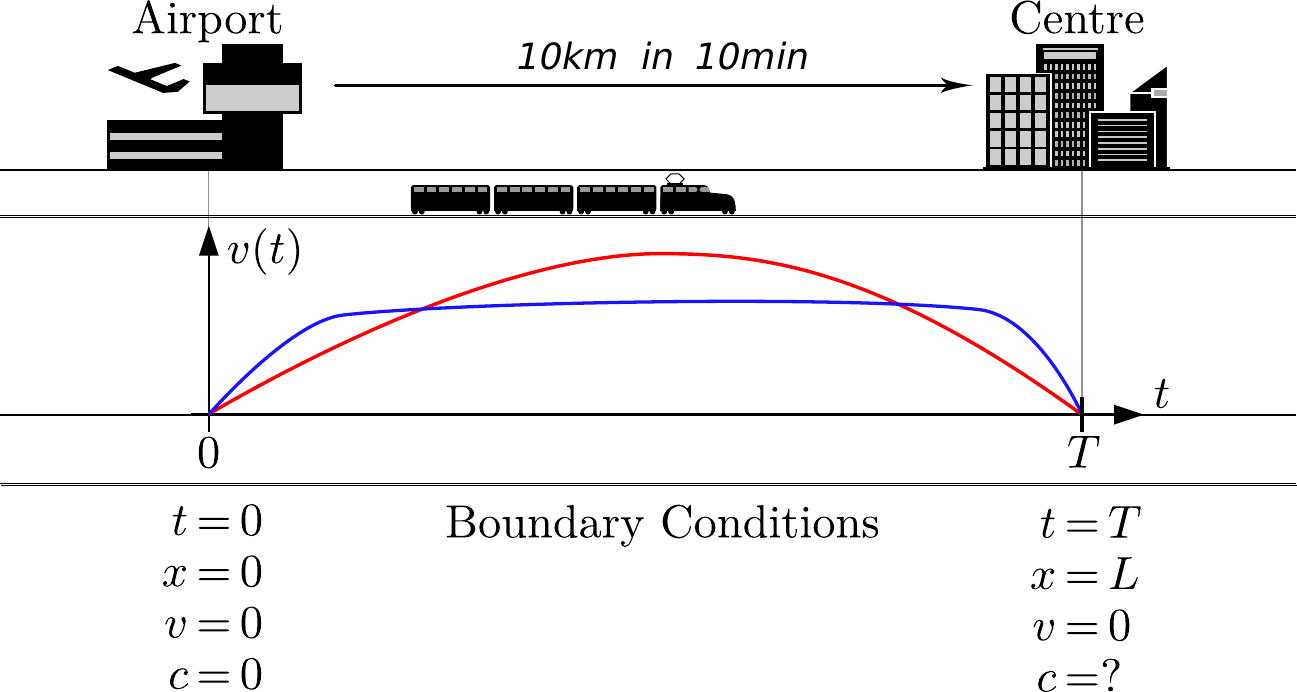}
	\caption{Commute train between Airport and Centre. The train covers a distance of $L=10^4$ metres in an interval of $T=600$ seconds.}
	\label{fig:commutetrain}
\end{figure}

Figure~\ref{fig:commutetrain} depicts the commute route: The train starts from Airport at time zero. At later time $T$, the train is supposed to stop at Centre, which is at distance $L$. For the example, we use $T=600$ and $L=10^5$. As indicated in Figure~\ref{fig:commutetrain}, the conditions
\begin{align*}
x(0)&=0,& v(0)&=0, & c(0)&=0, & x(T)&=L, & v(T)&=0\,.
\end{align*}
are called \emph{boundary conditions}. These just describe that the train stands still at both platforms, and that the train travels a distance of $L$ in time $T$.

There are also \emph{bound constraints}: 
\begin{align*}
v(t)\leq 20\,,\qquad -0.25\leq a(t)\leq 0.2\,.
\end{align*}
For the purpose of this example, these bounds arise from security policies and power limitations of the train.

We wish to find an optimal solution for $x(t),\ v(t),\ a(t),\ c(t)$ such that the final cost $c(T)$ is minimized. As depicted in Figure~\ref{fig:commutetrain} in red and blue, multiple solutions for $v(t)$ are possible that all yield the same traveled distance (area under the graph of $v$) and satisfy the boundary conditions.

The optimal solution is depicted in Figure~\ref{fig:commutetrainsolution}: The optimal acceleration profile starts at the maximum value of $0.2$ and then steadily decays to zero. At about $t\approx 480$, the train will abruptly decelerate with maximum strength. Neither optimization alone nor integration alone could compute this solution.

\begin{figure}
	\centering
	\includegraphics[width=0.7\linewidth]{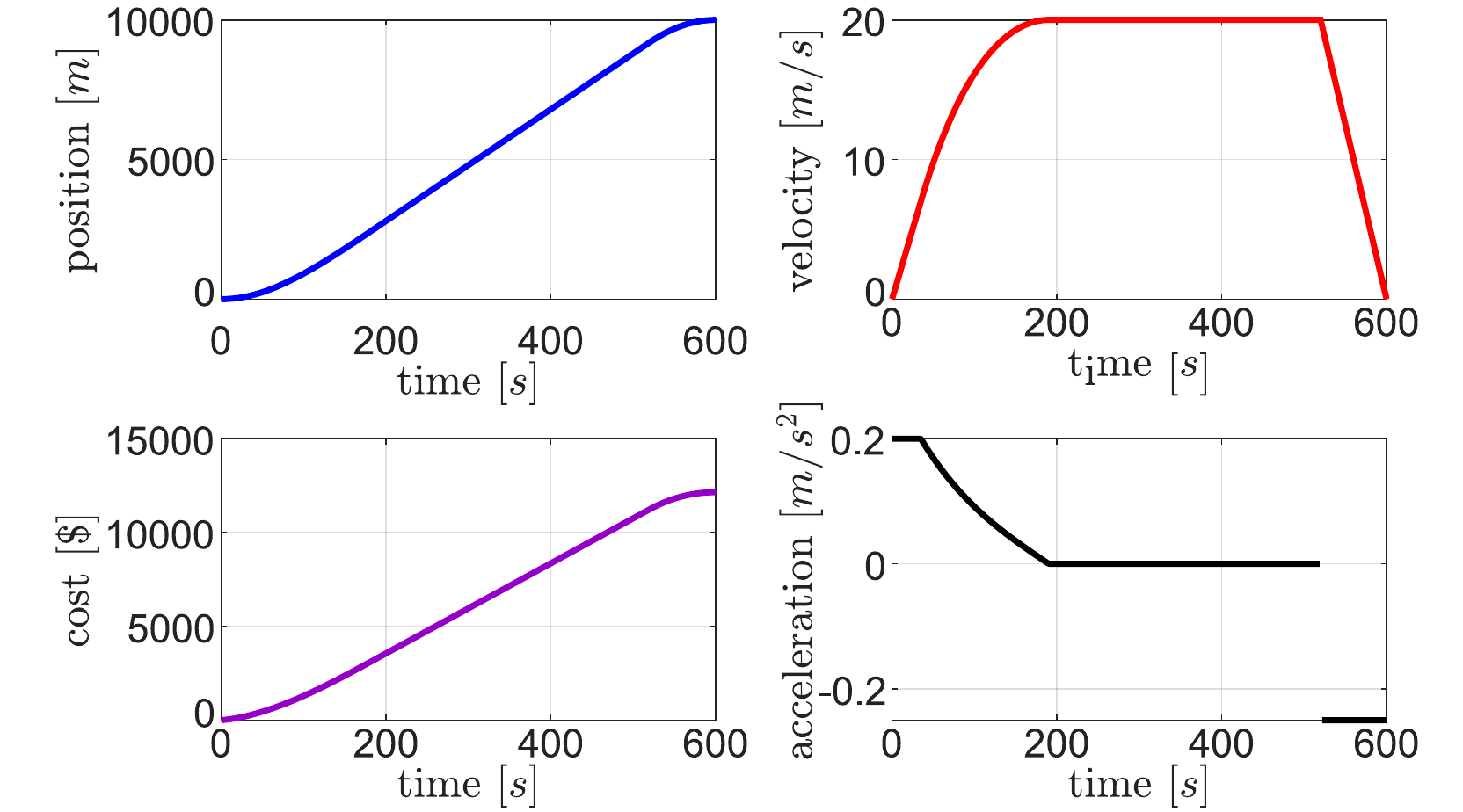}
	\caption{Solution to the Commute Train Example from Section~\ref{sec:CommuteTrainExample}.}
	\label{fig:commutetrainsolution}
\end{figure}

\subsection{Terminology of Optimization}
In the next subsection, we show how problem~\eqref{eqn:exampleOCP} can be solved numerically. Before doing so, it is advantageous to agree on certain terms.

\paragraph{Classification}
Figure~\ref{fig:optimizationterminology}~(a) shows an incomplete classification of optimization problems. These can be distinguished into finite and infinite optimization problems. Figure~\ref{fig:exampleeuleroptim}~(b) shows an example of a finite-dimensional optimization problem, because the solution is sought in the two-, i.e., finite-dimensional space $\R^2$. In contrast, optimal control problems belong into the class of infinite-dimensional optimization problems because continuous functions $y,u$ of $t$ cannot be represented as vectors in $\R^n$ without loss of information.

Within the class of finite-dimensional optimization, problems can be distinguished into different categories: linear programming (LP), quadratic programming (QP), and nonlinear programming (NLP) \cite{Nocedal}. This categorization is done because dedicated numerical solution software is available for problems of each category. NLP is the most general. Thus, numerical solution algorithms for NLP are also capable of solving LP and QP. Section~\ref{sec:NLP} discusses NLP and numerical solution algorithms for NLP instances.

\paragraph{Sets}
In contrast to format~\eqref{eqn:Min}, most optimization problems feature constraints. Problem~\eqref{eqn:exampleOCP} is an example of a constrained optimization problem. Figure~\ref{fig:optimizationterminology}~(b) provides names of related sets: $\cX$ denotes the set of candidates. In problem~\eqref{eqn:Min}, all candidates must live in $\R^2$, hence $\cX=\R^2$. In optimal control, the candidate spaces are so-called \emph{Sobolev spaces} \cite{RossSobolev,Maurer,MR4046772,Hager2000}, introduced later. When the optimization problem features constraints, we use the notation $\cB$ for the set of all candidates that are feasible w.r.t. satisfying the constraints. Among all feasible candidates, candidates that locally minimize the objective are called minimizers and are denoted with superscript $\star$.

\begin{figure}
	\centering
	\includegraphics[width=0.7\linewidth]{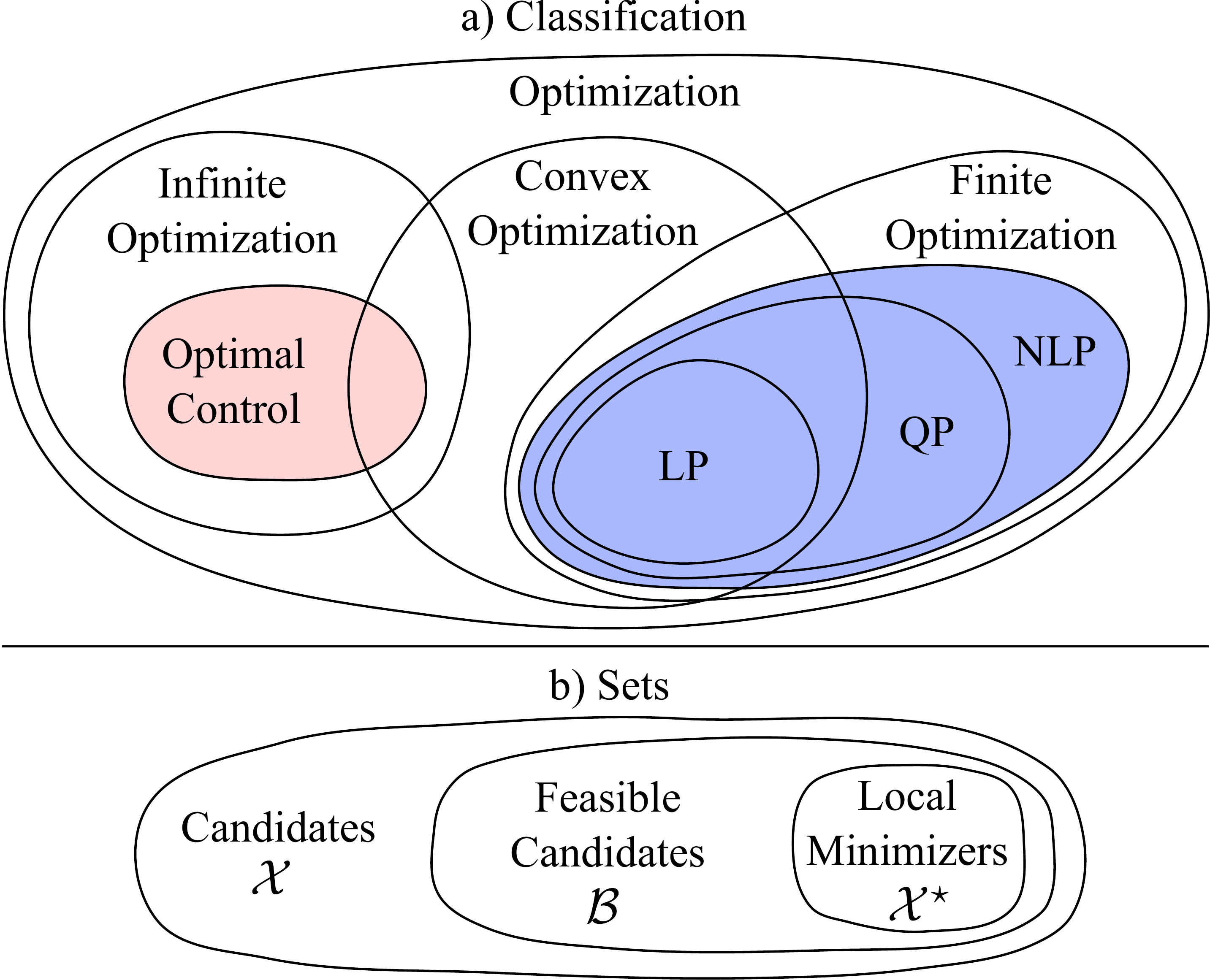}
	\caption{(a) Classification of types of optimization problems. (b) Sets of candidates/minimizers in optimization.}
	\label{fig:optimizationterminology}
\end{figure}

\subsection{Example of a Direct Transcription Method}\label{sec:example:OCPsolver}
In rare cases, optimal control problems can be solved with pen and paper. This is in fact the case for problem~\eqref{eqn:exampleOCP}. The exact minimizer is $y^\star(t)=10\cdot \cosh(\sqrt 2 \cdot t) - 9.4292... \cdot \sinh( \sqrt{2} \cdot t)$, and $u^\star(t)=\dot{y}^\star(t)+y^\star(t)$. This minimizer is shown in the left of Figure~\ref{fig:speedchangeproblem}. The figure also shows a numerical minimizer $y_h^\star,u_h^\star$ in the right.
Engineers and other applied mathematicians are very interested in optimal control solutions because they provide critical insight into the efficient operation of assets. In the given example, the optimal control solution provides the counter-intuitive insight that it is most wear-economic to first decelerate before accelerating at a higher rate.

In practice, most optimal control problems can only be solved numerically. This section presents how \eqref{eqn:exampleOCP} can be solved numerically for $y_h^\star,u_h^\star$. The presentation uses a numerical method called \emph{direct transcription}. In the example, the particular method used within the direct transcription is the explicit Euler method. 

Direct methods are beneficial when numerical solutions must be found without detailed prior knowledge on what the solution might look like (in terms of, e.g., its shape or any properties). This thesis' main focus is on direct methods. Section~\ref{sec:literatureReview:Overview} will provide a literature review, where we also present and compare alternative methods for the numerical solution of optimal control problems.

\begin{figure}[!htbp]
	\centering
	\includegraphics[width=1\linewidth]{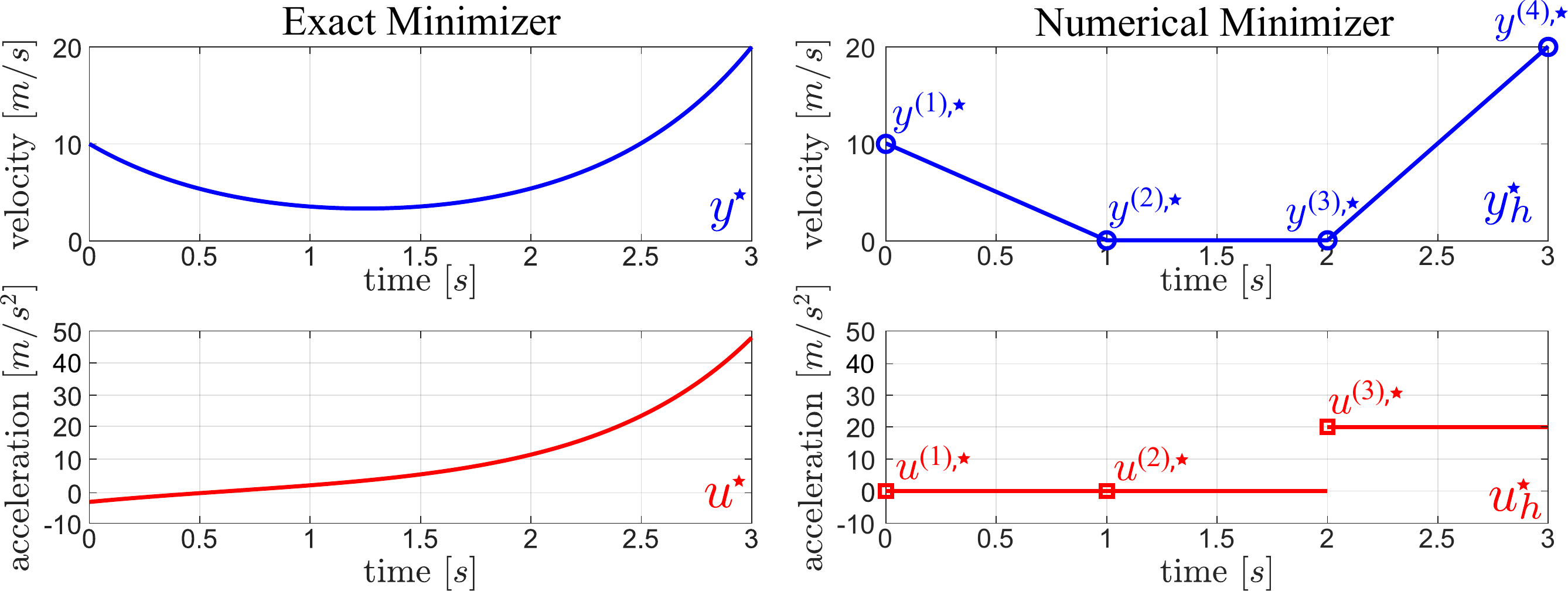}
	\caption{Left: Exact solution to problem~\eqref{eqn:exampleOCP}. Right: Numerical solution to problem~\eqref{eqn:exampleOCP} by means of direct transcription via the explicit Euler method.}
	\label{fig:speedchangeproblem}
\end{figure}

\paragraph{Overview}
Figure~\ref{fig:optimizationterminology} shows the classes of optimal control in red and of NLP in blue. The key idea in direct transcription is to approximate the optimal control problem \eqref{eqn:exampleOCP} with an NLP. The NLP in turn can be solved with methods like gradient-descent. The gradient-descent solution in turn can be used to construct $y^\star_h, u^\star_h$. The following example guides the reader through this process. For the purpose of this example, we use the explicit Euler method. More sophisticated methods are used in practice. These are introduced later in Section~\ref{sec:DirectCollocation}.

\paragraph{Approximating the Optimal Control Problem with the Explicit Euler Method}
The explicit Euler method uses a step-size $h=T/N$ for some $N\in\N$ to approximate differentials. Here, we do this for $N=3$. The differentials in \eqref{eqn:exampleODE} are then approximated with the finite difference approximations
\begin{align*}
\frac{y((i+1) \cdot h) - y(i \cdot h)}{h} \approx \dot{y}(i \cdot h)\qquad \text{ for }i=0,1,\dots,N-1\,.
\end{align*}
We use nodal approximations $y^{(i)} \approx y((i-1) \cdot h)$ and $u^{(i)} \approx u((i-1) \cdot h)$. Inserting the finite difference approximations and nodal approximations, the differential equation~\eqref{eqn:exampleODE} is replaced with
\begin{align*}
\frac{y^{(i+1)} - y^{(i)}}{h} = u^{(i)} - y^{(i)} \qquad \text{ for }i=1,2,\dots,N\,.
\end{align*}

\paragraph{Transcription into NLP}
Using the above explicit Euler approximation and left Riemann sums for the integral~\eqref{eqn:exampleIntegral}, the problem~\eqref{eqn:exampleOCP} reduces into:
\begin{equation}
\label{eqn:ExampleEuler:NLP}
\left\lbrace
\begin{aligned}
&\operatornamewithlimits{min}_{y^{(1)},y^{(2)},y^{(3)},y^{(4)},u^{(1)},u^{(2)},u^{(3)} \in \R}	& h \cdot &\left(\, \sum_{i=1}^{N}\big(y^{(i)}\big)^2 + \big(u^{(i)}\big)^2\, \right) & &\\
& \text{s.t.}  						& y^{(1)}&= 10\,,\\
& 									& y^{(4)}&= 20\,,\\
& 									& \frac{y^{(2)}-y^{(1)}}{h}&=u^{(1)}-y^{(1)}\,,\\
& 									& \frac{y^{(3)}-y^{(2)}}{h}&=u^{(2)}-y^{(2)}\,,\\
& 									& \frac{y^{(4)}-y^{(3)}}{h}&=u^{(3)}-y^{(3)}\,.
\end{aligned}
\right\rbrace
\end{equation}
This is an NLP (actually also a QP) with a solution $\bx=(y^{(1)},y^{(2)},y^{(3)},y^{(4)},u^{(1)},u^{(2)},u^{(3)}) \in \R^7$. A local minimizer $\bx^\star$ can be computed with numerical optimization algorithms. These algorithms are presented in Section~\ref{sec:NLP}. We denote the values in $\bx^\star$ with $y^{(1),\star},y^{(2),\star},y^{(3),\star},y^{(4),\star},u^{(1),\star},u^{(2),\star},u^{(3),\star}$.

The numerical minimizer $y_h^\star,u^\star_h$ of the optimal control problem \eqref{eqn:exampleOCP} in the right of Figure~\ref{fig:speedchangeproblem} can be constructed from the values  $y^{(1),\star},y^{(2),\star},y^{(3),\star},y^{(4),\star},u^{(1),\star},u^{(2),\star},u^{(3),\star}$ of the numerical minimizer of~\eqref{eqn:ExampleEuler:NLP}. The locations of these values are indicated with little circles and squares.

\paragraph{Solving the Optimization Problem with the Gradient-Descent Method}
For the purpose of this demonstration, we want to solve the NLP~\eqref{eqn:ExampleEuler:NLP} with accessible methods. By eliminating the variables $y^{(1)},u^{(1)},u^{(2)},y^{(4)},u^{(3)}$ as follows,
\begin{align*}
y^{(1)}&=10\,, & u^{(1)}&=y^{(1)} + \frac{y^{(2)}-y^{(1)}}{h} = 10 + \frac{y^{(2)}-10}{1} = y^{(2)}\,,\ u^{(2)} = y^{(2)} + \frac{y^{(3)}-y^{(2)}}{h} = y^{(3)}\,,\\ 
y^{(4)}&=20\,, & u^{(3)}&= y^{(3)} + \frac{y^{(4)}-y^{(3)}}{h} = y^{(4)} = 20\,,
\end{align*}
we obtain an unconstrained minimization problem in the remaining two variables $y^{(2)},y^{(3)}$. Arranging these into a vector $\bx = (y^{(2)},y^{(3)})$, the objective function becomes \eqref{eqn:examplefMin}. Figure~\ref{fig:exampleeuleroptim}~(b) shows the minimization procedure of \eqref{eqn:examplefMin} with the gradient-descent method. Once we know $y^{(2),\star},y^{(3),\star}$, we can compute all the other values and construct the solution in the right of Figure~\ref{fig:speedchangeproblem}. The interpolation of the values $y^{(1)},\dots,y^{(4)}$ and $u^{(1)},\dots,u^{(3)}$ is according to \cite{BettsChap2}.

\paragraph{Accuracy of the Numerical Minimizer}
Judging from Figure~\ref{fig:speedchangeproblem}, the numerical solution $y_h^\star,u_h^\star$ is rather inaccurate. This is so because in this example $h=1$ is relatively large. We hope that for smaller values of $h$ the numerical solution $y^\star_h,u^\star_h$ converges to the exact minimizer $y^\star,u^\star$. The following observation motivates this.

The exact control solution to problem~\ref{eqn:exampleOCP} is $u^\star(t)=4.7129... \cdot \sinh(\sqrt 2 \cdot t)- 3.3349... \cdot \cosh(\sqrt 2 \cdot t)$, hence $y_h^\star$ should solve the initial value problem~\eqref{eqn:exampleIvp} with $y_\text{ini}=10$. The explicit Euler solutions in  Figure~\ref{fig:exampleeuleroptim}~(a) become more accurate as $h$ decreases. Likewise, the intention behind direct transcription is that the numerical minimizer $y^\star_h,u^\star_h$ converges to the exact minimizer $y^\star,u^\star$ as $h \rightarrow 0$.

\section{Scope of this Thesis}\label{sec:Scope}
Figure~\ref{fig:frameworkdirecttranscriptionv2} summarizes the previous section: Optimal control problems have exact minimizers that are often impractical to compute analytically. To find a solution numerically, we use a direct transcription method. This gives us an NLP, whose solution we can practically compute. The NLP's solution gives us a numerical minimizer, which hopefully resembles the exact minimizer accurately. Different transcription methods result in different accuracies, and some methods do sometimes fail. Researchers and practitioners care about guarantees that can be given on the accuracy of a numerical solution for a given direct transcription method. Hence, this thesis is ultimately about the proposal of one particular direct transcription method that guarantees good accuracy under a large number of circumstances.
\begin{figure}[tb]
	\centering
	\includegraphics[width=0.85\linewidth]{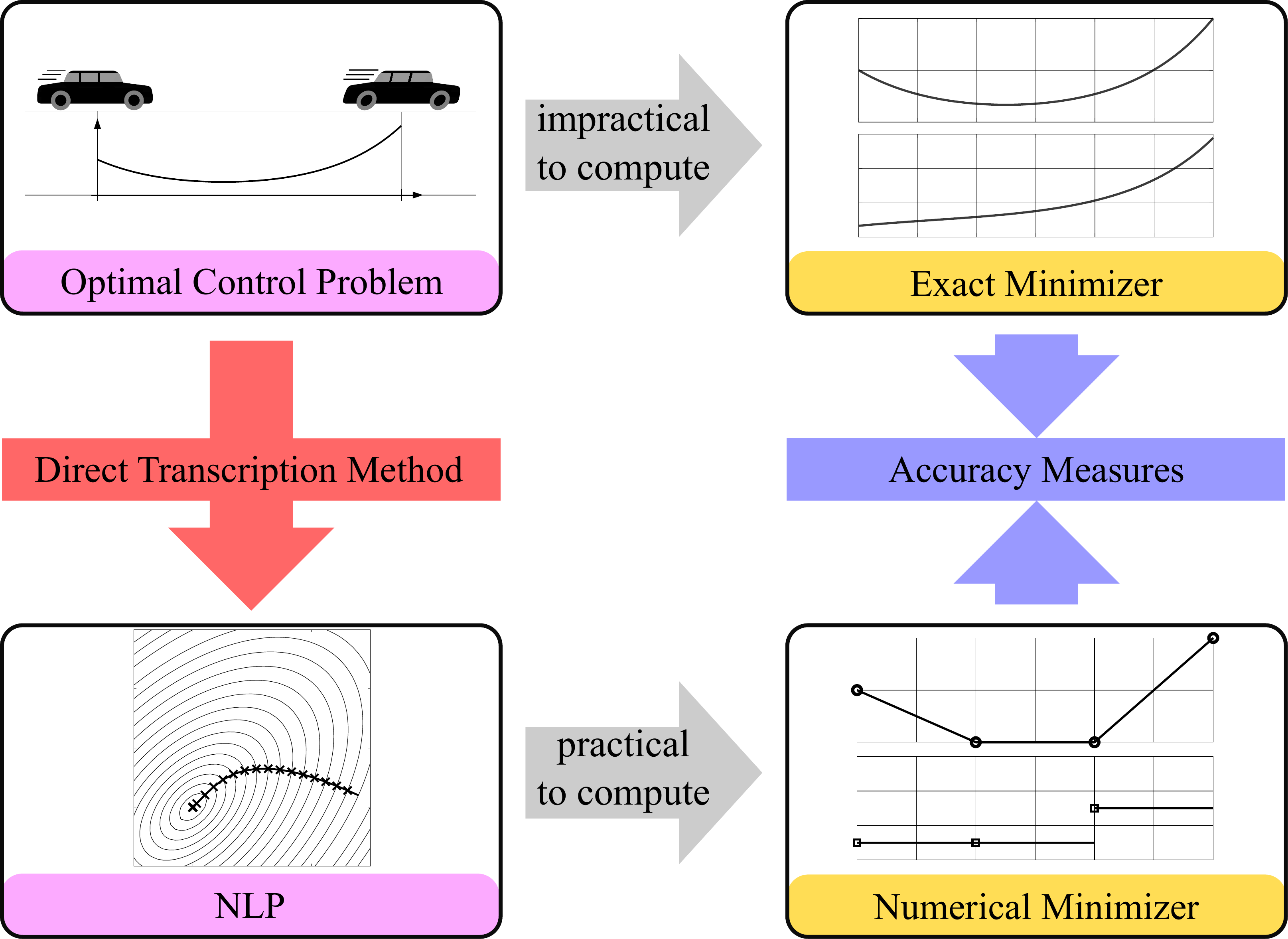}
	\caption{Flow chart of numerically solving optimal control problems via direct transcription: The direct transcription approximates the optimal control problem with an NLP. The solution of the NLP is used as an approximate solution of the optimal control problem. The accuracy of a numerical minimizer in comparison to an exact minimizer can be quantified using various measures. We introduce measures in Section~\ref{sec:ConvMeasures}.}
	\label{fig:frameworkdirecttranscriptionv2}
\end{figure}

The direct transcription method that we propose is a quadrature penalty method, whereas the current state of the art in the literature are collocation methods. This is why the scope of this thesis is to great extent on these two classes of direct transcription methods: collocation, detailed in Section~\ref{sec:DirectCollocation}; and quadrature penalties, detailed in Section~\ref{sec:DQIPM}.

Quadrature penalty methods and collocation methods use different principles to accomplish the same tasks. These are:
\begin{enumerate}[(a)]
	\item \emph{Approximating} the time-dependent functions $y,\,u$. As depicted in Figure~\ref{fig:speedchangeproblem}, the explicit Euler method uses piecewise linear functions for $y$ and piecewise constant functions for $u$.
	\item \emph{Relaxing} the differential equations that were actually supposed to be satisfied $\tforall\ t \in [0,T]$. As apparent from the NLP~\eqref{eqn:ExampleEuler:NLP} and Figure~\ref{fig:exampleeuleroptim}~(a), the explicit Euler solution $y^\star_h,\,u^\star_h$ is relaxed in the sense that it satisfies the differential equation only at the three points $t=0$, $t=1$, and $t=2$. The relaxation is necessary because the solution to \eqref{eqn:exampleODE} is not piecewise linear.
\end{enumerate}

Figure~\ref{fig:methodclasses} illustrates the tasks of (a)~approximation and (b)~relaxation. The relaxation uses (i)~a number of points in which the differential equations are solved to (ii)~a certain accuracy. We will come back to this diagram in Section~\ref{sec:future:QuadPen}.

\subsection{Direct Collocation Methods}
Most direct transcription methods are generalizations of the explicit Euler method: They first (a)~approximate $y,u$ with piecewise polynomial functions instead of piecewise linear functions. They then (b)~relax the differential equations to be satisfied at only a (i)~finite number of points $t$ in which they are solved (ii)~exactly.

\begin{figure}
	\centering
	\includegraphics[width=1\linewidth]{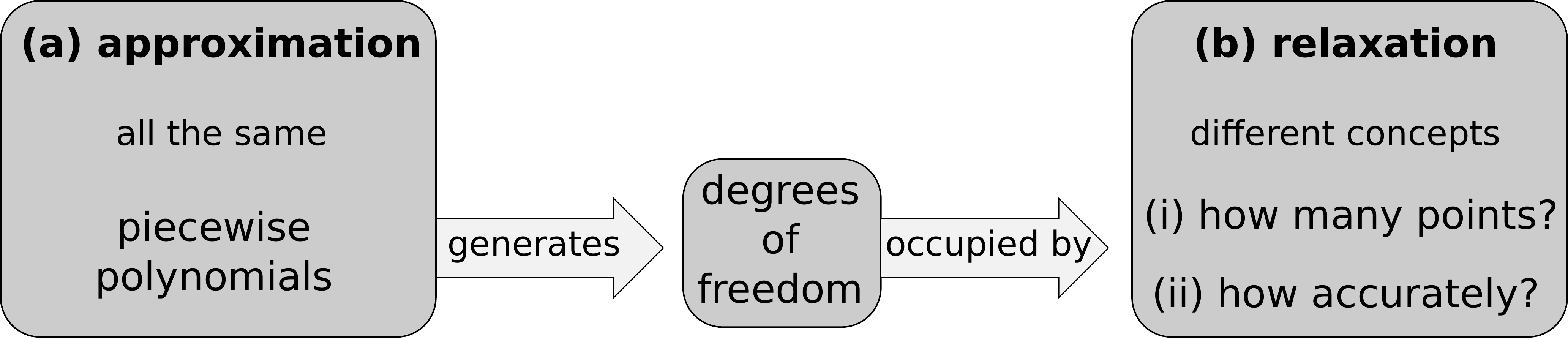}
	\caption{Tasks of approximation and relaxation. The approximation typically uses piecewise polynomials. These generate degrees of freedom. These degrees of freedom are occupied by the relaxed differential equations. Direct collocation methods vary in (i)~the number of points and (ii)~the accuracy in these points.}
	\label{fig:methodclasses}
\end{figure}

The principle of solving differential equations exactly at a finite number of points is called \emph{collocation}. This principle is used in the explicit \cite{Maurer} and implicit Euler methods \cite{MR4046772}, in Runge-Kutta methods \cite{Hager2000}, in multi-step methods \cite{MR500418}, in single \cite{MR1990061} and multiple shooting \cite{BOCK1984} methods, and in a class of methods that is called collocation methods \cite{BettsChap2}. An overview of various numerical methods for optimal control is given in Section~\ref{sec:literatureReview}.

Today, \emph{direct collocation methods} are considered as the state of the art for the numerical solution of optimal control problems. They will be presented in detail in Section~\ref{sec:DirectCollocation}.

As we discussed, the intention behind direct transcription methods is that the numerical minimizer $y^\star_h,u^\star_h$ converges to the exact minimizer $y^\star,u^\star$ as $h \rightarrow 0$. Unfortunately, there are various practical kinds of optimal control problems for which the principle of collocation does not converge, including:
\begin{itemize}
	\item nonlinear optimal control problems;
	\item consistently over-determined optimal control problems;
	\item singular-arc optimal control problems.
\end{itemize}
The lack of convergence in collocation methods for these classes of problems is analyzed in Section~\ref{sec:Misconceptions} and in Section~\ref{sec:NumExp:Round2}.

\subsection{Quadrature Penalty Methods}\label{sec:future:QuadPen}
Coming back to Figure~\ref{fig:methodclasses}, all direct transcription methods (a)~approximate $y,u$ with piecewise polynomial functions, just like collocation methods. However, there are different classes of direct transcription methods with respect to the (b)~relaxation of the differential equations. These different classes can be located in a two-dimensional graph that answers the following two questions from Figure~\ref{fig:methodclasses}: (i)~At how many points should we satisfy the differential equations? (ii)~To which accuracy should we satisfy the differential equations at these points? The answers to these questions are opposed: Any approximation gives us a limited number of degrees of freedom. However, if we wish to satisfy the differential equations at more points then we cannot satisfy them as accurately at each of these points. Diagrams of meeting conflicting goals are called \emph{Pareto fronts}.

Figure~\ref{fig:paretocurveqpm} presents this Pareto front for the above two questions: The choice of an (a)~approximation determines a number of \emph{degrees of freedom}. For example, the explicit Euler method with $N=100$ has $100$ degrees of freedom (DOF). As the red curve shows, we have to decide how to invest these DOF. We can either satisfy the differential equations at (i)~only $100$ points but with (ii)~exact accuracy\ ---\,this is what collocation methods do. Alternatively, we can satisfy the differential equations at (i)~many more points but to (ii)~only moderate accuracy\ ---\,this is what the so-called \emph{quadrature penalty methods} do.

Quadrature penalty methods live between the extremes of collocation methods and \emph{integral penalty methods}. Integral penalty methods are theoretical methods that satisfy the differential equations at (i)~all points but to (ii)~only low accuracy. Collocation methods and integral penalty methods are extreme in the sense that they live at the ends of the Pareto front: Collocation methods solve equations \emph{exactly}, whereas integral penalty methods solve equations \emph{everywhere}.

Collocation methods may fail to converge, and integral penalty methods may be impractical (e.g., when the exact solution to the integral cannot be found). The promise with quadrature penalty methods is that they meet both ends: they do converge and they are practical. When choosing many quadrature points, the quadrature penalty method approximates an integral penalty method. When choosing only a few quadrature points, the quadrature penalty method resembles a collocation method. This theoretically interesting because we can now understand the two extreme classes of methods as one class and unify their mathematical analysis. The practical interest for these methods is in their efficiency, as documented in Section~\ref{sec:NumExp:Round1}, and in their reliability, as illustrated in Section~\ref{sec:NumExp:Round2}. Convergence of high order (hence numerical efficiency) and under mild assumptions (hence numerical reliability) for this method is proven in Part~\ref{part:convproof}.

\begin{figure}[!htbp]
	\centering
	\includegraphics[width=0.6\linewidth]{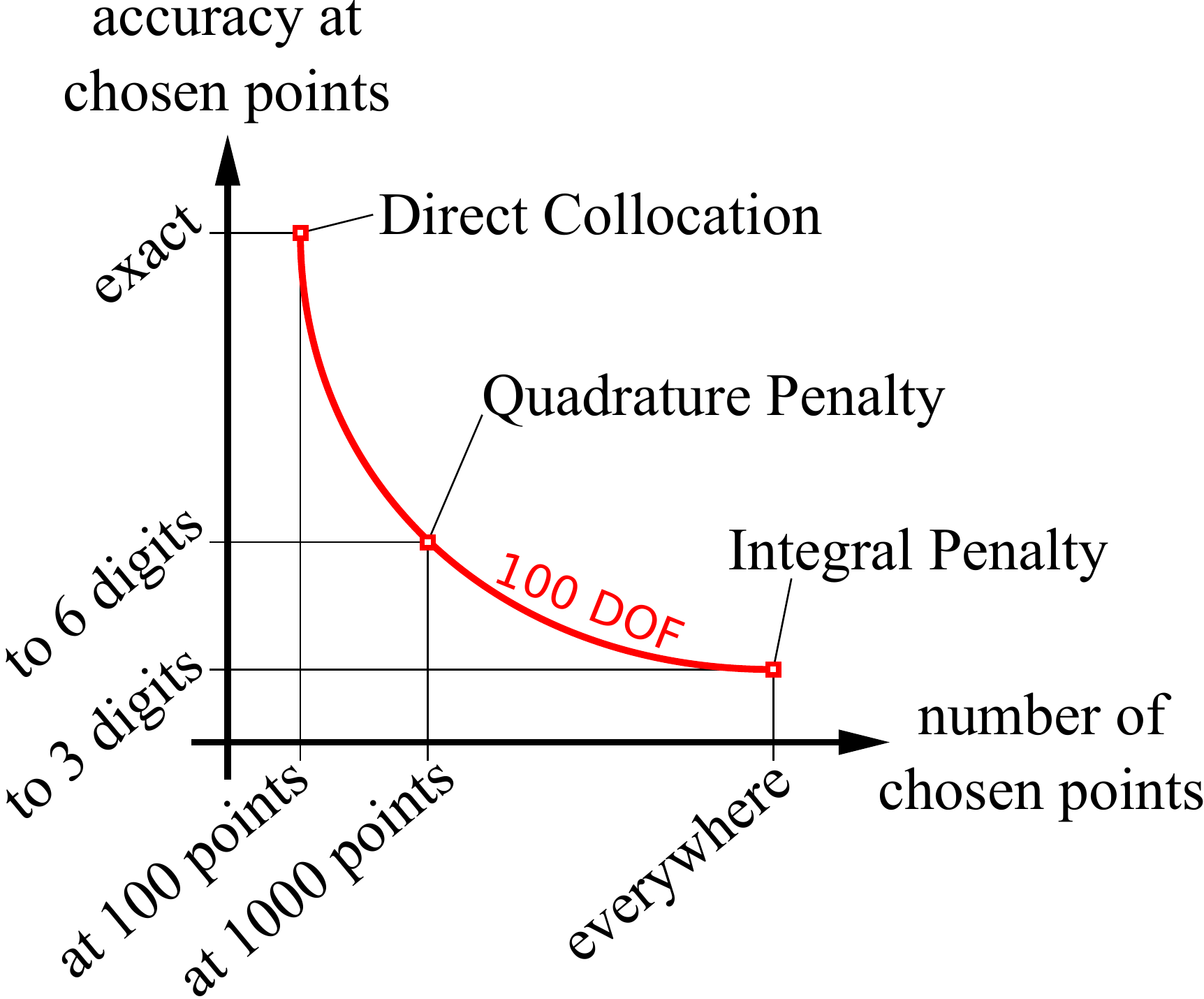}
	\caption{Pareto Front: For a given amount of degrees of freedom (DOF), it is possible to satisfy the equations with $f_1,f_2$ exactly at a finite number of points. Instead, when giving up on exactness, one can solve the equations approximately at a larger number of points. Quadrature penalty methods may be advantageous over direct collocation methods in the sense that they yield convergence of the integral residual under milder assumptions.}
	\label{fig:paretocurveqpm}
\end{figure}

\subsection{Contributions}
There are three main contributions in this thesis.

\subsubsection{\underline{Quadrature Penalty Direct Transcription Method}}
\paragraph{{Accessible Presentation}}
Our primary contribution is an accessible presentation of a reliable direct transcription method for the solution of a broad class of optimal control problems. This class are the \emph{direct quadrature penalty methods}. As depicted in Figure~\ref{fig:paretocurveqpm}, these methods generalize collocation methods and integral penalty methods. Part~\ref{part:QPM} presents these methods, illustrates their construction, and provides examples. We also discuss practical benefits in comparison to direct collocation methods. A wealth of numerical experiments demonstrate the practicality, performance, and robustness of quadrature penalty methods.

\paragraph{{Convergence Analysis}}
Our second contribution is a complete convergence analysis for quadrature penalty methods. This analysis is given in Part~\ref{part:convproof} and is based on only three mild assumptions, given in Section~\ref{sec:assumptions}. These are the mildest assumptions to date for a convergence proof of direct transcription methods, as confirmed through our extensive literature review.

\paragraph{{Numerical Demonstration}}
Our third contribution is the detailed illustration and discussion of optimal control examples where popular direct transcription methods struggle. These examples are given in Section~\ref{sec:Misconceptions} and Section~\ref{sec:NumExp:Round2}. These examples help motivate quadrature penalty methods and understand how they gain benefits in reliability. Prior to the presentation of quadrature penalty methods, we provide in Part~\ref{part:collocation} the necessary background on optimal control problems, on literature results, the context of different method classes and their relationships, in an accessible way.

\subsubsection{\underline{Modified Augmented Lagrangian}}
The direct quadratic penalty transcriptions, that this thesis is focused on, result in NLP that have large quadratic penalty terms. In order to minimize these special NLP reliably in practice, a modification of the augmented Lagrangian method is proposed in Chapter~\ref{chap:MALM} that works by solving a sequence of moderate penalty problems. Each of the moderate penalty problems can be solved reliably with general-purpose numerical optimization software. In contrast, solving one optimization problem with a large penalty term may cause general-purpose software to fail.

\subsubsection{\underline{Integral Penalty-Barrier Direct Transcription Method}}
Large sparse NLPs are typically solved with interior-point methods (IPMs). More accurate transcription result in larger NLPs. It is often observed that larger NLPs require more computation cycles (i.e., iterations of the IPM) to converge. The penalty-barrier method seeks to avoid this effect: The penalty functional of the quadratic penalty direct transcription is replaced with a penalty-\emph{barrier} functional. Minimizing said functional directly with a theoretical Newton/gradient-descent-type algorithm would result in a theoretical number $N$ of IPM iterations. Practical software on a transcription of finite dimensions may hopefully converge to this finite number $N$ of iterations as the dimension of the NLP grows. Again, as for the quadratic penalty method, the penalty-barrier method is presented in Chapter~\ref{chap:PBF_SICON} with a full proof of convergence and order of convergence that work under still relatively mild assumptions.

\chapter{The Numerical Optimal Control Problem}\label{sec:Formalities}
This chapter provides a mathematically precise format for optimal control problems and introduces the necessary function spaces. Afterwards, in alignment with Figure~\ref{fig:frameworkdirecttranscriptionv2}, we discuss suitable measures for the assessment of the accuracy of a numerical solution.

\section{Format of Optimal Control Problems}
The electric vehicle problem~\eqref{eqn:exampleOCP} was just one very particular example of an optimal control problem. This section describes a more general format for optimal control problems. We first explain the solution format and then the problem format.

\subsection{Solution Format}\label{sec:SolutionFormat}
The solution of an optimal control problem consists of two functions $y,u$. They are called \emph{states} and \emph{controls}, respectively. They must be of the following format:
\begin{align*}
y& \,: [0,T] \rightarrow \R^{n_y}\,,\ t \mapsto y(t)\,;\qquad & u& \,: [0,T] \rightarrow \R^{n_u}\,,\ t \mapsto u(t)\,.
\end{align*}
The integers $n_y,n_u\in \N$ are the number of states and number of controls, respectively. In the electric vehicle problem~\eqref{eqn:exampleOCP}, we had only $n_y=1$ state and $n_u=1$ control. This gives an informal description of the candidate space $\cX$ so that $(y,u) \in \cX$. A formal definition is in terms of Sobolev spaces and is given in Section~\ref{sec:SoboDef}.

\paragraph{Properties of States and Controls}
In an optimal control problem, the state trajectories may have edges (in particular: an countable infinite number of discontinuities in the first derivative, as formalized in Section~\ref{sec:CandidateSpaces} below via the Lebesgue measure), but they must be continuous everywhere. This is so because differential equations like \eqref{eqn:exampleODE} demand that states evolve continuously over time. In contrast, controls may have discontinuities. This is to enable sudden control actions. An example of a discontinuous control is given in the lower right of Figure~\ref{fig:speedchangeproblem}.

\subsection{Problem Format}\label{sec:ProbStat}
Many optimal control problems can be posed in Bolza form with fixed initial and final time \cite{BettsChap2}. In favour of less notation and ease of analysis, this thesis considers control problems in the following form:
\begin{equation}
\label{eqn:OCP}
\left\lbrace
\begin{aligned}
&\operatornamewithlimits{min}_{(y,u) \in \cX}  			& M \big(\,y(0),y(T)\,\big)&			& & \\[10pt]
& \text{subject to} & b\big(\,y(0),y(T)\,\big)  		&= \bO\,,\\[10pt]
& 					& f_1\big(\,y(t),u(t),t\,\big)  	&=\dot{y}(t)   		&\tforall\ &t \in [0,T]\,,\\
& 					& f_2\big(\,y(t),u(t),t\,\big) 		&=\bO 				&\tforall\ &t \in [0,T]\,,\\[10pt]
& 					& \yL(t) \leq y(t)  				&\leq \yR(t)\quad 	&\forall\ &t \in [0,T]\,,\\
& 					& \uL(t) \leq u(t)  				&\leq \uR(t)\quad 	&\forall\ &t \in [0,T]\,.
\end{aligned}
\right\rbrace\quad \begin{matrix}
\vspace{0mm}\\[-8pt]
(\ref{eqn:OCP}\text{:M})\\[16pt]
(\ref{eqn:OCP}\text{:b})\\[14pt]
(\ref{eqn:OCP}\text{:f1})\\[6pt]
(\ref{eqn:OCP}\text{:f2})\\[14pt]
(\ref{eqn:OCP}\text{:y})\\[6pt]
(\ref{eqn:OCP}\text{:u})
\end{matrix}
\end{equation}

Problem format~\eqref{eqn:OCP} consists of four row-blocks, separated by vertical margins. In the first block, (\ref{eqn:OCP}:M) states the objective $M$. In the second block, (\ref{eqn:OCP}:b) holds the boundary conditions $b$. The third row-block consists of differential equations (\ref{eqn:OCP}:f1) and algebraic equations (\ref{eqn:OCP}:f2). The last block expresses left and right bound constraints on the states (\ref{eqn:OCP}:y) and on the controls (\ref{eqn:OCP}:u).

\begin{remark}
	The Bolza form has an additional so-called \emph{Lagrange term}, i.e., an integral term in the objective. We saw such a term in example~\eqref{eqn:exampleOCP}, where the objective was not only on $y(0),y(T)$ but on an integral over $y,u$. Such a term is missing in (\ref{eqn:OCP}:M) for simplicity. A Lagrange term may be augmented to $M$ via the use of numerical quadrature (cf. first line in~\eqref{eqn:ExampleEuler:NLP}) or via the techniques described in \cite{BettsChap2}. Algebraic inequality constraints can be embedded into (\ref{eqn:OCP}:f2) and (\ref{eqn:OCP}:u) via use of so-called \emph{slack controls} \cite{BettsChap2}.
\end{remark}
\begin{remark}
	The differential and algebraic constraints may be violated at some points due to possible edges in $y$ and jumps in $u$. This is formally taken care of by the $\tforall$ notation. In contrast, for the bound constraints we can opt for the different notation $\forall\ t \in [0,T]$. This helps avoiding ambiguity with the error measure $\gamma$ in Section~\ref{sec:ConvMeasures}.
\end{remark}
\begin{remark}
	There exist numerical algorithms and analyses for optimal control problems of more general formats. Some formats include partial differential equations, unknown parameters, and random variables in the boundary conditions or differential and algebraic constraints. The format of the problem can have dramatic impacts on the efficiency and thus suitable choice of a numerical algorithm. Therefor, any problem statement that cannot be expressed in terms of \eqref{eqn:OCP} is beyond the scope of this thesis.
\end{remark}

\paragraph{Parameters and Functions in the Problem Format}
In problem format~\eqref{eqn:OCP}, the numbers $T \in \R_{>0}$, $n_y,n_u,n_c,n_b \in \N$ are given parameters and
\begin{align*}
M & : \R^{n_y} \times \R^{n_y} \rightarrow \R\,;\quad &
b & : \R^{n_y} \times \R^{n_y} \rightarrow \R^{n_b}\,;\\
f_1 & : \R^{n_y} \times \R^{n_u} \times [0,T] \rightarrow \R^{n_y}\,;\quad &
f_2 & : \R^{n_y} \times \R^{n_u} \times [0,T] \rightarrow \R^{n_c}\,;\\
\yL & : [0,T] \rightarrow \R^{n_y}\,;\quad &
\yR & : [0,T] \rightarrow \R^{n_y}\,;\\
\uL & : [0,T] \rightarrow \R^{n_u}\,;\quad &
\uR & : [0,T] \rightarrow \R^{n_u}
\end{align*}
are given functions that must satisfy $\yL(t)\leq\yR(t)$ and $\uL(t)\leq\uR(t)$ $\forall t \in [0,T]$, where $\leq$ is meant for each vector component.

In~\eqref{eqn:OCP}, the differential and algebraic equations have been split into two functions $f_1,f_2$ because this helps theoretical analysis. In other cases it is more handy instead to merge $f_1,f_2$ into the following notation:
\begin{align}
f\big(\dot{y}(t),y(t),u(t),t\big):=\begin{bmatrix}
f_1\big(y(t),u(t),t\big){}{}-{}\dot{y}(t)\\
f_2\big(y(t),u(t),t\big){}\phantom{{}-{}\dot{y}(t)}
\end{bmatrix}\,.\label{eqn:def:f}
\end{align}
Hence, \mbox{(\ref{eqn:OCP}:f1)--(\ref{eqn:OCP}:f2)} can be written equivalently as
\begin{align}
f\big(\dot{y}(t),y(t),u(t),t\big)=\bO \qquad \tforall\ t \in [0,T]\,. \tag{\ref{eqn:OCP}:f}\label{eqn:OCP:dae}
\end{align}

\subsection{Existence and Uniqueness of Solutions}
Typical literature in numerical mathematics is structured in three steps: First, there is a section to present the problem statement. Second, there is a section verifying the ``well-posedness'' of the stated problem. Finally, a numerical method is presented and its convergence is analyzed. Well-posedness means that the solution to a problem changes mildly when the problem-defining functions and data are perturbed.

Basically, if a problem is not well-posed then it cannot be treated numerically in a meaningful way. This is because the numerical method, as well as the floating-point arithmetic, by itself induce perturbations that could radically change the solution unless the problem is well-posed. Nonetheless, the subject of well-posedness is generally a part of the convergence analysis: Namely, if a numerical method converges under particular assumptions then these assumptions must imply the well-posedness of the problem.

For the problem \eqref{eqn:OCP}, which is the central problem statement of this thesis, it is clear that a solution might not exist. For instance, consider
\begin{align*}
\min y(0)\quad \text{subject to }\ y(0)=0\,,\quad y(1)=1\,,\quad \dot{y}(t)=0\,.
\end{align*}
Especially the theoretical realm of optimal control literature therefor considers rather restrictive problem statements, in order to enable the existence of clear conditions under which solutions exist and are well-posed. However, this has the disadvantage that practical optimal control problems are tedious to fit into the format of these literature's problem statements.

As mentioned before, the analysis on well-posedness is implied in the analysis on convergence. For the quadrature penalty method proposed in this thesis, we obtain convergence to a local solution $y^\star,u^\star$ when:
\begin{itemize}
	\item problem \eqref{eqn:OCP} has a global infimum;
	\item $y^\star,u^\star$ is feasible (needless to say);
	\item $M,f_1,f_2,b$ are point-wise local H\"older continuous in a neighborhood of $y^\star,u^\star$; and
	\item $\yL,\yR,\uL,\uR$ are bounded.
\end{itemize}
This does not mean that under these conditions $y^\star,u^\star$ is well-posed, i.e. would not change dramatically under small perturbations of $M,f_1,f_2,b$. Rather, it means that the convergence measures (discussed below) of the numerical solution will not be affected significantly under small perturbations of $M,f_1,f_2,b$ --unless the four above conditions are not all met.

Essentially, in this thesis we wanted to avoid the discussion of well-posedness altogether because we consider it unhelpful for any problem in general (not limited to optimal control): The real world dictates problems in a natural problem statement. Numerics must make sense of these problem statements regardless; not by restricting them but by devising suitable metrics that are well-posed under all possible practical scenarios. The four bulleted conditions above are not a restriction but merely a characterization of what a practical scenario is. Therefor, no compromises had to be made in order to arrive at a problem statement that meets the two goals of generality and numerical practicality. The only peculiarity is now that we do not use the difference between exact and numerical solution as convergence metric because indeed this would usually result in ill-posedness. Instead, we are going to measure the optimality gap and a residual measure of the constraint violations.

\section{Function Spaces in Optimal Control}\label{sec:SoboDef}
In the terminology of Figure~\ref{fig:optimizationterminology}, the description of properties of states and controls from Section~\ref{sec:SolutionFormat} above is an informal characterization of the candidate space $\cX$. Formal definitions of candidate spaces in the literature use either H\"older spaces or the so-called Sobolev spaces. Sobolev spaces are based on Lebesgue spaces. For accessibility and because definitions of Lebesgue and Sobolev norms vary in the literature, this section reviews H\"older spaces, Lebesgue spaces, and Sobolev spaces. At the end we give a formal definition of the candidate space $\cX$ that we use in \eqref{eqn:OCP}.

\subsection{H\"older Spaces}
H\"older spaces, Lebesgue spaces, and Sobolev spaces are \emph{function spaces}; i.e., spaces that contain functions. We care in particular about functions of the form
\begin{align*}
z\,:\ [0,T] \rightarrow \R^{n_z}\,,
\end{align*}
for some dimension $n_z \in \N$. The notation $z \in \cC^k$ may be familiar to express that $z$ be $k \in \N_0$ times \emph{continuously differentiable}.

A special form of continuity is \emph{H\"older continuity}. We say $z$ is $\lambda$-H\"older continuous of constant $\lambda \in (0,1]$ on the interval $I \subset \R$ if there is a constant $L \in \R_{>0}$ such that
\begin{align*}
\|z(\hat{t})-z(\check{t})\|_2 \leq L \cdot |\hat{t}-\check{t}|^\lambda \qquad \forall\ \hat{t},\check{t} \in I\,.
\end{align*}
We use the typical notation $\cC^{k,\lambda}(I)$ for the space of functions $z$ whose first $k \in \N_0$ derivatives are $\lambda$-H\"older continuous on $I$. The special case $\cC^{0,1}(I)$ is the space of functions $z$ that are \emph{Lipschitz continuous}, i.e. when $\lambda=1$. We use the short-hand $\cC^{k,\lambda}$ for $\cC^{k,\lambda}([0,T])$. In general, larger values of $\lambda$ mean smoother functions. Hence, H\"older continuity is a milder condition than Lipschitz continuity.

\subsection{Lebesgue Spaces}
H\"older spaces are characterized by values of the function $z$ at individual points $t \in [0,T]$. Lebesgue spaces and Sobolev spaces are fundamentally different. They only depend on integral-norms over $z$. We introduce \emph{Lebesgue spaces} $L^d$ for $d \in \N \cup \lbrace \infty \rbrace$ as the spaces of all functions $z$ that are bounded in the norm
\begin{align*}
\|z\|_{L^d} &:= \sqrt[\leftroot{-2}\uproot{2}d]{\int_0^T \|z(t)\|_d^d\,\mathrm{d}t\,}\,, \label{eqn:LpNorm}
\end{align*}
where $\|\bv\|_d:=\sqrt[d]{\sum_{i=1}^{n} |v_{[i]}|^d\,}$ for vectors $\bv=(v_{[1]},\dots,v_{[n]}) \in \R^n$.
In the limit $d=\infty$, the norms are
\begin{align*}
\|z\|_{L^\infty} &:= \operatornamewithlimits{ess\,sup}_{t \in [0,T]} \|z(t)\|_\infty\,,\\
\|\bv\|_{\infty} &:= \operatornamewithlimits{max}_{i \in \lbrace 1,\dots,n\rbrace } |v_{[i]}|\,,\\
\end{align*}
Like the supremum, the essential supremum always exists \cite[p.~172]{essSup}. The essential supremum induces the notion of \emph{essential boundedness}. For instance, the indicator function $\Xi_\Q$ is essentially bounded by zero. The space $L^2$ is a Hilbert space with scalar product
\begin{align*}
\langle w,z \rangle := \int_0^T w(t) \cdot z(t) \,\mathrm{d}t\,.
\end{align*}

\subsection{Sobolev Spaces}
We introduce \emph{Sobolev spaces} $W^{k,d}$ as the spaces of all $k$ times weakly differentiable functions $z$ that are bounded in the norm
\begin{align}
\|z\|_{W^{k,d}} := \sum_{j=0}^k \left\| \frac{\mathrm{d}^jz}{\mathrm{d}t^j} \right\|_{L^d}\,.
\end{align}
Therein, $\frac{\mathrm{d}^jz}{\mathrm{d}t^j}$ denotes the $j^{\text{th}}$ \emph{weak derivative} \cite{Brezis} of $z$.\footnote{For readers unfamiliar with a weak derivative, for simplicity, we can pretend that it is the conventional derivative and still follow the conceptual ideas. No emphasis on weak versus strong derivative is needed for the remainder of this thesis.} We denote with $\dot{z}$ the first weak derivative of $z$. Hence, Sobolev spaces extend the idea of Lebesgue spaces to derivatives. 

\subsection{Candidate Space}\label{sec:CandidateSpaces}
\paragraph{In the Literature}
Some literature use $\cX = \cC^{1} \times \cC^0$ \cite{MR3983447,Maurer}, i.e., $y$ continuously differentiable and $u$ continuous. Other choices are $\cX = W^{1,\infty} \times L^\infty$ \cite{Hager2000,SchwartzPolak:1996} or $\cX = W^{2,\infty} \times W^{1,\infty}$ \cite{MR4046772,MR1770350} or $\cX = W^{2,\infty} \times \cC^0$ \cite{gongEtAl:2006,gongEtAl:2008}. All of these choices imply $\|\dot{y}\|_{L^\infty}<\infty$, i.e., require $y$ to be Lipschitz-continuous. A less restrictive choice is $\cX = W^{1,2} \times L^2$ in \cite{DeJulio70,MR0271512}, only requiring $\dot{y} \in L^2$.

\paragraph{Our Definition}
The choice $\cX = W^{1,2} \times L^2$ implies $y \in L^\infty$ \cite[Thm.~8.8]{Brezis}, however $u \in L^2$ may not be essentially bounded. This complicates the assumptions on the growth of $f_1,f_2$, because $u$ may be unbounded at some points $t \in [0,T]$. To avoid this issue, we make use of the following space:
\begin{align*}
\cX := W^{1,2} \times L^\infty\,,
\end{align*}
meaning that $y$ and $u$ are essentially bounded. We can hence formalize or space for $\cX$ equivalently in the following more insightful form:
\begin{align}
\cX = \Big\lbrace\ (y,u)\ \Big\vert\ y \in L^\infty,\ u \in L^\infty,\ \dot{y} \in L^2\ \Big\rbrace\,. \label{eqn:defX}
\end{align}
In preparation for the convergence analysis, we also define the norm
\begin{align}
\|(y,u)\|_\cX := \|\dot{y}\|_{L^2} + \|(y,u)\|_{L^\infty}\,, \label{eqn:normX}
\end{align}
where we mean the $L^\infty$-norms for elements in $\cX$ by stacking $y,u$ into a vector:
\begin{align}
\|(y,u)\|_{L^\infty} &= \operatornamewithlimits{ess\,sup}_{t \in [0,T]} \left\|\begin{bmatrix}
y(t)\\
u(t)
\end{bmatrix}\right\|_\infty\label{eqn:normX:inf}\,.
\end{align}

\section{Accuracy Measures}\label{sec:ConvMeasures}
Figure~\ref{fig:frameworkdirecttranscriptionv2} presents the flow chart of the conceptual idea behind direct transcription: By approximating the optimal control problem with an NLP, we obtain an approximate numerical minimizer $y_h^\star,u_h^\star$ to an exact minimizer $(y^\star,u^\star) \in \cX^\star$. Accuracy measures can be used to quantify how good the numerical minimizer is in comparison to the exact one.

\paragraph{Connection between Optimality Conditions and Error Measures}
A conventional analysis of optimal control would pose \eqref{eqn:OCP} in a format that is well-posed under some assumptions. The analysis would then establish a set of equations that are necessarily satisfied by every local solution of the optimal control problem. These equations are called Pontryagin equations (cf. Section~\ref{sec:literatureReview}). A numerical method is then devised to solve these equations approximately. The accuracy of this approximation is quantified and a bound of the deviation between numerical and exact solution is derived.

This thesis takes an entirely different approach. We do not use optimality conditions for the optimal control problem. Hence, we also do not care how the solution accuracy to such optimality conditions relates to the solution accuracy with respect to the optimal control problem. Furthermore, we do not use an error as a metric but instead we use a gap. This is explained and motivated in the following.

\subsection{Motivation of Gap Measures}
For optimization problems in general, we can use either of two metrics when measuring the accuracy of a numerical minimizer:
\begin{enumerate}
	\item Error: the distance between the numerical and the exact minimizer.
	\item Gap: the difference between the objective values of the numerical and the exact minimizer.
\end{enumerate}
In optimization problems, and thus also in optimal control problems, the error does not necessarily converge. Hence, the gap may be the only suitable measure. To see why this is so, we provide some illustrations.

\paragraph{Illustration of Minimizers}
Minimizers can be strict, non-strict, or unique \cite[p.~13~\&~Lem~4.7~ii]{Nocedal}. Figure~\ref{fig:minimizertypes} illustrates general challenges with minimizers of optimization problems: (a) Usually, problems can have several different exact minimizers. These can have different properties. (b) For instance, minimizers are called non-strict when the objective has a plateau. In the depicted example, all the points on the blue interval are local minimizers. (c) The existence of a unique minimizer can only be asserted under special circumstances, such as strict convexity. For instance, a \emph{sufficient} condition for convexity in optimal control is called coercivity \cite{Maurer,Hager2000,RaoHager2018}.

Most often, numerical algorithms can only compute local minimizers. As illustrated in (d), the numerical minimizer will converge to an exact local minimizer when that local minimizer is strict; thus the error converges. In contrast, in the case (e) of a non-strict minimizer, only the gap converges.

\paragraph{Consequences for Measurement}
Because the error may not converge, we opted for the term \emph{accuracy measures} instead of convergence measures. We will only consider the one-sided gap metric in (f): For a given exact minimizer $\bx^\star$ under consideration, we measure the value of $\delta \in \R_{\geq 0}$ such that the objective of the numerical minimizer is bounded by $\bf(\bx^\star)+\delta$. The possible locations of all sufficiently optimal numerical minimizers with respect to $\bx^\star$ and $\delta$ are depicted in violet. Nothing prevents the numerical minimizer to converge to an even smaller local minimum than $\bf(\bx^\star)$, which is beneficial.

\begin{figure}
	\centering
	\includegraphics[width=1.0\linewidth]{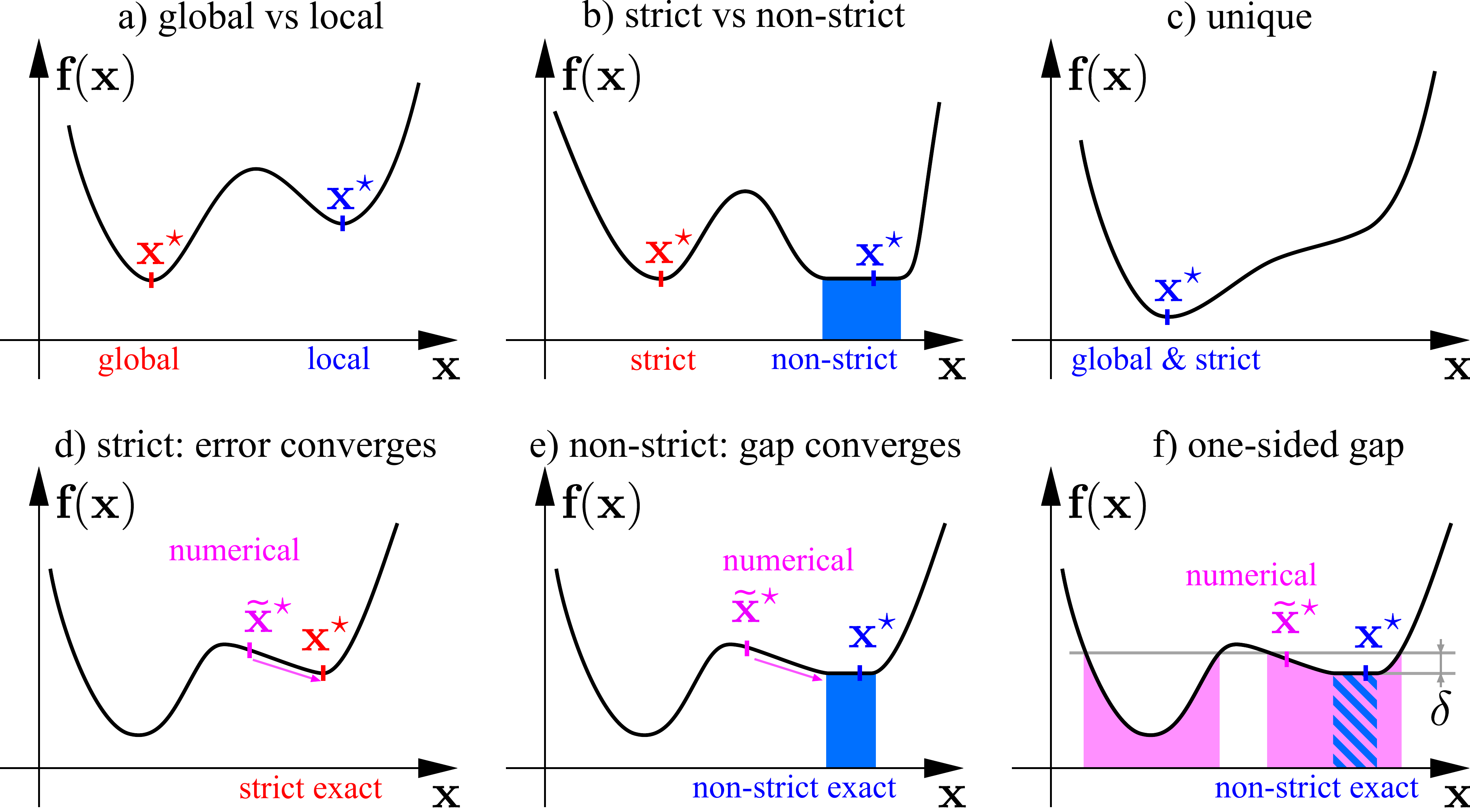}
	\caption{ (a)--(c): Types of exact minimizers. (d)--(f): Numerical minimizers vs exact minimizers.}
	\label{fig:minimizertypes}
\end{figure}

\subsection{Accuracy of Optimality}
An exact minimizer $y^\star,u^\star$ of problem~\eqref{eqn:OCP} is at least a local minimizer of the objective $M$. Hence, $y^\star_h,u^\star_h$ should at least be an approximate local minimizer. Hence, in accordance to Figure~\ref{fig:minimizertypes}~(f), the relation
\begin{align}
M\big(y^\star_h(0),y^\star_h(T)\big) \leq M\big(y^\star(0),y^\star(T)\big) + \delta \label{eqn:meas:delta}
\end{align}
should hold for a small \emph{optimality gap} $\delta \in \R_{\geq 0}$.

\subsection{Accuracy of Feasibility to Equality Constraints}
Consider the measure $r$, that takes two functions $y,u$ and maps them onto a non-negative value:
\begin{align}
r(y,u):=\sqrt{ \int_0^T \left\|f\big(\dot{y}(t),y(t),u(t),t\big)\right\|_2^2\,\mathrm{d}t + \left\|b\big(y(0),y(T)\big) \right\|_2^2 }\,, \label{eqn:IntRes}
\end{align}
where $\|\cdot\|_2$ means the Euclidean norm of a vector in $\R^n$. Thus, $r$ quantifies: i) how well $y,u$ satisfy all the boundary constraints with $b$ at $t=0$ and $t=T$; ii) how well $y,u$ satisfy all the differential equations with $f_1$ and algebraic equations with $f_2$ for almost every $t \in [0,T]$, in the sense of $\tforall$. This is in analogy to how \eqref{eqn:LebesgueIntegral} is equivalent to $\Xi_\Q(t)=0\ \tforall\,t\in\R$.

A solution $y^\star,u^\star$ of problem~\eqref{eqn:OCP} satisfies the constraints exactly because it is feasible. The numerical solution should hence be at least approximately feasible. Hence, the relation
\begin{align}
r(y^\star_h,u^\star_h) \leq \rho \label{eqn:meas:rho}
\end{align}
should hold for a small \emph{equality residual} $\rho \in \R_{\geq 0}$.

\subsection{Accuracy of Feasibility to Inequality Constraints}
Of special importance to safety-critical applications, such as collision avoidance, is the property of numerical solutions to satisfy the left bounds $\yL,\uL$ and right bounds $\yR,\uR$ to high accuracy \cite{strictSatisfaction}. Hence, the relation
\begin{subequations}
	\label{eqn:boxFeas} \label{eqn:meas:gamma}%
	\begin{align}
	\yL(t)-\gamma \leq &y_h(t) \leq \yR(t)+\gamma \quad &\forall t&\in[0,T]\\
	\uL(t)-\gamma \leq &u_h(t) \leq \uR(t)+\gamma \quad &\forall t&\in[0,T]
	\end{align}
\end{subequations}
should hold for a small \emph{inequality residual} $\gamma \in \R_{\geq 0}$.

\subsection{Final Remarks}
In Part~\ref{part:QPM} we present a direct transcription method for which we can prove that all three above measures converge to zero under mild assumptions. Part~\ref{part:convproof} provides the convergence proof.

A proof under mild assumptions is only possible because of the way how we formulate the convergence measures: We use the gap instead of the error. In addition to the gap, we use residual measures to assess the feasibility of a numerical minimizer.

Regarding these residuals, $\rho$ is an integral measure, whereas $\gamma$ bounds violations in any point. This is a compromise for the numerics: We can, with very mild assumptions, assert that \eqref{eqn:boxFeas} is satisfiable numerically. In contrast, a property like $\left\|f\big(\dot{y}(t),y(t),u(t),t\big)\right\|_2\leq\rho$ $\tforall\, t \in [0,T]$ would necessitate stronger assumptions on the smoothness of $f$ and of $y^\star,u^\star$.

\part{Status Quo of Numerical Optimal Control} 
\label{part:collocation}

\chapter{Literature Review}
\label{sec:literatureReview}

Now that we have defined the problem statement and solution measures, it is natural to ask for the availability of numerical methods that are proven to converge for problem~\eqref{eqn:OCP} under mild assumptions. A tabular literature overview of available convergence results is given at the end of this section. This will motivate our proof in Part~\ref{part:convproof}, which works under favorable assumptions.

We first provide a broad {overview of concepts} and a {historical survey} of different developments. This is to answer two important questions: What is the state of research in the development of numerical algorithms for optimal control? Why do we only focus on direct transcription methods?

\section{Overview of Concepts}
\label{sec:literatureReview:Overview}

Optimal control problems have several structural {properties}. Each property suggests a certain solution {principle}. Principles can be combined. To better explain these properties and principles, for the scope of this overview, we consider the following model problem:
\begin{align}
J(t_0,t_E;y_0,y_E) :=& \operatornamewithlimits{min}_{y}\, \int_{t_0}^{t_E} g\big(\dot{y}(\tau),y(\tau)\big)\,\mathrm{d}\tau\ \text{ s.t. }y(t_0)=y_0\,,\ y(t_E)=y_E \label{eqn:modelOCP} %
\end{align}
In this model problem, we can suppose the control $u = \dot{y}$.

As far as known today, optimal control problems inherit four structural properties. These are depicted in Figure~\ref{fig:principlesvisual}. We discuss them in order.
\begin{figure}
	\centering
	\includegraphics[width=0.85\linewidth]{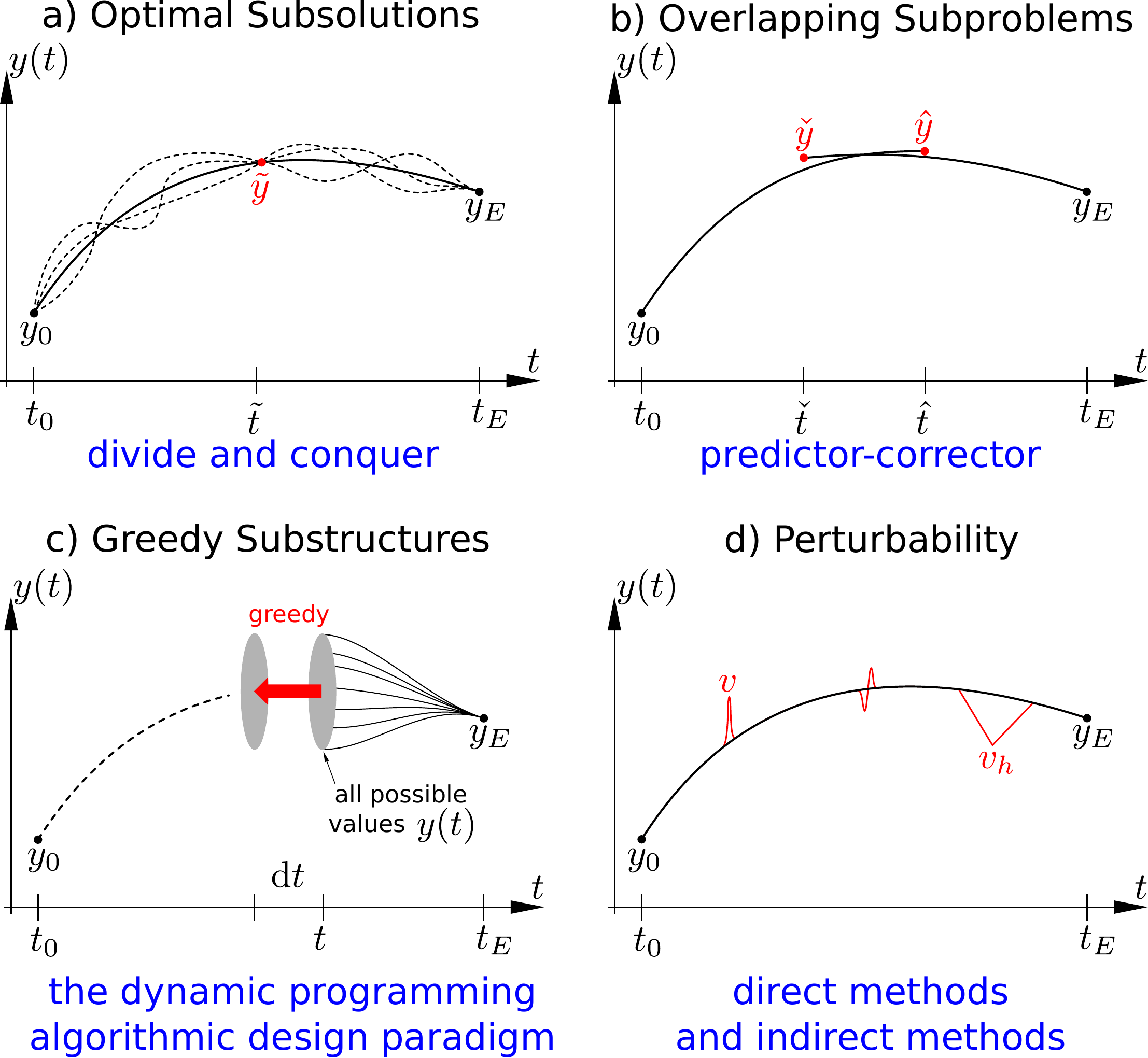}
	\caption{Four structural properties in optimal control problems (black) and solution principles (blue) that emerge from these properties.}
	\label{fig:principlesvisual}
\end{figure}

\subsection{Optimal Subsolutions}
Illustrated in Figure~\ref{fig:principlesvisual}~(a), it holds that
\begin{align*}
J(t_0,t_E;y_0,y_E) = \operatornamewithlimits{min}_{\tilde{y}} J(t_0,\tilde{t};y_0,\tilde{y}) +J(\tilde{t},t_E;\tilde{y},y_E)
\end{align*}
because every minimizer $y^\star$ of $J(t_0,t_E;y_0,y_E)$ also minimizes the integrals over the sub-intervals $[t_0,\tilde{t}]$ and $[\tilde{t},t_E]$. Figure~\ref{fig:principlesvisual}~(a) illustrates several non-optimal arcs in dashed besides the optimal arc in solid. The optimal arc over $[t_0,t_E]$ consists of the optimal arcs over $[t_0,\tilde{t}]$ and $[\tilde{t},t_E]$. Divide-and-conquer strategies use this property by dividing one optimal control problem into several smaller ones \cite{MR3514781}.

\begin{example}
	Direct transcription with the explicit Euler method divides $[0,T]$ into $N$ sub-intervals. $y$ is modelled piecewise linear on each sub-interval. The values at the red junction points are found by constrained optimization.
\end{example}

\subsection{Overlapping Subproblems}
As shown in Figure~\ref{fig:principlesvisual}~(b), solutions of optimal control problems overlap. Figure~\ref{fig:principlesvisual}~(b) plots the solutions of $J(t_0,\hat{t};y_0,\hat{y})$ and $J(\check{t},t_E;\check{y},y_E)$. If we move both $\check{y},\hat{y}$ down just a little bit then their arcs lie on top of each other and match the solution of $J(t_0,t_E;y_0,y_E)$. Such strategy is used for example in the Schwarz predictor-corrector algorithm \cite{DBLP_Schwarz}.

\subsection{Greedy Substructures}
Figure~\ref{fig:principlesvisual}~(c) illustrates a greedy principle. Optimal substructures over an infinitesimal time-interval are straight lines. Hence, they can be computed greedily by forcing $J(t-\mathrm{d}t,t;y(t-\mathrm{d}t),y(t))$ minimal. This provides a means for algorithms of the dynamic programming design paradigm: These propagate all optimal arcs $y(t)$ backwards in time \cite{MR61289}. Afterwards, they select that particular arc that satisfies $y(t_0)=y_0$ (dashed in the figure).

\begin{example}
	Bellman's value function \mbox{$V(t,x):=J(t,t_E;x,y_E)$} satisfies the hyperbolic end-value problem
	\begin{subequations}
		\label{eqn:BellmanEqn}
		\begin{align}
		V(t_E,x)&=V_E(x) \,,\\
		\partial_t V(t,x)&=-\operatornamewithlimits{min}_{u(t,x)}\Big\lbrace\, \partial_x V(t,x)\t \cdot u(t,x) + g\big(u(t,x),x\big) \,\Big\rbrace\quad\forall x \in \R^{n_y}\,,\ t \in [t_0,t_E].
		\end{align}
	\end{subequations}
	The end-conditions $V_E$ are a function of $g$ and $y_E$. In case there are no end-conditions $y_E$, there holds that $V_E(x)=0$. This partial differential equation is known as the Hamilton-Jacobi-Bellman equation \cite{MR61289,Conway2012}. Upon solution of $V(t,x)$, $y^\star$ can be found as the solution of the initial value problem $y(t_0)=y_0$, $\dot{y}(t)=u(\,t,y(t)\,)$.
\end{example}

\subsection{Perturbability}
Figure~\ref{fig:principlesvisual}~(d) shows an optimal solution in black and perturbations of it in red.
When a solution $y$ of $t$ is already optimal then the objective cannot be decreased any further via any of these perturbations. The red perturbation $v$ of $t$ shows an infinitesimal needle perturbation, whereas $v_h$ of $t$ shows a finite perturbation.

\vspace{2mm}\noindent
\underline{Direct Methods}\\
Perturbations of finite length like $v_h$ result in so-called \emph{direct methods}. These are methods that compose $y,u$ via splines such that an objective is minimized. Direct transcription with explicit Euler is an example of a direct method because the method constructs $y$ from piecewise linear functions such that $J$ is minimized.

\vspace{2mm}\noindent
\underline{Indirect Methods}\\
By means of advanced calculus and limits of infinitely short needle perturbations, it is possible to derive the fact that exact minimizers of \eqref{eqn:modelOCP} must satisfy the following Euler-Lagrange conditions:
\begin{align*}
\frac{\partial g}{\partial y} - \frac{\mathrm{d}}{\mathrm{d}t}\left(\frac{\partial g}{\partial \dot{y}}\right)=0\,,\ y(t_0)=y_0\,,\ y(t_E)=y_E\,.
\end{align*}
This is a so-called \emph{boundary value problem} (BVP), in analogy to an initial value problem (IVP), because a differential equation must be solved subject to an initial and a final value. A generalization of the Euler-Lagrange conditions are the Pontryagin conditions \cite{MR0084444}. These conditions also result in BVP. Methods that work via solving either of these boundary value problems are called \emph{indirect methods}. Numerical methods for the solution of boundary value problems are introduced in Section~\ref{sec:ELODE}.

\vspace{2mm}
Methods of perturbations are distinguished into direct methods and indirect methods because they reduce the optimal control problem into either of two different formats: direct methods result in NLP, whereas indirect methods result in boundary value problems.

\subsection{Practicality}
\label{sec:literatureReview:Overview:Practicality}

The greedy substructures face many technical issues, e.g., when $u(t,x)$ cannot be determined uniquely or $V$ admits no classical solution \cite{MR4238012}. Also, the back-propagation aspect and the spatial domain $\R^{n_y}$ for $x$ in $V(t,x)$ render dynamic programming algorithmic principles impractical due to memory-limitations when $n_y$ is large.

Predictor-corrector algorithms constitute a broad and blurry class. It is often not clear in which way to select a predictor-corrector algorithm in order to solve a particular optimal control problem.

Divide-and-conquer methods have been used widely in multiple direct/indirect shooting methods. These compute the subsolutions with direct/indirect methods and determine optimal values $\tilde{y}$ in a direct/indirect fashion. Mesh-refinement algorithms can also be considered as divide-and-conquer algorithms.

We now discuss methods based on perturbability. Indirect methods in this realm are unpopular for several reasons: Firstly, the boundary value problem resulting from the optimality conditions is often difficult to solve numerically \cite{rao_asurvey,MR500418}. Numerical methods may diverge unless accurate initial guesses for the solution of the boundary value problem are given \cite{Bryson1975}. Secondly, for complicated optimal control problems it is difficult or impossible to determine the optimality conditions to begin with, hence there are problems which cannot be solved by indirect methods \cite{Boehme2017}. Even when possible, state-constrained problems require estimates of the intervals at which inequality constraints are active. Workarounds for state constraints, such as by saturation functions, have been proposed in~\cite{WANG20174070}. Thirdly, for singular-arc problems, optimality conditions of higher order need to be used in order to determine a solution \cite[Sec.~1.4]{LAMNABHILAGARRIGUE1987173}. Finally, the optimality conditions can have non-unique solutions.

In contrast, direct methods in this realm are very popular nowadays because they provide a direct means to discretizing the optimal control problem into an NLP. The NLP in turn can be solved with readily available solvers \cite{Nocedal,BettsChap2,KellyMatthew,ForsgrenGillSIREV}. However, this is the situation only since recently, as we depict in the following historical brief.

\section{Historical Brief}
\label{sec:literatureReview:History}

Direct transcription methods are both a very recent and a very old class of methods. To understand how this is possible, we have to look into the historical developments of numerical methods for optimal control problems. Table~\ref{fig:historydiagram} visualizes a timeline of different directions of development for numerical methods in optimal control. Revisiting these developments helps in assessing the maturity of each area and in identifying potential starting points for future research. We stress that the figure is not exhaustive and the depicted landmark contributions have certainly been inspired by fore-going research.

As the figure shows, computers existed only since 1930. However, today's purpose of numerical methods and computers is that numerical methods solve mathematical problems on computers. This makes it fruitlessly debatable how early researchers would have intended their methods to be used on computers. For example, one could argue from Section~\ref{sec:Intro:OCP} that direct collocation was invented by 1847: in Section~\ref{sec:Intro:OCP} we construct a direct collocation method from the explicit Euler and the gradient-descent method. Both methods existed by 1847. However, the method of direct collocation is formally attributed to Hargraves et al. in 1987 \cite{HARGRAVES87}. We appreciate that attribution of contribution is an undecidable problem. Some methods, e.g., collocation and quadrature, are even so old that they cannot be clearly attributed to anyone.

\begin{figure}
	\centering
	\includegraphics[width=0.9\linewidth]{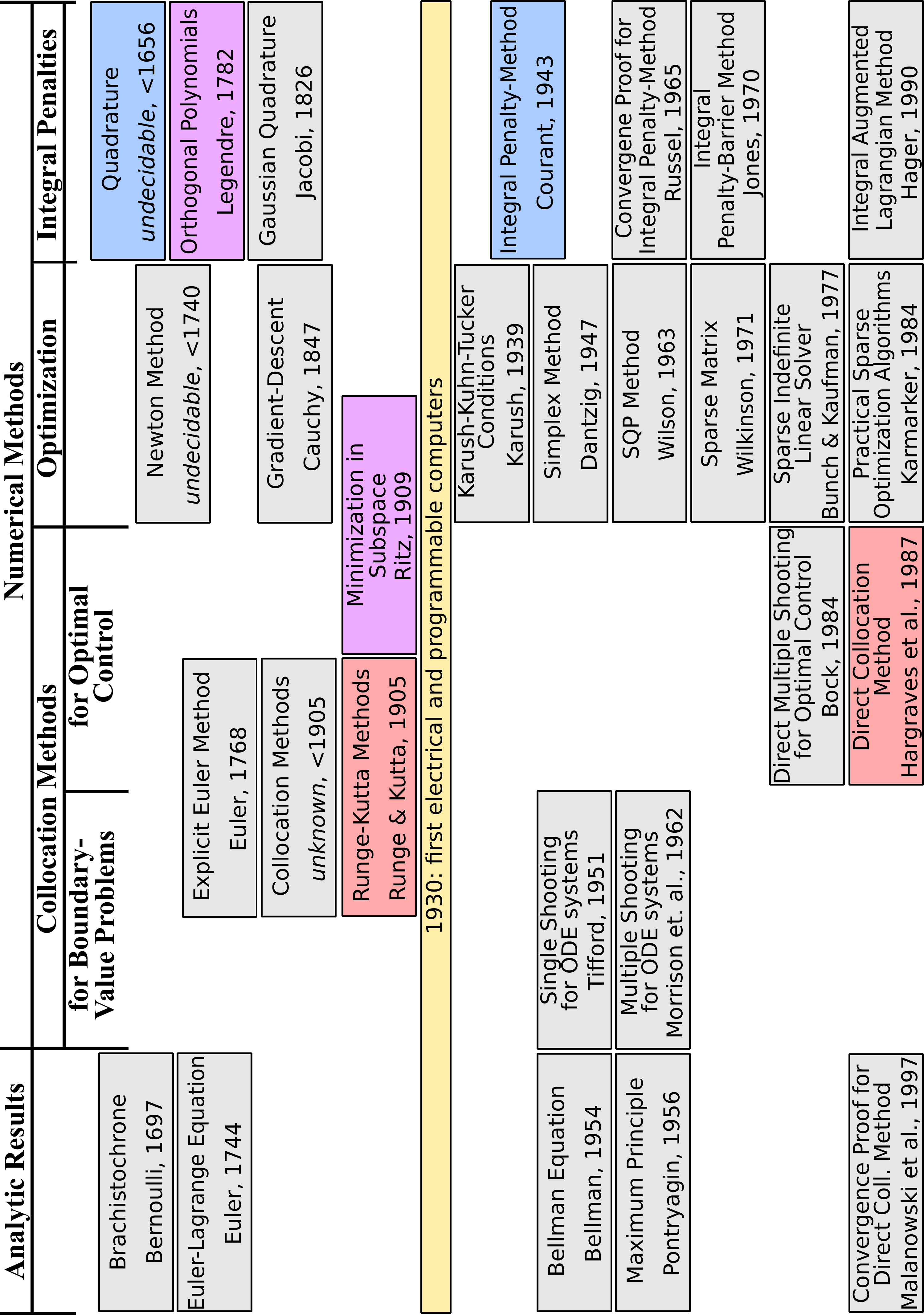}
	\caption{Milestones in the history of mathematical research in optimal control. Items are sorted vertically by time and in columns by research direction.}
	\label{fig:historydiagram}
\end{figure}

The figure distinguishes several research directions: Analytic results mean equations and concepts that give analytic characterizations of solutions. Numerical methods provide formulas to generate approximations. Basic methods such as Newton's method, quadrature, Euler's method and gradient-descent form the origin of all modern numerical methods.

\subsection{Euler-Lagrange Equations and Boundary Value Problems}\label{sec:ELODE}
The birth of optimal control is sometimes attributed to Johann Bernoulli for his discussion of the brachistochrone problem in 1697 \cite{Sussmann97}. In 1744, Euler identified the Euler-Lagrange differential equation as an analytic formula for how problems such as Bernoulli's brachistochrone could be solved in a generic way \cite{MR0056522}. The solution of optimal control problems via Euler-Lagrange equations results in boundary value problems.

Unlike initial value problems like \eqref{eqn:IVP}, boundary value problems replace some conditions on the initial values $y(0)$ with conditions on the end values $y(T)$. Methods for initial value problems are due to Euler in 1768 and have been generalized by Runge and Kutta in 1905. The first numerical method for boundary value problems however was the \emph{single shooting method}. It was implemented on ENIAC, the first electrical and programmable computer, to calculate ballistic curves for the US military \cite{mccartney1999eniac,reed1952firing}. Shooting methods work by iteratively guessing suitable initial values $y(0)$ such that the solution of the initial value problem meets all conditions on the end values $y(T)$. The iterative scheme works via Newton's method. This scheme was first described in the literature by Tifford only in 1951 \cite{tifford1951solution}.

As Figure~\ref{fig:singlevsmultipleshooting} shows, there are multiple other ways than single shooting in terms of how Newton's method can be combined with Runge-Kutta methods: (a) recalls the explicit Euler method as one example for a Runge-Kutta method. Euler's method can be used to compute the blue nodes of $y(t)$ from a given red initial value $y(0)$. (b) The single shooting method determines the initial value via Newton's method such that a particular end condition is met. (c) Multiple shooting divides the interval into multiple (the figure shows two) sub-intervals. The initial values of each sub-interval are computed via Newton's method such that the end-condition and finite-difference equations are satisfied. (d) Collocation methods include and solve all nodal values via Newton's method.

\begin{figure}
	\centering
	\includegraphics[width=0.8\linewidth]{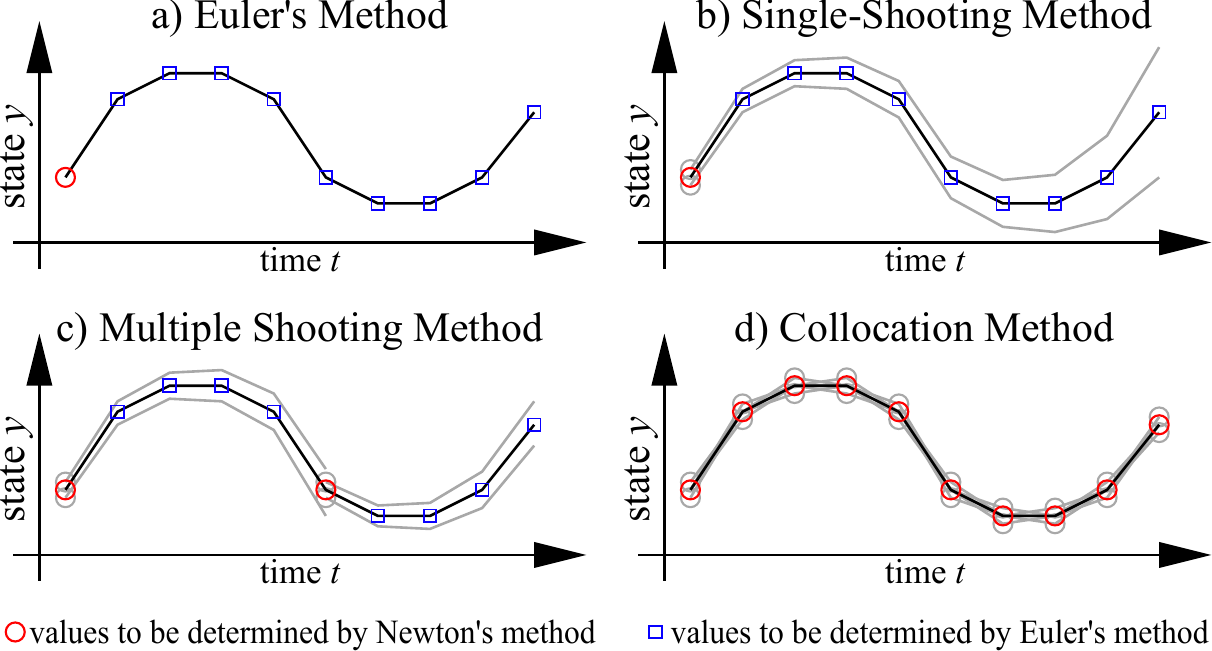}
	\caption{Three classes of methods for the solution of boundary value problems: Single Shooting, Multiple Shooting, and Collocation.}
	\label{fig:singlevsmultipleshooting}
\end{figure}

The solutions of all three methods are algebraically equivalent, because they solve the exact same set of equations. However, methods with smaller propagation of perturbations of the red nodes (visualized in the figure in grey) result in better numerical conditioning and thus superior robustness for the Newton iteration. In addition, multiple shooting yields a smaller Newton system than collocation. Yet, multiple shooting may break when on any shooting interval the differential equation has a non-unique solution. This is so because any perturbation may result in switching between two different ODE solutions. Collocation avoids this issue altogether because the Newton system does not use ODE solutions but ODE residuals. Hence, collocation is preferred in practice.

Limitations of the Euler-Lagrange equations arise when the optimal control problem has inequality constraints. In this case, Pontryagin's maximum principle (1956) provides a more general system of necessary conditions than the Euler-Lagrange equation \cite{MR0084444}. In contrast to necessary conditions, the Bellman equation (1954) provides a constructional principle for globally optimal solutions to all local minimizers of optimal control problems \cite{MR61289}. The limitations to applying these analytic results for practical computations have been discussed in Section~\ref{sec:literatureReview:Overview:Practicality}.

\subsection{Direct Multiple Shooting and Direct Collocation}
We saw in Figure~\ref{fig:speedchangeproblem} how the Euler method approximates $y$ with a piecewise linear function $y_h$. However, doing so in the context of an optimization problem was first proposed by Ritz: To minimize, e.g., the functional
\begin{align*}
\operatornamewithlimits{min}_{y \in \cX} \ F(y)=\int_0^1 \Big(\dot{y}(t)^2 + y(t)^2 - y(t)\Big)\,\mathrm{d}t\,,
\end{align*}
Ritz proposed approximating $y$ with a polynomial of fixed degree $n-1$; thus, resulting in a finite-dimensional optimization problem in $\R^n$ for the $n$ coefficients of the polynomial. Implementations of Ritz methods typically use orthogonal polynomials due to numerical stability reasons; cf.~(P1)--(P2) in~\cite{RaoHager2018}. Section~\ref{sec:GalerkinReview} presents the the Ritz method on the above example in detail for $n=4$.

Since Ritz, optimal control problems can be reduced directly into NLP. This is called direct transcription. Thus, optimization algorithms for NLP could be utilized for the numerical solution of optimal control problems. However, these algorithms had not been invented yet. Because of that, the NLP from Ritz method was solved via indirect methods, introduced by Galerkin. These are reviewed in Section~\ref{sec:GalerkinReview}. Modern direct solution methods for NLP are based on the equations derived by Karush in 1939 \cite{MR2936770}.

NLP can have equality and inequality constraints. The first widely used algorithm for inequality constrained optimization was Dantzig's simplex method dated to 1947\footnote{There is no publication of Dantzig in this year. He invented it earlier and it was undisclosed in that year.} \cite{nash2000dantzig}. This algorithm was only for LP, not NLP. The first method for NLP was the sequential quadratic programming (SQP) method by Wilson in 1963 \cite{wilson1963simplicial}. However, at that time it would not have been possible to solve NLP that stem from the direct transcription of optimal control problems of relevant sizes. This is because these NLP are typically of large dimension, whereas computers at that time were slow and efficient data structures for large matrices did not yet exist.

It was only in 1971 that Wilkinson introduced the notion of a sparse matrix data structure. Collocation methods, multiple shooting methods, and Galerkin methods result in equation systems with large matrices that comprise mostly of zeros. A sparse matrix data structure avoids multiplications with and storage of zeros in these large matrices; thus making computations with the aforementioned methods possible at all. Algorithms with sparse matrix data structures emerged for solving equation systems and optimization problems of larger dimension; such as the Bunch-Kaufman factorization in 1977 \cite{MR428694} and a first practical interior-point algorithm by Karmarkar in 1984 \cite{MR779900}.

The time at which optimization algorithms and practical optimization software became available explains why direct multiple shooting and direct collocation methods for solving optimal control problems entered the literature only by the 1980s. Before that point, it would have been difficult to propose such discretizations because there was no algorithm available to solve the resulting NLP. Due to similar reasons, as discussed along Figure~\ref{fig:singlevsmultipleshooting}, relating to conditioning and regularity, direct collocation is preferred over direct multiple shooting.

The constructional simplicity and practical success of direct collocation methods motivated research in convergence analysis. The first convergence proof for a direct collocation method is due to Malanowski et al. in 1997 \cite{Maurer}. They prove convergence of direct collocation with the explicit Euler method for optimal control problems where the dynamics are described as in \eqref{eqn:OCP}, but with prescribed values for $y(0)$. This proof is very sophisticated as apparent from its assumptions. Earlier convergence results are available but consider significantly less general formats. More detailed convergence results are reviewed in Section~\ref{sec:literatureReview:table}.

\subsection{Integral Penalty Methods and Quadrature Penalty Methods}
Integral penalty methods were introduced by Courant in 1943 within a Rayleigh-Ritz method for partial differential equations \cite{Courant43}. This method uses quadratic penalization to enforce homogeneous Dirichlet boundary conditions. In the same manner that Galerkin generalized the necessary conditions of the Ritz method, as will be detailed in Section~\ref{sec:GalerkinReview}, Nitsche (1971) \cite{MR341903} generalized the necessary conditions of the quadratic integral penalty method.

Today, \emph{Nitsche methods} are well-known and widely used for the solution \cite{MR351118} and optimal control \cite{WangYangXie_Nitsche,troltzsch2010optimal} of partial differential equations (PDE). For optimal control of nonlinear ordinary ODE and DAE however, collocation-type methods have been preferred over penalty-methods \cite{betts1998survey,rao_asurvey}. This may be related to the fact that Runge-Kutta methods are usually preferred over finite element/volume/difference methods when solving ODE.

In \cite{Russell65}, existence and convergence of solutions to two integral penalty functions of generic form were analyzed, but without a discretization scheme. Quadratic penalty functions of a less generic form, suiting optimal control problems with explicit initial conditions, were studied in~\cite{Balakrishnan68}. The analysis focuses on the maximum principles that arise from the penalty function and their connection (under suitable assumptions on smoothness and uniqueness) to Pontryagin's maximum principle, laying groundwork for an indirect solution to the penalty function. A numerical method is not proposed. In~\cite{Jones70}, the analysis is extended to inequality path constraints with a fractional barrier function. Using suitable smoothness and boundedness assumptions on the problem-defining functions, it is shown that the unconstrained minimizer converges from the interior to the original solution. As in~\cite{Balakrishnan68}, the analysis uses first-order necessary conditions. A numerical scheme on how to minimize the penalty functional is not presented. Limitations are in the smoothness assumptions.

In \cite{DeJulio70} the penalty function of~\cite{Balakrishnan68} is used for problems with explicit initial conditions, linear dynamics and no path constraints. Under a local uniqueness assumption, convergence is proven for a direct discretization with piecewise constant functions for $y_h,u_h$. The approach is extended in~\cite{Hager90} to augmented Lagrangian methods with piecewise linear elements for the states. The analysis is mainly for linear-quadratic optimal control, which is used for approximately solving the inner iterations. The error in the outer iteration (augmented Lagrangian updates) contracts if the initial guess is sufficiently accurate~\cite[Lem.~3]{Hager90}.

A convergence analysis for the use of integral penalties for equality constraints and integral geometric barriers for inequality constraints is given in \cite{Jones70}. This analysis shows that the constraint residuals and optimality gap of the penalty-barrier problem converge with respect to the original problem. The paper sketches how this analysis could be useful within a practical numerical scheme. The work in \cite{Neuenhofen2020AnIP} provides a practical numerical method that uses quadratic penalties and logarithmic barriers instead of geometric barriers. Logarithmic barriers can be solved more efficiently in finite-dimensional optimization \cite{ForsgrenGillSIREV}.

\subsection{Final Remarks}
Coming back to Table~\ref{fig:historydiagram}, research in numerical methods for optimal control accelerated with the advent and development of modern computers (beige). Today, it is possible to solve large sparse finite-dimensional optimization problems on affordable consumer computers.

Due to these developments, the most widely used class of numerical methods for solving optimal control problems today is direct transcription. The concepts of direct transcription methods are rooted in the Ritz method and orthogonal polynomials, highlighted in purple in Table~\ref{fig:historydiagram}. Among direct transcription methods, there are two classes of methods: Collocation methods are based on Runge-Kutta methods (red). Penalty methods are based on quadrature (blue).

\section{Relation between Integral Penalty and Galerkin Methods}\label{sec:GalerkinReview}
Integrals arise in integral penalty methods and in Galerkin methods. Both classes of methods are rarely used for the optimal control of ODE. Furthermore, Galerkin methods are a large class of non-trivial methods that may be unfamiliar to a wide readership in optimal control of ODE. Due to these reasons, this section reviews the conceptual ideas behind Galerkin methods and penalty methods in the context of optimal control problems. We use numerical examples to highlight differences and trade-offs between different methods.

A discussion of finite element methods, variational calculus, (bi-) linear forms, certain terminology and geometric interpretation (e.g., Galerkin-orthogonality) is avoided when possible because their introduction does not enhance accessibility. Interested readers are referred to \cite{Brezis,troltzsch2010optimal,MR3097958}.

\paragraph{Organization}
Table~\ref{fig:galerkindiagramrot} gives an organigram of three classes of methods for the solution of three classes of problems. The methods are: Ritz, Ritz-Galerkin, and weighted residual. The problems are: Unconstrained minimization, boundary value problems, and constrained minimization. Different methods may differ only in details, hence why the organigram works with three example problems that we crafted such that they are all mathematically equivalent. However, each numerical method yields a different numerical solution. In the following we explain each problem and method in the figure.

\begin{table}
	\centering
	\includegraphics[width=0.85\linewidth]{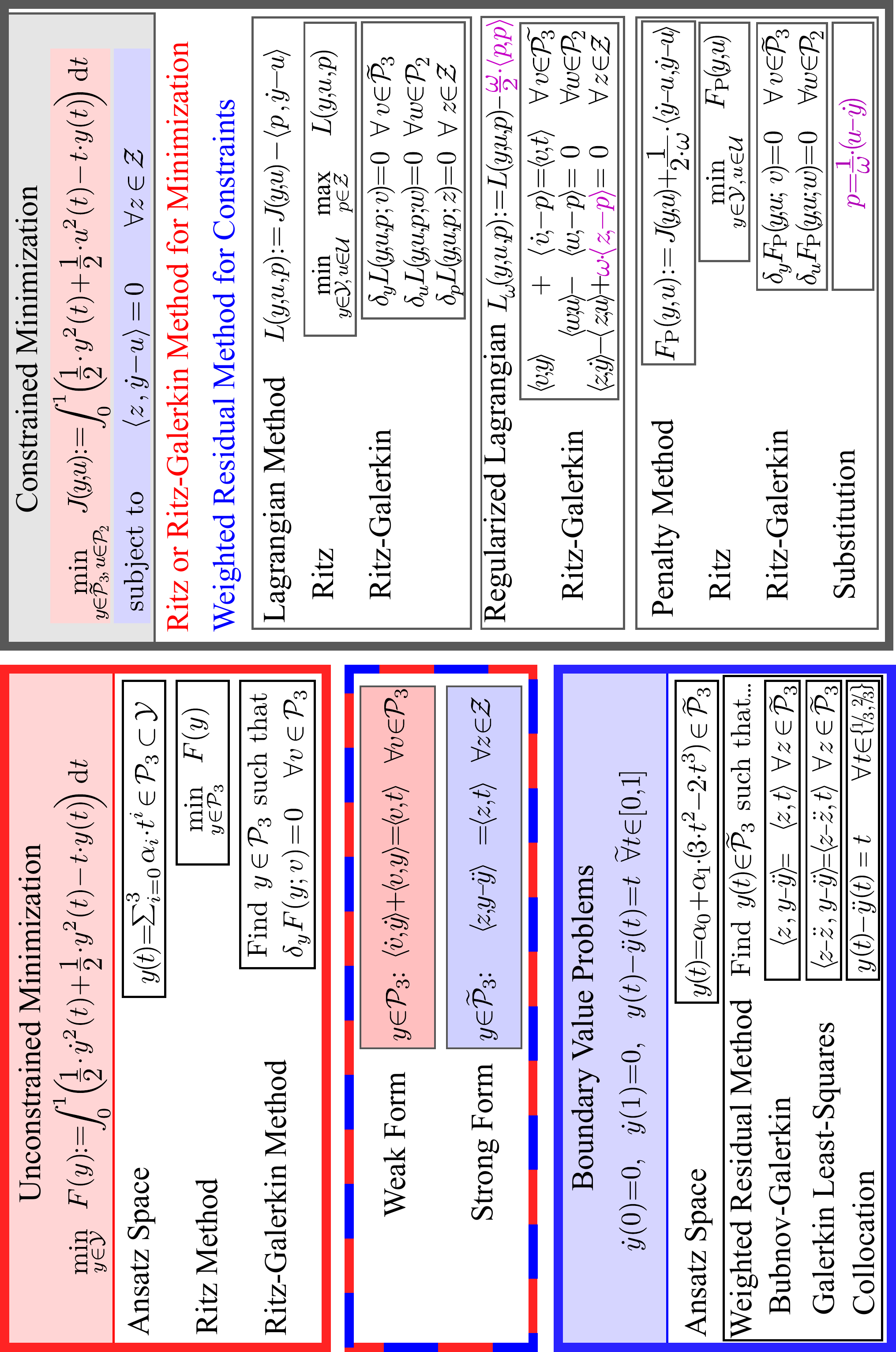}
	\caption{Methods of Ritz, Galerkin, and weighted residuals for the solution of minimization and boundary value problems.}
	\label{fig:galerkindiagramrot}
\end{table}

\paragraph{Equivalence of Problems}
We show that the problems in Table~\ref{fig:galerkindiagramrot} are equivalent. A suitable candidate space $\cY$ for the depicted unconstrained minimization problem is $W^{1,2}$.

If we solve problem \eqref{eqn:Min} in a subspace $\cV \subset \R^n$ then the first-order necessary condition of optimality is:
\begin{align*}
\bx \in \cV\,:\quad \bv\t \cdot \nabla_{\bx} \bf(\bx)=0\ \ \forall\, \bv \in \cV\,.
\end{align*}
In the above, $\bv\t \cdot \nabla_{\bx} \bf(\bx)$ is the directional derivative of $\bf$ in the direction $\bv$.
When $\cV=\R^n$ then this is equivalent to $\nabla_{\bx} \bf(\bx)=\bO$.

Analogously, the necessary condition of optimality of
\begin{align*}
\operatornamewithlimits{min}_{y \in \cV}\ F(y)
\end{align*}
for any vector space $\cV \subseteq \cY$ is:
\begin{align}
y \in \cV\,:\quad \delta_{y} F(y;v)=0\ \ \forall\, v \in \cV\,, \label{eqn:Galerkin:indirect}
\end{align}
where
\begin{align}
\delta_y F(y;v) = \int_0^1 \Big( \dot{v}(t) \cdot \dot{y}(t) + v(t) \cdot y(t) - v(t) \cdot t \Big)\,\mathrm{d}t \label{eqn:Galerkin:derivativeF}
\end{align}
is the directional derivative.

The analytic solution of the unconstrained optimization problem and the boundary value problem is
\begin{align*}
y^\star(t) = t + \frac{\exp(1-t) - \exp(-t)}{\exp(1)+1}\,.
\end{align*}
Substituting $u=\dot{y}$, we see that the unconstrained and the constrained optimization problem in Table~\ref{fig:galerkindiagramrot} are related. In the following, we discuss numerical methods for solving each of the three depicted problem classes in the figure.

\subsection{Unconstrained Minimization}

\paragraph{Ritz Method}
In order to solve unconstrained optimization problems with readily available methods, Ritz (1909) proposed to replace the infinite-dimensional candidate space $\cY$ with a finite-dimensional \emph{ansatz space}\footnote{We use the terms \emph{candidate space} and \emph{ansatz space} to avoid the use of the term \emph{search space}, which could be confused with either of the two. A candidate space is a property of the problem statement. An ansatz space is a property of a numerical method.}. For the purpose of this example, we choose the space of cubic polynomials $\cP_3$. This is a discretization of dimension $N=4$. This gives us the ansatz
\begin{align*}
y(t) &= \alpha_0 + \alpha_1 \cdot t + \alpha_2 \cdot t^2 + \alpha_3 \cdot t^3\,,\qquad \bx:=(\alpha_0,\dots,\alpha_3) \in \R^N\,, \tageq\label{eqn:Galerkin:x}\\
\dot{y}(t) &= \alpha_1 + 2 \cdot \alpha_2 \cdot t + 3 \cdot \alpha_3 \cdot t^2\,,
\end{align*}
which we can insert into the integral expression of $F(y)$:
\begin{align*}
F(y)&=\int_0^1 \Big( \frac{1}{2} \cdot \big(\alpha_1 + 2 \cdot\alpha_2 \cdot t + 3 \cdot \alpha_3 \cdot t^2\big)^2 + \frac{1}{2} \cdot \big( \alpha_0 + \alpha_1 \cdot t + \alpha_2 \cdot t^2 +\alpha_3 \cdot t^3\big)^2\\
&\qquad - t \cdot \big(\alpha_0 +\alpha_1 \cdot t +\alpha_2 \cdot t^2 +\alpha_3 \cdot t^3\big) \Big)\,\mathrm{d}t\\[10pt]
&=\frac{\alpha_0^2}{2} + \frac{\alpha_0 \cdot\alpha_1}{2} + \frac{\alpha_0 \cdot\alpha_2}{3} + \frac{\alpha_0 \cdot\alpha_3}{4} - \frac{\alpha_0}{2} + \frac{2 \cdot\alpha_1^2}{3} + \frac{5 \cdot\alpha_1 \cdot\alpha_3}{4}\\
&\qquad + \frac{6 \cdot\alpha_1 \cdot\alpha_3}{5} - \frac{\alpha_1}{3} + \frac{23 \cdot\alpha_2^2}{20} + \frac{5 \cdot\alpha_2 \cdot\alpha_3}{3} - \frac{\alpha_2}{4}=:\bf(\bx)
\end{align*}
The \emph{Ritz method} is the minimization of $F$ in the ansatz space; hence an NLP of the form \eqref{eqn:examplefMin}.
We can compute a minimizer for $\bf(\bx)$ with the gradient-descent method, resulting in~\mbox{$\bx \approx (0.4621,\,0.0006,\,0.2255,\,-0.1503)$}. Thus, the solution of the Ritz method is
\begin{align*}
y(t) \approx 0.4621 + 0.0006 \cdot t + 0.2255 \cdot t^2 - 0.1503 \cdot t^3\,.
\end{align*}
This is a very accurate numerical solution. Figure~\ref{fig:galerkinvariants} shows the error between the Ritz solution and $y^\star$ in blue. We see that the error is smaller than $10^{-4}$ everywhere.
\begin{figure}[tb]
	\centering
	\includegraphics[width=1\linewidth]{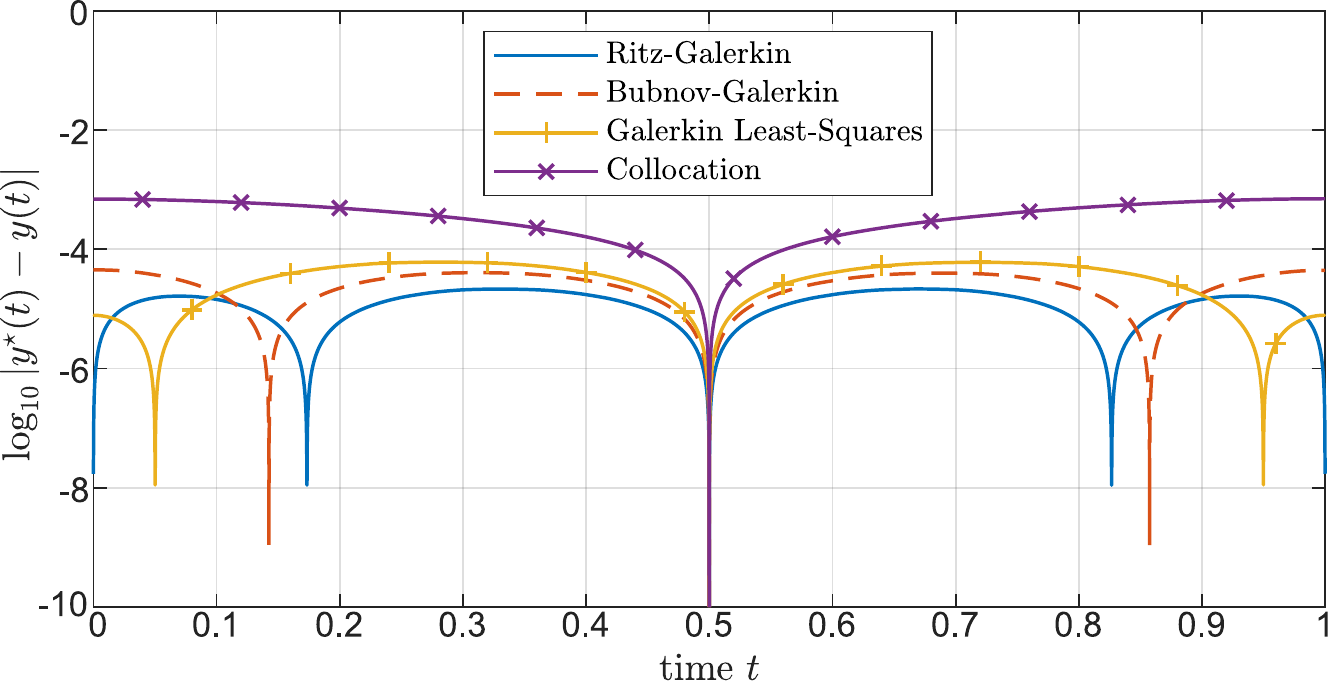}
	\caption{Absolute error of different Galerkin-type methods and a collocation method for a polynomial of degree $3$ for the boundary value problem in Table~\ref{fig:galerkindiagramrot}. Each numerical method generates a different approximate solution.}
	\label{fig:galerkinvariants}
\end{figure}

\paragraph{Ritz-Galerkin Method}
The Ritz method minimizes $F$ directly in the ansatz space. Recalling the necessary optimality condition \eqref{eqn:Galerkin:indirect}, solutions of the Ritz method must satisfy
\begin{align}
y \in \cP_3\,:\quad \delta_y F(y;v) = 0\ \ \forall\,v \in \cP_3\,. \label{eqn:Galerkin:RitzGalerkin}
\end{align}
The solution of equations of this form was first proposed by Galerkin in 1915 \cite{galerkin1915rods}. For the aforementioned reasons, this yields the same solution as the Ritz method; hence why this method is called the \emph{Ritz-Galerkin method}. Inserting \eqref{eqn:Galerkin:derivativeF} and $a,b$ as stated in Table~\ref{fig:galerkindiagramrot} under \emph{Weak Form}, we may write~\eqref{eqn:Galerkin:RitzGalerkin} as
\begin{align}
y \in \cP_3\,:\quad a(v,y) = b(v) \ \ \forall\,v \in \cP_3\,. \label{eqn:Galerkin:Eqn}
\end{align}
We explain later why it is called the \emph{Weak Form}. Using the representations~\eqref{eqn:Galerkin:x} and
\begin{align*}
v(t) = \sum_{j=0}^3 \beta_j \cdot t^j \in \cP_3\,,\quad \bbeta = (\beta_0,\dots,\beta_3) \in \R^4\,,
\end{align*}
the condition~\eqref{eqn:Galerkin:Eqn} can be represented as
\begin{align}
\bbeta\t \cdot \underbrace{\begin{bmatrix}
	1  	& 1/2 & 1/3 & 1/4\\
	1/2 & 4/3 & 5/4 & 6/5\\
	1/3 & 5/4 &23/15& 5/3\\
	1/4 & 6/5 & 5/3 &68/35
	\end{bmatrix} \cdot \bx}_{T(\bx)} = \bbeta\t \cdot \begin{bmatrix}
1/2\\
1/3\\
1/4\\
1/5
\end{bmatrix}\qquad \forall \bbeta \in \R^N\,,\label{eqn:Galerkin:LinSys}
\end{align}
which can be solved as linear system for $\bx$. The entries in the matrix and right-hand side vector are $\langle t^j , t^i \rangle+\langle \frac{\mathrm{d}}{\mathrm{d}t}(t^j),\frac{\mathrm{d}}{\mathrm{d}t}(t^i)\rangle$ and $\langle t^j, t\rangle$, $i,j=0,\dots,3$. When $\delta_y F(y;v)$ is nonlinear in $y$ then $T(\bx)$ is nonlinear in $\bx$ and one must resort to, e.g., Newton iterations.

\subsection{Boundary Value Problems}
The above formulation is called \emph{weak form} because it only assumes that $\dot{y} \in L^2$. In contrast, if we make the \emph{strong assumption} that $\ddot{y}$ is well-defined $\tforall t \in [0,1]$ then the directional derivative $\delta_y F(y;v)$ from~\eqref{eqn:Galerkin:derivativeF} can be re-expressed by virtue of partial integration:
\begin{align*}
\delta_y F(y;z) &= z(1) \cdot \dot{y}(1) - z(0) \cdot \dot{y}(0) + \int_0^1\Big( z(t)\cdot y(t) - z(t) \cdot \ddot{y}(t) - z(t)\cdot t \Big)\,\mathrm{d}t\\
&= -z(0) \cdot \dot{y}(0) + z(1) \cdot \dot{y}(1) + \langle z , y-\ddot{y}-t\rangle
\end{align*}
Apart from the partial integration, in the directional derivative we also replaced the symbol of a direction $v$ with a so-called \emph{test function} $z$ in a suitable \emph{test space} $\cZ$. If we want to satisfy $\delta_y F(y;z)=0$ independent of any values for $z(0)$, $z(1)$, and $z(t)$, then we must find $y$ that satisfies the boundary value problem in Table~\ref{fig:galerkindiagramrot}.

\paragraph{Weighted Residual Method}
We can satisfy
\begin{align*}
-z(0) \cdot \dot{y}(0) + z(1) \cdot \dot{y}(1) + \langle z , y-\ddot{y}-t\rangle=0
\end{align*}
by choosing the new approach
\begin{align*}
y(t) = \alpha_0 + \alpha_1 \cdot (3 \cdot t^2 - 2 \cdot t^3) \in \tcP_3\,,\qquad \bx = (\alpha_0,\alpha_1) \in \R^N
\end{align*}
for $N=2$. Thus, $\dot{y}(0)=\dot{y}(1)=0$ satisfies the boundary conditions naturally. We can then solve the boundary value problem numerically as depicted in Table~\ref{fig:galerkindiagramrot} under \emph{Strong Form}.

There are three common weighted residual methods. The \emph{Bubnov-Galerkin method} uses $\cZ=\tcP_3$, i.e. the same test space as ansatz space. The \emph{Galerkin Least-Squares method} uses a space $\cZ$ that depends on the differential equation at hand: The method forms the least-squares functional
\begin{align*}
F_{\text{LS}}(y):= \frac{1}{2} \cdot \int_0^1 \Big(y(t)-\ddot{y}(t)-t\Big)^2 \,\mathrm{d}t\,.
\end{align*}
The method then applies the Ritz-Galerkin method to the minimization of that functional. This results in the following equation:
\begin{align*}
\delta_y F_{\text{LS}}(y;z)= \langle z-\ddot{z},y-\ddot{y}\rangle - \langle z-\ddot{z},t\rangle=0\quad \forall z \in \tcP_3\,.
\end{align*}
Thus, the Galerkin Least-Squares method can be written in the same way as the Bubnov-Galerkin method in Table~\ref{fig:galerkindiagramrot} by selecting $\cZ = \lbrace z - \ddot{z} \ \vert\ z \in \tcP_3 \rbrace$. The Galerkin Least-Squares method always yields a symmetric linear system. In contrast, \emph{Petrov-Galerkin methods} are methods where the system matrix is non-symmetric and $\cZ\neq \tcP_3$. Lastly, there are collocation methods. These read: Find $y \in \tcP_3$, such that
\begin{align*}
y(t)-\ddot{y}(t) = t \quad \forall t \in \cT\,,
\end{align*}
where $\cT \subset [0,1]$ is set of $N$ points; i.e., two points in our example.

\paragraph{Numerical Comparison}
We now compare the methods of Ritz-Galerkin, Bubnov-Galerkin, Galerkin Least-Squares, and Collocation with $\cT=\lbrace 1/3,\, 2/3 \rbrace$ numerically. Figure~\ref{fig:galerkinvariants} plots the error of each method. Due to symmetry and overlap, the error vanishes at $t=0.5$ and at some other method-specific points. The figure confirms that indeed each method yields a different solution.

\subsection{Constrained Minimization}
Optimal control problems use \emph{mixed formulations} \cite{MR3097958}. These are problems that combine a Ritz or Ritz-Galerkin method for the ansatz space with a weighted residual method for the constraints. Table~\ref{fig:galerkindiagramrot} color-indicates both portions of the problem. In the figure, $\cZ$ is a place-holder for an arbitrary weighted residual method. There are at least three ways for numerically solving the constrained minimization problem. 

\paragraph{Lagrange Method}
Using a Lagrange multiplier $p \in \cZ$, we can apply a Ritz method or Ritz-Galerkin method to find the stationary point of the Lagrangian functional. The directional derivatives are:
\begin{align*}
\delta_y L(y,u,p;v) &= \int_0^1 \Big(v(t) \cdot y(t) - v(t) \cdot t - \dot{v}(t) \cdot p(t)\Big)\,\mathrm{d}t\\
\delta_u L(y,u,p;w) &= \int_0^1 \Big(w(t) \cdot u(t) + w(t) \cdot p(t) \Big)\,\mathrm{d}t\\
\delta_p L(y,u,p;z) &= \int_0^1 \Big(z(t) \cdot \dot{y}(t) - z(t) \cdot u(t)\Big)\,\mathrm{d}t
\end{align*}
Setting them to zero for a given test space $\cZ$ yields a symmetric linear system, just like \eqref{eqn:Galerkin:LinSys}. We introduce these systems in the next paragraph. However, the matrix of this linear system is often singular unless $\cZ$ is chosen in a special way such that the \emph{Ladyzhenskaya–Babuska–Brezzi condition} is satisfied \cite{MR3097958}. Finding a suitable $\cZ$ requires sophisticated functional analysis for each constrained minimization problem at hand.

\paragraph{Regularized Lagrange Method}
It thus makes sense to regularize the Lagrangian via the purple term with a small regularization parameter $\omega \in \R_{>0}$. In the symmetric linear system this will yield a negative definite block in the lower right, shown in purple in Table~\ref{fig:galerkindiagramrot} under \emph{Regularized Lagrangian, Ritz-Galerkin}. This may help regularizing the linear system. On the flip-side, the magnitude of $\omega$ alters the numerical solution of $y,u$. Also, the Lagrangian method necessitates a basis representation for $\cZ$; cf. $\bbeta$ in \eqref{eqn:Galerkin:LinSys}. We will show different system matrices below in the numerical comparison. Equation systems of the above form have been proposed by Nitsche in 1971 \cite{MR341903}.

\paragraph{Quadratic Penalty Method}
Because of the hassle with a suitable basis for $\cZ$ and regularization, it seems attractive to use an alternative method altogether. The \emph{quadratic penalty method} (Courant, 1943) is such an alternative. The method works by forming
\begin{align*}
\operatornamewithlimits{min}_{y \in \cY,u \in \cU}\quad F_{\text{Penalty}}(y,u):=J(y,u) + \frac{1}{2 \cdot \omega} \cdot \int_0^1 \Big(\dot{y}(t)-u(t)\Big)^2\,\mathrm{d}t
\end{align*}
and solving the unconstrained minimization problem with a direct method (e.g., conjugate gradients). In the above, $\omega \in \R_{>0}$ is a small penalty parameter.

Since the Ritz and the Ritz-Galerkin method both generate the same solution, we can now use the latter to derive an equation system. The directional derivatives are:
\begin{align*}
\delta_y F_{\text{Penalty}}(y,u;v) &= \int_0^1 \Big(v(t) \cdot y(t) - v(t) \cdot t + \frac{1}{\omega} \cdot \dot{v} \cdot \big(\dot{y}(t)-u(t)\big) \Big)\,\mathrm{d}t \\
\delta_u F_{\text{Penalty}}(y,u;w) &= \int_0^1 \Big(w(t) \cdot u(t) - \frac{1}{\omega} \cdot w(t) \cdot \big(\dot{y}(t)-u(t)\big) \Big)\,\mathrm{d}t
\end{align*}
Hence, the Ritz-Galerkin method reads: Find $y \in \cY, u \in \cU$, such that
\begin{align}
\begin{matrix}
\vspace{2mm}\langle v,y \rangle &+&\frac{1}{\omega} \cdot \langle \dot{v},\dot{y}-u\rangle &=& \langle v, t \rangle  &\forall\ v\in \cY\,,\\
\langle w,u \rangle &-&\frac{1}{\omega} \cdot \langle      w ,\dot{y}-u\rangle &=& 0                     &\forall\ w\in \cU\,.
\end{matrix}\label{eqn:Galerkin:Nitsche}
\end{align}
If $\frac{1}{\omega} \cdot (u-\dot{y}) \in \cZ$ then we can substitute $p=\frac{1}{\omega} \cdot (u-\dot{y})$ in the Ritz-Galerkin method for the regularized Lagrangian. Thus, under this special choice of $\cZ$, the quadratic penalty method and the regularized Lagrangian method are equivalent.

\subsection{Comparison: Galerkin Least-Squares Method vs Quadratic Penalty Method}
The quadratic penalty method may be confused with the Galerkin Least-Squares method. However, both methods are unrelated: The Galerkin Least-Squares method is a non-trivial choice of $\cZ$, whereas the quadratic penalty method is a method where $\cZ$ is absent.

In the given example, because we suppose that the constraint uniquely determines $y$ from given values of $u$, the Galerkin Least-Squares method reads
\begin{align*}
\langle \dot{v} , \dot{y}-u\rangle = 0\quad \forall v \in \tcP_3\,,
\end{align*}
thus $\cZ=\cZ_{\text{LS}} = \lbrace \dot{v} \,\vert \ v \in \tcP_3 \rangle$.

We numerically compare the linear systems of the aforementioned methods to clarify that they yield different numerical solutions. The linear systems are for
\begin{align*}
y(t) &= \alpha_{0,y} + \alpha_{1,y} \cdot (3 \cdot t^2 - 2 \cdot t^3) & &\in \tcP_3\,,\\
u(t) &= \alpha_{0,u} + \alpha_{1,u} \cdot t^1 + \alpha_{2,u} \cdot t^2 & & \in \cP_2\,,\\
p(t) &= \alpha_{0,p} \cdot (6 \cdot t - 6 \cdot t^2) & & \in \cZ_{\text{LS}}\,.
\end{align*}
\begin{itemize}
	\item The Ritz-Galerkin system for the \textcolor{violet}{regularized} Lagrangian method with $\cZ_{LS}$ is
	\begin{align*}
	\left[\begin{array}{ccccc|cc}
	1   &  1/2  & 0   & 0   & 0      &  0   \\
	1/2 & 13/35 & 0   & 0   & 0      & 5/4  \\
	0 &  0    & 1   & 1/2 & 1/3    & -1   \\
	0 &  0    & 1/2 & 1/3 & 1/4    & -1/2 \\
	0	&  0    & 1/3 & 1/4 & 1/5    & -3/10\\
	\hline
	0 &  5/4  & -1  & -1/2& -3/10  &  \textcolor{violet}{-\omega \cdot 1250}\\
	\end{array}\right] \cdot \left[\begin{array}{c}
	\alpha_{0,y}\\
	\alpha_{1,y}\\
	\alpha_{0,u}\\
	\alpha_{1,u}\\
	\alpha_{2,u}\\
	\hline
	-\alpha_{0,p}
	\end{array}\right] = \left[\begin{array}{c}
	1/2\\
	7/20\\
	0\\
	0\\
	0\\
	\hline
	0\\
	\end{array}\right]\,.
	\end{align*}
	\item The Ritz-Galerkin system for the quadratic penalty method is
	\begin{align*}
	\left(\left[\begin{array}{ccccc}
	1   &  1/2  & 0   & 0   & 0      \\
	1/2 & 13/35 & 0   & 0   & 0      \\
	0   &  0    & 1   & 1/2 & 1/3    \\
	0   &  0    & 1/2 & 1/3 & 1/4    \\
	0	&  0    & 1/3 & 1/4 & 1/5    
	\end{array}\right] + \frac{1}{\omega} \cdot \bB \right) \cdot \left[\begin{array}{c}
	\alpha_{0,y}\\
	\alpha_{1,y}\\
	\alpha_{0,u}\\
	\alpha_{1,u}\\
	\alpha_{2,u}
	\end{array}\right] = \left[\begin{array}{c}
	1/2\\
	7/20\\
	0\\
	0\\
	0
	\end{array}\right]\,,
	\end{align*}
\end{itemize}
where
\begin{align*}
\bB := \left[\begin{array}{ccccc}
0 & 0     & 0   &  0   & 0 \\
0 & 5/4   & -1  & -1/2 & -3/10 \\
0 & -1    &  1  &  1/2 & 1/3 \\
0 & -1/2  & 1/2 &  1/3 & 1/4 \\
0 & -3/10 & 1/3 &  1/4 & 1/5
\end{array}\right] \neq \left[\begin{array}{ccccc|cc}
0   \\
5/4  \\
-1   \\
-1/2 \\
-3/10
\end{array}\right] \cdot \left[\begin{array}{ccccc|cc}
0   \\
5/4  \\
-1   \\
-1/2 \\
-3/10
\end{array}\right]\t\,.
\end{align*}
The matrix $\bB$ has rank $4$, hence any regularized Lagrangian method with $\dim(\cZ)<4$ cannot be equivalent to the quadratic penalty method. Even then, it is non-trivial to find a space $\cZ$ such that the reduced linear system matches the one with $\bB$.

\subsection{Comparison: Galerkin vs Collocation}
Galerkin methods are more complicated to implement and more computationally expensive than collocation methods, due to the quadrature. However, Galerkin methods may yield more accurate solutions than a collocation method of similar discretization size $N$. Figure~\ref{fig:galerkinvariants} shows an example where a Galerkin method finds a more accurate solution in the same ansatz space $\tcP_3$. However, using a Galerkin method is not a guarantor for superior accuracy.

We now demonstrate via the following numerical experiment the intricacy of choosing the right basis: The convection-diffusion boundary value problem
\begin{align*}
y(0)=0,\quad y(1)=0,\qquad -\nu \cdot \ddot{y}(t)+\dot{y}(t)=1\ \ \tforall\,t \in [0,1]
\end{align*}
has the exact solution $y^\star(t)=t + \frac{1-\exp(t/\nu)}{\exp(1/\nu)-1}$ but will be solved numerically via four methods for the parameters $\nu=1$ and $\nu=0.001$ in the ansatz space
\begin{align*}
\cV := \Bigg\lbrace y \, \Bigg\vert\ y(t)=\sum_{i=1}^{8} \alpha_i \cdot \sin(\pi\cdot i \cdot t)\ \Bigg\rbrace\,.
\end{align*}
This choice permits easy construction and differentiation of $y$ and satisfaction of the boundary conditions. The weak form is
\begin{align*}
\nu \cdot \langle \dot{v},\dot{y}\rangle + \langle z,\dot{y}\rangle = \langle v,1\rangle\quad \forall v \in \cV\,.
\end{align*}
The strong form is
\begin{align*}
-\nu \cdot \langle {z},\ddot{y}\rangle + \langle z,\dot{y}\rangle = \langle z,1\rangle\quad \forall z \in \cZ\,,
\end{align*}
which in this instance yields the same solution as the weak form when $\cZ=\cV$ (Galerkin). Collocation uses the strong form with $\cT=\lbrace 1/9,\,2/9,\dots,8/9 \rbrace$. The Galerkin Least-Squares method solves
\begin{align*}
\langle -\nu \cdot \ddot{z}+\dot{z},-\nu \cdot \ddot{y} + \dot{y} \rangle = \langle  -\nu \cdot \ddot{z}-\dot{z},1 \rangle\quad \forall z \in \cV\,.
\end{align*}

Figure~\ref{fig:cdrexperiment} shows the solutions and residuals. In the case $\nu=1$ all solutions look accurate, even though the residuals are uniformly large. In contrast, for $\nu=0.001$ all methods yield equally useless inaccurate solutions. This is because the exact solution has a very sharp edge that cannot be resolved in the ansatz space. One exception is the streamline-upwind Petrov-Galerkin method (SUPG), which we included as an example for a Petrov-Galerkin method. The SUPG method takes the special choice
\begin{align*}
\langle z+\omega \cdot \dot{z},-\nu \cdot \ddot{y}+\dot{y}\rangle = \langle z+\omega \cdot \dot{z},1\rangle\quad \forall z \in \cV\,,
\end{align*}
which in the physical context of this particular differential equation can be interpreted as a regularized upwind discretization with regularization parameter $\omega=10^{-1}$.

\begin{figure}
	\centering
	\includegraphics[width=1\linewidth]{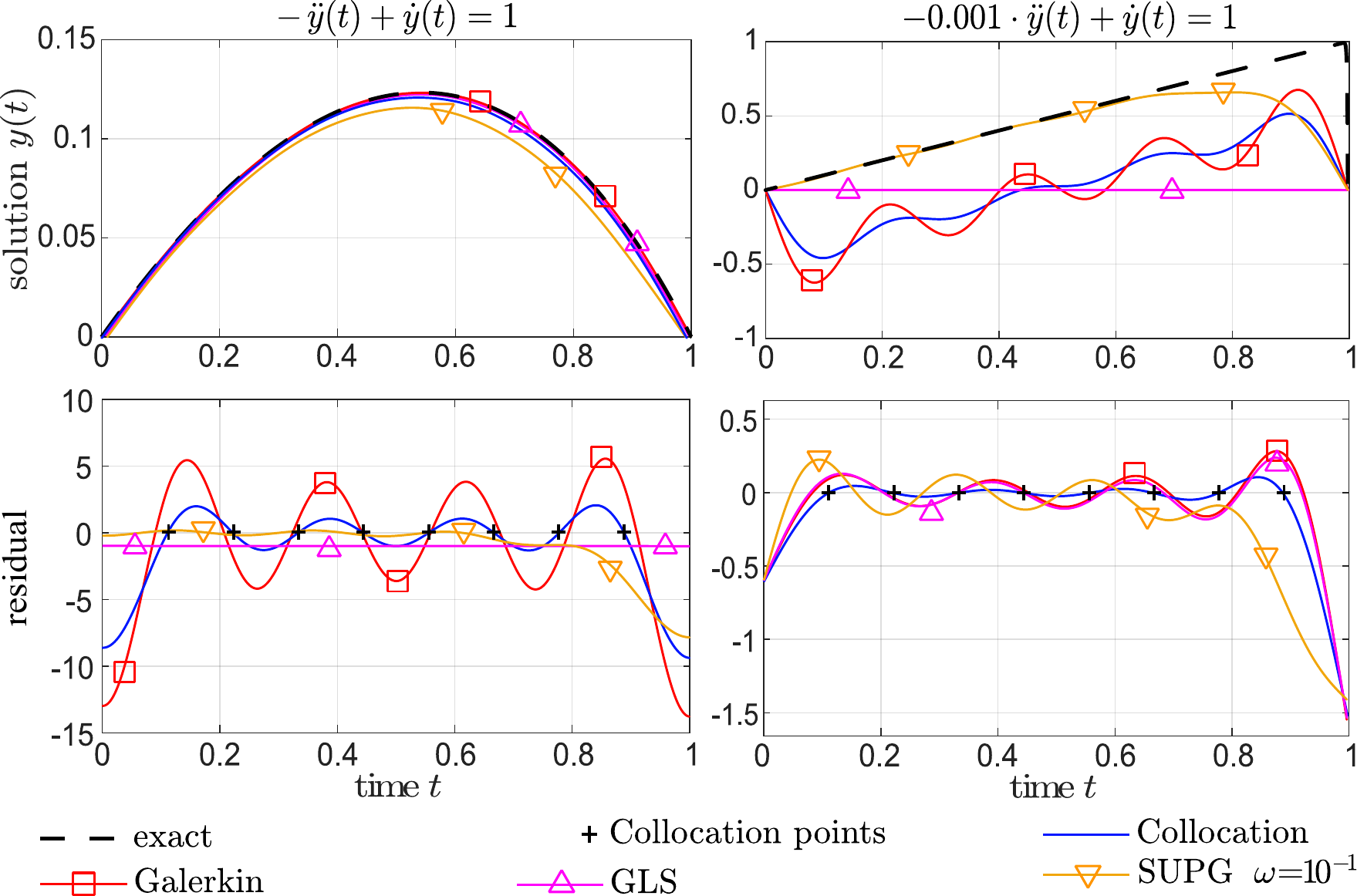}
	\caption{Comparison of different numerical solutions to a convection-diffusion equation for different values of the diffusion parameter.}
	\label{fig:cdrexperiment}
\end{figure}

Concluding from the numerical example, the right combination of ansatz space $\cV$ and test space $\cZ$ is critical for each individual boundary value problem. For instance, the SUPG method is only good for the presented example of convection-diffusion equations. For black-box nonlinear optimal control problems it seems impossible to find an ideal method every time from scratch. Hence, opting for collocation as the cheapest and simplest method seems reasonable. On the other hand, optimal control problems with a-priori known linear partial differential constraints often use a tailored Galerkin method \cite{troltzsch2010optimal}.

\section{Available Convergence Proofs for Direct Transcription Methods}
\label{sec:literatureReview:table}

This section presents an overview of known convergence results for various classes of direct transcription methods in the literature. Table~\ref{tab:litconv} gives an overview of various methods: explicit and implicit Euler discretization, linear multi-step methods, Runge-Kutta methods, pseudo-spectral methods, penalty methods, penalty-barrier methods, and two other classes of methods that are not discussed in this survey: pseudo-spectral and augmented Lagrangian methods. They are not discussed because they are conceptually similar to collocation and penalty methods. The last row lists the convergence result of Part~\ref{part:convproof} in this thesis, hence there is no reference.

\subsection{Simplifications for the Sake of Comparability}
The table only states a limited number of attributes in a non-parametric form. For instance, a particular parametric order of convergence results in parametric assumptions that make the table difficult to read. Thus, the table states the minimum assumptions that are required by each respective convergence proof. Except for the last row, the assumptions listed in Table~\ref{tab:litconv} are not necessarily complete. For instance, \cite{Maurer} uses an additional assumption on the surjectivity of the linearized equality constraints and on the existence of a solution to a particular Riccati boundary value problem. Finally, some papers define the candidate space implicitly via boundedness of certain barrier-functionals \cite{Jones70}. Therefore, the table can only give a broad idea of the typical assumptions used when proving convergence of a certain type of method for a certain problem format.

As the table shows, different convergence analyses for different methods vary in the problem format that they treat and in the assumptions that they make. Usually, the candidate space is a Sobolev space and the functions $M,f_1,f_2,b$ from the problem statement are assumed to live in H\"older spaces. Only some of the literature results are for optimal control problems in the general format~\eqref{eqn:OCP}. Most convergence results are established in the measure of error, rather than optimality or feasibility.

\subsection{Assumptions for Convergence of Error}
As discussed along Figure~\ref{fig:minimizertypes} in Section~\ref{sec:ConvMeasures}, the analysis of convergence in terms of an error necessitates local uniqueness of the exact minimizer $y^\star,u^\star$. This is often established via a \emph{coercivity} assumption, which is a sufficient condition for a strict minimizer \cite{Maurer,Hager2000,RaoHager2018}.

Another typical assumption is the \emph{homogeneous rank assumption}, which is a relevant assumption for the uniqueness of the dual solution. Further assumptions are on the boundedness of either the exact minimizer $y^\star,u^\star$, the numerical minimizer $y^\star_h,u^\star_h$, on $f_1$, or on the objective function from below $-M$ or from above $M$.

\subsection{Striking a Good Balance}
Table~\ref{tab:litconv} highlights one cell of the most preferable attribute per column in yellow. For example, methods and convergence proofs for optimal control problems of general problem format are preferred over those that can only solve initial-value problems.

In contrast, when it comes to candidate spaces and assumptions, there are trade-offs. For instance, the insignificantly more general space $W^{1,2} \times L^2$ on the one hand necessitates significantly stronger assumptions on the smoothness of $f_1$ on the other hand, because potential singularities in $u$ must remain measurable in $f_1$.

With regards to boundedness, the absence of any explicit boundedness assumption seems to be replaced via a hidden boundedness assertion implied by the regularity of the optimality system, by virtue of assumptions on coercivity and homogeneous rank. From a computational engineering stance, the verification of a lower bound on $M$ is more practical than the verification of local coercivity.

\renewcommand{\arraystretch}{1.2}
\begin{table}[]
	\centering
	\resizebox{\columnwidth}{!}{%
		\begin{adjustbox}{angle=90}
			\begin{tabular}{
					!{\vrule width 3pt}
					>{\centering\arraybackslash}p{0.8cm}| 	
					>{\centering\arraybackslash}p{0.6cm}
					!{\vrule width 3pt}
					>{\centering\arraybackslash}p{2.0cm}|
					>{\centering\arraybackslash}p{2.4cm}
					!{\vrule width 3pt}
					>{\centering\arraybackslash}p{3.4cm}|
					>{\centering\arraybackslash}p{2.4cm}
					!{\vrule width 2pt}
					>{\centering\arraybackslash}p{1cm}|
					>{\centering\arraybackslash}p{1cm}
					!{\vrule width 3pt}
					>{\centering\arraybackslash}p{1cm}|
					>{\centering\arraybackslash}p{1cm}
					!{\vrule width 3pt}
				}
				\noalign{\hrule height 3pt}
				\multicolumn{2}{!{\vrule width 3pt}c!{\vrule width 3pt}}{\textbf{method}}     	& \multicolumn{2}{c!{\vrule width 3pt}}{\textbf{problem format}}             																				& \multicolumn{4}{c!{\vrule width 3pt}}{\textbf{assumptions}}                                                                        																											& \multicolumn{2}{c!{\vrule width 3pt}}{\textbf{convergence}}   				\\ \hline
				\textbf{\rotatebox{90}{type}}		& \textbf{\rotatebox{90}{reference}}		& \textbf{\rotatebox{90}{dynamics}} 									& \textbf{\rotatebox{90}{\parbox{2.0cm}{candidate\\ space $\mathcal{X}$}}}			& \textbf{\rotatebox{90}{\parbox{2.0cm}{smooth-\\ ness}}} 		& \textbf{\rotatebox{90}{\parbox{2.0cm}{bounded-\\ ness}}} 		& \textbf{\rotatebox{90}{coercivity}} 	& \textbf{\rotatebox{90}{\parbox{2.0cm}{homoge-\\ neousity}}} 	& \textbf{\rotatebox{90}{type}} 	& \textbf{\rotatebox{90}{order}}			\\ \hline \hline\noalign{\hrule height 1.5pt}
				EE            						& \cite{MR0267439}  						& IVP                   												&  implicit																			& $M,f \in \cC^0$             									& $y^\star,u^\star      $											&                     					&                   												& error     						& low         								\\ \hline
				EE            						& \cite{Maurer}  							& CIVP              													& $\cC^{1} \times \cC^0       				$										& $f \in\cC^{1}$               								& \cellcolor{yellow!25}												& \checkmark       						& \checkmark        												& error     						& low         								\\ \hline
				EE            						& \cite{MR1779739} 							& CIVP 																	& $W^{2,\infty} \times W^{1,\infty} 		$										& $M,f_1 \in\cC^{2,1};\ f_2 \in\cC^{3,1}$ 						& 																	& \checkmark 							& \checkmark        												& error     						& low         								\\ \hline
				EE            						& \cite{MR3119154} 							& general               												& $W^{1,\infty} \times L^\infty  			$										& $M,f,b \in\cC^{1,1}$           								& $y_h^\star,u_h^\star  $											&                     					&                   												& error     						& low         								\\ \hline
				IE            						& \cite{MR4046772} 							& CIVP             														& $W^{2,\infty} \times W^{1,\infty} 		$										& $M,f_1 \in\cC^{2,1};\ f_2 \in\cC^{3,1}$ 						& $                     $											& \checkmark          					& \checkmark        												& error     						& low         								\\ \hline \hline\noalign{\hrule height 1.5pt}
				LM            						& \cite{MR500418} 							& IVP                   												&  implicit									 										& $M,f_1 \in \cC^1$												& $y^\star,u^\star      $											&                     					&                   												& error     						& high        								\\ \hline
				RK            						& \cite{MR1770350} 							& IVP                   												& $W^{2,\infty} \times W^{1,\infty} 		$										& $M \in \cC^{0,1};\ f_1 \in\cC^{2,1}$    						& $u^\star              $											& \checkmark          					&                   												& error     						& high        								\\ \hline
				RK            						& \cite{Hager2000} 							& IVP                   												& $W^{1,\infty} \times L^\infty  			$										& $M,f \in\cC^{1,1}$             								&  																	& \checkmark          					& \checkmark        												& error     						& high        								\\ \hline
				RK            						& \cite{SchwartzPolak:1996} 				& IVP                   												& $W^{1,\infty} \times L^\infty  			$										& $M \in \cC^{0,1};\ f \in\cC^{1,1}$     						& $y^\star,u^\star      $											& \checkmark          					&                   												& error     						& high        								\\ \hline \hline\noalign{\hrule height 1.5pt}
				PS            						& \cite{RaoHager2018} 						& IVP                   												& $\cC^{1} \times \cC^0       				$										& $M,f_1 \in\cC^{1,1}$            								& $y^\star,u^\star      $											& \checkmark          					& \checkmark        												& error     						& high        								\\ \hline
				PS            						& \cite{gongEtAl:2006} 						& custom																& $W^{2,\infty} \times \cC^0    			$										& $M, f, b \in \cC^{0,1}$         								& $y_h^\star,u_h^\star  $											&                     					&                   												& error     						& high        								\\ \hline
				PS            						& \cite{gongEtAl:2008} 						& general               												& $W^{2,\infty} \times \cC^0    			$										& $M,f,b \in\cC^{1,1}$           								& $y_h^\star,u_h^\star  $											&                     					&                   												& error     						& high        								\\ \hline \hline\noalign{\hrule height 1.5pt}
				DC            						& \cite{reddien1979collocation} 			& IVP                   												& $W^{1,2} \times L^2       				$										& $M, f_1 \in\cC^{2}$            								&                      												& \checkmark          					& \checkmark        												& error     						& high        								\\ \hline
				DC            						& \cite{kameswaran_biegler_2008} 			& IVP                   												&  implicit									 										& $M,f_1 \in \cC^1$	   											& $y_h^\star,u_h^\star  $											&                     					&                   												& error     						& high        								\\ \hline
				DC            						& \cite{MR3983447} 							& IVP                   												& $\cC^{1} \times \cC^0       				$										& $M, f_1 \in\cC^{2,1}$           								& $y^\star,u^\star      $											& \checkmark          					& \checkmark        												& error     						& high        								\\ \hline \hline\noalign{\hrule height 1.5pt}
				AL            						& \cite{Hager90} 							& IVP                   												& $W^{1,\infty} \times L^\infty  			$										& $M, f_1 \in\cC^{2}$            								& $y_h^\star,u_h^\star  $											& \checkmark          					&                   												& error     						& low         								\\ \hline
				PM            						& \cite{Russell65} 							& IVP                   												&  implicit        																	& $M, f_1 \in\cC^{2}$            								& $M,y^\star            $											&                     					&                   												& G\&R      						& none        								\\ \hline
				PM            						& \cite{DeJulio70} 							& LTI IVP               												& $W^{1,2} \times L^2      					$										& $M \in \cC^0$                									& $-M,y^\star,u^\star   $											&                     					&                   												& G\&R      						& low         								\\ \hline
				PM            						& \cite{Balakrishnan68} 					& IVP                   												&  implicit  								 										& $M, f_1 \in\cC^{1}$            								& $-M, f_1, y^\star     $											&                     					&                   												& G\&R      						& none        								\\ \hline
				PM            						& \cite{MR0271512} 							& PDE         															& $W^{1,2} \times L^2  						$										& $M, f_1 \in\cC^{1}$            								& $-M                   $											& \checkmark          					& \checkmark        												& error     						& high        								\\ \hline
				PBM            						& \cite{Jones70} 							& CIVP          														&  implicit      					 												& \cellcolor{yellow!25}$M, f \in \cC^0$							& $f_1,y^\star,u^\star  $											&                     					&                   												& G\&R      						& none        								\\ \hline
				PBM            						& \cite{Neuenhofen2020AnIP} 				& general               												& \cellcolor{yellow!25}$W^{1,2} \times L^2      					$										& $M, f, b \in \cC^{0,1}$          								& $-M,y^\star,u^\star   $											&                     					&                   												& G\&R      						& high        								\\ \hline \hline\noalign{\hrule height 3pt}
				\multicolumn{10}{!{\vrule width 3pt}c!{\vrule width 3pt}}{\phantom{$\frac{\frac{a}{b}}{c}$}\textbf{The Method of This Paper}\phantom{$\frac{\frac{a}{b}}{c}$}}																																																																																																				\\ \hline \hline
				PM            						&  											& \cellcolor{yellow!25}general											& $W^{1,2} \times L^\infty    				$					& $M, f, b \in \cC^{0,\lambda}$									& $-M                   $											& \cellcolor{yellow!25}					& \cellcolor{yellow!25}												& G\&R      						& \cellcolor{yellow!25}high    				\\ \hline
			\end{tabular}
		\end{adjustbox}
	}
	\caption{List of literature results for the convergence analyses of various direct transcription methods. Abbreviations: ``EE''=Explicit Euler; ``IE''=Implicit Euler; ``LM''=Linear Multi-step; ``PS''=pseudo-spectral; ``DC''=Direct Collocation; ``AL''=Augmented Lagrangian; ``PM''=Penalty Method; ``CIVP''=constrained IVP; ``LTI''=linear time-invariant; ``G\&R''= optimality gap and feasibility residuals. Convergence orders $\geq 2$ are considered high. Non-numerical methods do not possess convergence orders. Attributes of each method are given in columns. Yellow highlights the most desirable characteristic for each attribute.}\label{tab:litconv}
\end{table}
\renewcommand{\arraystretch}{1.0}

\chapter{Direct Transcription via Collocation Methods}\label{sec:DirectCollocation}
Direct collocation methods are considered as the state of the art for numerically solving optimal control problems.
Direct collocation methods are direct transcription methods that generalize the idea of the explicit Euler method.

Because there exist many variants, heuristics, and deviations in technical details, this section presents direct collocation methods according to their default definition \cite[Def.~7.6]{Hairer1}. Our presentation is illustrated with one particular direct collocation method\footnote{Remark for experts: The illustrated method is Legendre-Gauss-Radau collocation. This causes no harm because this particular method is most widely used in practice and will also be used in our numerical experiments.}.

\section{Construction}
As discussed in Section~\ref{sec:Scope} and conveyed in Figure~\ref{fig:methodclasses}, all direct transcription methods must achieve two tasks: (a)~approximation of the states and controls, and (b)~relaxation of the constraints. The following two subsections explain how collocation methods construct the approximation and the relaxation, respectively.

\begin{figure}
	\centering
	\includegraphics[width=1\linewidth]{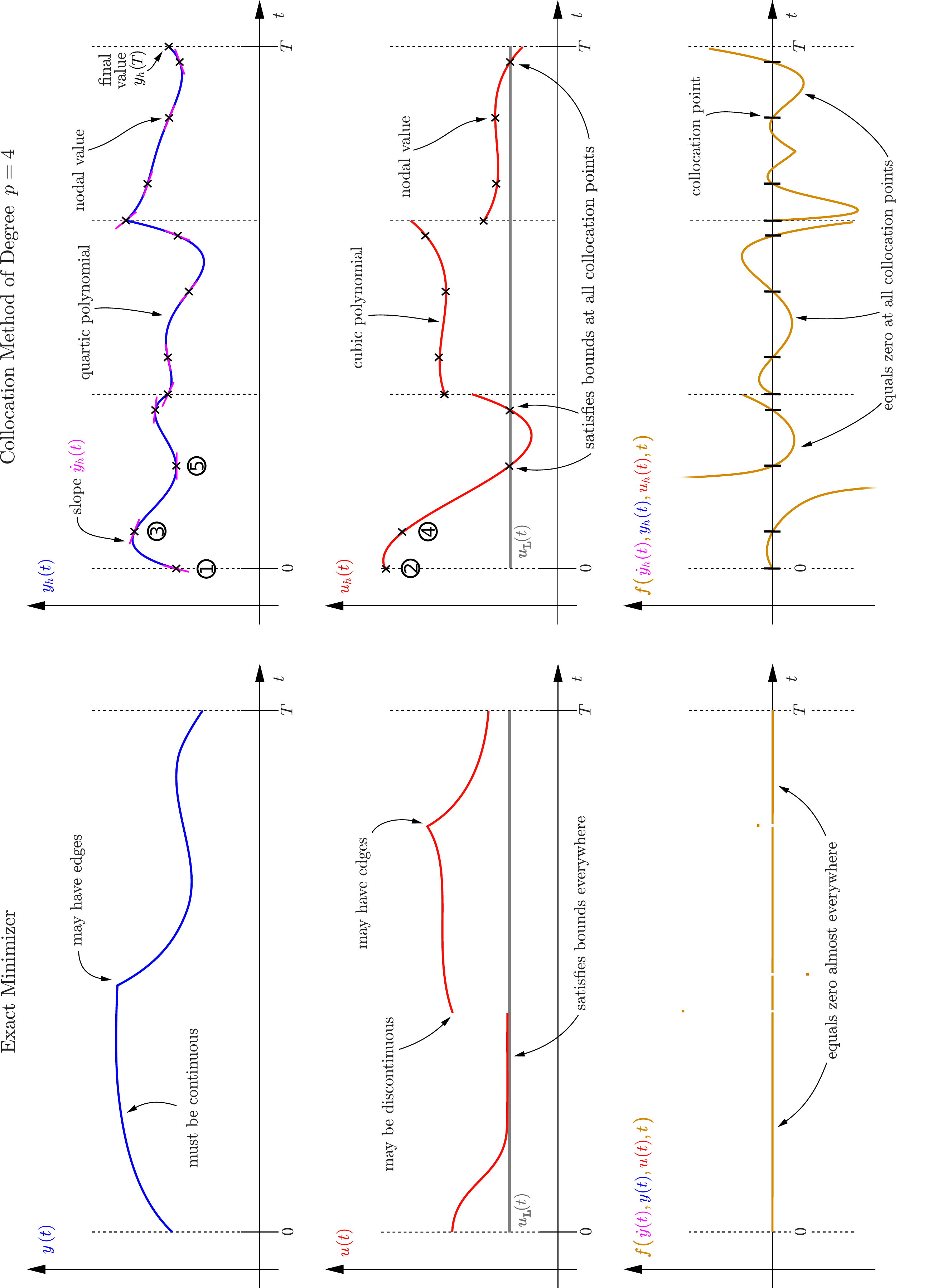}
	\caption{The construction of direct collocation methods. Collocation points and nodal values are indicated with markers and labels. The shapes of numerical functions are compared with functions of the exact minimizer.}
	\label{fig:collocationconstructv2rot}
\end{figure}

\subsection{Approximation of States and Controls}
As depicted in Figure~\ref{fig:speedchangeproblem} and in the optimization problem~\eqref{eqn:ExampleEuler:NLP}, direct transcription via the explicit Euler method uses nodal values $y^{(1)},u^{(1)},y^{(2)},u^{(2)},\dots$ for the states and controls. These nodal values were placed at the following fixed points of time:
\begin{align*}
y^{(1)}&\approx y(0 \cdot h)\,,& y^{(2)}&\approx y(1 \cdot h)\,,& y^{(3)}&\approx y(2 \cdot h)\,,\dots\,,\\
u^{(1)}&\approx u(0 \cdot h)\,,& u^{(2)}&\approx u(1 \cdot h)\,,& u^{(3)}&\approx u(2 \cdot h)\,,\dots\,.
\end{align*}

Collocation methods generalize this concept of the explicit Euler method. They pack the nodal values into groups of parametric size $p \in \N$, which is called the \emph{polynomial degree} of the collocation method, for reasons given in the next paragraph. Figure~\ref{fig:collocationconstructv2rot} shows an example of a collocation method for $p=4$. The separation into groups is illustrated with dashed vertical lines. The nodal values are indicated with crosses. We use the number $N\in\N$ to denote the number of groups. In the figure, $N=3$.

Collocation methods approximate the states $y$ and controls $u$ with piecewise polynomials. This is done by uniquely interpolating the nodal values of each group. Figure~\ref{fig:collocationconstructv2rot} illustrates this: For $y_h$, we interpolate $p+1$ points per group uniquely into a polynomial of degree $p$. For the depicted example, this generates a quartic polynomial. In contrast, for $u_h$, we interpolate only $p$ points per group, thus obtaining a polynomial of only degree $p-1$.

As depicted in Figure~\ref{fig:collocationconstructv2rot}, this particular way of interpolation results in a continuous function for $y_h$ and in a discontinuous function for $u_h$. In particular, the functions $y_h$ and $u_h$ are piecewise polynomials of degree $p$ and $p-1$, respectively. Both $y_h,u_h$ may have edges. This way of approximation makes sense because exact minimizers may have edges in $y,u$ and discontinuities in $u$, cf. Section~\ref{sec:SolutionFormat}.

\subsection{Relaxation of Constraints}
As depicted in Figure~\ref{fig:speedchangeproblem} and in~\eqref{eqn:ExampleEuler:NLP}, direct transcription via the explicit Euler method satisfies the differential equation only at the following fixed points:
$$ t=0\,,\quad t=h\,,\quad t=2\cdot h\,,\dots\quad t=(N-1) \cdot h\,.$$

Collocation methods take over this concept of the explicit Euler method by satisfying all constraints only at a finite number of fixed points for $t$. This concept is so fundamental to collocation methods that these fixed points for $t$ are called \emph{collocation points}. This means, collocation methods relax \eqref{eqn:OCP:dae} and (\ref{eqn:OCP}:y) and (\ref{eqn:OCP}:u) into
\begin{align*}
f\big(\,\dot{y}_h(t),y_h(t),u_h(t),t\,\big)&=\bO 	& &\text{ at each collocation point }t\,,\\
\yL(t)\leq y(t)&\leq\yR(t) 	& &\text{ at each collocation point }t\,,\\
\uL(t)\leq u(t)&\leq\uR(t) 	& &\text{ at each collocation point }t\,.
\end{align*}

\begin{remark}
	The name ``collocation method'' emerged from initial value problems. In an equation like~\eqref{eqn:IVP}, a polynomial $y_h$ of degree $p$ is used such that the equations $y_{\text{ini}}=y_h(0)$ and $\dot{y}_h(t) = f(y_h(t),t)$ are satisfied at $p$ fixed points $t$ on the interval $[0,T]$; thus, co-locating (=colloquere, Latin) the polynomial's derivative $\dot{y}_h(t)$ with the flux-function $f(y_h(t),t)$ at each co-location point $t$ in order to determine $y_h$ uniquely. The coefficients of the interpolating polynomial were typically found via Newton's method \cite[Sec.~3.1]{Hairer1}.
\end{remark}

All collocation methods have their own individual sets of collocation points. Figure~\ref{fig:collocationconstructv2rot} shows the location of the collocation points for a method of degree $p=4$ via axis tics.

\subsection{Restriction on the Number of Collocation Points}
In general, each collocation method must use exactly $p$ collocation points per group. The reason for this can be inferred from the above remark: For any other number of collocation points, the polynomials of $y_h$ would not be determined uniquely from co-location of~(\ref{eqn:OCP}:f1). Using more than $p$ points would result in an equation system that possesses no solution; using fewer than $p$ points would result in an equation that has infinitely many solutions. In either case, the method would fail because the collocation principle is supposed to (locally) uniquely determine the solution to the differential equations at hand. For a further discussion of the fact that collocation methods of polynomial degree $p$ must use precisely $p$ collocation points, we refer to \cite[Sec.~7]{Hairer1}.

We discuss below some effects of collocation.

\paragraph{Effects on Equality Constraints}
As per requirement of collocation methods, \mbox{$f\big(\,\dot{y}_h(t),y_h(t),u_h(t),t\,\big)=\bO$} must hold at all collocation points. Figure~\ref{fig:speedchangeproblem} depicts $f$ in orange. Outside of the collocation points, $f\big(\,\dot{y}_h(t),y_h(t),u_h(t),t\,\big)$ can take on arbitrary values, can have poles, discontinuities, and edges.  This is in contrast to exact minimizers, which yield $f\big(\,\dot{y}(t),y(t),u(t),t\,\big)=\bO$ almost everywhere, as is required per~\eqref{eqn:OCP:dae}.

\paragraph{Effects on Inequality Constraints}
Likewise, as per requirement of collocation methods, $\uL(t)\leq u_h(t)$ at all collocation points. Figure~\ref{fig:speedchangeproblem} depicts $\uL$ in grey. Outside of the collocation points, $\uL(t)\leq u_h(t)$ can be violated arbitrarily. This is in contrast to exact minimizers, which obey $\uL(t)\leq u(t)$ everywhere, as is required by~\mbox{(\ref{eqn:OCP}:y)--(\ref{eqn:OCP}:u)}.

\section{Implementation as NLP}
The purpose of the approximation and the relaxation is to transcribe the optimal control problem~\eqref{eqn:OCP} into an NLP. An example of an NLP is~\eqref{eqn:ExampleEuler:NLP}, which we used when demonstrating in Section~\ref{sec:Intro:OCP} how to solve the optimal control problem~\eqref{eqn:exampleOCP} with the explicit Euler method.

In the present section, we formalize the NLP that direct collocation methods use when solving problems of format~\eqref{eqn:OCP}. To improve readability, we introduce some notation first.

\subsection{Notation of Piecewise Polynomials and Collocation Points}\label{def:spaceXhp}

\paragraph{Mesh}
In Figure~\ref{fig:collocationconstructv2rot}, the arrangement into groups has been indicated with dashed vertical lines. These lines separate the span $[0,T]$ into $N$ non-overlapping \emph{intervals} $I_i:=[t_i,t_{i+1}]$, $i=1,\dots,N$. This separation is called \emph{mesh}. The points $t_i$ satisfy $0=t_1<t_2<\dots<t_{N+1}=T$.

The parameter $h$ denotes the \emph{mesh size}. This is the diameter of the longest interval $I_i$:
\begin{align*}
h := \operatornamewithlimits{max}_{1\leq i\leq N} t_{i+1}-t_i\,.
\end{align*}

\paragraph{Sets of Collocation Points}
We use the notation $\cT_{h,p}$ for a set of $N \cdot p$ distinct collocation points on $[0,T]$, such that there are $p$ collocation points on each mesh interval $I_i$.

In Figure~\ref{fig:collocationconstructv2rot}, the collocation points are spaced in the same pattern on each mesh interval $I_i$. This reflects the typical case. Section~\ref{sec:COL:examples} gives examples of popular collocation methods in the literature.

\paragraph{Spaces of Piecewise Polynomials}
Section~\ref{sec:SolutionFormat} introduced the candidate space $\cX$ of states $y$ and controls $u$ so that $y,u$ may have edges and $u$ may have discontinuities, cf. the illustration of exact minimizers in Figure~\ref{fig:collocationconstructv2rot}. The direct collocation method constructs $y_h,u_h$ via piecewise polynomials such that they live in $\cX$. We denote this space of piecewise polynomials with $\cX_{h,p}$. It holds that $\cX_{h,p} \subset \cX$ because all $(y_h,u_h) \in \cX_{h,p}$ satisfy the aforementioned conditions.

We now formally define and illustrate the space $\cX_{h,p}$ in preparation of later results and to improve understanding of possible shapes that $y_h,u_h$ can attain. We make use of the spaces $\cP_p(I)$, that contain all functions that equal a polynomial of degree $\leq p \in \N$ on a bounded interval $I \subset \R$. Figure~\ref{fig:spacepiecewisepolynomials}~(a) shows an example of a pathological function $\phi$ (in blue) that has infinitely many jumps\footnote{cf. Cantor function}, has edges, and has a pole. Regardless, $\phi$ lives in $\cP_3([a,b])$ because on the interior of the interval $[a,b]$ it takes on the form of a cubic polynomial that is indicated in red.

We can define $\cX_{h,p}$ via the spaces $\cP_p(I_i)$:
\begin{align}
\cX_{h,p}=\left\lbrace\ (y_h,u_h)\ \Bigg\vert\ y_h \in \cC^0([0,T])\ \land\ y_h \in \cP_{p}(I_i)\ \land\ u_h \in \cP_{p-1}(I_i) \quad \text{for } i=1,\dots,N. \ \right\rbrace\label{eqn:def:spaceXhp}
\end{align}
Figure~\ref{fig:spacepiecewisepolynomials}~(b) shows an example of piecewise polynomials $y_h,u_h$ on a non-equidistant mesh of $N=3$ intervals. The function $y_h$ is continuous and consists of piecewise parabolas. In contrast, $u_h$ is discontinuous and piecewise linear. Hence, $(y_h,u_h) \in \cX_{h,p}$ for $p=2$. In the depicted case, $u_h=\dot{y}_h$.

\begin{figure}
	\centering
	\includegraphics[width=0.85\linewidth]{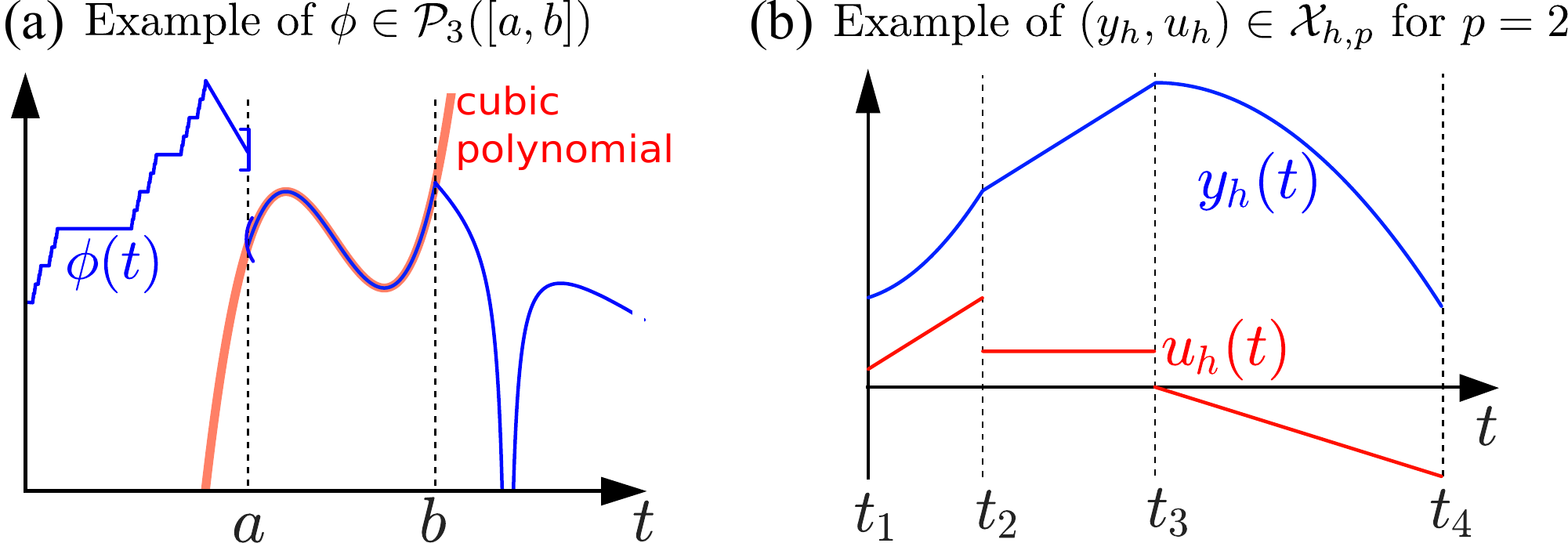}
	\caption{Examples of functions in the spaces $\cP_p(I)$ and $\cX_{h,p}$: (a) The space $\cP_3([a,b])$ contains all functions that take on the shape of a polynomial of degree $\leq 3$ on the interval $(a,b)$. (b) The space $\cX_{h,2}$ contains all piecewise polynomials where $y_h$ is of degree $2$ and $u_h$ of degree $1$.}
	\label{fig:spacepiecewisepolynomials}
\end{figure}

\subsection{NLP in Optimal Control Notation}
Using the notation with $\cX_{h,p}$ and $\cT_{h,p}$, we can state the transcribed optimal control problem in collocation methods as follows:
\begin{equation}
\label{eqn:COL}
\left\lbrace
\begin{aligned}
&\operatornamewithlimits{min}_{(y_h,u_h) \in \cX_{h,p}} & M \big(\,y_h(0),y_h(T)\,\big) & & &\\
& \text{subject to} & b\big(\,y_h(0),y_h(T)\,\big) &= \bO\,,\\
& 					& f\big(\,\dot{y}_h(t),y_h(t),u_h(t),t\,\big)&=\bO &\forall\ &t \in \cT_{h,p}\,,\\
& 					& \yL(t) \leq y_h(t) &\leq \yR(t)\quad&\forall\ &t \in \cT_{h,p}\,,\\
& 					& \uL(t) \leq u_h(t) &\leq \uR(t)\quad&\forall\ &t \in \cT_{h,p}\,.
\end{aligned}
\right\rbrace
\end{equation}
In the problem statement, the states and controls are approximated with piecewise polynomials in the space $\cX_{h,p}$. The differential and algebraic constraints~(\ref{eqn:OCP}:f) together with the bound constraints~\mbox{(\ref{eqn:OCP}:y)--(\ref{eqn:OCP}:u)} are relaxed to the collocation points in the set $\cT_{h,p}$.

\subsection{NLP in Standard Notation}\label{sec:COL:optim}\label{sec:NLP_of_DCM}
In order to solve \eqref{eqn:COL} with available numerical algorithms, it is helpful to re-express \eqref{eqn:COL} in the format
\begin{align}
{\left\lbrace
	\label{eqn:COL:NLP}
	\begin{aligned}
	&\operatornamewithlimits{min}_{\bx \in \R^{n_{\bx}}}  	& 	&\bf(\bx)\\
	&\text{subject to}										& 	&\bc(\bx)=\bO\,,\\
	&       												& 	&\bbL \leq \bA \cdot \bx \leq \bbR\,.
	\end{aligned}
	\right\rbrace}
\end{align}
We explain in the following how this can be achieved.

The functions $y_h,u_h$ can be identified with a vector $\bx \in \R^{n_{\bx}}$ of dimension
\begin{align*}
n_{\bx} := p \cdot \big( N \cdot (n_y+n_u) + n_y \big)\,.
\end{align*}
This vector contains all the nodal values of $y_h,u_h$ that are illustrated in Figure~\ref{fig:collocationconstructv2rot}. The figure also shows encircled numbers. These give one possible order in which the nodal values of $y_h,u_h$ and the final value $y_h(T)$ can be listed in $\bx$. We write
\begin{align*}
\bx = \left[\begin{array}{c}
\vdots\\
y_h(t)\\
u_h(t)\\
\vdots\\
\hline
y_h(T)
\end{array}\right]\quad\forall t \in \cT_{h,p}\,.
\end{align*}
I.e., $\bx$ is a vector that contains a list of all values of $y_h,u_h$ at all collocation points -- plus the additional final node of at $y_h(T)$ because $t=T$ is not a collocation point in the depicted setting of Figure~\ref{fig:collocationconstructv2rot}. Given the values of $y_h(t),u_h(t)$ at the abscissae $t$ in $\bx$, we can interpolate the piecewise polynomials of $y_h,u_h$ over $[0,T]$ uniquely. The values in $\bx$ define the vertical position of the black crosses in Figure~\ref{fig:collocationconstructv2rot}. We see that the blue curve is continuous while the red interpolation is not. This is so because the blue interpolation takes one additional point per mesh interval.

Using $\bx$, we can evaluate the functions $y_h,u_h$ that $\bx$ represents. We do this to construct the properties $\bf,\bc,\bA,\bbL,\bbR$ in \eqref{eqn:COL:NLP}:
\begin{align*}
\bf(\bx)&:= M\big(y_h(0),y_h(T)\big)\,, & &\\[3pt]
\bc(\bx)&:= \left[\begin{array}{c}
b\big(y_h(0),y_h(T)\big)\\[2pt]
\hline
\vdots\\
f\big(\dot{y}_h(t),{y}_h(t),u_h(t),t\big)\\
\vdots
\end{array}\right]\quad \forall t \in \cT_{h,p}\,,\qquad&
\bbL&:=\begin{bmatrix}
\vdots\\
\yL(t)\\
\uL(t)\\
\vdots
\end{bmatrix}\,,\ 
\bbR:=\begin{bmatrix}
\vdots\\
\yR(t)\\
\uR(t)\\
\vdots
\end{bmatrix}\quad \forall t \in \cT_{h,p}\,.
\end{align*}
The matrix $\bA$ is constructed such that
\begin{align*}
\bA \cdot \bx = \begin{bmatrix}
\vdots\\
y_h(t)\\
u_h(t)\\
\vdots
\end{bmatrix}\quad\forall t \in \cT_{h,p}\,.
\end{align*}
From the above definitions, we obtain a vectorial function $\bc : \R^{n_{\bx}} \rightarrow \R^{n_{\bc}}$ and vectors $\bbL,\bbR\in\R^{n_{\bb}}$ of dimensions
\begin{align*}
n_{\bc} &:= n_b + N \cdot p \cdot (n_y + n_c)\,,\qquad &n_{\bb} &:= N \cdot p \cdot (n_y + n_u)\,.
\end{align*}

\subsection{Numerical Solution of the NLP}
The problem~\eqref{eqn:COL:NLP} matches precisely with the problem format~\eqref{eqn:NLP} in Section~\ref{sec:NLP} below. Hence, the method presented in Section~\ref{sec:NLP}, which discusses numerical solution algorithms for NLP, can be uesd to solve this optimization problem numerically.

Some optimal control problems have a large number of states $n_y$ and/or of controls $n_u$. Further, some problems have very long time-horizons $[0,T]$ or dynamic phenomena that require a very small mesh-size $h$ in order to resolve. In any of these scenarios, the dimensions $n_\bx,\ n_\bc,\ n_\bb$ of the NLP can become very large. This can lead to numerical issues in the NLP solver. One issue is with the linear equation system: The time and RAM needed to solve the system increases with the dimension; and the accuracy of the computed solution decreases. There is a point where the system is so large that the computer runs out of RAM or the user runs out of computation time or the computed solution is uselessly inaccurate. Another issue is with the convergence of the NLP solver: For larger problems, the rate of convergence from a remote initial guess $\bx_0$ towards a local minimizer can be very slow. This can result in an impractically large amount of NLP solver iterations in order to converge, thereby rendering the solution procedure impractical.

\section{Examples of Collocation Methods}\label{sec:COL:examples}

Collocation methods differ only in their polynomial degree $p$ and in the way how they place the collocation points.
Typically, the collocation points are placed in the same way on each mesh interval. An example of this can be best observed in the graph of $u_h$ in Figure~\ref{fig:collocationconstructv2rot}: Each mesh interval uses the same four points. Mapped on the reference interval $\Iref = [-1,1]$, these points are (rounded to four digits):
\begin{align*}
\cT_{\text{ref},4}^\text{LGR} = \lbrace\,-1.0000,\,-0.5753,\,0.1811,\,0.8228\,\rbrace\,.
\end{align*}
We call a set on $\Iref$ a \emph{reference set}. In collocation methods of degree $p$, reference sets of collocation points hold $p$ points.

We use the subscript $\text{ref}$ to indicate reference sets $\cT_{\text{ref},p}$. We write $\cT_{i,p}$ for the transformation of $\cT_{\text{ref},p}$ from $\Iref$ onto $I_i$. We write $\cT_{h,p}$ for the union of all sets $\cT_{i,p}$ for $i=1,\dots,N$ for a mesh of size $h$.

In the above example, the superscript LGR denotes the method name: Legendre-Gauss-Radau collocation. This is one particular collocation method. In the following, we introduce several examples.

\paragraph{Explicit and Implicit Euler Method}
Direct collocation with the explicit Euler (EE) and the implicit Euler (IE) method is of degree $p=1$ and uses the collocation sets
\begin{align*}
\cT_{\text{ref},1}^\text{EE} = \lbrace\,-1.0000\,\rbrace\,,\qquad \cT_{\text{ref},1}^\text{IE} = \lbrace\,1.0000\,\rbrace\,.
\end{align*}
These examples are given to improve accessibility of the notation.

\paragraph{Trapezoidal Method}
The trapezoidal method (TZ) is a collocation method of degree $p=2$. It uses the collocation set
\begin{align*}
\cT_{\text{ref},1}^\text{TZ} &= \lbrace\,-1.0000,\,1.0000\,\rbrace\,.
\end{align*}
This example is given because we will use the trapezoidal method later for one numerical demonstration. Due to its symmetry, the trapezoidal method has some interesting properties such as symplecticity \cite{HairerSymplect}.

\paragraph{Gauss-Legendre Collocation Method}
The Gauss-Legendre (LG) collocation method is an \emph{orthogonal collocation method}, meaning that it uses collocation points that are the roots of orthogonal polynomials. The method is of parametric degree $p$. The collocation points of the LG collocation of degree $p$ are the $p$ roots of the Gauss-Legendre polynomial \cite{GaussQuad} of degree $p$. For $p=1,\dots,4$, these are (to four digits):
\begin{align*}
\cT_{\text{ref},1}^\text{LG} &= \lbrace\, 0.0000\,\rbrace\,, &\cT_{\text{ref},2}^\text{LG} &= \lbrace\,-0.5774,\,0.5774\,\rbrace\,,\\ \cT_{\text{ref},3}^\text{LG} &= \lbrace\,-0.7746,\,0.0000,\,0.7746\,\rbrace\,, & \cT_{\text{ref},4}^\text{LG} &= \lbrace\,-0.8611,\,-0.3400,\,0.3400,\,0.8611\,\rbrace\,.
\end{align*}
The Gauss-Legendre collocation method will be used later in one of the numerical illustrations.

As is known, the Gauss-Legendre quadrature uses these points in conjunction with suitable positive quadrature weights to approximate integrals. We will also use quadrature with Gauss-Legendre points later in Section~\ref{sec:DQIPM} for the construction of a particular quadrature penalty method.

\paragraph{Gauss-Legendre-Radau Collocation Method}
The Gauss-Legendre-Radau (LGR) collocation method is another example of an orthogonal collocation method. LGR collocation of degree $p$ uses the $p$ roots of the Gauss-Legendre-Radau polynomial of degree $p$. Just like Gauss-Legendre polynomials, the Gauss-Legendre-Radau polynomials are a particular sequence of orthogonal polynomials. For $p=1,\dots,3$, these are (to four digits):
\begin{align*}
\cT_{\text{ref},1}^\text{LGR} &= \lbrace\,-1.0000\,\rbrace\,, &\cT_{\text{ref},2}^\text{LGR} &= \lbrace\,-1.0000,\,0.3333\rbrace\,, & \cT_{\text{ref},3}^\text{LGR} &= \lbrace\,-1.0000,\,-0.2899,\,0.6899\,\rbrace\,.
\end{align*}
The set $\cT_{\text{ref},4}^\text{LGR}$ is given above.

\paragraph{Chebyshev-Gauss-Lobatto Points}
The Chebyshev-Gauss-Lobatto (CGL) points of degree $m \in \N$ on $\Iref$ are the $m+1$ points
\begin{align}
\tau_k:=-\cos\left(\frac{k}{m} \cdot \pi\right)\qquad \text{for }\ k=0,1,\dots,m\,. \label{eqn:cheblob}
\end{align}
These are the constrained extrema of the $m^\text{th}$ Chebyshev polynomials of the first kind \cite{ChebLob} on $I_{\text{ref}}$. We write them into the set $\Tclref{m}$. The CGL collocation method of degree $p$ uses the CGL points of degree $m=p-1$ as collocation points. 

We write $\cT_{i,m}$ for the transformation of $\cT^{\text{CGL}}_{\text{ref},m}$ from $\Iref$ onto $I_i$. We write $\cT^{\text{CGL}}_{h,m}$ for the union of all $\cT^{\text{CGL}}_{i,m}$ for $i=1,\dots,N$ for a mesh of size $h$. We will use the CGL\footnote{The CGL points are unrelated to the Legendre-Gauss-Lobatto (LGL) points. This thesis does not use LGL points anywhere.} points later in Section~\ref{sec:DQIPM} for the construction of a particular quadrature penalty method.

\chapter{Numerical Solution of NLP}\label{sec:NLP}
The gradient-descent method is for minimizing unconstrained optimization problems in $\R^n$.
This section discusses the more general class of NLP. These are optimization problems in $\R^n$ where the objective and the constraints may be nonlinear and non-convex.

Figure~\ref{fig:frameworkdirecttranscriptionv2} illustrates the steps to be followed in order to numerically solve an optimal control problem: The procedure begins with the direct transcription method, that is used to approximate the optimal control problem with an NLP. However, the direct transcription is only the first step, and involves no computational cost\footnote{because it only specifies the formulas that a computer program is supposed to implement}. Thus, the second step is much more crucial: the numerical solution of this NLP.

Each direct transcription method generates a different NLP in order to solve the same optimal control problem. These NLP may differ dramatically in the computational cost that is required to solve them. If an NLP cannot be solved efficiently then the respective direct transcription method is practically useless. Hence, the discussion of solution algorithms for NLP is vital to the assessment of direct transcription methods.

We first discuss the format and properties of NLP. We then review one important class of practical algorithms for solving NLP. The presented class of algorithms is suitable in particular for those NLP that result from direct transcription methods.

\section{Format and Properties of NLP}
In the following we introduce one possible standard problem format of NLP together with a few relevant properties.

\subsection{Problem Format}\label{sec:NLP:format}%, Optimality System, and Quasi-Newton Methods}
Without loss of generality, NLPs can be posed in the following form \cite{IPOPT}:
\begin{align}
{\left\lbrace
	\begin{aligned}
	&\operatornamewithlimits{min}_{\bx \in \R^{n_{\bx}}}  & 	&\bf(\bx)\\
	&\text{subject to}									& 	&\bc(\bx)=\bO\,,\\
	&       											& 	&\bbL \leq \bA \cdot \bx \leq \bbR\,.
	\end{aligned}
	\right\rbrace}\label{eqn:NLP}
\end{align}
In this problem, one seeks a local minimizer $\bx \in \R^{n_{\bx}}$ that minimizes $\bf : \R^{n_{\bx}} \rightarrow \R$ subject to satisfying $n_\bc \in \N_0$ equality constraints $\bc\,:\,\R^{n_\bx} \rightarrow \R^{n_\bc}$ and $n_\bb \in \N_0$ linear inequality-constraints with $\bbL,\bbR \in \R^{n_\bb}$, $\bbL<\bbR$ and $\bA \in \R^{n_\bb \times n_\bx}$. 

Notice that $\bf$ has nothing to do with $f_1,f_2$; and $\bbL,\bbR$ have nothing to do with $b$. We opted to keep standard notation of each discipline. This keeps each subject readable in separate and avoids use of esoteric/obscure symbols.

\subsection{Degrees of Freedom, Determination, and Feasibility}\label{sec:overdetermination}
In problem~\eqref{eqn:NLP}, the dimension $n_\bx$ of the local minimizer $\bx$ is called the number of \emph{degrees of freedom}, or for short just degrees of freedom.

The \emph{determination} of an NLP describes the relation between the number of degrees of freedom and the number of constraints. The number of equality constraints $n_\bc$ in \eqref{eqn:NLP} is typically expected\footnote{Some versions of IPOPT, SNOPT, and Knitro reject problems when $n_\bc>n_\bx$.} to be bounded by $n_\bx$. This is because it is known from linear algebra that a vector $\bx$ can solve up to $n_\bx$ independent linear equations. If there are more than $n_x$ equations then the problem~\eqref{eqn:NLP} is called \emph{overdetermined} because there are potentially too many constraints to be satisfied.

The term \emph{feasibility} describes whether or not there exist vectors $\bx \in \R^{n_\bx}$ that satisfy the constraints of \eqref{eqn:NLP}. As a rule of thumb, each constraint in $\bc(\bx)=\bO$ takes one degree of freedom. Intuitively, a problem is more likely to be infeasible when it has fewer degrees of freedom. Likewise, overdetermined problems are more likely to be infeasible. Inequality constraints can render an NLP infeasible as well; e.g., when $\bA\cdot\bx=\bbL$ has no solution and $\bbL,\bbR$ are close together.

Overdetermination/infeasibility are undesirable because they imply that it is challenging/impossible to find a feasible local minimizer $\bx$ to~\eqref{eqn:NLP}.

\subsection{Optimality Conditions}
For each local minimizer $\bx$ of \eqref{eqn:NLP} there exist Lagrange multipliers $\by \in \R^{n_\bc}$, $\bzL,\bzR \in \R^{n_\bx}_{\geq 0}$ such that the \emph{Karush Kuhn-Tucker equations} are satisfied \cite{MR2936770,Nocedal}. We call $\bx$ the \emph{primal solution} and $\by,\bzL,\bzR$ the \emph{dual solution} \cite{ForsgrenGillSIREV}. Notice that the Lagrange multipliers $\by$ have nothing to do with the states $y$. We consider a regularization of the Karush Kuhn-Tucker equations:
\begin{subequations}
	\label{eqn:KKT}
	\begin{align}
	\nabla_\bx \bf(\bx) - \nabla_\bx \bc(\bx) \cdot \by - \bA\t \cdot (\bzL-\bzR) &= \bO\,,\\
	\bc(\bx)+\omega \cdot \by &= \bO\,,\\
	\opdiag(\bzL) \cdot (\bA \cdot \bx - \bbL)- \mu \cdot \be &=\bO\,,\\
	\opdiag(\bzR) \cdot (\bbR - \bA \cdot \bx)- \mu \cdot \be &=\bO\,.
	\end{align}
\end{subequations}
This is a system of $n_\bx+n_\bc + 2 \cdot n_\bb$ nonlinear equations for the~\mbox{$n_\bx+n_\bc + 2 \cdot n_\bb$} variables $\bx,\by,\bzL,\bzR$. In this equation system, $\omega,\mu \in \R_{>0}$ are small regularization parameters to make the equations easier to solve numerically.

\section{Optimization Algorithms for NLP}\label{sec:IPM}\label{sec:NLP:solver}
Direct collocation methods transcribe optimal control problems into NLP. These NLP become very large as $h$ decreases (because then $N$ increases). So-called \emph{sparse optimization algorithms} \cite[Chap.~2]{BettsChap2} must be used in order to solve these very large NLP. As we saw in the historical background in Section~\ref{sec:literatureReview:History}, direct transcription methods became popular only after these sparse optimization algorithms were invented. This is plausible because direct transcription into an NLP makes no practical sense unless that NLP can be solved reliably and with low computational effort.

To provide a picture how NLPs can be solved in this way, we review superficially one widely used class of algorithms for large sparse optimization. Our review is based on Newton iterations and a gradient-descent method. The main purpose of this section is to provide an understanding of the fundamental aspects that determine the computational cost of solving large NLP. These are: the number of iterations, and the computational cost per iteration. The latter is determined by the so-called \emph{sparsity pattern} of the NLP, a property that we discuss at the end.

\paragraph{Primal-Dual Optimization Algorithms}
Many solution algorithms for NLP work by solving \eqref{eqn:KKT} via application of Newton iterations for the primal and the dual solution. We call the left-hand sides in \eqref{eqn:KKT} the \emph{KKT residuals}. Hence, Newton's method is supposed to drive the KKT residuals to zero.

Here we review superficially one particular class of solution algorithms in order to highlight the computational cost and the working principle of Newton-based solution algorithms for constrained optimization. Section~\ref{sec:IPM} explains the working principle. Afterwards, Section~\ref{sec:NLP:cost} discusses aspects of computational cost. We refer to \cite{ForsgrenGill1998,Nocedal} for detailed recipes of solution algorithms. Below, we only illustrate the principle ideas.

Figure~\ref{fig:newtoniteration} presents a schematic diagram for the widely-used class of primal-dual penalty-interior-point line search methods \cite{Renke,IPOPT,ForsgrenGill1998,ForsgrenGillSIREV}. The steps in the figure will be explained in the following paragraphs. Methods of the depicted class are iterative in two levels: The \emph{outer loop} is depicted in violet, and the \emph{inner loop} is in light blue and light beige. Methods are initialized in (S.1) with a solution guess for \eqref{eqn:KKT} and regularization parameters of, e.g., $\omega=\mu=0.1\,$.

\begin{figure}
	\centering
	\includegraphics[width=0.8\linewidth]{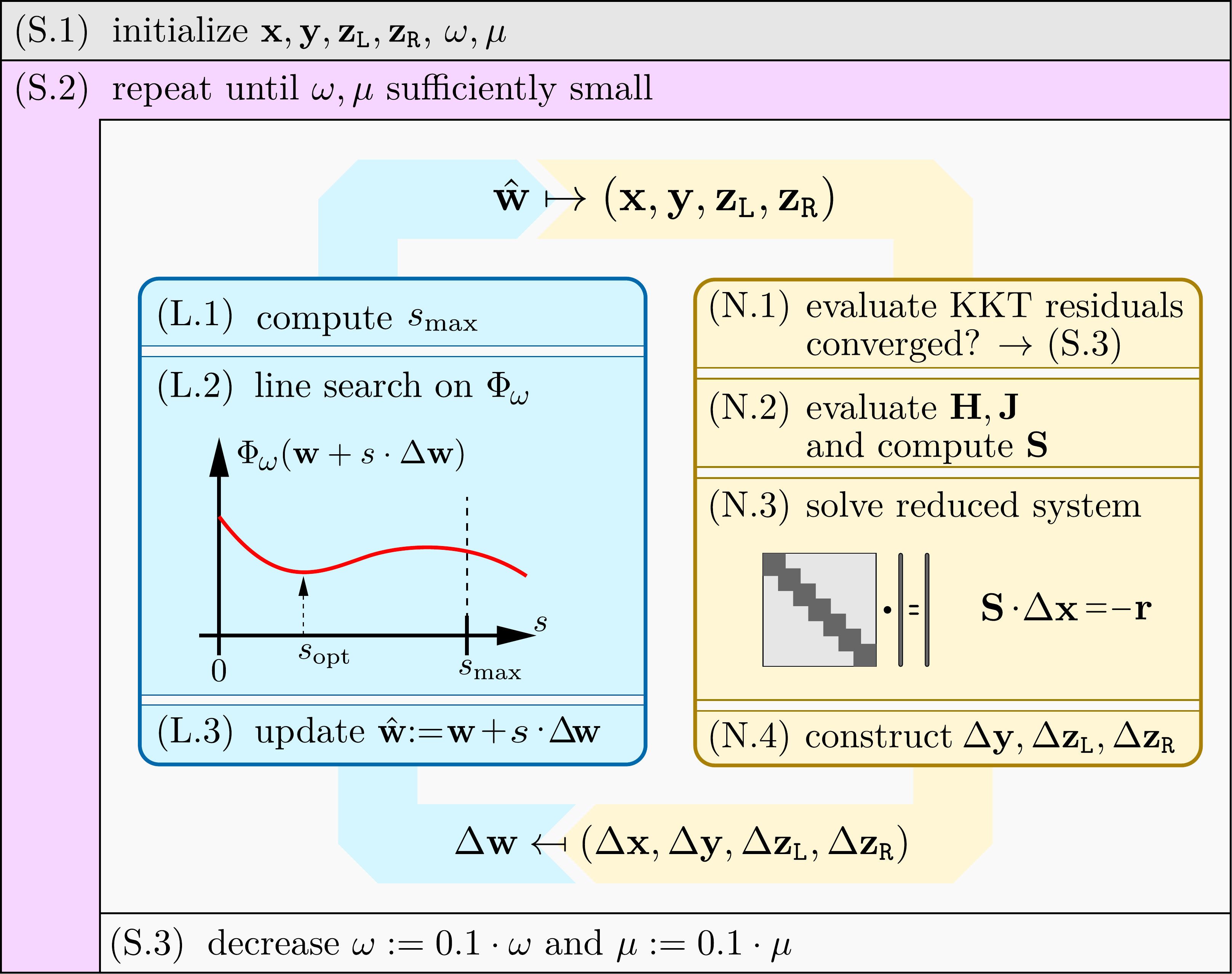}
	\caption{Working principle of primal-dual penalty-interior-point line search algorithms for the solution of \eqref{eqn:NLP}, based on line search (light blue) and Newton iterations (light beige) within an outer loop (violet).}
	\label{fig:newtoniteration}
\end{figure}

\subsection{Outer Loop}
Each outer loop (S.2) poses one minimization problem of the following form for one particular value of $\omega$:
\begin{align}
\label{eqn:Merit}
\left\lbrace
\begin{aligned}
&\operatornamewithlimits{min}_{\bx \in \R^{n_\bx}}  	& 	&\Phi_\omega(\bx):=\bf(\bx)+\frac{1}{2 \cdot \omega} \cdot \|\bc(\bx)\|_2^2\\
&\text{subject to}      								& 	&\bbL \leq \bA \cdot \bx \leq \bbR\,.
\end{aligned}
\right\rbrace
\end{align}
E.g., the first outer loop poses this problem for $\omega=10^{-1}$, the second for $\omega=10^{-2}$, and so on. The job of the inner loop is to compute a minimizer $\bx$ to the problem that the outer loop poses. The entire algorithm terminates and returns $\bx$ as a minimizer of \eqref{eqn:NLP} once $\omega,\mu$ are smaller than some prescribed tolerance.

In problem~\eqref{eqn:Merit}, the function
\begin{align}
\Phi_\omega(\bx) = \bf(\bx) + \frac{1}{2 \cdot \omega} \cdot \|\bc(\bx)\|_2^2
\end{align}
is called \emph{merit function}. This function is a bias of objective value and constraint violation \cite{Nocedal}. There are several different merit functions. The merit function depicted here is also called \emph{quadratic penalty function} \cite{Armand,ForsgrenGill1998} in the literature.

The inner loops solve the problem~\eqref{eqn:Merit} with the current outer loop's parameters $\omega,\mu$ in an iterative manner. Each inner loop performs two alternating steps:
i) Newton iteration;
ii) line search.

\subsection{Inner Loop}
In the inner loop, the Newton iteration is applied to find a solution $\bx,\by,\bzL,\bzR$ that drives the KKT residuals to zero, i.e., that solves the equation system \eqref{eqn:KKT}. One Newton iteration computes one \emph{Newton step} $\Delta\bx,\Delta\by,\Delta\bzL,\Delta\bzR$.

Linearizing \eqref{eqn:KKT} gives the \emph{Newton matrix}, that is used to compute the Newton step:
\begin{align}
\begin{bmatrix}
\bH 					& -\bJ\t 			&-\bA\t  				 		&\bA\t 	 							\\
\bJ 					& \omega \cdot \bI	& 		 				 		& 		 							\\
\phantom{-}{}\opdiag(\bzL)\cdot \bA 	& 			 		&\opdiag(\bA \cdot \bx-\bbL)	& 		 							\\
-{}\opdiag(\bzR)\cdot \bA & 					& 						 		& \opdiag(\bbR - \bA \cdot \bx)
\end{bmatrix} \label{eqn:KKTmatrix}
\end{align}
In~\eqref{eqn:KKTmatrix}, the symmetric matrix $\bH=\nabla^2_{\bx,\bx}\big(\bf(\bx)-\by \cdot \bc(\bx)\big) \in \R^{n_\bx \times n_\bx}$ is called \emph{Hessian of the Lagrangian} and $\bJ=\nabla_\bx\bC(\bx)\t \in \R^{n_\bx \times n_\bc}$ is called \emph{Jacobian of the equality constraints}.

The Newton iteration evaluates the KKT residuals from \eqref{eqn:KKT} at the current iterate $\bx,\by,\bzL,\bzR$. This is step (N.1) in Figure~\ref{fig:newtoniteration}, which also checks for convergence of Newton's method: If the KKT residuals are small then the inner loop terminates and $\omega,\mu$ are decreased (S.3)\,. Otherwise, Newton's method continues in (N.2) by evaluating the derivative matrices $\bH,\bJ$ at $\bx,\by$. The Newton step is computed in (N.3) by solving the Newton system. To save computations, this can be done in a reduced manner.

\paragraph{Reduced Linear System}
The authors in \cite{ALMIPM,Nocedal} explain how the Newton direction can be computed from a reduced system of linear equations with the matrix $\bS \in \R^{n_{\bx} \times n_{\bx}}$ below, where $\bD \in \R^{n_{\bb} \times n_{\bb}}$ is a positive definite diagonal matrix:
\begin{align*}
\underbrace{\left(\bH + \frac{1}{\omega} \cdot \bJ\t\cdot\bJ+\bA\t\cdot \bD \cdot \bA\right)}_{=:\bS} \cdot \Delta\bx = -\br\,. \tageq\label{eqn:ReducedKKTlinsys}
\end{align*}
In optimal control applications, this matrix is narrowly banded, as is illustrated in Figure~\ref{fig:newtoniteration}. This has to do with structural properties of the matrices $\bH,\bJ,\bA$, as will be discussed later in Section~\ref{sec:CostCompare:Sparsity} along Figure~\ref{fig:sparsity}.

A formula for the construction of $\br$ can be found in \cite[eqn.~10]{ALMIPM}.
Using the partial solution $\Delta\bx$ of the reduced linear equation system, the other parts $\Delta\by,\Delta\bzL,\Delta\bzR$ of the Newton step (N.4) can be reconstructed as described in \cite[eqn.~11]{ALMIPM}.

\paragraph{Line Search}
There are two kinds of methods that can be used to improve the convergence of the Newton method: line search methods and trust region methods.

Primal-dual penalty-interior-point line search methods \cite{Renke,ForsgrenGill1998} are popular. Their idea is to use the Newton direction $(\Delta\bx,\Delta\by,\Delta\bzL,\Delta\bzR)$ in a line search of step size $s \in ]0,s_\text{max}]$ such that the updated iterate \mbox{$(\hbx,\hby,\hbzL,\hbzR)=(\bx+s\cdot\Delta\bx,\by+s\cdot\Delta\by,\bzL+s\cdot\Delta\bzL,\bzR+s\cdot\Delta\bzR)$} satisfies the conditions $\bbL<\bA\cdot\hbx<\bbR$ and $\hbzL,\hbzR>\bO$ strictly. The strictness is asserted via the bound $s_\text{max}$, which is determined by a formula called \emph{fraction-to-boundary rule} \cite{IPOPT}. The value $s$ is selected such that $0<s\leq s_\text{max}$. We choose $s$ so to approximately minimize the merit function $\Phi_\omega$.

All in all, the line search achieves two purposes: i) keeping $\bbL<\bA\cdot\bx<\bbR$ and $\bzL,\bzR>0$; ii) promoting convergence to a minimizer of \eqref{eqn:Merit} by minimizing $\Phi_\omega$ along the Newton direction.

\section{Computational Cost}\label{sec:NLP:cost}
The computational cost for solving \eqref{eqn:NLP} via primal-dual interior-point methods can be decoded from Figure~\ref{fig:newtoniteration}: The steps (N.2) and (N.3) perform expensive matrix computations, whereas in contrast all the other \mbox{steps --- i.e.,~(S.1),} (S.2), (S.3), (L.1), (L.2), (L.3), (L.4), (N.1), \mbox{and~(N.4) --- only} perform negligibly cheap vector computations. The references \cite{IPOPT,ForsgrenGill1998,Nocedal} confirm that in Newton-type optimization algorithms the computation of derivatives and of solutions to the Newton systems overshadow the computational cost of everything else.

Our particular interest is in the computational cost when solving optimization problems that arise from direct transcription of optimal control problems. This case is depicted in Figure~\ref{fig:costdiagram}~(a) and will be discussed in Section~\ref{sec:CostCompare}.

\part{Quadrature Penalty Methods} 
\label{part:QPM}

\chapter{Limitations in Direct Collocation Methods}\label{sec:Misconceptions}
This part presents quadrature penalty methods as a reliable and practical direct transcription method.
We showcase how the quadrature penalty method surpasses numerical issues that collocation methods suffer from.
We first identify these issues. Afterwards, we present the construction of quadrature penalty methods. Eventually, we compare computational cost and provide numerical experiments.

There are two main limitations in the design of direct collocation methods. These are on the convergence of the feasibility to the equality constraints and inequality feasibility, respectively. Each limitation will be illustrated in a separate subsection with an example.

\section{Limitation on Equality Feasibility}
\subsection{Example}
Consider the below optimal control problem over the interval~\mbox{$[0,1]$}:
\begin{equation}
\left\lbrace
\begin{aligned}
&\operatornamewithlimits{min}_{y,u}  		& 	&y_{[3]}(1)\\
&\text{subject to} 							& 	&y_{[1]}(0)=0\,,\ \dot{y}_{[1]}(t)=u(t)\,,\\
& 											& 	&y_{[2]}(0)=0\,,\ \dot{y}_{[2]}(t)=-u(t)\,,\\
& 											& 	&y_{[3]}(0)=0\,,\ \dot{y}_{[3]}(t)=y_{[1]}(t)\,,\\
&											& 	&y_{[1]}(t)^2 = y_{[2]}(t)\,.
\end{aligned}
\right\rbrace\quad \begin{matrix}
(\ref{eqn:col_counter}\text{a})\\[7pt]
(\ref{eqn:col_counter}\text{b})\\[3pt]
(\ref{eqn:col_counter}\text{c})\\[3pt]
(\ref{eqn:col_counter}\text{d})\\[6pt]
(\ref{eqn:col_counter}\text{e})
\end{matrix}
\label{eqn:col_counter}
\end{equation}
The only feasible solution and hence unique global minimizer is $y(t)=\bO$, $u(t)=0$; with minimum $y_{[3]}(1)=0$.
To see this, notice that ${y}_{[2]}(t)=-{y}_{[1]}(t)$ follows from~\mbox{(\ref{eqn:col_counter}b)} and~\mbox{(\ref{eqn:col_counter}c)}, whereas $2 \cdot y_{[1]}(t) \cdot \dot{y}_1(t)=\dot{y}_2(t)$ follows from differentiation of~\mbox{(\ref{eqn:col_counter}e)}. This yields the unique solution $y_{[1]}=0$. Given $y_{[1]}$, the arcs for $u,y_{[2]},y_{[3]}$ are determined in order by the constraints~\mbox{(\ref{eqn:col_counter}b)},~\mbox{(\ref{eqn:col_counter}c)}, and~\mbox{(\ref{eqn:col_counter}d)}.

While it should be $y_{[1]}=0$, direct collocation solutions interpolate $y_{[1]}(t)=0$ at $t=0$ and $y_{[1]}(t)=-1$ at all collocation points $t>0$ because this yields a smaller value of $y_{[3]}(1)$. This is so because $y_{[3]}(1)$ is the average of all values of $y_{[1]}(t)$ over $t \in [0,1]$. Figure~\ref{fig:counter1col} illustrates the numerical solution of $y_{[1]}$ for Lgendre-Gauss-Radau collocation (LGRC) and Legendre-Gauss collocation (LGC) of degree $p=2$ on a mesh of three intervals.

\subsection{Limitation} The limitation is that the equality feasibility measure $\rho$ from equation~\ref{eqn:meas:rho} will not converge to zero as $h$ decreases or $p$ increases. The convergence failure arises because the process of collocation can result in extraneous and/or missing solutions. An extraneous solution is a solution that only exists for the numerical discretization, but is not a solution to the original optimal control problem. In the given example, the extraneous solutions of the collocation methods are depicted in Figure~\ref{fig:counter1col} centre and right. In Section~\ref{sec:NumExp:Round2:Sat} we give an example where collocation results in missing solutions. These are solutions of the optimal control problem that do not exist in the discretized version of the problem.

\begin{figure}
	\centering
	\includegraphics[width=0.8\linewidth]{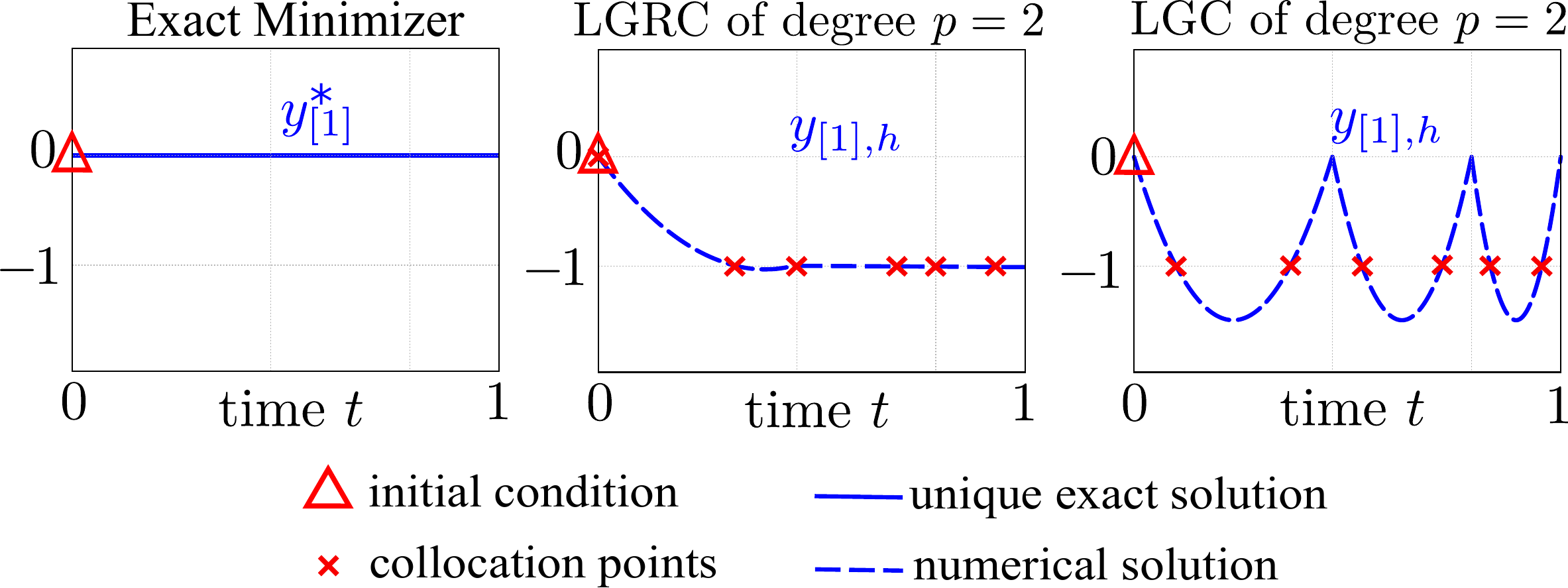}
	\caption{Comparison of the unique exact minimizer with two direct collocation solutions for Problem \eqref{eqn:col_counter}. The direct collocation uses Lgendre-Gauss-Radau collocation (LGRC) and Legendre-Gauss Collocation (LGC) of degree $p=2$ on a mesh of three intervals. Neither method converges.}
	\label{fig:counter1col}
\end{figure}

\section{Limitation on Inequality Feasibility}
\subsection{Example}
Consider the below optimal control problem over the interval $[0,2]$:
\begin{equation}
\label{eqn:box_counter}
\left\lbrace
\begin{aligned}
&\operatornamewithlimits{min}_{y,u}  		& 	&y(2)\\
&\text{subject to} 								& 	&y(0)=1\,,\ \dot{y}(t)=u(t)\,,\\
&											& 	&u(t) = \operatorname{sign}(t-1)\,,\\
& 											& 	&-1\leq u(t)\leq 1\,.
\end{aligned}
\right\rbrace\quad \begin{matrix}
(\ref{eqn:box_counter}\text{a})\\[7pt]
(\ref{eqn:box_counter}\text{b})\\[3pt]
(\ref{eqn:box_counter}\text{c})\\[3pt]
(\ref{eqn:box_counter}\text{d})
\end{matrix}
\end{equation}

The only feasible solution and hence unique global minimizer is $y(t)=|t-1|$, $u(t)=\operatorname{sign}(t-1)$. To see this, notice that $u(t)=\operatorname{sign}(t-1)$ is determined uniquely from the algebraic constraint, whereas $y(t)$ follows uniquely from the initial value problem $y(0)=1,\ \dot{y}(t)=\operatorname{sign}(t-1)$.

Figure~\ref{fig:counter2col} shows the numerical solution to this problem with LGRC of degree $p$ on an odd number $N$ of equidistant mesh intervals. We see that on an even number of equidistant mesh intervals the finite elements could exactly re-capture the shape of the analytic solution. However, in practice there are often discontinuous features that cannot be captured exactly unless additional techniques (e.g., mesh refinement, adaptive discontinuous elements, etc.), each with their own numerical issues in turn, are relied upon. Therefor, the purpose of this experiment is to explore what happens when a single discontinuity is not captured by the mesh.

As we see from the numerical solutions in the figure, neither a decrease of $h$ nor an increase of $p$ yields convergence of the inequality feasibility residual $\gamma$ from~\eqref{eqn:meas:gamma}. Instead, it remains at $\gamma \approx 0.25$\,. This phenomenon is known as the Gibbs phenomenon \cite{Gibbs}. This phenomenon states the observation that interpolations overshoot discontinuities by a fixed margin.

\begin{figure}
	\centering
	\includegraphics[width=0.8\linewidth]{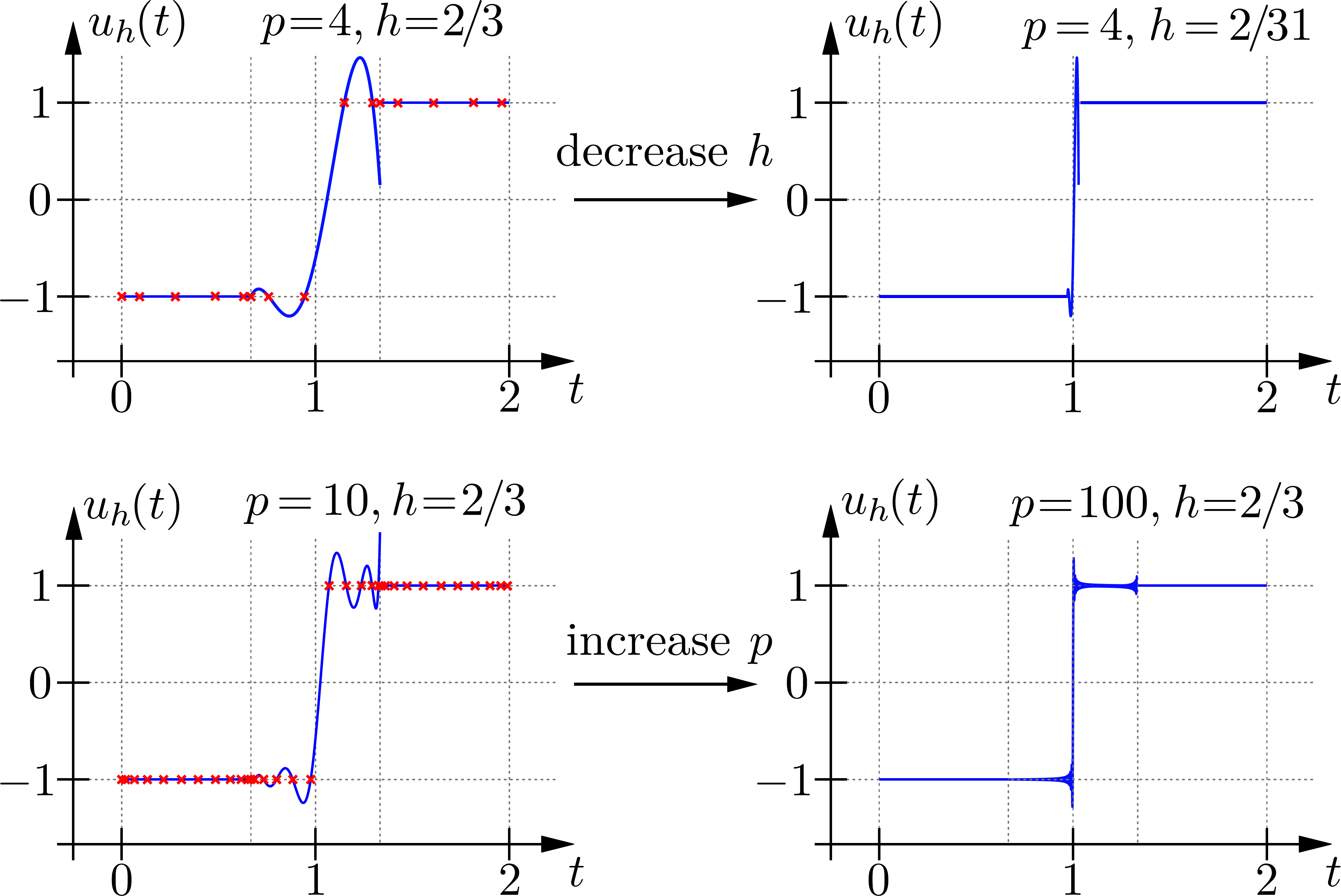}
	\caption{Direct collocation solution to \eqref{eqn:box_counter} with LGRC of various degree on various meshes. Red crosses mark the position of the collocation points. Neither method yields convergence of $\gamma \rightarrow 0$ from~\eqref{eqn:meas:gamma}. Instead, it remains at $\gamma \approx 0.25$\,.}
	\label{fig:counter2col}
\end{figure}

\subsection{Limitation}
The limitation is that $\gamma$ will not converge to zero as $h$ decreases or $p$ increases. Instead, $\gamma\approx 0.25$ in this example, regardless of $h$ and $p$. There are simply not enough collocation points to prevent $u_h$ from violating the bound constraints in-between the collocation points.

\subsection{Misconception}
It seems that by virtue of more collocation points we could enforce convergence of $\gamma$.
This is a misconception. In collocation, the solution $u_h$ is uniquely determined from interpolating $\operatornamewithlimits{sign}(t-1)$ at the collocation points. Due to the Gibbs phenomenon, this interpolation overshoots the bounds in any case. Thus, $\gamma$ cannot converge whenever solving (\ref{eqn:box_counter}c) by means of collocation.

\section{Summary and Outlook}
The prior examples have shown misconceptions on the convergence of collocation methods in the context of solving optimal control problems: Example~\eqref{eqn:col_counter} shows that collocation is inappropriate for relaxation of systems of differential and algebraic constraints~\eqref{eqn:OCP:dae} because the collocation can result in extraneous and/or missing solutions. Example~\eqref{eqn:box_counter} shows that a higher density of collocation points is necessary in the relaxation of~\mbox{(\ref{eqn:OCP}:y)--(\ref{eqn:OCP}:u)} to force convergence of $\gamma$; but that this cannot be done whenever \eqref{eqn:OCP:dae} is relaxed via collocation.
As a logical consequence, if one seeks to assert convergence of $\rho$ and $\gamma$ then one must abandon the concept of collocation and use a different relaxation concept instead. This alternative concept is introduced in the next section.

\chapter{Direct Transcription via Quadrature Penalty Methods}\label{sec:DQIPM}
As illustrated in Figure~\ref{fig:methodclasses}, direct transcription methods have two building blocks: approximation and relaxation. As discussed in Section~\ref{sec:Scope}, relaxation of constraints in optimal control problems can be conducted via two means: (i) the number of points in which the constraints are solved and (ii) the accuracy to which the constraints are solved at each of these points. Collocation methods use a (i)~relatively small number of points and (ii)~exact accuracy. The opposite concept to collocation are integral penalty methods. These integrate over (i)~all points but use only (i)~moderate accuracy in each point. A middle-ground and generalization of both methods is achieved with quadrature penalty methods. These work like integral penalty methods but replace the integral with a quadrature approximation. Thereby, (i)~the number of points can be chosen flexible via the number of quadrature points, and (ii)~the accuracy is parametrized via the penalty parameter.

This section presents direct transcription via quadrature penalty methods. The structure is identical to Section~\ref{sec:DirectCollocation}: We first explain the construction and then present the resulting optimization problem in a standard form.

\section{Construction}
In contrast to Section~\ref{sec:DirectCollocation}, here we first introduce the relaxation and then apply the approximation. In addition, there will be a final step of discretization.

\subsection{Relaxation of the Optimal Control Problem}
Unlike collocation, we treat the equality constraints~\mbox{(\ref{eqn:OCP}:b) and \eqref{eqn:OCP:dae}} with an integral penalty term. This yields:
\begin{equation}
\label{eqn:POCP}
\left\lbrace
\begin{aligned}
&\operatornamewithlimits{min}_{(y,u) \in \cX} & M \big(\,&y(0),y(T)\,\big) +\frac{1}{2 \cdot \omega} \cdot \Big( \int_0^T \|f\big(\dot{y}(t),y(t),u(t),t\big)\|_2^2\,\mathrm{d}t + \|b\big(y(0),y(T)\big)\|_2^2 \Big)\\
& \text{subject to} 		& &\yL(t) \leq y(t) \leq \yR(t)\quad 	\forall\ t \in [0,T]\,,\\
& 					& &\uL(t) \leq u(t) \leq \uR(t)\quad 	\forall\ t \in [0,T]\,,
\end{aligned}
\right\rbrace
\end{equation}
with some small penalty parameter $\omega \in \R_{>0}$; e.g., $\omega=10^{-6}$, which is the standard tolerance for most algorithms and most packages (e.g., all Matlab solvers, IPOPT, SNOPT, WORHP, EISPACK, RADAU). Solutions to this problem do not attempt to solve the constraints \eqref{eqn:OCP:dae} and (\ref{eqn:OCP}:b) exactly, but instead minimize a bias of objective $M\big(y_h(0),y_h(T)\big)$ and constraint violation $r(y_h,u_h)$ from~\eqref{eqn:IntRes}. Usually, the penalty results in $\rho \in \cO(\omega)$, hence $\omega=10^{-6}$ is suitable for most practical purposes.\footnote{Later in Theorem~\ref{thm:main} we only prove $\rho \in \cO(\sqrt{\omega})$, but this is rather due to our mild assumptions.}

\begin{remark}In the above, we use the \emph{inexact $\ell^2$-penalty} \cite{Nocedal,SUMT} because it is smooth. Other penalties, such as the non-smooth $\ell^1$-penalty, are common in constrained optimization algorithms \cite{IPOPT,WORHP,Knitro} because they are exact. Because of the approximation in transcription methods, exactness is lost anyways. Thus, we only consider the inexact $\ell^2$-penalty. This is advantageous because we can benefit from its smoothness.
\end{remark}

\subsection{Approximation of States and Controls}
Quadrature penalty methods use the same piecewise polynomial functions $(y_h,u_h) \in \cX_{h,p}$ as collocation methods. Hence, at this stage, the problem reads:
\begin{equation}
\label{eqn:POCP2}
\left\lbrace
\begin{aligned}
&\operatornamewithlimits{min}_{(y_h,u_h) \in \cX_{h,p}} & M \big(&y_h(0),y_h(T)\big) +\frac{1}{2 \cdot \omega} \cdot \Big( \int_0^T \|f\big(\dot{y}_h(t),y_h(t),u(t),t\big)\|_2^2\mathrm{d}t + \|b\big(y_h(0),y_h(T)\big)\|_2^2 \Big)\\
& \text{subject to} 		& &\yL(t) \leq y_h(t) \leq \yR(t)\quad 	\forall\ t \in [0,T]\,,\\
& 					& &\uL(t) \leq u_h(t) \leq \uR(t)\quad 	\forall\ t \in [0,T]\,,
\end{aligned}
\right\rbrace
\end{equation}

The risk of extraneous solutions is avoided because the integral penalty keeps the residual of~\eqref{eqn:OCP:dae} small almost everywhere. Also, the risk of missing solutions is avoided because the penalty integral avoids the necessity to solve any equality constraints exactly at any point $t$.

\subsection{Discretization of the Integral and Bound Constraints}
Problem~\eqref{eqn:POCP2} cannot be evaluated because of the integral and because of the bound constraints that are evaluated $\forall\ t \in [0,T]$. These two items are now discretized.

\subsubsection{Quadrature for the Integral}
The integral is discretized by using a set $\cQ_{h,q}$ of $q\in\N$ quadrature weights $\alpha \in \R_{>0}$ and quadrature points $t$ per mesh-interval. This permits the quadrature approximation
\begin{align}
\int_0^T \Big\|f\big(\dot{y}_h(t),y_h(t),u(t),t\big)\Big\|_2^2\,\mathrm{d}t \approx Q_{h,q}(y_h,u_h) := \sum_{(t,\alpha) \in \cQ_{h,q}} \alpha \cdot \Big\|f\big(\dot{y}_h(t),y_h(t),u(t),t\big)\Big\|_2^2\,.
\end{align}

\subsubsection{Sampling for the Bound Constraints}
The bound constraints are discretized to be evaluated $\forall\ t \in \cT_{h,m}$. Therein, $m \in \N$ is a method parameter that can exceed the degree $p$. We call $m$ the \emph{sampling degree}.

\section{Implementation as an NLP}
The construction of approximation, relaxation, and discretization are used to transcribe the optimal control problem~\eqref{eqn:OCP} into an NLP. In this section, we formalize the NLP that quadrature penalty methods use.

\subsection{NLP in Transcription Notation}
Using the notation with $\cX_{h,p}$, $Q_{h,q}$, $\cT_{h,m}$, we can state the transcribed optimal control problem in quadrature penalty methods as follows:
\begin{equation}
\label{eqn:POCPh}
\left\lbrace
\begin{aligned}
&\operatornamewithlimits{min}_{(y_h,u_h) \in \cX_{h,p}} & M\big(\,&y_h(0),y_h(T)\,\big)+\frac{1}{2 \cdot \omega} \cdot \bigg( Q_{h,q}(y_h,u_h) + \left\|b\big(y(0),y(T)\big)\right\|_2^2  \bigg) \\[12pt]
& \text{subject to} & \yL(t) &\leq y_h(t) \leq \yR(t)\qquad 	\forall\ t \in \cT_{h,m}\,,\\
& 					& \uL(t) &\leq u_h(t) \leq \uR(t)\qquad 	\forall\ t \in \cT_{h,m}\,.
\end{aligned}
\right\rbrace
\end{equation}

\subsection{NLP in Standard Notation}\label{sec:NLP_of_QPM}
In order to solve~\eqref{eqn:POCPh} with available numerical algorithms, it is helpful to re-express~\eqref{eqn:POCPh} in the following format:
\begin{align}
\label{eqn:Merit2}
\left\lbrace
\begin{aligned}
&\operatornamewithlimits{min}_{\bx \in \R^{n_\bx}}  	& 	&\bf(\bx)+\frac{1}{2 \cdot \omega} \cdot \|\bc(\bx)\|_2^2\\
&\text{subject to}      								& 	&\bbL \leq \bA \cdot \bx \leq \bbR\,.
\end{aligned}
\right\rbrace
\end{align}
In analogy to Section~\ref{sec:COL:optim}, we explain in the following how this can be achieved.

The functions $y_h,u_h$ are identified in exactly the same way with a vector $\bx \in \R^{n_\bx}$ of dimension $N\cdot (n_y+n_u)+n_y$ as in collocation, described in Section~\ref{sec:COL:optim}:
\begin{align*}
\bx = \left[\begin{array}{c}
\vdots\\
y_h(t)\\
u_h(t)\\
\vdots\\
\hline
y_h(T)
\end{array}\right]\quad\forall t \in \cT_{h,p}\,.
\end{align*}

Using $\bx$, we can evaluate the functions $y_h,u_h$ that $\bx$ represents. We do this to construct the properties $\bf,\bc,\bA,\bbL,\bbR$ in \eqref{eqn:COL:NLP}:
\begin{align*}
\bf(\bx)&:= M\big(y_h(0),y_h(T)\big)\,, &&\\[3pt]
\bc(\bx)&:= \left[\begin{array}{c}
b\big(y_h(0),y_h(T)\big)\\[2pt]
\hline
\vdots\\
\sqrt{\alpha} \cdot f\big(\dot{y}_h(t),{y}_h(t),u_h(t),t\big)\\
\vdots
\end{array}\right]\quad \forall\, (t,\alpha) \in \cQ_{h,q}\,,\\
\bbL&:=\begin{bmatrix}
\vdots\\
\yL(t)\\
\uL(t)\\
\vdots
\end{bmatrix},\ 
\bbR:=\begin{bmatrix}
\vdots\\
\yR(t)\\
\uR(t)\\
\vdots
\end{bmatrix}\quad \forall\, t \in \cT_{h,m}\,.
\end{align*}
The matrix $\bA$ is constructed such that
\begin{align*}
\bA \cdot \bx = \begin{bmatrix}
\vdots\\
y_h(t)\\
u_h(t)\\
\vdots
\end{bmatrix}\quad\forall t \in \cT_{h,m}\,.
\end{align*}
From the above definitions, we obtain a vectorial function $\bc : \R^{n_\bx} \rightarrow \R^{n_\bc}$ and vectors $\bbL,\bbR\in\R^{n_\bb}$ of dimensions
\begin{align*}
n_\bc &:= n_b + N \cdot q \cdot (n_y + n_c)\,,\qquad &n_\bb &:= N \cdot m \cdot (n_y + n_u)\,.
\end{align*}

\subsection{Numerical Solution of the Optimization Problem}
The problem~\eqref{eqn:COL:NLP} matches precisely with the problem format~\eqref{eqn:NLP}. Hence, the method presented in Section~\ref{sec:NLP} can be used to solve this optimization problem numerically.

\chapter{Discussion of Quadrature Penalty Methods}
This section discusses quadrature penalty methods. We first emphasize some potential benefits of the method. We then motivate the use of higher-order quadrature schemes within quadrature penalty methods. Finally, we give examples. Computational cost will be discussed in Section~\ref{sec:CostCompare}.

\section{Extension to Least-Square Collocation Methods}

Quadrature penalty methods treat the differential and algebraic constraints via integrals of the squared constraint residual. Upon discretization of the integral via quadrature, this is similar to least-squares collocation methods, described in \cite{ascher1978,HankeDAEcol1,HankeDAEcol2}. In contrast to conventional collocation discretizations, the number of collocation points exceeds the number of degrees of freedom of the polynomial interpolants. To resolve the overdetermination, the collocation conditions are solved in a nonlinear least-squares manner. Typical weightings for the nonlinear least-squares residual result from quadrature.

The references \cite{ascher1978,HankeDAEcol1,HankeDAEcol2} consider least-squares solutions to DAE boundary value problems. The quadrature penalty method presented in this thesis extends these methods to optimal control problems; involving an objective function and potentially inequality constraints. The objective makes it necessary to introduce a penalty parameter in order to tell in which relation the scales of constraint residual and objective value stay.

\section{Key Benefits in Quadrature Penalty Methods}
Collocation methods have one method parameter: (i)~the polynomial degree $p \in \N$. In contrast, quadrature penalty methods have three additional parameters: (ii)~the quadrature degree $q \in \N$; (iii)~the sampling degree $m \in \N$; and (iv)~the penalty parameter $\omega \in \R_{>0}$.

\subsubsection{Generalization of Direct Collocation Methods}
Quadrature Penalty methods are a true generalization of direct collocation methods. Choosing $m=q=p$ and letting $\omega\rightarrow 0$, the numerical minimizer of the quadrature penalty method matches with the numerical minimizer of a collocation method of degree $p$ that uses the quadrature points as collocation points. Thus, quadrature penalty methods can inherit all benefits of collocation methods. However, quadrature penalty methods can also do the following things that collocation methods cannot do.

\subsubsection{Prioritizability between Feasibility and Optimality}
Due to the penalty parameter $\omega$, quadrature penalty methods allow tuning between feasibility and optimality: On a given mesh, the functions $(y_h,u_h) \in \cX_{h,p}$ are likely unable to achieve $r(y_h,u_h)=0$. Instead, there will be some strictly positive lower bound on the smallest possible value for $r(y_h,u_h)$. In quadrature penalty methods, we can approach this value by selecting $\omega$ very small. Alternatively, we may opt for the opposite by selecting $\omega$ rather large. With collocation this is not possible.

\subsubsection{Improved Robustness in Comparison to Collocation}
In collocation, an increase of $p$ means more collocation points per mesh interval. This means \eqref{eqn:OCP:dae} is forced to zero at more points. This seems advantageous for driving $\rho\rightarrow 0$. However, an increase of $p$ also means that $y_h,u_h$ can take on more possible shapes. Thus, also $f$ can fluctuate more wildly in-between the collocation points. Practical experience show that collocation methods of higher degree $p$ can perform less reliably than methods of low degree.

Quadrature penalty methods decouple the number of quadrature points $q$ and sampling points $m$ from the polynomial degree $p$. This allows more robustness by simply choosing $q$ and $m$ larger than $p$. The motivation for doing this is given in the next section.

\section{Motivation for Higher-Order Quadrature Schemes}\label{sec:HigherOrderQuadMotiv}
It is advantageous for the robustness of a direct transcription method when the quadrature degree $q$ (i.e., the number of quadrature point per mesh interval) exceeds the polynomial degree $p$. In this section we illustrate the reasons for this. The section closes with a definition of a quadrature order that is suitable for optimal control problems.

\subsection{Quadrature Points and Collocation Points applied to Algebraic Constraints}
In general, collocation methods can struggle with solving algebraic equality constraints~(\ref{eqn:OCP}:f2). The following example illustrates the reasons for why these constraints are more likely to not converge. Consider the algebraic constraint
\begin{align*}
y_h(t)=0 \quad \tforall\,t \in [0,1]\,.
\end{align*}
According to the definition of the functional measure $r$ from~\eqref{eqn:IntRes}, convergence of the equality feasibility residual $\rho$ necessitates that the integral
\begin{align*}
\int_0^1 y_h^2(t)\,\mathrm{d}t
\end{align*}
converges to zero.

Quadrature penalty methods approximate the integral with a quadrature approximation. The approximation is driven to zero as $\omega$ is decreased:
\begin{align*}
\int_0^1 y_h^2(t)\,\mathrm{d}t \approx \sum_{(t,\alpha) \in \cQ_{h,q}} \alpha \cdot y_h^2(t)\xrightarrow{\omega\rightarrow 0} 0\,.
\end{align*}
This is supposed to yield convergence of $\rho \rightarrow 0$ as $\omega\rightarrow 0$ and $h\rightarrow 0$. Similarly, collocation methods use collocation points at which $y_h$ is set to zero:
\begin{align*}
& y_h(t)=0\qquad  \forall\ t \in \cT_{h,p}\\
\Rightarrow\quad & \int_0^1 y_h^2(t)\,\mathrm{d}t \approx \sum_{t \in \cT_{h,q}} \alpha(t) \cdot y_h^2(t) = 0\,,
\end{align*}
where the weights $\alpha(t)$ are some suitable quadrature weights with respect to the collocation points. They can be computed, e.g., as the integral of Lagrange basis polynomials. Likewise, this is supposed to yield convergence of $\rho \rightarrow 0$ as $h\rightarrow 0$.

The issue with the quadrature is that $y_h$ lives on the same mesh as the quadrature rule of $\cQ_{h,q}$ or $\cT_{h,p}$. Thus, when $h \rightarrow 0$ decreases then $y_h$ may oscillate more wildly. This can prevent the quadrature approximations from converging to the integral; hence, $\rho$ may fail to converge to zero. In the following, we illustrate a numerical example for the quadrature error in the above approximations.

\subsection{Numerical Example for the Quadrature Error}\label{sec:QuadError}
Figure~\ref{fig:quaddiv} shows a piecewise polynomial function $y_h$ of degree $p=2$ in red and the Gauss-Legendre quadrature points of degree $q=p$ in blue. The function $y_h$, and thus $y_h^2$, are zero at the quadrature points. Hence, $Q_{h,p}[y_h^2]=0$, independent of the mesh size $h$. However, $\int_0^1 y_h^2(t)\,\mathrm{dt}=0.2$, independent of $h$. This integral is the red area in the top of the figure. The quadrature does not converge to this integral as $h\rightarrow 0$.
\begin{figure}
	\centering
	\includegraphics[width=0.75\linewidth]{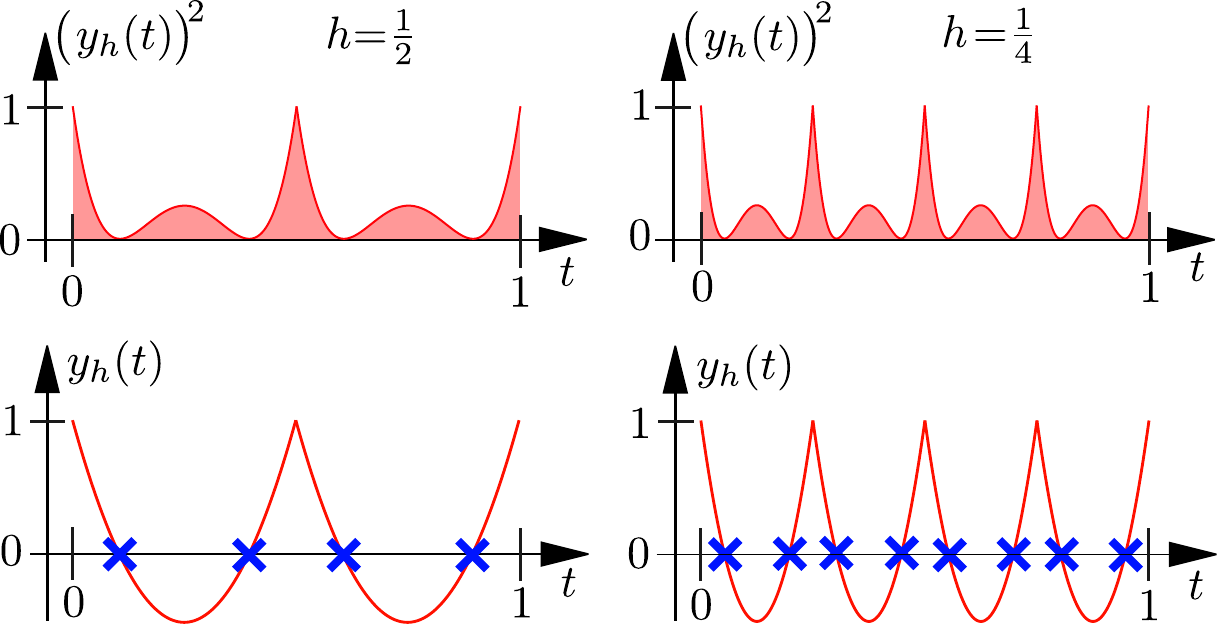}
	\caption{Bottom: Piecewise polynomial function $y_h$. Top: Integral over $y^2_h(t)$. Blue crosses show the points of Gauss-Legendre quadrature of degree $q=2$.}
	\label{fig:quaddiv}
\end{figure}

\subsection{Definition of Quadrature Order in the Literature}
In the literature, the quadrature order is the order at which a quadrature error decreases to zero as $h\rightarrow 0$. The Gauss-Legendre quadrature $G_q$ of degree $q \in \N$ over a sufficiently smooth function $g : [0,T] \rightarrow \R$ yields \cite{GaussQuad}:
\begin{align}
\left| \int_0^T g(t)\,\mathrm{d}t - G_q(g) \right| \leq const \cdot h^{2 \cdot q-1} \qquad \forall h\leq \text{ some positive constant} \label{eqn:def:quadOrder}
\end{align}
Using Landau notation, this can be abbreviated as
\begin{align*}
\left| \int_0^T g(t)\,\mathrm{d}t - G_q(g) \right| \in \cO(h^{2\cdot q-1})\,.
\end{align*}
We say: Gauss-Legendre quadrature of degree $q$ has the \emph{quadrature order} $2\cdot q-1$. For comparison, Newton-Cotes quadrature only achieves a quadrature order of $q$ \cite{GaussQuad}.

In contrary to what the quadrature order seems to imply, we just witnessed in Section~\ref{sec:QuadError} that Gauss-Legendre quadrature of degree $q=2$ is insufficient to approximate the integral over $y_h^2$ in a convergent manner. This is so because $y_h$ is not sufficiently smooth. To resolve the issue, we next introduce a stricter measure for the quadrature order. This measure is called \emph{piecewise polynomials quadrature order}.

\subsection{Definition of Piecewise Polynomials Quadrature Order}\label{sec:QuadOrder}
Suppose $(y_h,u_h) \in \cX_{h,p}$ from \eqref{eqn:def:spaceXhp}. Then a quadrature rule $Q_{h,q}$ has the \emph{piecewise polynomials quadrature order} $\ell \in \R_{>0}$ if the following condition holds: 
\begin{mdframedwithfoot}
	There exist a finite constant $\Cquad \in \R_{>0}$ and a constant $h_0 \in \R_{>0}$ finitely above zero\footnote{Strictly speaking, the relation $0.\overline{9}=1$ holds. In a less strict mathematical sense, the number $1-0.\overline{9}$ might appear strictly larger than zero but not finitely larger than zero. We write ``finitely above zero'' to avoid misunderstandings.}, such that $\forall h\leq h_0$ the following holds:
	\begin{equation}
	\begin{aligned}
	\Bigg|\int_0^T \left\|f\big(\dot{y}_h(t),y_h(t),u_h(t),t\big)\right\|_2^2\,\mathrm{d}t-Q_{h,q}(y_h,u_h)\Bigg|
	\leq \Cquad \cdot h^\ell \quad \forall\, (y_h,u_h) \in \cX_{h,p}
	\end{aligned}
	\label{eqn:quadcond}%
	\end{equation}
\end{mdframedwithfoot}
\noindent
For short, the quadrature error must live in $\cO(h^\ell)\quad \forall\,(y_h,u_h) \in \cX_{h,p}$.

The difference to the definition of the conventional quadrature order from~\eqref{eqn:def:quadOrder} is in the detail that it must hold $\forall\ (y_h,u_h) \in \cX_{h,p}$. This is significant because a refinement of $h$ leads to an increase of $\cX_{h,p}$. Thus, as $h\rightarrow 0$, the quadrature must become more accurate for an increasing space $\cX_{h,p}$ of possible functions. In comparison, the definition of conventional quadrature order in \eqref{eqn:def:quadOrder} only considers convergence for a fixed arbitrary function $g$.

For example, if in Figure~\ref{fig:quaddiv} we had chosen Gauss-Legendre quadrature of degree $q\geq p+1$ then $\ell=2\cdot q-1$. As a rule of thumb, using Gauss-Legendre quadrature of degree $q$ sufficiently larger than $p$ results in~\mbox{$\ell=2\cdot q-1$}.

\section{Examples of Quadrature Penalty Methods}
We give a few examples of quadrature penalty methods. These showcase ways in which quadrature penalty methods can generalize collocation methods.

\subsubsection{Explicit Euler Method}
We demonstrate that the presented framework of quadrature penalty methods generalizes collocation via explicit Euler.
We set $p=1$ and $q=m=1$. The reference sets of quadrature and the reference sets of bounds sampling are:
\begin{align*}
\cQ_{\text{ref},q} &= \lbrace (-1,2)\rbrace\,,\\
\cT_{\text{ref},m} &= \lbrace -1\rbrace\,.
\end{align*}
I.e., the explicit Euler method is attained by using quadrature with left Riemann sums and sampling the left-most point on each mesh interval. However, we motivated in Section~\ref{sec:HigherOrderQuadMotiv} that it can be advantageous to choose a higher-order quadrature scheme of degree $q$ that exceeds the piecewise polynomial degree $p$.

\subsubsection{Penalty Euler Method}
Collocation via Euler methods is characterized by choosing $p=1$ and setting the collocation point to either $-1$ or $+1$ on $\Iref$. In quadrature penalty methods, we do not have to decide for either point because we can choose the quadrature degree $q$ and sampling degree $m$ independent of $p$. We propose $q=3$, $m=2$, with the following sets:
\begin{align*}
\cQ_{\text{ref},q} &= \lbrace (-1,0.5),(0,1),(1,0.5)\rbrace\,,\\
\cT_{\text{ref},m} &= \lbrace -1,0,1\rbrace\,.
\end{align*}
I.e., we can use Hermite-Simpson quadrature and sample the bound constraints at the CGL points of degree $m=2$. This may improve the robustness of the method.

\subsubsection{Gauss-Legendre Quadrature and Chebyshev-Gauss-Lobatto Sampling}
We now present a more sophisticated quadrature penalty method. This method is parametric in the choice of all three parameters $p,q,m$. This method will be used for the numerical experiments in Section~\ref{sec:NumExp} and for the theoretical analysis in Part~\ref{part:convproof}.

This method uses the quadrature points and weights $\cQ_{\text{ref},q}$ of the Gauss-Legendre quadrature of degree $q$. Further, it samples the bound constraints at the Chebyshev-Gauss-Lobatto points $\cT_{\text{ref},m}$ of degree $m$.

\section{Comparison of Computational Cost between QPM and DCM}\label{sec:CostCompare}
This section compares the two direct transcription methods discussed in Section~\ref{sec:DirectCollocation} and Section~\ref{sec:DQIPM}: Direct Collocation Method (DCM) and Quadrature Penalty Method (QPM).

The computational cost (in terms of computation time) of both methods is identical to the computational cost for solving the NLPs that their transcriptions result in. As discussed in Section~\ref{sec:NLP:cost}, the cost for solving an NLP depends on the following aspects:
\begin{itemize}
	\item number of iterations
	\item computation time per iteration
\end{itemize}
An experimental analysis of the number of iterations between DCM and QPM is given in Section~\ref{sec:NumExp}. Because the number of iterations can vary a lot, a theoretical analysis is possible only for the computation time per iteration. This analysis is provided in the following.

\subsection{Overview}
Figure~\ref{fig:costdiagram} shows the computation time per iteration for DCM and QPM for the parameters described in Section~\ref{sec:NumExp:Setup}. The figure shows the computation time per iteration of the NLP solver. Step (N.3) takes the most time per iteration in each method. DCM and QPM require the exact same time for step (N.3) because their reduced linear systems have the exact same dimension and sparsity pattern, as we will showcase below. The assembly step (N.2) is more expensive in QPM than in DCM. The cost of all other steps is negligible in both methods. The reasons for this are given below.

Figure~\ref{fig:costdiagram}~(b) shows that DCM takes only $79\%$ as much time per iteration as QPM. The parallel versions of both methods perform more similarly in speed. This is so because some tasks can benefit dramatically from parallelisation, however the most expensive part only benefits by a limited amount. Further details on these aspects are discussed below.

The starting point of all discussions is the sparsity pattern. This is discussed in the next subsection. Afterwards, we analyze cost and parallelizability in the assembly step (N.2) and in the computation step (N.3).

\begin{figure}
	\centering
	\includegraphics[width=0.7\linewidth]{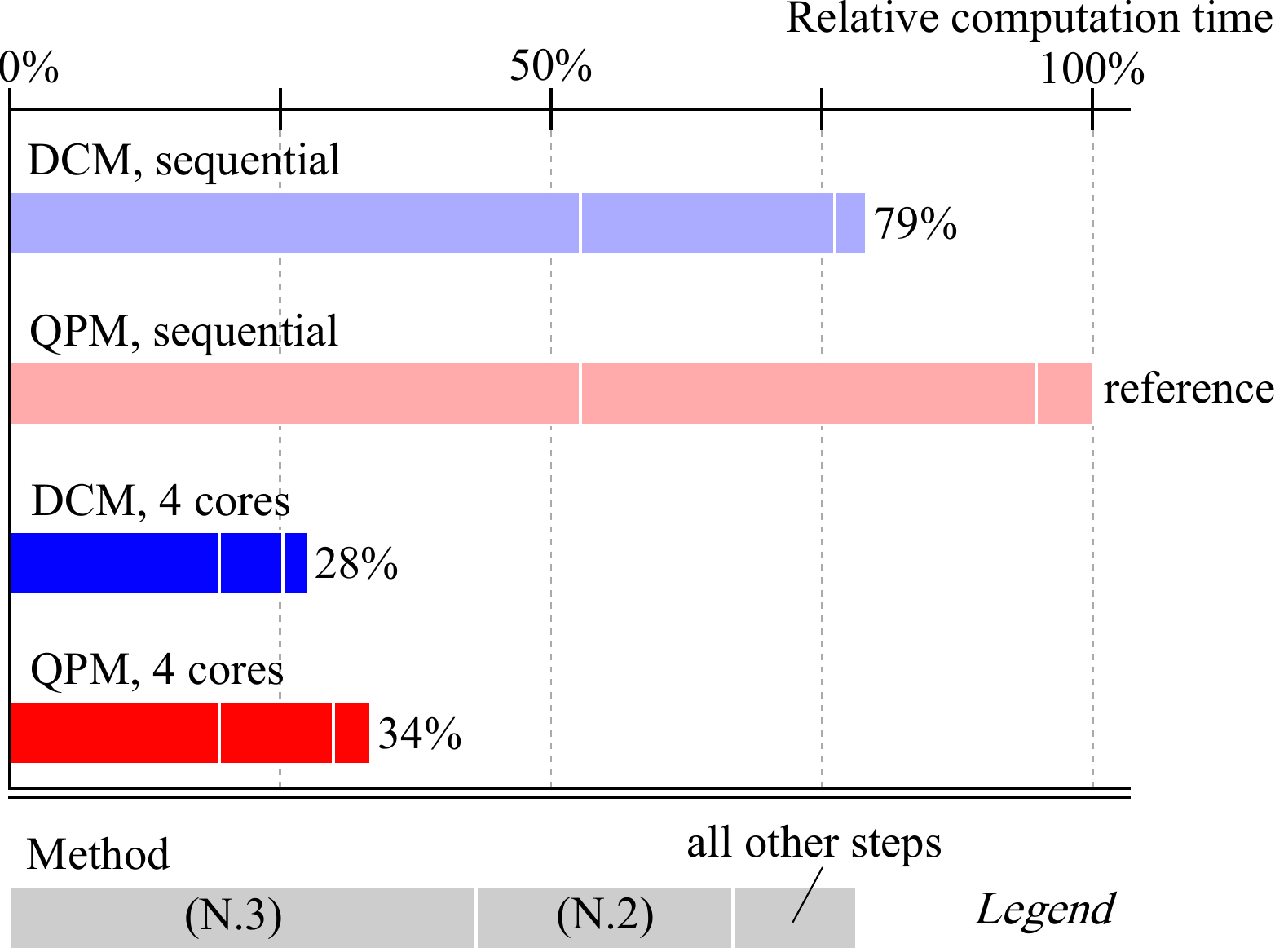}
	\caption{Computation time per iteration of the constrained optimization algorithm. Different methods and numbers of CPU cores are compared. The steps (N.3), (N.2) from Figure~\ref{fig:newtoniteration} are indicated.}
	\label{fig:costdiagram}
\end{figure}

\subsection{Sparsity Structure in the NLP}\label{sec:CostCompare:Sparsity}
Figure~\ref{fig:sparsity} shows the worst-case sparsity patterns of the matrices $\bH,\bJ,\bA$ in the Newton matrix~\eqref{eqn:KKTmatrix} of DCM and QPM when the problem dimensions are $n_y=2,\ n_u=1,\ n_c=0,\ n_b=2$. Both methods use degree $p=5$ and $N=10$ mesh intervals. QPM uses $q=7$ and $m=10$. The number of non-zeros (nnz) for each matrix is given below each sparsity pattern.

\begin{figure}
	\centering
	\includegraphics[width=0.85\linewidth]{Images/Sparsity}
	\caption{Sparsity of matrices $\bH,\bJ,\bA$ in \eqref{eqn:KKTmatrix} for DCM and QPM of degree $p=5$. QPM uses $q=7$ and $m=10$.}
	\label{fig:sparsity}
\end{figure}

Both methods feature banded matrices of overlapping dense blocks. In the special case of collocation, the matrix $\bA$ matches with the identity matrix. Depending on the nonlinearity of $f$ in $\dot{y},y,u$, some of these blocks in $\bH,\bJ$ could actually be sparser in both methods; this is ignored here for simplicity. Importantly, the reduced Newton matrix $\bS$ in the reduced system~\eqref{eqn:ReducedKKTlinsys} has the same sparsity pattern as $\bH$.

In the following we explain the cost of steps (N.2) and (N.3) in Figure~\ref{fig:costdiagram} based on the sparsity patterns.

\subsection{Assembly of Derivative Matrices}\label{sec:cost:assembly}
The matrices $\bH,\bJ$ must be recomputed in each iteration in step~(N.2) because they depend on $\bx,\by$, which may change in each Newton iteration. In contrast, the matrix $\bA$ needs only be computed once for both methods. For QPM, this results in about $q/p$ times as many computations as for DCM because the derivatives of $f$ must be computed at $q$ instead of $p$ quadrature points per mesh interval. Because the computation of derivatives can be performed in parallel on each mesh interval, dramatic time savings are possible when parallelizing these computations onto multiple cores \cite{BandedMatmul}.

\subsection{Computation of the Newton Direction}
The Newton direction is computed by factorizing the reduced Newton matrix $\bS$ in \eqref{eqn:ReducedKKTlinsys}. Because this matrix has the same size and pattern for both methods, the computational cost of this task is identical in both methods.

The factorization of $\bS$ is most time-consuming and can only be parallelized by a limited amount \cite{parallelfactorization}. This is why on parallel computers the computation times per iteration are quite similar between DCM and QPM.

It remains to answer whether one method pathologically requires more Newton iterations to converge than the other. The numerical experiments in Section~\ref{sec:NumExp} do not indicate that this is the case.

\chapter{Numerical Experiments}\label{sec:NumExp}
Part~\ref{part:convproof} proves convergence of QPM in general whereas the counter-examples in Section~\ref{sec:Misconceptions} disprove convergence of DCM in general. However, these examples were practically irrelevant crafted edge cases. We now compare DCM and QPM on a variety of practical examples for optimal control problems in order to compare their practical performance.

\section{Experimental Setting}\label{sec:NumExp:Setup}
We compare DCM and QPM in two rounds of numerical experiments: i) The first round comprises realistic test problems from the literature for which DCM is known to converge. In this round, we compare accuracy and computational cost of both methods. ii) The second round showcases two problems for which DCM are known to struggle.

\subsubsection{Method Parameters}
DCM and QPM are method classes. We give here the particular parameters and options that we use for each method. For DCM we use Legendre-Gauss-Radau collocation (LGRC) from Section~\ref{sec:COL:examples} because this is the most widely implemented DCM. We choose the polynomial degree $p=4$ because this is a good trade-off between convergence rate and sparsity. We construct a comparable QPM by letting $p=4$, and choosing Gauss-Legendre quadrature of degree $q=8$ for $Q_{h,q}$. We select the sampling spaces $\Tcl{i,m}$ of degree $m=8$. Since $m,q\leq 2\cdot p$, this QPM is at most twice as expensive in terms of computation time per iteration as this DCM.

\section{Round 1: Problems where DCM works well}\label{sec:NumExp:Round1}

The first round comprises of 21 problems. Each problem is solved with DCM and QPM on three meshes; a coarse mesh of $N$ intervals, a medium mesh of $4\cdot N$ intervals, and a fine mesh of $16 \cdot N$ intervals. To improve comparability, both methods use the same three meshes and the same initial guesses. We use equidistant meshes because different refinement strategies may favor either method, resulting in unfair comparison. For each method on each mesh, we measure solution accuracy in terms of $\delta,\rho,\gamma$, the number of Newton iterations in the constrained optimization algorithm, and the solution time in seconds.

For better comparability, both transcriptions use the exact same NLP solver from Section~\ref{sec:NLP:solver} in Figure~\ref{fig:newtoniteration}. The constrained optimization algorithm uses exact first and second derivatives and terminates when \eqref{eqn:KKT} are solved to $\|\cdot\|_\infty$-accuracy $\leq \tol=10^{-7}$.

\subsection{Test Problems}
Table~\ref{tab:testprobs} depicts the test problems with their respective properties from left to right: active inequality constraints on $y,u$; smoothness properties of the literature solution $y^\star,u^\star$; properties of the minimizer, such as whether it features a so-called singular arc \cite{BettsChap2} or is a unique or strict/non-strict minimizer (cf.~Figure~\ref{fig:minimizertypes}); and numerical properties such as stiffness of the optimality system (in terms of Euler-Lagrange equations or generalizations thereof; cf.~Section~\ref{sec:ELODE}), scaling issues due to large discrepancy in magnitude of variables, and long timespans $[0,T]$.

The problems are sorted into categories. Some problems permit analytic solutions, while others are models from engineering applications. Finally, there are two classes of challenges, commonly seen in nonlinear optimal control. These challenges are explained in the following.

The first class of challenges deals with non-strictness and non-uniqueness of solutions: We compute two distinct minimizers to the same problem, to confirm that both methods are able to converge to both minimizers. We also compute non-strict minimizers for a landing-abortion problem that features a family of equally good solutions with regards to how the plane escapes from the abortion zone. The second class of challenges deals with irregular constraints: The constrained brachistochrone problem features a singular Jacobi matrix \cite{BettsChap2}; the pendulum determines the beam force implicitly from a differential-algebraic equation of varying index and eventually also imposes a bound on the beam force. This results in singular optimality conditions. Details on each problem are given in the references in the table.

% here goes table

\begin{table}[]
	\centering
	\scalebox{0.94}{
		\begin{tabular}{
				!{\vrule width 3pt}>{\centering}p{0.5cm}|
				>{\centering}l!{\vrule width 3pt}
				>{\centering}p{0.42cm}|
				>{\centering}p{0.42cm}!{\vrule width 2pt}
				>{\centering}p{0.42cm}|
				>{\centering}p{0.42cm}!{\vrule width 2pt}
				>{\centering}p{0.42cm}|
				>{\centering}p{1.1cm}!{\vrule width 2pt}
				>{\centering}p{0.42cm}|
				>{\centering}p{0.42cm}|
				>{\centering}p{0.42cm}!{\vrule width 3pt}
				l!{\vrule width 3pt}
			}
			\noalign{\hrule height 3pt}
			\multicolumn{2}{!{\vrule width 3pt}c!{\vrule width 3pt}}{\textbf{problem}} 									& \multicolumn{9}{c!{\vrule width 3pt}}{\textbf{properties}} 																																																																									& 	\\ \cline{1-11}
			&  															& \multicolumn{2}{c!{\vrule width 2pt}}{\textbf{{ineq.}}} 			& \multicolumn{2}{c!{\vrule width 2pt}}{\textbf{{cont.}}} 		& \multicolumn{2}{c!{\vrule width 2pt}}{\textbf{{minimizer}}} 						& \multicolumn{3}{c!{\vrule width 3pt}}{\SetTracking{encoding=*}{-50}\lsstyle\textbf{{conditioning}}} 										&   										\\ \cline{3-11}
			\multicolumn{1}{!{\vrule width 3pt}c|}{\textbf{\rotatebox{90}{index}}}				& \multicolumn{1}{c!{\vrule width 3pt}}{\textbf{\rotatebox{90}{name}}}						& \textbf{\rotatebox{90}{bound $y$}}	& \textbf{\rotatebox{90}{bound $u$}}	& \textbf{\rotatebox{90}{jump $u$}}	& \textbf{\rotatebox{90}{edge}}	& \textbf{\rotatebox{90}{singular}}	& \multicolumn{1}{c!{\vrule width 2pt}}{\textbf{\rotatebox{90}{kind}}}	& \textbf{\rotatebox{90}{stiff}}	& \textbf{\rotatebox{90}{bad scale\phantom{x}}}	& \textbf{\rotatebox{90}{long span}}	& \multicolumn{1}{c!{\vrule width 3pt}}{\textbf{\rotatebox{90}{reference}}}		 									\\ \noalign{\hrule height 1.5pt} \hline \hline \noalign{\hrule height 3pt}
			\multicolumn{12}{!{\vrule width 3pt}c!{\vrule width 3pt}}{\phantom{$\frac{\frac{a}{b}}{c}$}\textbf{Analytic solution available}\phantom{$\frac{\frac{a}{b}}{c}$}}																																																																																															 		\\ \hline \hline \noalign{\hrule height 1.5pt}
			1 														& Hager Problem										&  									&  									&  									&    							&  									& unique 						&  									&  												&  										& \cite{kameswaran_biegler_2008}		\\ \noalign{\hrule height 1.5pt}
			2 														& Bryson-Denham	Problem								& \checkmark						&  									&  									& \checkmark 					&  									& unique 						&  									&  												&  										& \cite{bryson1999dynamic}				\\ \noalign{\hrule height 1.5pt}
			%3 														& Bryson Problem									& 									& \checkmark						& \checkmark						& \checkmark 					& \checkmark						& unique 						&  									&  												&  										& \cite{Bryson1975}						\\ \noalign{\hrule height 1.5pt}
			%5 														& Aly-Chan Problem									& 									& \checkmark						&  									& 								& \checkmark						& strict 						&  									&  												&  										& \cite{aly1978computation}				\\ \noalign{\hrule height 1.5pt} 
			3 														& Singular Regulator								& 									& \checkmark						& \checkmark						& \checkmark 					& \checkmark						& unique 						&  									&  												&  										& \cite{AlyChan}						\\ \noalign{\hrule height 1.5pt}\hline \hline \noalign{\hrule height 3pt}
			\multicolumn{12}{!{\vrule width 3pt}c!{\vrule width 3pt}}{\phantom{$\frac{\frac{a}{b}}{c}$}\textbf{Applications}\phantom{$\frac{\frac{a}{b}}{c}$}}																																																																																															 		\\ \hline \hline
			\multicolumn{12}{!{\vrule width 3pt}c!{\vrule width 3pt}}{\textit{Robotics}}		\\ \hline \noalign{\hrule height 1.5pt}
			%6 														& Van-der-Pol Controller							& 									& \checkmark						& \checkmark						& 								& \checkmark						& strict 						&  									&  												&  										& \cite{maurer2007theory}				\\ \noalign{\hrule height 1.5pt}
			4 														& Two-Link Robot Arm								& 									& \checkmark						& \checkmark						& \checkmark					&            						& strict 						&  									&  												&  										& \cite{luus2019iterative}				\\ \noalign{\hrule height 1.5pt}
			5 														& Container Crane									& 									& \checkmark						&  									& \checkmark					&            						& strict 						&  									&  												&  										& \cite{augustin2001sensitivity}		\\ \noalign{\hrule height 1.5pt} \hline \hline
			\multicolumn{12}{!{\vrule width 3pt}c!{\vrule width 3pt}}{\textit{Aircrafts}}		\\ \hline \noalign{\hrule height 1.5pt}
			6 														& Alp-Rider 										& 									& \checkmark						&           						& 								& 									& strict 						& \checkmark						& \checkmark									&  										& \cite{BettsChap2}						\\ \noalign{\hrule height 1.5pt}
			7 														& Dynamic Soaring									& \checkmark						&  									& 									& \checkmark					& 									& strict 						& \checkmark						&  												&  										& \cite{zhao2004optimal}				\\ \noalign{\hrule height 1.5pt}
			\hline \hline
			\multicolumn{12}{!{\vrule width 3pt}c!{\vrule width 3pt}}{\textit{Rockets}}		\\ \hline \noalign{\hrule height 1.5pt}
			8														& Goddard Rocket Max Height							& \checkmark						& \checkmark						& \checkmark   						& \checkmark					& \checkmark						& strict 						&  									&  												&  										& \cite{BettsChap2}						\\ \noalign{\hrule height 1.5pt}
			9														& Spaceship Control									& 									& \checkmark						&           						& 								& 									& strict 						&  									&  												&  										& \cite{battin1999introduction} 							\\ \noalign{\hrule height 1.5pt}
			10 														& Spaceshuttle Reentry								& \checkmark						& \checkmark						& 									& \checkmark					& 									& strict 						& \checkmark						& \checkmark 									& \checkmark							& \cite{BettsChap2} 							\\ \noalign{\hrule height 1.5pt}
			\hline \hline
			\multicolumn{12}{!{\vrule width 3pt}c!{\vrule width 3pt}}{\textit{Satellites}}		\\ \hline \noalign{\hrule height 1.5pt}
			11														& Orbit Raising										& 									& \checkmark						&           						& 								& 									& strict 						&  									&  												&  										& \cite{Bryson1975} 							\\ \noalign{\hrule height 1.5pt}
			12 														& Low-Thrust MEO-GEO Transfer						& 									& \checkmark						& 									&  								& 									& strict 						&  									& 			 									& \checkmark							& \cite{kluever1995optimal} 							\\ \noalign{\hrule height 1.5pt}  \hline \hline
			\multicolumn{12}{!{\vrule width 3pt}c!{\vrule width 3pt}}{\textit{Biochemistry}}		\\ \noalign{\hrule height 1.5pt}
			13														& Tuberculosis Treatment							& \checkmark						& \checkmark						&  			   						& \checkmark					&  									& strict 						& \checkmark						& \checkmark									& \checkmark							& \cite{jung2002optimal}				\\ \noalign{\hrule height 1.5pt}
			14														& Batch Fermentation								& \checkmark						& \checkmark						& \checkmark						& \checkmark					& \checkmark						& strict 						& \checkmark 						&  												&  										& \cite{cuthrell1989simultaneous}		\\ \noalign{\hrule height 1.5pt}
			15 														& Kiln Heating PDE									& 									& \checkmark						& 									& \checkmark					& 									& strict 						& \checkmark						& 			 									&   									& \cite{BettsChap2}						\\ \noalign{\hrule height 1.5pt} \hline \hline \noalign{\hrule height 3pt}
			\multicolumn{12}{!{\vrule width 3pt}c!{\vrule width 3pt}}{\phantom{$\frac{\frac{a}{b}}{c}$}\textbf{Challenges}\phantom{$\frac{\frac{a}{b}}{c}$}} \\ \hline \hline 
			
			\multicolumn{12}{!{\vrule width 3pt}c!{\vrule width 3pt}}{\textit{due to non-uniqueness}}		\\ \noalign{\hrule height 1.5pt}
			
			16a														& Obstacle Avoidance below							&   								& \checkmark						&  			   						& \checkmark					&  									& strict 						&  									&  												&  										& \cite{neuenhofen2018dynamic}			\\ \noalign{\hrule height 0.4pt}
			16b														& \phantom{Obstacle Avoidance} above				&   								& \checkmark						&  			   						& \checkmark					&  									& strict 						&  									&  												&  										& 										\\ \noalign{\hrule height 1.5pt}
			17a														& Free-Flying Robot book							&   								& \checkmark						& \checkmark   						& \checkmark					&  									& strict 						&  									&  												&  										& \cite{BettsChap2}						\\ \noalign{\hrule height 0.4pt}
			17b														& \phantom{Free-Flying Robot} asymmetric			&   								& \checkmark						& \checkmark   						& \checkmark					&  									& strict 						&  									&  												&  										& 										\\ \noalign{\hrule height 1.5pt}
			18a														& Landing Abortion low Exit							& \checkmark						& \checkmark						& \checkmark   						& \checkmark					& \checkmark						& {\SetTracking{encoding=*}{-70}\lsstyle\mbox{non-strict}} 			&  		& \checkmark							&  										& \cite{BettsChap2}						\\ \noalign{\hrule height 0.4pt}
			18b														& \phantom{Landing Abortion} high Exit				& \checkmark						& \checkmark						& \checkmark   						& \checkmark					& \checkmark						& {\SetTracking{encoding=*}{-70}\lsstyle\mbox{non-strict}} 			&  		& \checkmark							&	  									&  										\\ \noalign{\hrule height 1.5pt}\hline \hline 
			\multicolumn{12}{!{\vrule width 3pt}c!{\vrule width 3pt}}{\textit{due to non-regularity}}		\\ \noalign{\hrule height 1.5pt}
			19a														& Brachistochrone unconstr.							&  									&  									&  			   						&  								&  									& unique 						&  									&  												&  										& \cite{BettsChap2}						\\ \noalign{\hrule height 0.4pt}
			19b														& \phantom{Brachistochrone} constr. $h=0.1$			&  									& \checkmark						&  			   						& \checkmark					&  									& unique 						&  									&  												&  										&  										\\ \noalign{\hrule height 0.4pt}
			19c														& \phantom{Brachistochrone} constr. $h=0$			&  									& \checkmark						&  			   						& \checkmark					& \checkmark						& unique 						&  									&  												&  										& 										\\ \noalign{\hrule height 1.5pt}
			20a														& Pendulum Index 1									&  									&  									&  			   						&  								& \checkmark						& strict 						&  									&  												&  										& \cite{campbell2016solving}			\\ \noalign{\hrule height 0.4pt}
			20b														& \phantom{Pendulum} Index 2 						&  									&  									&  			   						&  								& \checkmark						& strict 						&  									&  												&  										& 										\\ \noalign{\hrule height 0.4pt}
			20c														& \phantom{Pendulum} Index 3 						&  									&  									&  			   						&  								& \checkmark						& strict 						&  									&  												&  										& 										\\ \noalign{\hrule height 0.4pt}
			20d														& \phantom{Pendulum} Index 3 + constr.				&  									& \checkmark						&  			   						& \checkmark					& \checkmark						& strict 						&  									&  												&  										& \cite{neuenhofen2018dynamic}			\\ \noalign{\hrule height 1.5pt}
	\end{tabular}}
	\caption{List of numerical test problems for Round~1. Abbreviations: ``ineq.''=inequalities; ``cont.''=continuity.}\label{tab:testprobs}
\end{table}

\subsection{Results and Discussion}
\begin{table}
	\centering
	\includegraphics[width=1\linewidth]{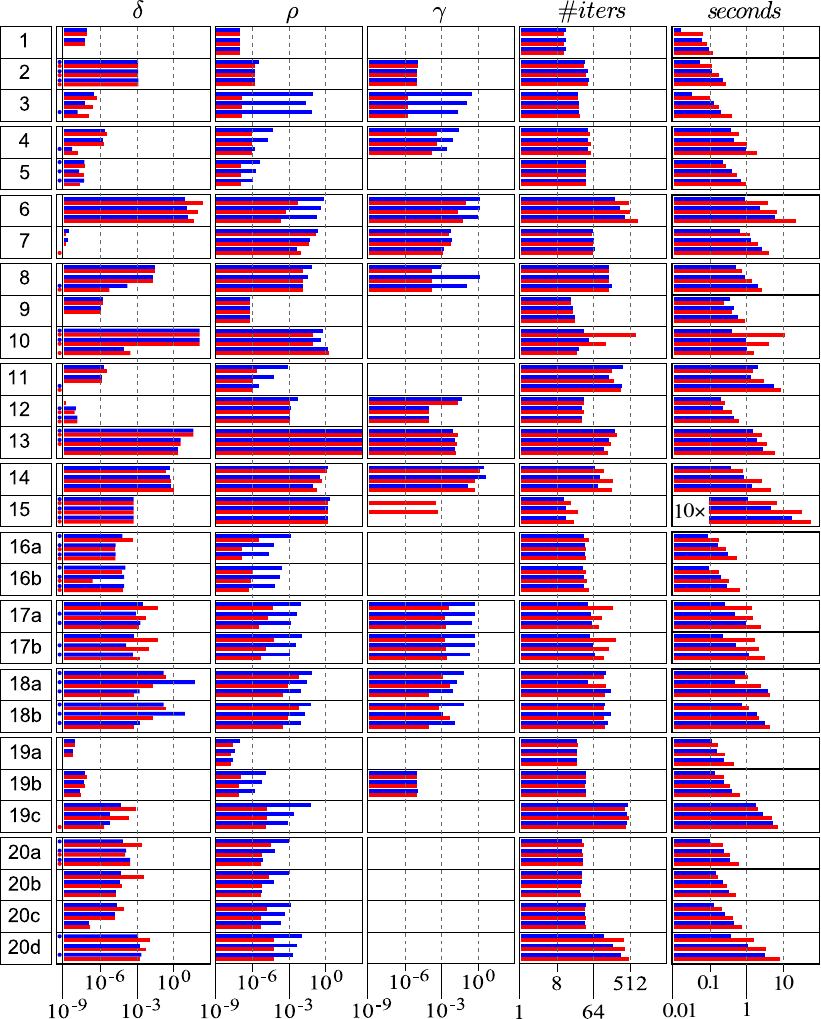}
	\caption{Convergence measures from~Section~\ref{sec:ConvMeasures} for DCM (blue) and QPM (red), both of degree $p=4$. Three bars per method per cell give the measure from coarsest (top) to finest mesh (bottom).
		Problems where the optimality gap converges from below are indicated with a dot in the left of the $\delta$ column. Problem 15 took $10\times$ as many seconds as plotted.}
	\label{fig:table2}
\end{table}

Table~\ref{fig:table2} presents the convergence measures from Section~\ref{sec:ConvMeasures} on each mesh of each problem from Table~\ref{tab:testprobs} for both DCM and QPM. Both direct transcriptions succeed on all problems in the sense that they generate reasonably good numerical solutions. Also, the NLP solver converges in a reasonable number of iterations for each problem.

\subsubsection{General Observations}
\begin{itemize}
	\item The computational time of QPM is at most twice as large as that of DCM.
	\item QPM yields on average three orders of magnitude smaller   equality constraint residuals $\rho$.
	\item QPM yields on average one   order  of magnitude smaller inequality constraint residuals $\gamma$.
	\item The optimality gap is similar for QPM and DCM.
	\item The number of iterations is similar for QPM and DCM when both methods yield similarly accurate solutions.
\end{itemize}

\subsubsection{Exceptions}
The three biochemistry problems (index 14--16) are very stiff, hence some states' derivatives have very large values. Therefore, the equality feasibility residual $\rho$ is large in these cases.

On the problems with index 6, 18, 21d (Alp-Rider, Free-Flying Robot, constrained Pendulum), the iteration count and optimality gap of QPM are significantly larger than of DCM. All these are problems where QPM converges to very accurate solutions (in terms of feasibility) whereas DCM converges to rather inaccurate solutions. However, comparison of solver cost (in terms of iterations and computation time) and optimality gap make sense only when both methods have similarly feasible solutions.

\FloatBarrier

\section{Round 2: Problems where DCM struggles}\label{sec:NumExp:Round2}
We now show two important practical classes of optimal control applications for which DCM is known to have convergence issues.
As a disclaimer, the problems that we present here are intended to be comprehensible. Hence, there are experts who are able to rewrite or modify these problems in a way such that DCM succeeds. Problems in industry may be incomprehensible or unreadable in terms of time available. Also, it can be impractical or uneconomic to rewrite or modify optimal control problems on a regular basis in order to be able to solve them. Additional reasons are discussed in \cite{MR4046772,MR1990061,campbell2016solving}.

We focus on two classes of optimal control problems: i) overdetermined optimal control problems; ii) problems with singular controls. The first class are problems where some of the constraints can only be met due to conservation properties of the states. The second class arises naturally in many areas. Both types of problems can be very difficult to spot in practical-sized optimal control problems.

\subsection{Reorientation of a Satellite}\label{sec:NumExp:Round2:Sat}
\subsubsection{Lateral Newtonian Boundary Value Problems}
Pushing a stone with position $x$ and speed $v$ can be described by the dynamics
\begin{align*}
\dot{x}(t) &= v(t)\,, 	&
\dot{v}(t) &= u(t)\,,
\end{align*}
with the control force $u$. Supposing that we want to move the stone by $0.96$ meters, we can prescribe the following boundary conditions:
\begin{align*}
x(0)&=0\,,& x(T)&=0.96\,,\\
v(0)&=0\,,& v(T)&=0\,.
\end{align*}
This is a shooting problem of positional placement because we have a second-order differential equation of the position $x \in \R$ and we can use the control $u\in\R$ to move the stone to its destination.

\subsubsection{Angular Newtonian Boundary Value Problems}
In this section we consider the optimal control problem from \cite[eqn.~6.123]{BettsChap2}, which is a shooting problem of \emph{angular} placement. A torque profile must be determined to turn NASA’s X-ray Timing Explorer in Figure~\ref{fig:satellitereorientation} around:
\newcommand{\bomega}{\boldsymbol{\omega}\xspace}
\begin{align*}
\dot{\bq}(t)&=\frac{1}{2}\cdot
\bomega(t)
\boldsymbol{\cdot} \bq(t)\,,&
\dot{\bomega}(t)&=\bM\inv \cdot \bu(t)\,.
\end{align*}
Therein, $\bq = (\cos(\theta),\sin(\theta)\cdot\bn) \in \R^4$ is a so-called quaternion. The quaternion describes the satellite's orientation via the axis vector $\bn$ and the rotation angle $\theta$. The state $\bomega \in \R^3$ describes the angular velocity of the satellite. The control input is the torque $\bu \in \R^3$.

\begin{figure}
	\centering
	\includegraphics[width=1\linewidth]{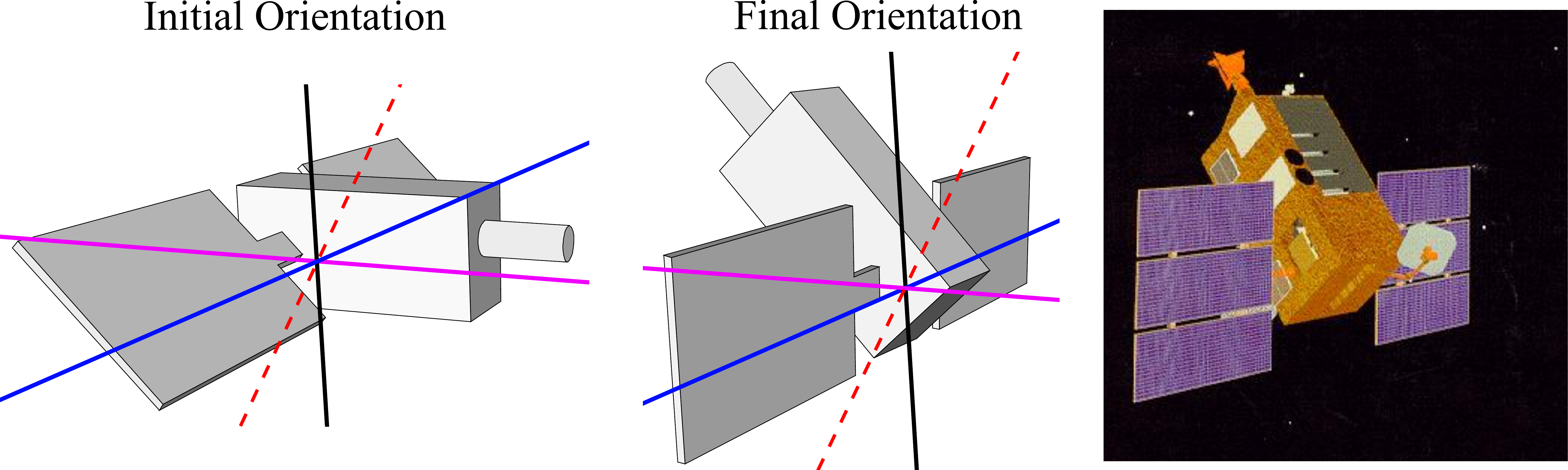}
	\caption{Left: Satellite's initial orientation. Centre: Satellite's final orientation. Right: Computer animation of the satellite.}
	\label{fig:satellitereorientation}
\end{figure}

In analogy to the stone, there are two initial and two end conditions on the angular position and on the angular speed: The satellite will start from the prescribed initial orientation and will stop at the prescribed final orientation. These boundary conditions read:
\begin{align*}
\bq(0)&=[0,0,0,1]^{\textsf{T}}\,,& \textcolor{red}{\bq(T)}&\textcolor{red}{=[0.96,0,0,0.28]^{\textsf{T}}}\,,\\
\bomega(0)&=\bO\,,& \dot{\bomega}(T)&=\bO\,.
\end{align*}
Essentially, the satellite will be rotated by $147^\circ$ around the blue axis in Figure~\ref{fig:satellitereorientation}. Because the satellite is long and narrow, the optimal control solution will actually not rotate the satellite around the blue axis, but instead simultaneously rotate by $180^\circ$ around the red and violet axes. The exact manner of doing this is non-trivial, hence requiring a numerical solution.

\subsubsection{Consistent Overdetermination}
Unfortunately, we only have three controls $\bu\in\R^3$ but four orientational states $\bq \in \R^4$. Hence, the problem is overdetermined. The problem is not infeasible though\footnote{because obviously one can rotate the satellite as depicted}. Solutions of the initial value problem naturally satisfy
\begin{align*}
\textcolor{blue}{\boldsymbol{-1} \leq \bq(t)}&\textcolor{blue}{\leq \be}\,,&\textcolor{violet}{1{}}{}&{}\textcolor{violet}{{}=\|\bq(t)\|_2^2}\,.
\end{align*}
Thus, the problem is overdetermined but still feasible.

The optimal control problem is completed by adding bound constraints that prevent rapid accelerations and spinning:
\begin{align*}
\boldsymbol{-20} \leq \bomega(t)&\leq \boldsymbol{20}\,,
& -\boldsymbol{50}\leq \bu(t)&\leq \boldsymbol{50}\,.
\end{align*}

\subsubsection{Illustration of Computational Issues}
Unless by modification, it is impossible to solve the above equations and bounds with any collocation method according to the definition in~\cite[Def.~7.6]{Hairer1}. This is due to the angular dynamics. To visualize the issue, we now indeed restrict the problem into the plane that is perpendicular to the blue axis in Figure~\ref{fig:satellitereorientation}. The restricted optimal control problem reads:
\begin{align*}
\bq(0)&=[0,\,0]\t\,,& \begin{bmatrix}
\dot{q}_{[1]}(t)\\
\dot{q}_{[2]}(t)
\end{bmatrix}&=\begin{bmatrix}
{}-{} {}q_{[2]}(t)\cdot\omega(t)\\
\phantom{{}-{}}{}q_{[1]}(t)\cdot\omega(t)
\end{bmatrix}\,,& \textcolor{red}{\bq(T)}&\textcolor{red}{=[-0.96,\ 0.28]\t}\,,\\
\omega(0)&=0\,,& \dot{\omega}(t)&=u(t)\,, & \omega(T)&=0\,,\\
\textcolor{blue}{-\be\leq \bq(t)}&\textcolor{blue}{\leq\be}\,,& -20\leq \omega(t)&\leq 20\,,& -50\leq u(t)& \leq 50\,,\\
\textcolor{violet}{\|\bq(t)\|_2^2}&\textcolor{violet}{=1}\,.& & & & &
\end{align*}
Of course, this is an oversimplification of the original problem because the satellite now indeed has to rotate around the blue axis. Nonetheless, the simpler example still suffices for illustrating where the numerical solution breaks.

\subsubsection{Numerical Results}
Figure~\ref{fig:satellitereorientationcollsolution} shows three collocation solutions for an adapation of this (otherwise unsolvable) problem. The adaptation removes the violet algebraic constraint, the blue bound constraints, and replaces the red end condition with $0.28\cdot q_{[1]}(T)+0.96\cdot q_{[2]}(T)=0$. The adaptation fixes $T=5$ and minimizes the integral over $u^2 + \omega^2$.

The figure shows numerical solutions of Explicit Euler, Trapezoidal Method (which is another collocation method), and LGR collocation of degree $p=2$. We see that all three methods violate all three color-indicated constraints. This explains why also the satellite reorientation problem cannot be solved with collocation methods. A larger number $N$ of mesh-intervals could decrease the magnitudes of these inconsistencies, but clearly they will not suddenly drop to zero. Hence, the conceptual idea of solving constraints exactly at certain points turns out once more to be flawed on this example.

\begin{remark}
	We included the trapezoidal method into the test because it is a \emph{symplectic} method, meaning that under special circumstances this it would satisfy $\|\bq(t)\|_2=1$ exactly at each collocation point $t$ \cite{HairerSymplect}. However, these special circumstances are not satisfied in this example, because $\omega$ is nonlinear in $t$.
\end{remark}
\begin{remark}
	When solved in practice, \cite{BettsChap2} proposes removal of the differential equation for $\bq_{[1]}$ and treating it as a control. However, this requires human intervention, is physically unintuitive, and compromises the  continuity (and hence differentiability) of $\bq_{[1]}$.
\end{remark}

\begin{figure}
	\centering
	\includegraphics[width=\linewidth]{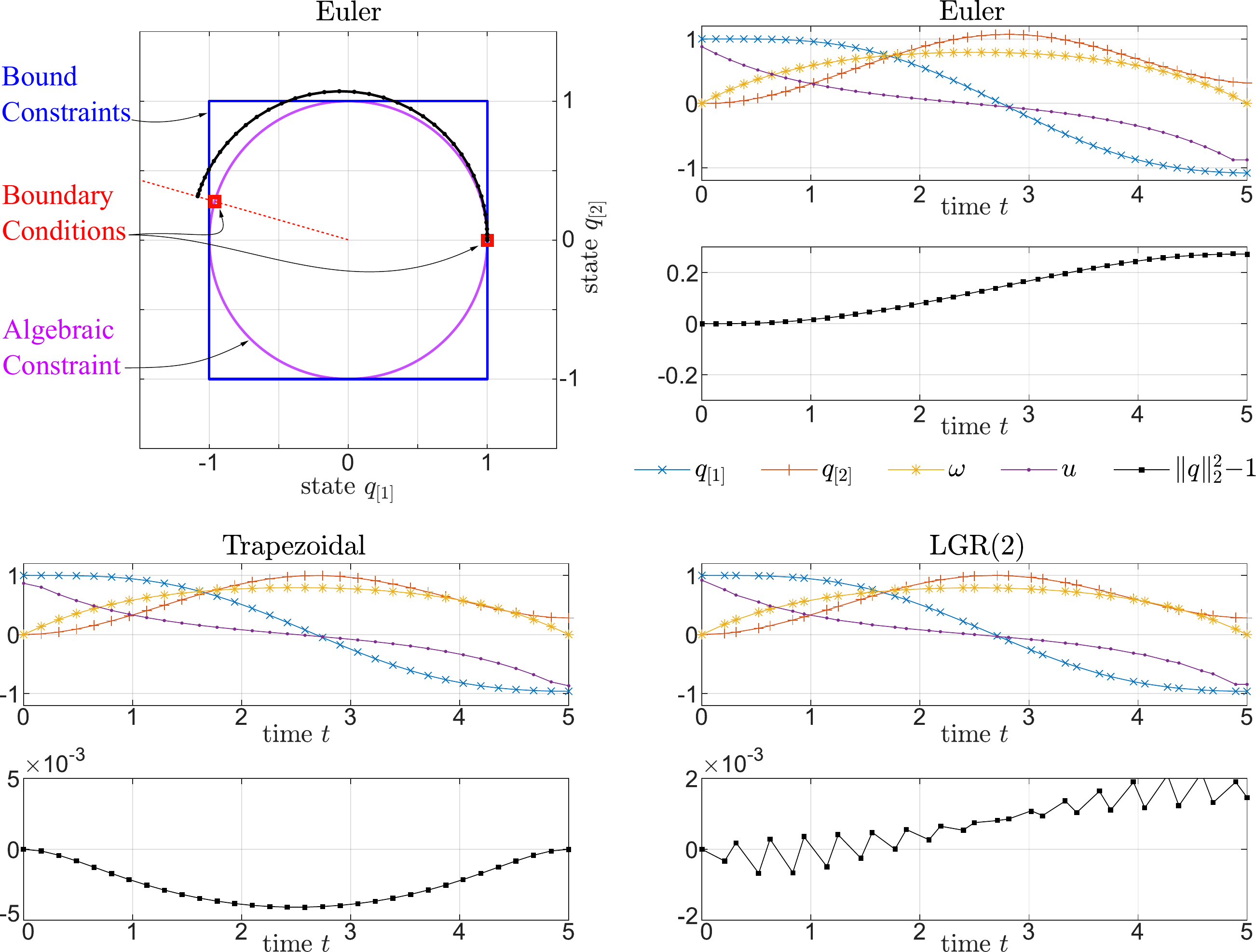}
	\caption{Top: Explicit Euler solution in state space and trajectory space. Bottom: Two collocation solutions in trajectory space. Observation: Each of these methods violates the equation $\|\bq(t)\|_2^2=1$.}
	\label{fig:satellitereorientationcollsolution}
\end{figure}

\subsection{Control of a Resonant Electric Circuit}
Optimization of electric circuits arises frequently in optimal control; e.g., the maximization of computational throughput subject to bounds on the CPU's temperature is an optimal control problem.

Electric circuits are modeled modularly: Each module (resistor, inductor, capacitor, diode, triode,...) has a model equation. A connection between any two modules is reflected in the model by equating certain variables.

\begin{figure}
	\centering
	\includegraphics[width=0.65\linewidth]{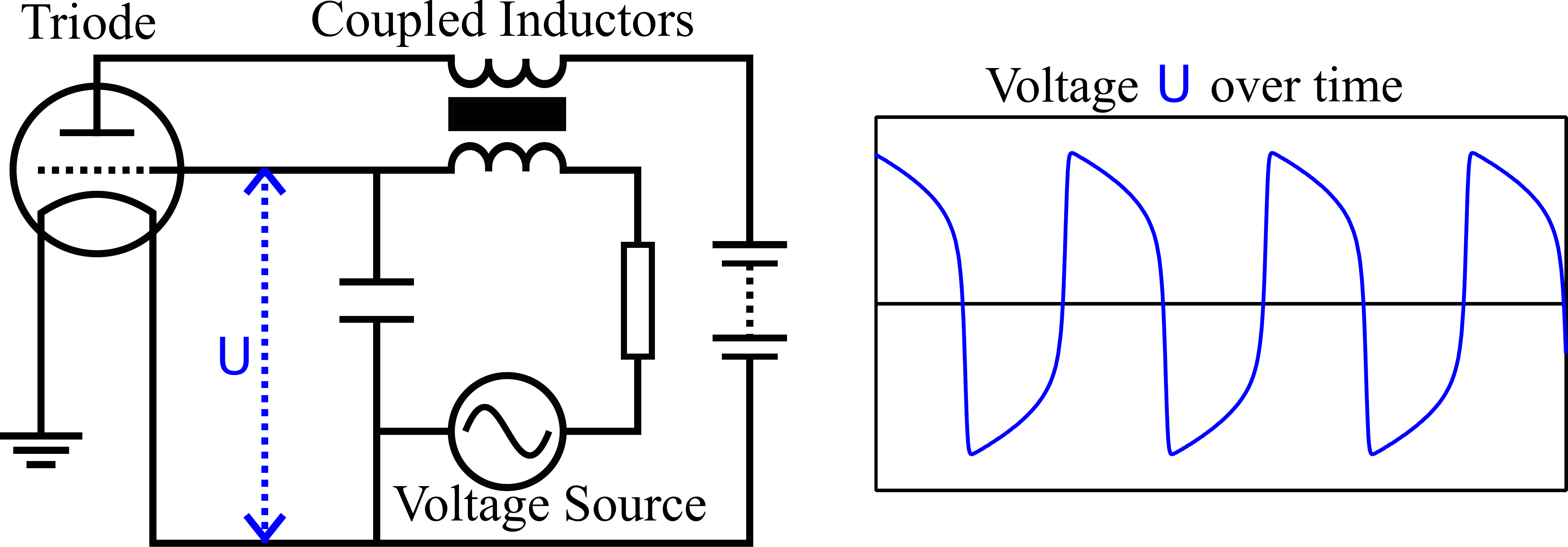}
	\caption{Left: Diagram of an electric circuit. Right: Dynamics of measured voltage reveals a Van-der-Pol oscillation.}
	\label{fig:vanderpolcircuit}
\end{figure}

In this example, we consider the Van-der-Pol optimal control problem from \cite{maurer2007theory}. 

Originally, the Van-der-Pol equation was a model for the electrodynamics in a triode, as depicted in Figure~\ref{fig:vanderpolcircuit}. Today, it is a basic oscillatory model for processes in physics, electronics, biology, neurology, sociology and economics \cite{MR952149}.

For the given electric circuit, an optimal control will be computed for the voltage source in Figure~\ref{fig:vanderpolcircuit} in order to damp the oscillator's amplitude. The problem reads:
\begin{equation*}
\left\lbrace
\begin{aligned}
&\operatornamewithlimits{min}_{y,u} & 	&\frac{1}{2} \cdot \int_0^4 \Big(\, y_{[1]}(t)^2 + y_{[2]}(t)^2 \,\Big)\,\mathrm{d}t\\
&\text{s.t.} 						& 	&y_{[1]}(0)=0\,,\quad \dot{y}_{[1]}(t)=y_{[2]}(t)\,,\quad -1\leq u_{[1]}(t)\leq 1\,,\\
& 									& 	&y_{[2]}(0)=1\,,\quad \dot{y}_{[2]}(t)=-y_{[1]}(t) + y_{[2]}(t) \cdot \big(1-y_{[1]}(t)^2\big) + u_{[1]}(t)\,.
\end{aligned}
\right\rbrace
\end{equation*}

We solve this problem with the same DCM and QPM methods specified in Section~\ref{sec:NumExp:Setup}.
To avoid doubts on the NLP accuracy, we solve DCM in ICLOCS2 \cite{Nie2018ICLOCS2TT} with IPOPT Version 3.12.3. We use $N=100$ mesh intervals. Figure~\ref{fig:vanderpolsolution} shows the numerical solution for $u_{h,[1]}$. As the figure shows, only the DCM solution overshoots the bound constraints and only the DCM solution rings on the last sub-arc.

We have seen in Section~\ref{sec:HigherOrderQuadMotiv} that higher-degree quadrature schemes can help preventing oscillations of numerical solutions. We observe in Figure~\ref{fig:vanderpolsolution} that the higher-order quadrature in QPM suppresses the ringing. This can be verified by comparing the values of $\rho$. DCM achieves $\rho\approx 3.7 \cdot 10^{-3}$ whereas QPM achieves $\rho\approx 6.5 \cdot 10^{-6}$.

The overshoot in DCM results in $\gamma \approx 8.9\cdot 10^{-1}$. In contrast, QPM keeps $\gamma\approx 3.1 \cdot 10^{-3}$.

\begin{figure}
	\centering
	\includegraphics[width=0.9\linewidth]{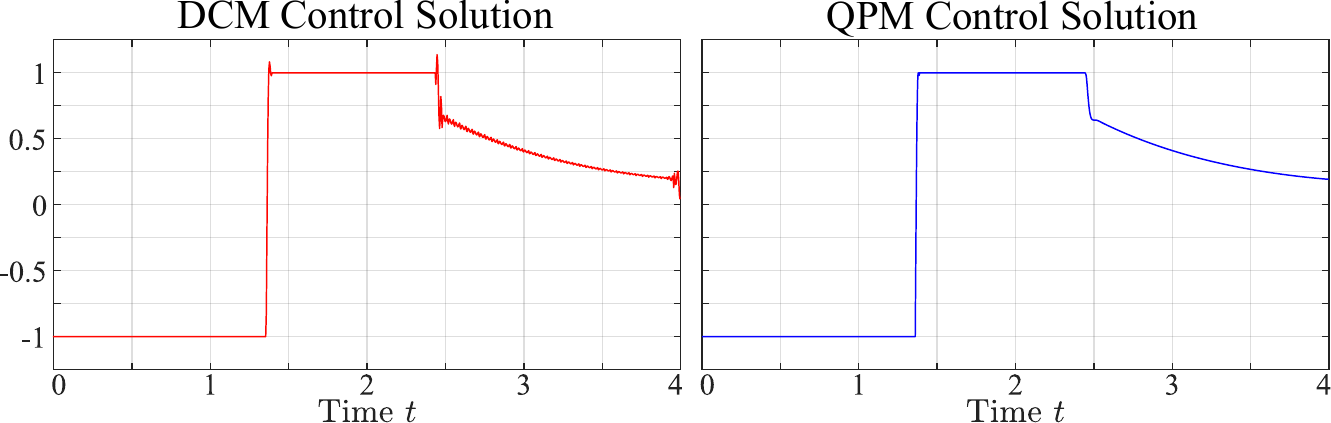}
	\caption{Solution to Van-der-Pol Optimal Control Problem with DCM and QPM as specified in Section~\ref{sec:NumExp:Setup}.}
	\label{fig:vanderpolsolution}
\end{figure}

\newpage

\part{Tailored Computation and Discretization Methods}
\label{part:TailoredMethods}
\chapter{Computation: Modified Augmented Lagranian Method}
\label{chap:MALM}

\section{Introduction}
This chapter describes a tailored mathematical algorithm for the computation of numerical solutions of instances of NLP when the objective contains large quadratic penalty terms. This method can be useful for solving instances of NLP that arise from direct transcriptions of optimal control problems via either the quadrature penalty or penalty-barrier method. The latter method is presented in Chapter~\ref{chap:PBF_SICON}.

\subsection{Problem Statement}
This chapter describes and analyzes a modified augmented Lagrangian method (MALM) for the numerical solution of a \textit{quadratic penalty program}:
\begin{align*}
\min_{\bx \in \cB} \ \Phi_{\pval}(\bx)&:= f(\bx) + \frac{1}{2  \pval}  \|c(\bx)\|_2^2\,, \tag{QPP}\label{eqn:CQPP}\\
\cB&:= \lbrace \bx \in \R^n\, \vert\, g(\bx)\geq \bO \rbrace\,,
\end{align*}
where $f : \R^n \rightarrow \R$, $c : \R^n \rightarrow \R^m$, $g : \R^n \rightarrow \R^p$ are possibly non-convex and nonlinear functions; $\geq$ is meant for each vector component; $\cB$ is the feasible set; $m,n,p \in \N$ are dimensions; $\pval \in \R_{>0}$ is part of the problem data.

\subsubsection{Relation to Constrained Programs (CP)}\label{sec:Intro:Relate}
When $\pval>0$ is close to zero then the penalty forces $c(\bx)\approx \bO$, provided such a point exists. Hence, the problem may be considered to be related to:
\begin{equation*}
\min_{\bx \in \cB} \qquad f(\bx) \quad\text{s.t.  }c(\bx)=\bO
\tag{CP}\label{eqn:CP}
\end{equation*}

We define the associated \emph{Lagrangian function} $\cL(\bx,\bdual,\boldsymbol{\eta}):=f(\bx)-\bdual\t\cdot c(\bx) - \boldsymbol{\eta}\t\cdot g(\bx)$, with Lagrange multipliers $\bdual \in \R^m$, $\boldsymbol{\eta} \in \R^p_{\geq 0}$. In~\eqref{eqn:CP}, $c,g$ are equality and inequality constraint functions with Lagrange multipliers $\bdual \in \R^m$, $\boldsymbol{\eta} \in \R^p_{\geq 0}$.

\subsubsection{Inconsistency}
\eqref{eqn:CP} only makes sense when $c(\bx)=\bO$ is consistent. However, for the scope of this work we are particularly interested in the case when $c$ is inconsistent. Experiments show that for inconsistent $c$ the solution of~\eqref{eqn:CQPP} depends significantly on the value of $\pval$; cf. Section~\ref{sec:Intro:Motiv} and Figure~\ref{fig:ocp_study}.

\subsubsection{Optimality Conditions}
From \cite[Thm~12.1]{Nocedal}:
\begin{align*}
&\underbrace{\nabla f(\bx) - \nabla c(\bx) \cdot \frac{-1}{\pval} \cdot c(\bx)}_{\equiv \nabla \Phi_{\pval}(\bx)} - \nabla g(\bx) \cdot \boldsymbol{\eta} = \bO\tag{KKT1'}\label{eqn:KKT1'}\\
&g_i(\bx)=0 	 \text{ and }\boldsymbol{\eta}_i\geq 0 		\qquad \forall i \in \cA\tag{KKT2a}\\
&g_i(\bx)>0  	 \text{ and }\boldsymbol{\eta}_i =   0     	\qquad \forall i \notin \cA\tag{KKT2b}
\end{align*}
where $\cA \subseteq \lbrace 1,\dots,p\rbrace$ is the \textit{active set}, and $g_i$ is the $i^\text{th}$ component of the vector $g(\bx)$.

Substituting $\bdual = \frac{-1}{\pval} \cdot c(\bx)$, we can re-express~\eqref{eqn:KKT1'}:
\begin{align}
\nabla_\bx \cL(\bx,\bdual,\boldsymbol{\eta})=\bO\,,\qquad c(\bx) + \pval \cdot \bdual =\bO
\tag{KKT1}\label{eqn:KKT1}
\end{align}
\eqnKKT (i.e., \eqnKKTa and \eqnKKTb) determines $\bx,\bdual,\boldsymbol{\eta}$. \eqnKKT are the optimality conditions of~\eqref{eqn:CQPP} when $\pval>0$ and the optimality conditions of~\eqref{eqn:CP} when $\pval=0$.

\subsection{Motivation}

\subsubsection{Necessity of Tailored Solvers for~\eqref{eqn:CQPP}}\label{sec:Intro:Motiv:Num}
Minimizing~\eqref{eqn:CQPP} directly appears natural but, unless $c$ is affine, will result in many iterations. This is caused by the bad scaling of the penalties.

As a demonstration, consider the instance
\begin{subequations}
	\begin{align}
	f(\bx)&:= -x_1\label{eqn:CircleF}\,,\quad g(\bx) :=\begin{bmatrix} x_1\\
	x_2-x_1
	\end{bmatrix} \in \R^{2}\\
	c(\bx)&:= \begin{bmatrix}
	(x_1 +\varepsilon)^2 + x_2^2 - 2\\
	(x_1 -\varepsilon)^2 + x_2^2 - 2
	\end{bmatrix} \in \R^{2}\label{eqn:CircleC}
	\end{align}
	\label{eqn:Circle}%
\end{subequations}
with primal and dual initial guesses $\bx_0 :=[2 \quad 1]\t$ and $\bdual_0 :=\bO$, for $\varepsilon=0$. We discuss later with \abbtab~\ref{tab:Exp_Circ_iter} that minimization of~\eqref{eqn:CQPP} of~\eqref{eqn:Circle} with a direct minimization method takes $134$ iterations when $\pval=10^{-6}$. This is inefficient when compared to our later proposed MALM, which solves the same instance in only $39$ iterations.

\subsubsection{Relevant Instances of~\eqref{eqn:CQPP}}\label{sec:Intro:Motiv}
Integral penalty methods \cite{Balakrishnan68,Hager90,Neuenhofen2020AnIP} are an alternative to collocation methods for solving dynamic optimization problems. Integral penalty methods can solve dynamic optimization problems with singular arcs and high-index differential-algebraic path-constraints as a problem of form~\eqref{eqn:CQPP}. Consider the bang-singular example
\begin{equation*}
\begin{aligned}
&\min_{y,u} &\quad &J:=\int_0^5 \left(\,y(t)^2 + t\,u(t)\,\right)\,\mathrm{d}t,\\
&\text{s.t.} & y(0)&=0.5,\quad\dot{y}(t)=\frac{1}{2}y(t)^2+u(t)\,,\\
&  & y(t),u(t)& \in [-1,1]\quad\forall t \in [0,5]\,.
\end{aligned}\tag{OCP}\label{eqn:ExampleOCP}
\end{equation*}
In integral-penalty-methods, the idea is to force $y(t)^2/2+u-\dot{y}=0$ not only at collocation points, but instead add an integral penalty $r=\int_0^5 \|y^2/2+u-\dot{y}\|_2^2\,\mathrm{d}t$ to the objective.

Consider using continuous piecewise linear finite elements $y_h$ for $y$ and discontinuous ones $u_h$ for $u$ on a uniform mesh of $N \in \N$ intervals (mesh size $h=5/N$); represented with $\bx := [y_h(h),\dots, y_h(Nh), u^+_h(0), u^-_h(h),u^+_h(h) \dots, u^-_h(Nh)]\t \in \R^n$, $n:={3 N}$. $y_h(0)=0.5$ is fixed and removed from $\bx$. We can minimize a quadrature approximation of $J + \frac{1}{2\pval}r$ by solving an instance of~\eqref{eqn:CQPP}, where
\begin{subequations}
	\label{eqn:OCPdisc}
	\begin{align}
	f(\bx)&:= \sum_{j=1}^{Nq} \wquad_{j} \left( y_h(\tau_{j})^2 + \tau_{j} \, u_h(\tau_j) \right)\label{eqn:OCPdiscF}\\
	c(\bx)&:= \begin{bmatrix} \vdots \\
	\sqrt{\wquad_{j}} \big(y_h(\tau_j)^2/2 + u_h(\tau_j)-\dot{y}_h(\tau_{j}) \big)\\
	\vdots
	\end{bmatrix} \in \R^{m}\label{eqn:OCPdiscC}\\
	g(\bx)&:= \begin{bmatrix}
	\be-\bx\\
	\be+\bx
	\end{bmatrix} \in \R^{p}\label{eqn:OCPdiscG}
	\end{align}
\end{subequations}
with $q$ quadrature points $\tau_j$ and weights $\alpha_j>0$ per mesh-interval. $\pval \in \R_{>0}$ is ideally chosen in $\cO(1/N)$ \cite{Hager90,neuenhofen2018dynamic}. Figure~\ref{fig:ocp_study} plots numerical solution $y_h,u_h$ for different values of $\pval$ against the analytic solution. The numerical solutions are vastly different for different $\pval$. Problem~\eqref{eqn:CP} is infeasible for~\eqref{eqn:OCPdisc} because $c(\bx)\neq \bO$ $\forall \bx \in \cB$.

In conclusion: Problems~\eqref{eqn:CQPP} and~\eqref{eqn:CP} have different solutions. Solutions of~\eqref{eqn:CQPP} depend on $\pval$. For a discussion on integral penalty methods, implementation, and choice of $\pval$, we refer to \cite{Hager90,Balakrishnan68,Neuenhofen2020AnIP}. The experiments in \cite{neuenhofen2018dynamic,Neuenhofen2020AnIP} present singular-arc and differential-algebraic optimal control problems where collocation methods fail to converge, but integral penalty methods converge.

\begin{figure}
	\centering
	\includegraphics[width=0.65\linewidth]{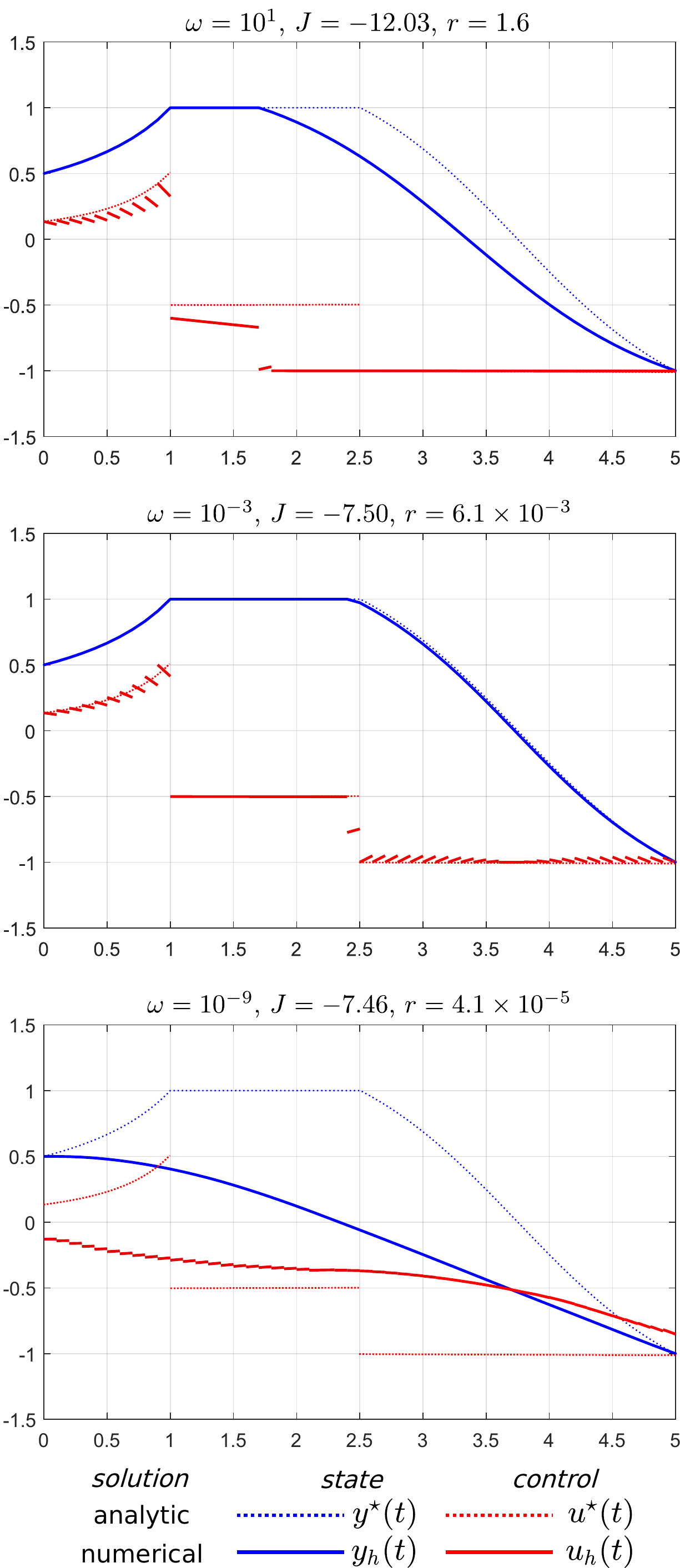}
	\caption{Numerical solution to~\eqref{eqn:ExampleOCP} for $N=40$ and different values of $\pval$.}
	\label{fig:ocp_study}
\end{figure}

\subsection{Literature Review}
We saw in Section~\ref{sec:Intro:Motiv:Num} that straightforward numerical minimization of~\eqref{eqn:CQPP} is inefficient due to bad scaling when $\pval$ is close to zero, hence necessitating tailored algorithms.

\subsubsection{Penalty Method (PM)}
PMs compensate for the bad scaling by iteratively minimizing a sequence of problems~\eqref{eqn:CQPP}. Therein, $\pval$ is replaced by a sequence of values $\lbrace\pmval_k\rbrace_{k \in \N_0}$ that converges to $\pval$ from above. We refer to \cite{Courant43,SUMT} for details. Actually, these methods have been proposed for problem~\eqref{eqn:CP}, i.e.\ when $\pval=0$; but they can also be used for~\eqref{eqn:CQPP}. This is so because PMs solve penalty problems of form~\eqref{eqn:CQPP}. PMs can converge slowly due to bad scaling \cite{Murray71}.

\subsubsection{Augmented Lagrangian Method (ALM)}
ALMs have been developed as a replacement for PMs when solving~\eqref{eqn:CP}. They work like PMs, but augment $\Phi_{\pmval_k}(\bx)$ with the term $-\bdual_{k}\t \cdot c(\bx)$. This term with $\bdual_k \in \R^m$ creates a term such that the inequality constrained minimizer of $\Phi_{\pmval_k}(\bx)-\bdual_{k}\t \cdot c(\bx)$ eventually matches with the minimizer of~\eqref{eqn:CP}. We refer to \cite[Alg.~17.3]{Nocedal} and the references therein for all details on how $\bdual_k \in \R^m$ is iteratively refined to achieve this. Convergence of $\bdual_k$ is asserted under suitable conditions \cite{Bertsekas1,Conn95}.

\subsubsection{Extensions of ALM to Inequality Constraints}
Originally, ALM treated only equality constraints \cite{Hestenes1,Powell1} by means of quadratic penalties of $c$ and update schemes for $\bdual$. In this case, inequalities can be \emph{subjected}~\cite{Conn95}, i.e.\ minimize the sequence of functions $\Phi_{\pmval_k}(\bx)-\bdual_{k}\t \cdot c(\bx)$ subject to $\bx \in \cB$. Alternatively, penalty or barrier terms of $g$ can be \emph{augmented} \cite{Rockafellar1} with according update strategies for $\boldsymbol{\eta}$. Subjections are considered more efficient in practice than augmentations~\cite{Lancelot}. Augmentations can suffer from non-smooth, non-differentiable, or low-order smooth penalties/barriers, and can converge slower or less reliably.

\subsubsection{Extensions of ALM to~\eqref{eqn:CQPP}}
Originally, ALM treated only~\eqref{eqn:CP} as opposed to~\eqref{eqn:CQPP}. The work \cite{SHARIFF2003257} proposes a modified scheme (MALM) for~\eqref{eqn:CQPP} when $f$ is quadratic, $c$ linear, $\Phi_{\pval}$ convex, and $\cB=\R^n$. They prove global convergence of their scheme. Our previous work in \cite{MALM_CDC20} extended the algorithm of MALM to problems where $f$ is nonlinear and may be non-convex, $c$ nonlinear, without additional inequality constraints $g$. Also, there is no convergence proof yet in the literature for the case when $f,c$ are general, regardless of the presence of $g$. This chapter will present such proofs.

\subsection{Challenges}
Our goal is in devising a method that solves~\eqref{eqn:CQPP} by solving a sequence of penalty problems with moderate penalty parameter $\pmval \gg \pval$, and prove its convergence. In the limit $\pval \rightarrow 0$, MALM should match ALM due to the relation of the problems~\eqref{eqn:CQPP} and~\eqref{eqn:CP} as described in Section~\ref{sec:Intro:Relate}.

Proving convergence for non-convex $\Phi_{\pval}$ is challenging because solutions of sub-problems may be non-unique and hence alternating. Convergence of $\boldsymbol{\eta}_k$ may be challenging to prove because the solution $\boldsymbol{\eta}$ of \eqnKKT may be non-unique. We will assert uniqueness of $\boldsymbol{\eta}$ from a strict complementarity assumption. Striking the right balance between mild assumptions and strong convergence assertions appears non-trivial in this context.

\subsection{Contributions}
We present MALM for general functions $f,c,g$ (Algorithm~1). We prove convergence for the case when $f,c$ are twice continuously differentiable and $g$ is linear (Theorem~\ref{thm:globconv}). Furthermore, we give a local rate-of-convergence result for the case when $f,c,g$ are twice local Lipschitz-differentiable (Theorem~\ref{thm:locconv}).

Theorem~1 is not easily extendable to nonlinear $g$ because it uses a result for ALM on~\eqref{eqn:CP} for linear~$g$. Theorem~2 works for general $g$ but assumes convergence and Lipschitz-continuous second derivatives of $f,c,g$. In the iteration limit, convergence of ALM can only be guaranteed to be at least at a linear rate \cite[Thm~17.6]{Nocedal}. Likewise, our rate-of-convergence result for MALM asserts only a linear rate. However, this linear rate is slightly better than the linear rate of ALM. Hence, our work draws connections between the rate of convergence between MALM and ALM.

\section{Derivation of the Algorithm of MALM}
\label{sec:MALM}
MALM is a solution method for~\eqref{eqn:CQPP}. MALM has been presented in \cite{SHARIFF2003257} for the special case when $f$ is quadratic, $c$ linear, and $\Phi_{\pval}$ convex. The method has been presented for the case where $f,c$ are general in \cite{MALM_CDC20} but without inequality constraints and without a convergence analysis. Here, we derive MALM for general nonlinear non-convex $f,c,g$, and in a stronger relation to its origins in ALM \cite{Hestenes1,Powell1}. For the method presented here, we give global and local convergence proofs.

The derivation poses an auxiliary problem, applies ALM to it, and then eliminates variables.

\subsection{Auxiliary Problem}
The following problem is equivalent to~\eqref{eqn:CQPP} but of the form~\eqref{eqn:CP}:
\begin{subequations}
	\begin{align}
	&\operatornamewithlimits{min}_{\hbx:=(\bx,\bxi) \in \R^{(n+m)}} 	&\quad 	\hat{f}(\hbx)&:=f(\bx) + \frac{\pval}{2}  \|\bxi\|_2^2 \label{eqn:SubstObjective}	\\
	&\text{s.t.} 	&		\hat{c}(\hbx)&:=c(\bx) + \pval  \bxi =\bO\,, \label{eqn:SubstConstraints}\\[3pt]
	& 				 	&		\hat{g}(\hbx)&:=g(\bx)\geq \bO\,.
	\end{align}
	\label{eqn:Subst}%
\end{subequations}
The optimality conditions of~\eqref{eqn:Subst} are \eqnKKTb and
\begin{subequations}
	\label{eqn:KKT1_Subst}
	\begin{align}
	\begin{bmatrix}
	\nabla f(\bx)\\
	\pval \bxi
	\end{bmatrix} - \begin{bmatrix}
	\nabla c(\bx)\\
	\pval  \bI
	\end{bmatrix}  \bdual -\begin{bmatrix}
	\nabla g(\bx)\\
	\bO
	\end{bmatrix} \boldsymbol{\eta} &=\bO\\
	c(\bx) + \pval  \bxi &= \bO\,. \label{eqn:5b}
	\end{align}
\end{subequations}

\subsection{Augmented Optimality System}
Since~\eqref{eqn:Subst} is of form~\eqref{eqn:CP}, we can apply ALM with augmented inequality constraints as in \cite{Conn95}. To this end, we introduce an auxiliary vector $\bz \in \R^m$ and a moderate penalty parameter $\pmval>0$. These are added to~\eqref{eqn:SubstConstraints} and in the gradient of the Lagrangian function:
\begin{subequations}
	\begin{align}
	\begin{bmatrix}
	\nabla f(\bx)\\
	\pval\bxi
	\end{bmatrix} - \begin{bmatrix}
	\nabla c(\bx)\\
	\pval  \bI
	\end{bmatrix}  (\bdual+\bz) - \begin{bmatrix}
	\nabla g(\bx)\\
	\bO
	\end{bmatrix}\boldsymbol{\eta} &=\bO\label{eqn:KKT1_Subst_ALM_unreduced1}\\
	c(\bx) + \pval  \bxi + \pmval  \bz &= \bO\,.\label{eqn:KKT1_Subst_ALM_unreduced2}
	\end{align}
	\label{eqn:KKT1_Subst_ALM_unreduced}%
\end{subequations}

We could use~\eqref{eqn:KKT1_Subst_ALM_unreduced} directly in order to form an ALM iteration. That iteration would consist of two alternating steps: 1) solving the optimality system~\eqref{eqn:KKT1_Subst_ALM_unreduced} together with \eqnKKTb for $(\bx,\bxi,\bz,\boldsymbol{\eta},\cA)$ where $\bdual$ is fixed; 2) updating $\bdual \leftarrow \bdual + \bz$, being equivalent to $\bdual \leftarrow \bdual - \frac{1}{\pmval}\left(c(\bx)+\pval\bxi\right)$.

\subsection{Elimination of the Auxiliary Vector}
Instead, we propose to eliminate $\bxi = \bdual+\bz$ to obtain
\begin{subequations}
	\begin{align}
	\nabla f(\bx) - \nabla c(\bx)  (\bdual+\bz) - \nabla g(\bx) \boldsymbol{\eta} &=\bO\\
	c(\bx) + \pval  \bdual + (\pval+\pmval)  \bz &= \bO \label{eqn:AlternatedKKT_MALM_primal}\,.
	\end{align}
	\label{eqn:AlternatedKKT_MALM}%
\end{subequations}
As in ALM, we solve~\eqref{eqn:AlternatedKKT_MALM} and \eqnKKTb with an iteration of two alternating steps:
\begin{enumerate}
	\item Keep the value of $\bdual$ fixed, and solve~\eqref{eqn:AlternatedKKT_MALM} and \eqnKKTb for $(\bx,\bz,\boldsymbol{\eta},\cA)$.
	\item Update $\bdual$ as $\bdual \leftarrow \bdual +\bz\,.$
\end{enumerate}
Analogous to ALM, the first step can be realized by minimizing an augmented Lagrangian function
for $\bx$ at fixed $\bdual$ subject to $\bx \in \cB$, whereas in the second step $\bz$ can be expressed in terms of $\bx$ from~\eqref{eqn:AlternatedKKT_MALM_primal}. Using this, the method can be expressed  in \abbalg~\ref{algo:MALM}, where
\begin{align}
\Psi^\pval_{k+1}(\bx) := \cL(\bx,\bdual_{k},\bO) + \frac{0.5}{\pval + \pmval} \left\|c(\bx)+\pval  \bdual_{k}\right\|_2^2 \label{eqn:ALF}
\end{align}
is the augmented Lagrangian function, with $\cL(\bx,\bdual,\bO)\equiv f(\bx)-\bdual\t \cdot c(\bx)$.

\begin{algorithm}[tb]
	\caption{Modified Augmented Lagrangian Method}
	\label{algo:MALM}
	\begin{algorithmic}[1]
		\Procedure{MALM}{$f,c,g,\pval,\bx_0,\bdual_0,\tol$}
		\State $\pmval \leftarrow \pmval_0$
		\For{$k=1,2,3,\dots,k_\text{max}$}
		\State Compute $\bx_k$, and optionally $\boldsymbol{\eta}_k$, by solving
		\begin{align}
		&\min_{\bx \in \R^{n}}\quad\Psi^\omega_{k}(\bx)\quad\text{s.t. }\ g(\bx)\geq \bO\,. \label{eqn:ALF_Prob}
		\end{align}
		\State Update $\bdual_k \leftarrow \bdual_{k-1} - \frac{1}{\pval + \pmval} \left(c(\bx_k)+\pval \bdual_{k-1}\right)$
		\vspace{2mm}
		\If{$\|c(\bx_k)+\pval \bdual_k\|_\infty\leq\tol$}
		\State \Return $\bx_k,\bdual_k$ and optionally $\boldsymbol{\eta}_k$
		\Else
		\State Decrease $\pmval\leftarrow c_{\pmval} \pmval$ to promote convergence.
		\EndIf
		\EndFor
		\EndProcedure
	\end{algorithmic}
\end{algorithm}

\subsection{Practical Aspects}\label{sec:pracAspects}

\paragraph{Parameters}
Values that we have found work well in practice are $\tol=10^{-8}, c_\pmval=0.1, \pmval_0=0.1$. Care must be taken that $\Psi^\omega_k$ in~\eqref{eqn:ALF_Prob} is bounded below. To this end, practical methods impose box constraints $\bx_L \leq \bx \leq \bx_U$ \cite[eq.~3.2.2]{Lancelot}, expressible via $g(\bx)\geq \bO$, or a trust-region constraint $g(\bx)=\Delta^2 - \|\bx-\bx_{k-1}\|_2^2 \geq 0$ \cite[eq.~3.2.4]{Lancelot} with trust-region radius $\Delta>0$.

\paragraph{Inner Optimization Algorithm}
In order to minimize~\eqref{eqn:ALF_Prob}, one can use any numerical method for inequality constrained nonlinear minimization; e.g.\ an interior-point method like IPOPT \cite{IPOPT} or an active set method like SNOPT \cite{GilMS05}. These methods can use second-order information and thereby attain a quadratic local rate of convergence for each inner optimization problem. Alternatively, pure first-order methods such as the projected gradient descent algorithm in \cite{Nocedal} could be used.

\paragraph{Linear Systems}
We refer to \cite[eq.~17.21]{Nocedal} for details on how the quasi-Newton direction for the quadratic penalty function can be computed in a more numerically stable fashion from a saddle-point linear equation system.

\paragraph{Preconditioning}
Apart from solution procedures for the linear systems involved in any inner optimization algorithm, one could imagine the use of preconditioning techniques to accelerate the rate of convergence for the outer iteration; that is, to yield faster convergence of the sequence $\lbrace\blambda_k\rbrace$. Section~\ref{sec:convAnalysisMALM} provides a local convergence analysis with a matrix $\bM$ in \eqref{eqn:Banach} that yields the update for $\blambda$. If we could precondition the optimization problem in such a way that $\|\bM\|_2$ is reduced then this would dramatically improve the rate of convergence. For general NLPs, however, such a scheme is unknown.

\subsection{Discussion}
\subsubsection{True Generalization of ALM}\label{sec:truegeneralization}
MALM is a true generalization of ALM, because they differ only by the parameter $\pval$. In particular, if $\pval=0$ then MALM in Algorithm~\ref{algo:MALM} is identical to ALM in \cite[Algorithm~3.1]{Conn95}. In contrast, when selecting $\pval>0$, we show below that MALM converges to critical points of~\eqref{eqn:CQPP} with the given~$\pval$.

\subsubsection{Benefit}
MALM solves the penalty function $\Phi_{\pval}$ in~\eqref{eqn:CQPP} by minimizing a sequence of penalty functions~$\Psi^\omega_k$. When does this make sense? If we select $\pmval \gg \pval$. Thereby, the penalty functions~$\Psi^\omega_k$ have better scaling and hence can often be minimized more efficiently in comparison to a single minimization of $\Phi_{\pval}$. The computational performance results in Section~\ref{sec:NumExp} verify this claim.

\section{Convergence Analysis}\label{sec:convAnalysisMALM}
The below analyses assume that all sub-problems~\eqref{eqn:ALF_Prob} are solved exactly, and that computations are performed in exact arithmetic. Throughout this subsection, MALM means the callback-function in \abbalg~\ref{algo:MALM}, wherein any black-box method can be used to solve~\eqref{eqn:ALF_Prob}.

\subsection{Global Convergence}
For the analyses, we consider a call of Algorithm~1 with instance $\cI:=(f,c,g,\pval,\bx_0,\bdual_0,\tol)$. MALM will create a sequence of iterates $\bx_k,\bdual_k$. 

\begin{lem}[Equivalence]
	MALM on the instance $\cJ:=(\hat{f},\hat{c},\hat{g},0,\hbx_0,\bdual_0,\tol)$ from~\eqref{eqn:Subst} will generate the same iterates $(\bx_k,\bxi_k),\bdual_k$ as MALM on the instance $\cI$ in terms of $\bx_k,\bdual_k$.
\end{lem}
\noindent
\textit{Proof:} By induction over $k$. Base: For $k=0$ the proposition holds by construction of the initial guesses. Step: Let the proposition hold for $k-1$. We now show that the proposition holds for $k$. The iterate $\hbx_k$ from $\cJ$ in line~4 necessarily satisfies $\nabla_{\hbx}\hat{\Psi}^0_k(\hbx_k)- \nabla_{\hbx} \hat{g}(\hbx_k)\cdot\boldsymbol{\eta}_k=\bO$, which is equivalent to $\nabla_{\bx}{\Psi}^\pval_k(\bx_k)- \nabla_{\bx} {g}(\bx_k)\cdot\boldsymbol{\eta}_k=\bO$ when $\hbx_k:=(\bx_k,\bxi_k)$ with
\begin{align}
\bxi_k := \frac{1}{\pval+\pmval} \big(\pmval\bdual_{k-1}-c(\bx_k)\big)\,.\label{eqn:bxi_func}
\end{align}
Thus, the first component $\bx_k$ of $\hbx_k$ is a valid $k^\text{th}$ iterate of MALM on $\cI$.
Finally, by insertion of \eqref{eqn:bxi_func} into the below, notice that $\bdual_k$ in $\cI,\cJ$ are identical because the increments $\bdual_k-\bdual_{k-1}$ are identical:
$$ -\frac{1}{\rho} \hat{c}(\hbx_k) = -\frac{1}{\rho+\pval}\big(c(\bx_k)+\pval\blambda_{k-1}\big)\,.$$
q.e.d.

In turn, MALM with $\pval=0$ is identical to ALM in \cite[Algorithm~3.1]{Conn95}. We can hence use the convergence result from \cite[Thm~4.6]{Conn95}:

\begin{theorem}[Global Convergence]\label{thm:globconv}
	Choose a bounded domain $\Omega \subset \R^n$. Let $\pval>0$, $c$ be bounded on $\Omega$, and let $f,c$ be twice continuously differentiable in $\Omega$, and $g$ affine. Suppose all iterates $\lbrace\bx_k\rbrace_{k \in \N_0}$ of MALM live in $\Omega$. If $\rho_0$ is sufficiently small then $\lbrace\bx_k\rbrace_{k \in \N_0}$ converges to a critical point of~\eqref{eqn:CQPP}.
\end{theorem}
\noindent
\textit{Proof:} \cite[Thm~4.6]{Conn95} shows convergence of ALM for $\cJ$ under four assumptions (AS1)-(AS4). It suffices to show that $\hat{f},\hat{c},\hat{g}$ satisfy these assumptions.

Feasibility \cite[AS1]{Conn95} of~\eqref{eqn:Subst} holds naturally by $\bxi= \frac{-1}{\pval} \,c(\bx)$.
Twice continuous differentiability \cite[AS2]{Conn95} of $\hat{f},\hat{c}$ holds per requirement.
Boundedness \cite[AS3]{Conn95} of all $\hbx_k \in \Omega \times c(\Omega)$ follows from boundedness of $\Omega$ and $c$ on $\Omega$.

The last assumption \cite[AS4]{Conn95} is more technical. Since $g$ is affine, we can express $g(\bx)=\bA\cdot\bx-\bb$, and likewise $\hat{g}(\hbx)=\hbA\cdot\hbx-\bb$, where $\hbA = [\bA\ \bO]$. We define the matrix $\bZ$ of orthonormal columns that span the null-space of $\hbA_\cA$, i.e.\ the matrix of sub-rows of $\hbA$ of the active constraints at $\hbx$. (AS4) requires $\nabla\hat{c}(\hbx)\cdot\bZ$ to be of column rank $\geq m$. Due to the special structure of $\hbA$, we see that $\bZ$ has a structure like

\begin{align*}
\bZ = \begin{bmatrix}
\begin{matrix}
\bO\\
\bI
\end{matrix}&
\begin{matrix}
\dots\\
\dots
\end{matrix}
\end{bmatrix}\,.
\end{align*}
Since $\nabla_{\hbx}\hat{c}(\hbx)\t=[\nabla c(\hbx)\t \ \pval\bI]$ has full row rank, the rank of $\nabla\hat{c}(\hbx)\t\cdot\bZ$ is bounded below by the number of columns of $\bZ$, i.e.\ bounded below by $m$. q.e.d.

Some of the requirements in Theorem~\ref{thm:globconv} may be forcible: Section~\ref{sec:pracAspects} explains how $\cB$ can be bounded. In this case, choosing $\Omega=\cB$ yields $\lbrace \bx_k\rbrace \subset \Omega$. Also, $c$ may be bounded over $\Omega$ by approximating $c(\bx)$ with $\arctan\big(c(\bx)\big)$. If $\|c(\bx)\|_2$ is very small at the minimizer of~\eqref{eqn:CQPP} then the approximation error of $\arctan$ is negligible. To make $g$ affine, several practical ALM implementations (Lancelot, MINOS) convert inequalities to equalities via the addition of slack variables $\bs\geq 0$~\cite[Sec.~17.4]{Nocedal}. The constraints $g(\bx)-\bs=\bO$ (as in~\cite[eqn~17.47]{Nocedal}) can be merged into $c$ and scaled such that they hold tightly. Also, interior-point methods like IPOPT~\cite{IPOPT} use slacks to ensure iterates are strictly interior.

\subsection{Local Convergence}\label{app:Local}
\cite[Thm~17.6]{Nocedal} asserts linear convergence of ALM when $\nabla_{\bx\bx}^2\cL,\nabla c,\nabla_\bx\cL$ are local Lipschitz-continuous and $\pmval$ is a constant. Section~\ref{sec:NumExp:Circle:Rate} and Figure~\ref{fig:convrate} show this. Likewise, MALM attains a linear rate in the limit when $\pmval$ is a constant. Upper bounds for these rates can be computed. In this section we prove that the rate of MALM is strictly smaller than that of ALM.

For the following result, we compare the iteration of ALM and MALM from the same initial guess $\bx_0,\bdual_0$ and the same problem-defining functions $f,c,g$. We assume that $\lbrace \bx_k \rbrace \subset \cU$, where $\cU\subset \R^n$ is an open neighborhood which contains unique local minimizers of both~\eqref{eqn:CP} and~\eqref{eqn:CQPP}.

\begin{theorem}[Local Convergence]\label{thm:locconv}
	Let $\nabla f$, $\nabla c$, $\nabla_{\bx\bx}^2 \cL$ be Lipschitz-continuous $\forall \bx \in \cU$ and let all iterates of ALM and MALM remain in $\cU$. Let the local minimizers satisfy strict complementarity. Apply ALM and MALM with fixed penalty parameter $\pmval$ to solve either problem, each starting from $\bx_k$. If both methods converge and if $\bx_k$ is sufficiently close to the local minimizer of~\eqref{eqn:CQPP}, then the linear rates of convergence of MALM and ALM satisfy the relation $C_\text{MALM}=\frac{\pmval}{\pmval+\pval} \cdot C_\text{ALM}<C_\text{ALM}$.
\end{theorem}
\noindent
\textit{Proof:}
We use the Taylor series
\begin{align*}
{}&{}\nabla \Psi^\omega_k(\bx_k,\bdual_{k-1})\\
={}&{} \bH\bx_k+\bg-\frac{1}{\pval+\pmval}\bJ\t\left(\pmval\bdual_{k-1}+\bc\right) + R_L(\bx_k,\bdual_{k-1})
\end{align*}
with $\bJ\t := \nabla c(\bx_\infty)$, $\bH := \nabla^2_{\bx\bx}\cL(\bx_\infty,\bdual_\infty,\bO) + \frac{1}{\pval+\rho}\bJ\t\bJ$, $\bc := \bJ \bx_\infty - c(\bx_\infty)$ and $\bg:=\nabla f(\bx_\infty)$ has the Lagrange remainder $\|R_L(\bx_k,\bdual_k)\|_2 \leq \frac{L}{\rho+\pval} (\|\bx_k-\bx_\infty\|_2 + \|\bdual_{k-1}-\bdual_\infty\|_2)^2$, where $L$ is the Lipschitz constant.

We now first consider the case where $p=0$, i.e.\ when there are no inequality constraints. Since $\bx_k$ is convergent by requirement, $\bH$ must be positive semi-definite and, if $\bx_\infty$ is locally unique, $\bH$ must be positive definite. Clearly, local convergence to a unique point depends quantitatively on uniqueness, hence we imply $\lambda_\text{min}(\bH)\geq \mu>0$. For the induced 2-norm it follows that $\|\bH\inv\|_2\leq \mu$, hence
\begin{align*}
&{}\left\|\bx_k-\bH\inv\left(\bg-\frac{1}{\pval+\pmval}\bJ\t(\pmval\bdual_{k-1}+\bc)\right)\right\|_2\\
\leq{}&{}\frac{L}{\mu(\rho+\pval)}\|\bdual_{k-1}-\bdual_\infty\|^2_2.
\end{align*}
Inserting the estimate for $\bx_k$ into line~5 in \abbalg~\ref{algo:MALM} gives a formula for $\bdual_k$ that only depends on $\bdual_{k-1}$:
\begin{align}
\bdual_k = \bM \cdot \bdual_{k-1} + \mathbf{f} + R_{\bdual}(\bdual_k) \label{eqn:Banach}
\end{align}
with $\bM \in \R^{m \times m}$  below, some $\mathbf{f} \in \R^m$, and $\|R_{\bdual}(\bdual_k)\|_2 \leq \frac{1}{\mu}\left(\frac{L}{\rho+\pval}\right)^2 \|\bdual_k-\bdual_\infty\|_2^2$. Rearranging  reveals
\begin{align*}
\bM &= \frac{\pmval}{\pval+\pmval}\left(\bI - \frac{1}{\pval+\pmval}\bJ\bH\inv\bJ\t\right)\,.
\end{align*}
Since \abbthm~\ref{thm:globconv} asserts convergence of $\bdual_k$, the second order terms become negligible compared to the first-order terms and can hence be ignored in the limit. Then,~\eqref{eqn:Banach} is a Banach iteration. Thus, in the limit, the rate of convergence for $\bdual_k$ is linear with contraction $\|\bM\|_2<1$. The analysis holds regardless of whether $\pval=0$ or $>0$.

We see that in the limit MALM converges faster than ALM because $\frac{\pmval}{\pval+\pmval}<1$ when $\pval>0$, whereas $\frac{\pmval}{\pval+\pmval}=1$ when $\pval=0$. Hence, in the limit $k\rightarrow \infty$ MALM yields a stronger contraction for the errors per iteration than ALM. This is in particular an advantage in cases where ALM would converge slowly. For instance, choosing $\rho=10\pval$ guarantees convergence in the limit with at least a rate of contraction of $\frac{\pmval}{\pval+\pmval}<0.91$\,.

\balance

From the above, when dropping the Lagrange remainder terms, we can 
identify the local rate of convergence by that of the following quadratic model iteration: 1) Solve
\begin{align*}
\operatornamewithlimits{min}_{x} \frac{1}{2}\bx\t\bH\bx + \left(\bg+\frac{1}{\pval+\pmval}\bJ\t(\bJ\bx_{k-1}-\bc-\pval\bdual_{k-1})\right)\t\bx\,.
\end{align*}
2) Update $\bdual_k:=\bdual_{k-1}-\frac{1}{\pval+\pmval}(\bJ\bx_k-\bc-\pval\bdual_{k-1})$.

We  discuss the case when $p>0$, i.e.\ when inequality constraints are present. We use our assumption on strict complementarity, i.e.\ $i \in \cA \Leftrightarrow \boldsymbol{\eta}_i>\beta$ for some real $\beta>0$. Since $\bx_k$ converges by requirement, $\bdual_{k-1}$ converges and thus also $\nabla \Psi^\omega_k(\bx_k)$ converges. Hence, $\boldsymbol{\eta}_k$ must converge in order to yield $\nabla_\bx \Psi^\omega_k(\bx_k)-\nabla g(\bx)\,\boldsymbol{\eta}_k=\bO$. Once $\boldsymbol{\eta}_k$ changes less than $\beta$ at some finite $k_0 \in \N$, the active set $\cA_k$ will remain unchanged $\cA_\infty$ for all subsequent iterations $k\geq k_0$. We use $\bg_\infty$ for only the active constraints of $g$ and define $\bA_\infty := \nabla \bg_\infty(\bx_\infty)\t$, $\bb_\infty:=\nabla \bg_\infty(\bx_\infty)\t \cdot \bx_\infty$; hence $g_\infty(\bx) = \bA_\infty \cdot \bx - \bb_\infty + \cO(\|\bx-\bx_\infty\|_2^2)$.

Given the above intermezzo, the appropriate model iteration in the limit becomes obvious: 1) Solve
\begin{align*}
&\operatornamewithlimits{min}_{x} \frac{1}{2}\bx\t\bH\bx + \left(\bg+\frac{1}{\pval+\pmval}\bJ\t(\bJ\bx_{k-1}-\bc-\pval\bdual_{k-1})\right)\t\bx\\
&\text{s.t. }\bA_\infty \cdot \bx = \bb_\infty\,.
\end{align*}
2) Update $\bdual_k:=\bdual_{k-1}-\frac{1}{\pval+\pmval}(\bJ\bx_k-\bc-\pval\bdual_{k-1})$.

This is just a projection of the iteration above. Thus, we can project the iteration for $\bx_k$ onto the nullspace of $\bA_\infty$, identifying $\bx_k = \bx_r + \bN \tbx_k$ $\forall k \geq k_0$, where $\bx_r \in \R^n$ has active set $\cA_\infty$, $\tbx_k \in \R^{n-\dim(\cA_\infty)}$ and $\bN$ is a matrix of orthogonal columns that span the nullspace of $\nabla g_\infty(\bx_\infty)$. Defining $\tbH:=\bN\t\bH\bN$, $\tbJ:=\bJ\bN$, and $\tbg,\tbc$ appropriately, we arrive at the former unconstrained quadratic model iteration form, but with $\bH,\bg,\bJ,\bc,\bx_k$ replaced by the tilded quantities. Accordingly, the Banach iteration matrix $\bM$ is replaced with the matrix 
$$ 	\tbM = \frac{\pmval}{\pval+\pmval} \left(\bI-\frac{1}{\pval+\pmval}\tbJ\tbH\inv\tbJ\t\right)\,. 	$$
The resulting contraction matrix $\tbM$ for the Banach iteration of the inequality constrained case has  a factor $\frac{\rho}{\pval+\pmval}$ in front, just like for the case when $p=0$. Thus, for $\pmval>0$ the method converges locally faster in the limit $k\rightarrow\infty$. q.e.d.

\section{Numerical Experiments}\label{sec:NumExp}
For our tests we use two instances:~\eqref{eqn:Circle} and~\eqref{eqn:OCPdisc}. Each instance will be considered once as~\eqref{eqn:CQPP} and once as~\eqref{eqn:CP}. Both instances are parametric: The inconsistency of~\eqref{eqn:Circle} grows in the order of $\varepsilon$ and inconsistency of~\eqref{eqn:OCPdisc} grows in the order of the mesh size $h$. The sub-problems in~\eqref{eqn:ALF_Prob} are solved with IPOPT version 12.0.3. For tests on examples with equality constraints only, we refer to \cite{MALM_CDC20}.

\subsection{Circle Problem}
\subsubsection{Setting}
\paragraph{Initial Guess and Numerical Methods}
We use the initial guess $\bx_0=[2\,\ 1]\t$, $\bdual_0=\bO$. Fig.~\ref{fig:example1geometry} shows the instance's geometry. The figure also shows two points $\bx_A := [0\,\ \sqrt{2}]\t,\ \bx_B := [1\,\ 1]\t$.

\begin{figure}
	\centering
	\includegraphics[width=0.5\linewidth]{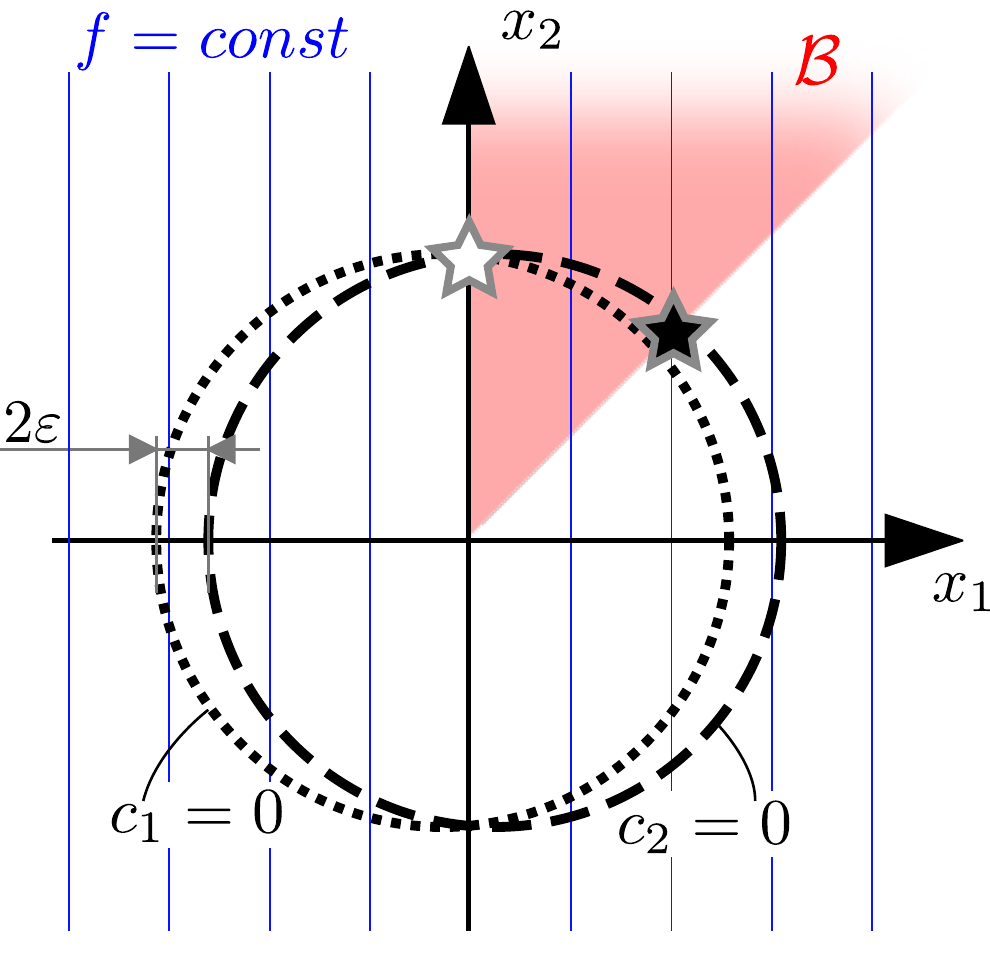}
	\caption{Geometry of the Circle Problem, with level sets of $f,c_1,c_2$ in blue solid, black dotted, and black dashed lines, respectively. The domain $\cB$ is highlighted in red. The points $\bx_A,\bx_B$ are marked as white and black star, respectively.}
	\label{fig:example1geometry}
\end{figure}

\paragraph{Expected Minimizers}
When considering the instance as~\eqref{eqn:CP} then we expect that $\bx_A$ would be the solution. To see this, notice that $c(\bx)=\bO$ is only satisfied at $\bx=\bx_A$. When $\varepsilon \rightarrow 0$, \eqnKKT becomes ill-conditioned for $\bx_A$. Once $\varepsilon=0$, the minimizer is suddenly $\bx_B$.

When considering the instance as~\eqref{eqn:CQPP} then a point close to $\bx_B$ should be the solution unless $\varepsilon$ becomes large relative in comparison to $\pval$. To see this, note that $\bx_B$ minimizes $f$ among all points in $\cB$ that yield $\|c(\bx)\|_2$ small relative to $\pval$.

\paragraph{Scope}
Both ways~\eqref{eqn:CP} and~\eqref{eqn:CQPP} of interpreting the instance~\eqref{eqn:Circle} and both solutions $\bx_A,\bx_B$ make sense in their own right. We want to find out which solver works best for solving a respective combination $\pval,\varepsilon$.

\subsubsection{Computational Results}
We observe that all iterates of all methods remain in $\Omega=\cB \cap \lbrace \bx \in \R^2 \vert x_2 \leq 2 \rbrace$. Hence, Theorem~\ref{thm:globconv} asserts that MALM and ALM converge because $f,c$ are twice continuously differentiable on $\Omega$ and $g$ is affine.

We solve the instance with MALM and PM, for various values of $\varepsilon,\pval$, including~0. We implement PM by solving~\eqref{eqn:CQPP} directly in IPOPT with objective $\Phi_{\pval}$. Recall that MALM=ALM for $\pval=0$ and that PM is not applicable (n.a.) when $\pval=0$, since $\Phi_{\pval}$ is undefined.

\paragraph{Confirmation of Expected Minimizers}
We first analyze the limit points $\bx_\infty$ (which are identical for both tested methods throughout all tests) for each $\varepsilon,\pval$, by measuring the quantities
\begin{align*}
e_A := \|\bx_\infty - \bx_A\|_2\,,\qquad e_B:=\|\bx_\infty-\bx_B\|_2\,.
\end{align*}

\abbTab~\ref{tab:Exp_Circ_conv} shows the quantities $e_A,e_B$ for respective $\varepsilon,\pval$. Dividing the table into a lower left and an upper right triangle, we see that indeed solutions in the lower triangle are close to $\bx_A$ and those on the diagonal and in the upper right are close to $\bx_B$. This confirms that solutions of~\eqref{eqn:CP} and~\eqref{eqn:CQPP} can be very distinct and the latter depend on the value of $\pval$.

\begin{table}%[tb]
	\caption{Solution of the Circle Problem with respect to $\varepsilon,\pval$. Smaller values mean closer convergence to either point. Cells in the lower left converge to $\bx_A$, cells in the upper right to $\bx_B$.}
	\label{tab:Exp_Circ_conv}
	\centering
	\begin{tabular}{cl||c|c|c|c|c|}    \cline{3-7}\multicolumn{1}{l}{}&&\multicolumn{5}{c|}{$\varepsilon$}\\ \cline{2-7}  
		\multicolumn{1}{l|}{}& $\begin{matrix}e_A\\e_B\end{matrix}$			&$1.0\text{e--}1$											&$1.0\text{e--}2$ 				 								&$1.0\text{e--}4$ 												&$1.0\text{e--}6$ 												&$0.0$\\ \hline\hline
		\multicolumn{1}{|c|}{\multirow{5}{*}{$\pval$}} 	& $1.0\text{e--}1$	&$\begin{matrix}1.1\text{e--}0\\2.6\text{e--}3\end{matrix}$	&$\begin{matrix}1.1\text{e+}0	\\4.3\text{e--}3\end{matrix}$	&$\begin{matrix}1.1\text{e+}0	\\4.4\text{e--}3\end{matrix}$	&$\begin{matrix}1.1\text{e+}0	\\4.4\text{e--}3\end{matrix}$ 	&$\begin{matrix}1.1\text{e+}0	\\4.4\text{e--}3\end{matrix}$\\ \cline{2-7}
		\multicolumn{1}{|c|}{\multirow{5}{*}{}}  		& $1.0\text{e--}2$	&$\begin{matrix}1.2\text{e--}1\\9.6\text{e--}1\end{matrix}$	&$\begin{matrix}1.1\text{e+}0	\\3.7\text{e--}4\end{matrix}$	&$\begin{matrix}1.1\text{e+}0	\\4.4\text{e--}4\end{matrix}$	&$\begin{matrix}1.1\text{e+}0	\\4.4\text{e--}4\end{matrix}$ 	&$\begin{matrix}1.1\text{e+}0	\\4.4\text{e--}4\end{matrix}$\\ \cline{2-7}
		\multicolumn{1}{|c|}{\multirow{5}{*}{}} 		& $1.0\text{e--}4$	&$\begin{matrix}3.8\text{e--}3\\1.1\text{e+}0\end{matrix}$	&$\begin{matrix}1.2\text{e--}1	\\9.6\text{e--}1\end{matrix}$	&$\begin{matrix}1.1\text{e+}0	\\4.4\text{e--}6\end{matrix}$	&$\begin{matrix}1.1\text{e+}0	\\4.4\text{e--}6\end{matrix}$ 	&$\begin{matrix}1.1\text{e+}0	\\4.4\text{e--}6\end{matrix}$\\ \cline{2-7}
		\multicolumn{1}{|c|}{\multirow{5}{*}{}} 		& $1.0\text{e--}6$	&$\begin{matrix}3.5\text{e--}3\\1.1\text{e+}0\end{matrix}$	&$\begin{matrix}1.3\text{e--}3	\\1.1\text{e+}0\end{matrix}$	&$\begin{matrix}1.1\text{e+}0	\\3.7\text{e--}8\end{matrix}$	&$\begin{matrix}1.1\text{e+}0	\\4.4\text{e--}8\end{matrix}$ 	&$\begin{matrix}1.1\text{e+}0	\\4.4\text{e--}8\end{matrix}$\\ \cline{2-7}
		\multicolumn{1}{|c|}{\multirow{5}{*}{}} 	    & $1.0\text{e--}8$	&$\begin{matrix}3.5\text{e--}3\\1.1\text{e+}0\end{matrix}$	&$\begin{matrix}3.7\text{e--}5	\\1.1\text{e+}0\end{matrix}$	&$\begin{matrix}1.2\text{e--}1	\\9.6\text{e--}1\end{matrix}$	&$\begin{matrix}1.3\text{+}0	\\7.1\text{e--}9\end{matrix}$ 	&$\begin{matrix}1.3\text{e+}0	\\7.1\text{e--}9\end{matrix}$\\ \cline{2-7}
		\multicolumn{1}{|c|}{\multirow{5}{*}{}} 	    & $0.0$	&$\begin{matrix}3.5\text{e--}3\\1.1\text{e+}0\end{matrix}$	&$\begin{matrix}3.7\text{e--}5	\\1.1\text{e+}0\end{matrix}$	&$\begin{matrix}1.2\text{e--}1	\\9.6\text{e--}1\end{matrix}$	&$\begin{matrix}1.3\text{+}0	\\7.1\text{e--}9\end{matrix}$ 	&$\begin{matrix}1.3\text{e+}0	\\0.0\end{matrix}$\\ \hline
	\end{tabular}
\end{table}

\paragraph{Computational Performance}
\abbTab~\ref{tab:Exp_Circ_iter} shows the sum of the number of all inner iterations of PM and MALM for respective $\varepsilon,\pval$. We see a trend for each of the two methods: PM converges in a few iterations when $\pval$ is moderate. However, when $\varepsilon,\pval$ both decrease, the iteration count blows up. The trend for MALM is different. MALM converges reliably for all $\varepsilon,\pval$ in the upper right triangle, including those where $\varepsilon,\pval$ are very small.

The last row of \abbTab~\ref{tab:Exp_Circ_iter} shows ALM. ALM converges quickly to $\bx_B$ when $\varepsilon=0$. In contrast, when $\varepsilon\neq 0$ then ALM should converge to $\bx_A$ but its iteration count blows up for small $\varepsilon>0$. In two instances ALM did not converge (n.c.) within $1000$ iterations. In conclusion, ALM is inefficient when $c$ has small inconsistencies.

\begin{table}%[tb]
	\caption{Total number of IPOPT iterations for MALM and PM for the Circle Problem with respect to $\varepsilon,\pval$.  Fewer iterations mean better computational efficiency; highlighting best in slanted (PM) or bold (MALM).}
	\label{tab:Exp_Circ_iter}
	\centering
	\begin{tabular}{cl||c|c|c|c|c|}    \cline{3-7}\multicolumn{1}{l}{}&&\multicolumn{5}{c|}{$\varepsilon$}\\ \cline{2-7}  
		\multicolumn{1}{l|}{}&$\begin{matrix}\#_\text{MALM}\\\#_\text{PM}\end{matrix}$& $1.0\text{e--}1$& $1.0\text{e--}2$& $1.0\text{e--}4$& $1.0\text{e--}6$& $0.0$\\ \hline\hline
		\multicolumn{1}{|c|}{\multirow{5}{*}{$\pval$}} & $1.0\text{e--}1$	& $\begin{matrix}\text{28}\\ \textsl{14}\end{matrix}$ 	& $\begin{matrix}\text{22}\\ \textsl{13}\end{matrix}$& 	$\begin{matrix}\text{22}\\ 		\textsl{13}\end{matrix}$&  	$\begin{matrix}\text{19}\\ \textsl{13}\end{matrix}$& 		$\begin{matrix}\text{19}\\ \textsl{13}\end{matrix}$\\ \cline{2-7}
		\multicolumn{1}{|c|}{\multirow{5}{*}{}} & $1.0\text{e--}2$ 			& $\begin{matrix}\text{36}\\ \textsl{12}\end{matrix}$ 	& $\begin{matrix}\text{28}\\ \textsl{16}\end{matrix}$& 	$\begin{matrix}\text{16}\\ 		\text{16}\end{matrix}$& 	$\begin{matrix}\text{23}\\ \textsl{16}\end{matrix}$& 		$\begin{matrix}\text{20}\\ \textsl{16}\end{matrix}$\\ \cline{2-7}
		\multicolumn{1}{|c|}{\multirow{5}{*}{}} & $1.0\text{e--}4$ 			& $\begin{matrix}\text{21}\\ \textsl{16}\end{matrix}$ 	& $\begin{matrix}\text{56}\\ \textsl{36}\end{matrix}$& 	$\begin{matrix}\textbf{32}\\ 		\text{43}\end{matrix}$& 	$\begin{matrix}\textbf{29}\\ \text{43}\end{matrix}$& 		$\begin{matrix}\textbf{23}\\ \text{43}\end{matrix}$\\ \cline{2-7}
		\multicolumn{1}{|c|}{\multirow{5}{*}{}} & $1.0\text{e--}6$ 			& $\begin{matrix}\text{29}\\ \textsl{16}\end{matrix}$ 	& $\begin{matrix}\text{68}\\ \textsl{35}\end{matrix}$& 	$\begin{matrix}\textbf{45}\\ 		\text{138}\end{matrix}$& 	$\begin{matrix}\textbf{39}\\ \text{134}\end{matrix}$& 	$\begin{matrix}\textbf{31}\\ \text{134}\end{matrix}$\\ \cline{2-7}
		\multicolumn{1}{|c|}{\multirow{5}{*}{}} & $1.0\text{e--}8$			& $\begin{matrix}\text{34}\\ \text{n.~c.}\end{matrix}$	& $\begin{matrix}\text{60}\\ \text{n.~c.}\end{matrix}$& $\begin{matrix}\text{n.~c.}\\ 	\text{n.~c.}\end{matrix}$& 	$\begin{matrix}\textbf{52}\\ \text{429}\end{matrix}$& 	$\begin{matrix}\textbf{40}\\ \text{374}\end{matrix}$\\ \cline{2-7}
		\multicolumn{1}{|c|}{\multirow{5}{*}{}} & $0.0$ 					& $\begin{matrix}\text{34}\\ \text{n.~a.}\end{matrix}$	& $\begin{matrix}\text{60}\\ \text{n.~a.}\end{matrix}$& $\begin{matrix}\text{n.~c.}\\ 	\text{n.~a.}\end{matrix}$& 	$\begin{matrix}\text{52}\\ \text{n.~a.}\end{matrix}$& 	$\begin{matrix}\text{40}\\ \text{n.~a.}\end{matrix}$\\ \hline
	\end{tabular}
\end{table}

\paragraph{Rate-of-Convergence Comparison}\label{sec:NumExp:Circle:Rate}
We compare the rate of local convergence of MALM and ALM to the theoretical prediction from Theorem~\ref{thm:locconv}. We use $\varepsilon=0$ and $\rho=1$\,. MALM solves~\eqref{eqn:CQPP} with $\pval=10^{-1}$ whereas ALM solves~\eqref{eqn:CP}. Both minimizers are close to $\bx_B$. As shown in Theorem~\ref{thm:locconv}, local convergence is only linear for our constant choice of $\rho$. Hence, by Cauchy criterion, $\|\bdual_{k}-\bdual_{k-1}\|_2$ converges at the same rate as $\|\bdual_{k}-\bdual^\star\|_2$, where $\bdual^\star$ is the exact dual solution. Because $\bdual^\star$ is unknown, Figure~\ref{fig:convrate} plots $\|\bdual_{k}-\bdual_{k-1}\|_2$ for both methods over the outer iteration index $k$ of Algorithm~\ref{algo:MALM}. Thereby, we find the rate of convergence for $\bdual_k$ and thus for $\bx_{k+1}$. We observe convergence at linear rates. We see that both methods converge in very few outer iterations to the order of machine precision. At $k\geq 9$ both methods have roughly attained their limit convergence rates.

\begin{figure}
	\centering
	\includegraphics[width=0.7\linewidth]{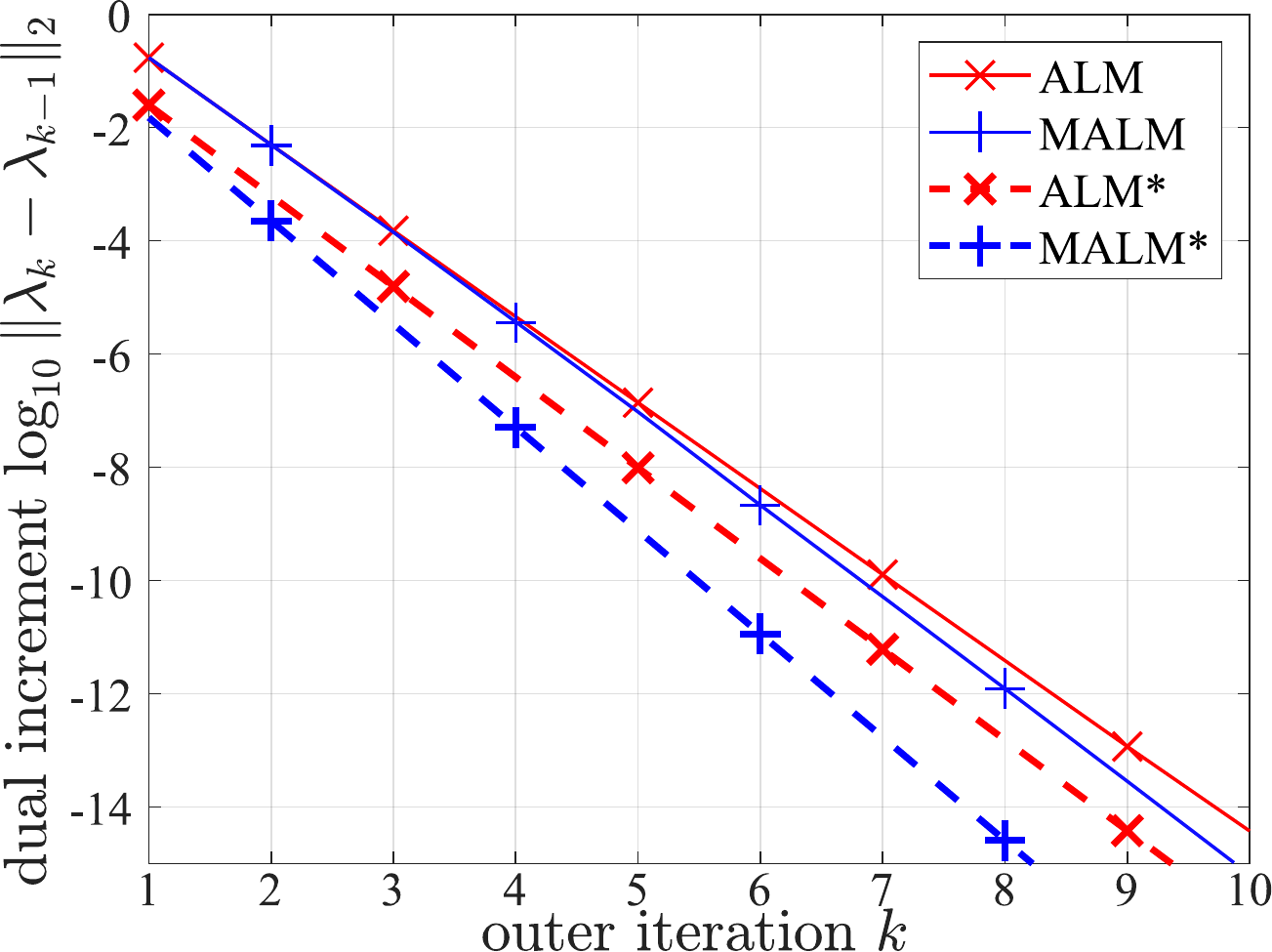}
	\caption{Comparison of convergence for ALM and MALM for the circle problem. Solid lines are measured convergence rates. Dotted lines indicate the theoretical rates of ALM and MALM in the limit $k \rightarrow \infty$.}
	\label{fig:convrate}
\end{figure}

\subsection{Integral Penalty-Discretization for Optimal Control}
\subsubsection{Setting}
\paragraph{Initial Guess and Solvers}
We solve the instance~\eqref{eqn:OCPdisc} with MALM and PM for various values of $h,\pval$ from $\bx_0=\bO,\bdual_0=\bO$. Recall that $h$ is the mesh size and $\|c(\bx)\|_2^2\equiv \int_0^5 (y_h^2/2 + u_h-\dot{y}_h)^2\,\mathrm{d}t$.

For this example, using the bounded domain $\Omega=\cB$, Theorem~\ref{thm:globconv} asserts a priori that MALM converges because $f,c$ are twice continuously differentiable on $\Omega$ and $g$ is affine.

\paragraph{Expected Minimizers}
We expect that the optimality gap and feasibility residual
\newcommand{\symgap}{\delta J}
\begin{align*}
\delta J= f(\bx)-J(y^\star,u^\star)\,,\quad r= \|c(\bx)\|_2^2
\end{align*}
both converge for increasing mesh sizes $N$ when choosing $\pval \in \cO(h)$; cf. discussion in Section~\ref{sec:Intro:Motiv}. For $\pval$ too large, $r$ should not converge and for $\pval$ too small $\delta J$ should not converge. To see this, notice that $\Phi_{\pval}=J+\frac{1}{2 \pval} r$; thus, minimization of $\Phi_{\pval}$ only strikes a balance between minimizing both terms when $\pval$ is chosen in the right order of magnitude

\paragraph{Scope}
We expect that again PM will be faster than MALM when $\pval$ is moderate and vice versa when $\pval$ is very small.
We shall also try ALM (i.e., Algorithm~\ref{algo:MALM} with $\pval=0$) but just for completeness, because this will not converge to the optimal control solution.

\subsubsection{Computational Results}
\paragraph{Confirmation of Expected Minimizers}
\abbTab~\ref{tab:Exp_OCP_conv} shows the quantities $\symgap,r$ for respective $h,\pval$. Dividing the table into a lower left and an upper right triangle, we find our expected minimizers confirmed: solutions in the lower left of the table achieve small $r$ but large $\delta J$, whereas solutions in the upper right of the table are not sufficiently feasible with respect to the path-constraints. For a given mesh size $h$, the most accurate control solutions are found on the diagonal cells of the table.

\begin{table}%[tb]
	\caption{Solution of the Optimal Control Problem with respect to $N,\pval$. For a given mesh size $N$, the value for $\pval$ is suitable when $\delta J$ (optimality gap) and $r$ (feasibility residual) have similar magnitude.}
	\label{tab:Exp_OCP_conv}
	\centering
	\begin{tabular}{cc||c|c|c|c|c|}    \cline{3-7}\multicolumn{1}{l}{}&&\multicolumn{5}{c|}{$h$}\\ \cline{2-7}  
		\multicolumn{1}{l|}{}& $\begin{matrix}\symgap\\r\end{matrix}$ 			&$1.0\text{e--}1$ 																	& $2.0\text{e--}2$ 																& $1.0\text{e--}2$ 																	& $5.0\text{e--}3$ 																& $2.5\text{e--}3$\\ \hline\hline
		\multicolumn{1}{|c|}{\multirow{5}{*}{$\pval$}} 	& $1.0\text{e--}2$		&$\begin{matrix}\text{-}1.7\text{e--}1\\9.2\text{e--}2\end{matrix}$		&$\begin{matrix}\text{-}1.7\text{e--}1	\\4.3\text{e--}2\end{matrix}$ 	&$\begin{matrix}\text{-}1.7\text{e--}1	\\4.3\text{e--}2\end{matrix}$		&$\begin{matrix}\text{-}1.7\text{e--}1\\4.2\text{e--}2\end{matrix}$		&$\begin{matrix}\text{-}1.7\text{e--}1\\4.3\text{e--}2\end{matrix}$\\ \cline{2-7}
		\multicolumn{1}{|c|}{\multirow{5}{*}{}} 		& $1.0\text{e--}3$ 		&$\begin{matrix}        4.3\text{e--}2\\2.7\text{e--}2\end{matrix}$		&$\begin{matrix}\text{-}8.6\text{e--}3	\\5.3\text{e--}3\end{matrix}$ 	&$\begin{matrix}\text{-}9.6\text{e--}3	\\4.6\text{e--}3\end{matrix}$		&$\begin{matrix}\text{-}9.8\text{e--}3\\4.4\text{e--}3\end{matrix}$		&$\begin{matrix}\text{-}9.9\text{e--}3\\4.4\text{e--}3\end{matrix}$\\ \cline{2-7}
		\multicolumn{1}{|c|}{\multirow{5}{*}{}} 		& $1.0\text{e--}4$ 		&$\begin{matrix}        7.6\text{e--}2\\8.8\text{e--}3\end{matrix}$		&$\begin{matrix}1.9\text{e--}2 			\\6.1\text{e--}4\end{matrix}$ 	&$\begin{matrix}1.2\text{e--}2			\\6.1\text{e--}4\end{matrix}$		&$\begin{matrix}\text{-}8.7\text{e--}3\\6.0\text{e--}4\end{matrix}$		&$\begin{matrix}\text{-}7.2\text{e--}3\\5.7\text{e--}4\end{matrix}$\\ \cline{2-7}
		\multicolumn{1}{|c|}{\multirow{5}{*}{}} 		& $1.0\text{e--}5$ 		&$\begin{matrix}        7.9\text{e--}2\\2.5\text{e--}3\end{matrix}$		&$\begin{matrix}2.3\text{e--}2			\\6.4\text{e--}5\end{matrix}$ 	&$\begin{matrix}1.6\text{e--}2			\\6.2\text{e--}5\end{matrix}$		&$\begin{matrix}        1.2\text{e--}2\\6.3\text{e--}5\end{matrix}$		&$\begin{matrix}        1.0\text{e--}2\\6.6\text{e--}5\end{matrix}$\\ \cline{2-7}
		\multicolumn{1}{|c|}{\multirow{5}{*}{}} 		& $1.0\text{e--}6$ 		&$\begin{matrix}        8.0\text{e--}2\\6.5\text{e--}4\end{matrix}$		&$\begin{matrix}2.3\text{e--}2			\\1.8\text{e--}5\end{matrix}$ 	&$\begin{matrix}1.6\text{e--}2			\\7.5\text{e--}6\end{matrix}$		&$\begin{matrix}        1.2\text{e--}2\\1.1\text{e--}5\end{matrix}$		&$\begin{matrix}        1.0\text{e--}2\\1.5\text{e--}5\end{matrix}$\\ \cline{2-7}
		\multicolumn{1}{|c|}{\multirow{5}{*}{}} 		& $0.0$ 				&$\begin{matrix}        7.2\text{e+}0\\0.0\end{matrix}$ 				&$\begin{matrix}7.2\text{e+}0			\\0.0\end{matrix}$ 				&$\begin{matrix}7.2\text{e+}0			\\0.0\end{matrix}$					&$\begin{matrix}        7.2\text{e+}0\\0.0           \end{matrix}$		&$\begin{matrix}        7.2\text{e+}0\\0.0           \end{matrix}$\\ \hline
	\end{tabular}
\end{table}

\paragraph{Computational Performance}
\abbTab~\ref{tab:Exp_OCP_iter} shows the sum of the number of all inner iterations of MALM and PM for respective $h,\pval$. We see the same trend as for the circle problem: PM converges faster than MALM when $\pval$ is moderate and vice versa when $\pval$ is small. We underline that MALM converges reliably for all $h,\pval$ in the upper right triangle, including those where $h,\pval$ are very small. Needless to say, accurate numerical optimal control solutions require $h,\pval$ very small; thus MALM seems very attractive for solving these classes of problems.

The last row shows that ALM does not converge (n.c.) within $500$ iterations for any mesh size.

\begin{table}%[tb]
	\caption{Total number of IPOPT iterations for MALM and PM for the Optimal Control Problem with respect to $N,\pval$. Fewer iterations mean better computational efficiency; highlighting best in slanted (PM) or bold (MALM).}
	\label{tab:Exp_OCP_iter}
	\centering
	\begin{tabular}{cc||c|c|c|c|c|}    \cline{3-7}\multicolumn{1}{l}{}&&\multicolumn{5}{c|}{$h$}\\ \cline{2-7}  
		\multicolumn{1}{l|}{}&$\begin{matrix}\#_\text{MALM}\\\#_\text{PM}\end{matrix}$& 	$1.0\text{e--}1$ 																	& $2.0\text{e--}2$ 																& $1.0\text{e--}2$ 																	& $5.0\text{e--}3$ 																& $2.5\text{e--}3$\\ \hline\hline
		\multicolumn{1}{|c|}{\multirow{5}{*}{$\pval$}} 	& $1.0\text{e--}2$& 				$\begin{matrix}\text{39}\\ 		\textsl{17}\end{matrix}$& 	$\begin{matrix}\text{53}\\ \textsl{24}\end{matrix}$& 			$\begin{matrix}\text{64}\\    \textsl{23}\end{matrix}$& 		$\begin{matrix}\text{69}\\    \textsl{26}\end{matrix}$& 		$\begin{matrix}\text{69}\\    \textsl{44}\end{matrix}$\\ \cline{2-7}
		\multicolumn{1}{|c|}{\multirow{5}{*}{}}			& $1.0\text{e--}3$& 				$\begin{matrix}\text{47}\\ 		\textsl{38}\end{matrix}$& 	$\begin{matrix}\text{63}\\ \textsl{44}\end{matrix}$& 			$\begin{matrix}\text{74}\\    \textsl{43}\end{matrix}$& 		$\begin{matrix}\text{90}\\    \textsl{44}\end{matrix}$& 		$\begin{matrix}\text{89}\\    \textsl{68}\end{matrix}$\\ \cline{2-7}
		\multicolumn{1}{|c|}{\multirow{5}{*}{}}  		& $1.0\text{e--}4$& 				$\begin{matrix}\text{61}\\ 		\textsl{57}\end{matrix}$& 	$\begin{matrix}\textbf{66}\\ \text{92}\end{matrix}$& 			$\begin{matrix}\textbf{80}\\    \text{85}\end{matrix}$& 		$\begin{matrix}\textbf{93}\\    \text{133}\end{matrix}$& 		$\begin{matrix}\textbf{121}\\   \text{124}\end{matrix}$\\ \cline{2-7}
		\multicolumn{1}{|c|}{\multirow{5}{*}{}} 		& $1.0\text{e--}5$& 				$\begin{matrix}\textbf{64}\\ 		\text{102}\end{matrix}$& 	$\begin{matrix}\textbf{78}\\ \text{161}\end{matrix}$& 		$\begin{matrix}\textbf{80}\\    \text{266}\end{matrix}$& 		$\begin{matrix}\textbf{93}\\    \text{242}\end{matrix}$& 		$\begin{matrix}\textbf{113}\\   \text{252}\end{matrix}$\\ \cline{2-7}
		\multicolumn{1}{|c|}{\multirow{5}{*}{}} 		& $1.0\text{e--}6$& 				$\begin{matrix}\textbf{81}\\ 		\text{163}\end{matrix}$& 	$\begin{matrix}\textbf{78}\\ \text{222}\end{matrix}$& 		$\begin{matrix}\textbf{110}\\   \text{328}\end{matrix}$& 		$\begin{matrix}\textbf{97}\\    \text{278}\end{matrix}$& 		$\begin{matrix}\textbf{126}\\   \text{224}\end{matrix}$\\ \cline{2-7}
		\multicolumn{1}{|c|}{\multirow{5}{*}{}} 		& $0.0$& 							$\begin{matrix}\text{n.~c.}\\ 	\text{n.~a.}\end{matrix}$& 	$\begin{matrix}\text{n.~c.}\\ \text{n.~a.}\end{matrix}$& 	$\begin{matrix}\text{n.~c.}\\ \text{n.~a.}\end{matrix}$& 	$\begin{matrix}\text{n.~c.}\\ \text{n.~a.}\end{matrix}$& 	$\begin{matrix}\text{n.~c.}\\ \text{n.~a.}\end{matrix}$\\ \hline
	\end{tabular}
\end{table}

\section{Conclusions}
We presented a modified augmented Lagrangian method (MALM), generalized to non-convex optimization problems with additional inequality constraints. We proved global convergence for our generalized method when the inequalities are affine. A local rate-of-convergence result shows that MALM inherits all the local convergence results of ALM while the regularization in $\pval>0$ also yields a slight benefit to its rate of local convergence in the iteration limit.

Our numerical experiments demonstrate that MALM outperforms PM when minimizing
quadratic penalty programs~\eqref{eqn:CQPP} in those situations where $\pval$ is very small, in a similar manner as ALM outperforms PM when solving equality constrained programs~\eqref{eqn:CP}. The experiments further show that ALM cannot solve~\eqref{eqn:CQPP}, but solves~\eqref{eqn:CP} instead. Hence, MALM is the best candidate for solving~\eqref{eqn:CQPP} when $\pval$ is very small.

In the experiments we have assumed that the sub-problems~\eqref{eqn:ALF_Prob} are solved to high accuracy. Future work could extend the approach to inexact iterations and sub-iterations to mild tolerances. This could reduce computations at sub-iterations where the dual is far from converged. Another open subject is the extension of global convergence analysis to the cases when $g$ is convex nonlinear or non-convex nonlinear.

\chapter{Discretization: Integral Penalty Barrier Method}
\label{chap:PBF_SICON}

	\section{Introduction}
This chapter presents and analyses a mathematical algorithm for the direct transcription of optimal control problems. In analogy to QPM, we use finite elements and nonlinear programming solvers to solve the optimal control problem as an NLP. However, in contrast to QPM, the method of this chapter uses not only quadratic penalty terms but also logarithmic barrier terms. This yields a better alignment between the merit-functional of the direct transcription and the merit-function of the interior-point algorithm. As motivated in Section~\ref{sec:Scope}, this may decrease the number of iterations for the NLP solver to converge. Still, the new way of discretization requires a separate proof of convergence, which is the emphasize of this chapter. The notation of this chapter is independent and uncorrelated to the rest of the thesis. For instance, the functions $f,c,b$ have different purposes. As another example, the symbol $\cX$ denotes the solution space in this chapter, too, but it uses a different definition from the one in Part~\ref{part:Intro}. Also, there are assumptions (A.1)--(A.5), but these are different from the assumptions that will be used in Part~\ref{part:convproof} of this thesis.

\subsection{An Important Class of Dynamic Optimization Problems}
\label{sec:classDOP}
Many optimal control, estimation, system identification and design problems can be written as a dynamic optimization problem in the Lagrange form
\newcommand{\refOCP}{(DOP)\xspace}
\begin{align}
\operatornamewithlimits{min}_{x:=(y,z) \in \cX} 	
\int_\Omega f( \dot{y}(t),y(t),z(t),t)&\, \mathrm{d}t  \label{eqn:OCP_SICON} \tag{DOPa}\\
\text{subject to\quad}  	
b\left(y(t_1),\ldots,y(t_M)\right) &=0,\label{eqn:point} \tag{DOPb}\\
c\left(\dot{y}(t),y(t),z(t),t\right)
&=0\quad\text{ f.a.e.\ }t \in \Omega, \label{eqn:dae} \tag{DOPc}\\
z(t) &\geq 0 \quad\text{ f.a.e.\ }t \in \Omega,\label{eqn:positive} \tag{DOPd}
\end{align}
where the open bounded interval
$ 	\Omega := (t_0,t_E) \subsetneq \R$, $\cX$ is an appropriately-defined Hilbert space for solution candidates $x:=(y,z)$ such that $y$ is continuous, and \enquote{f.a.e.} %in~\refOCP
means \enquote{for almost every} in the Lebesgue sense. Detailed definitions and assumptions are given  in Section~\ref{sec:assumptions_SICON}. An optimal solution will be denoted with $x^\star$.
We note that the form~\refOCP is quite general and adopted here to minimize notation.

Ordinary differential equations (ODEs) and path constraints are included via the differential-algebraic equations (DAE) in~\eqref{eqn:dae} and the inequalities~\eqref{eqn:positive}. The point constraints~\eqref{eqn:point} enforce boundary constraints, such as initial or final values on the state $y$, or include values obtained by measurements at given time instances.

With techniques presented in \cite[Sect.~3.1]{Patterson},\cite[Sect.~2--3]{Bryson1975},\cite[Chap.~4]{BettsChap2},
problems in the popular Bolza or Mayer forms with general inequalities and free initial- or end-time can be converted into the form~\refOCP. In turn, many problems from control,  estimation and system identification can be stated in Bolza or Mayer form~\cite{BettsChap2}.

%A Mayer term $E(y(t_0),y(t_1),\dots,y(t_E))=:\xi(t_E)/(t_M-t_0)$ can be represented with an auxiliary species $\xi$ with a derivative of zero, that is added into the Lagrange term. If $E$ is Lipschitz continuous in all arguments, then so remain $f,b,c$. Box constraints on $x$ can be expressed in the form of \eqref{eqn:positive} by augmenting $z$ with slack variables and $c$ with algebraic path constraints that define the values of these slacks. Thereby, boundedness of $\|x^\star\|_{L^\infty(\Omega)}$ can be ensured and $c$ remains Lipschitz-continuous if it was so before. Likewise, nonlinear inequality constraints can be incorporated into \eqref{eqn:dae} by augmenting $z$ with slack variables. If these inequality path constraints are Lipschitz then so remains $c$.

%Constant parameters that are to be determined, or unknown parameters that are to be estimated, can also be included in $y$. The  free variable $z$  can include manipulated control inputs to be determined or unknown external signals, such as measurement or process noise. %%Comment from Martin: trivial. need not be presented.

Problem~\refOCP is infinite-dimensional, because the optimization is over function spaces subject to an uncountable set of constraints. It is  very hard or impossible to compute an analytic solution, in general. Hence, one often has to resort to numerical methods to solve \refOCP.
When doing so, it is important to eliminate whether  features of the numerical solution have arisen from physical principles or numerical failure.
The need for a numerical method, which has a rigorous proof that the numerical solution convergences to the exact solution, is therefore essential in practice.

One of the more established choices for the numerical solution of \refOCP is to discretize via
direct collocation and finite elements~\cite{Conway2012,BettsChap2,rao_asurvey,Patterson,KellyMatthew}. 
Recall that explicit Runge-Kutta methods are unsuitable for stiff problems and that many popular implicit methods for solving differential equations, e.g.\ variants of Gauss schemes, can be interpreted as collocation methods. Collocation methods include certain classes of implicit Runge-Kutta, pseudospectral, as well as Adams and backward differentiation formula methods~\cite{KellyMatthew,BettsChap2,rao_asurvey,arevelo:2002}.
However, as is known \cite[Sect.~2.5\,\&\,4.14]{BettsChap2},\cite{ChenBiegler16,kameswaran_biegler_2008}, collocation methods can fail to converge if care is not taken. In \cite{Neuenhofen2020AnIP} we present an example where three commonly used collocation-based direct transcription methods diverge, and below in Section~\ref{sec:numExp:CQP} we give a parametric problem for which Legendre-Gauss-Radau collocation~\cite{BettsChap2} of any degree rings.

Notice that \refOCP includes problems with mixed differential and inequality path constraints, for which indirect methods~\cite{Conway2012,BettsChap2} have only a limited range of applicability. Even when applicable, indirect methods require sophisticated user knowledge to set up suitable co-state and switching-structure estimates \cite{Boehme2017}. A detailed discussion of available methods in the literature is given in Section~\ref{sec:literatureReview}.

There is a scarcity of rigorous proofs that show that high-order collocation schemes for dynamic optimization methods converge to a feasible or optimal solution as the discretization is refined. The assumptions in the literature are often highly technical, difficult to enforce or not very general.

\subsection{Contributions}

The penalty-barrier method aims to overcome the limitations of the  numerical methods mentioned above by presenting a novel direct transcription method  for solving~\refOCP. %together with  a proof of convergence under  mild assumptions.
Our method combines the following ingredients: quadratic integral penalties for the equality (path) constraints; logarithmic integral barriers for the inequality path constraints; and direct numerical solution via finite elements.
It is this combination, together with a rigorous proof of convergence, that amounts to a novel direct transcription method. %This method provides convergence guarantees under milder assumptions than existing methods.
We also provide order-of-convergence results.

As detailed in Section~\ref{sec:assumptions_SICON}, we  only require existence of a solution to~\refOCP and mild assumptions on the boundedness and Lipschitz continuity of $f,c,b$.
In contrast to existing convergence results:
\begin{itemize}
	\item The solution $x^\star$ %=(y^\star,z^\star)$ 
	does not need to be  unique.
	\item  $f,c,b$ can be non-differentiable everywhere.
	\item  We do not require the satisfaction of a constraint qualification for the discretized finite-dimensional optimization problem, such as the Linear Independence Constraint Qualification (LICQ), Mangasarian-Fromovitz Constraint Qualification (MFCQ) or Second-Order Sufficient Conditions (SOSC).
	\item Uniqueness or global smoothness of states or co-states/ adjoints do not need to hold.
	\item Local uniqueness assumptions, as in~\cite{arXiv:2017}, are removed.
\end{itemize}

The idea behind our new, Penalty-Barrier-Finite Element method (PBF), is to minimize the following unconstrained penalty-barrier function
\begin{align}
\Phi(x) := F(x) + \frac{1}{2\cdot \omega} \cdot r(x) + \tau \cdot \Gamma(x), \label{eqn:def:Phi}
\end{align} 
where
\begin{align}
F(x) := &\int_\Omega \, f\left(\dot{y}(t),y(t),z(t),t\right)\mathrm{d}t \label{eqn:def:F}\end{align}
is the objective,
\begin{align}
\begin{split}
r(x) := &\int_\Omega\|c\left(\dot{y}(t),y(t),z(t),t\right)\|_2^2\,\mathrm{d}t + \|b\left(y(t_1),y(t_2),\ldots,y(t_M)\right)\|_2^2
\end{split}\label{eqn:def:r}
\end{align}
is the integral quadratic penalty for the equality path- and point constraints, and
\begin{align}
\Gamma(x):=&-\sum_{j=1}^{n_z} \int_\Omega \log\big( z_{[j]}(t)\big)\,\mathrm{d}t\label{eqn:def:Gamma}
\end{align}
is an integral logarithmic barrier for the inequality path constraints. We provide an analysis that shows that one can construct trajectories $x_h$ that converge in the following \textit{tolerance-accurate} sense: the \textit{optimality gap}
\begin{align}
g_\text{opt} :=& \max\lbrace 0,F(x_h)-F(x^\star)\rbrace
\end{align}
and \textit{feasibility residual}
\begin{align}
r_\text{feas} :=& r(x_h)
\end{align}
converge to zero as the discretization mesh becomes finer and the parameters $\tau,\omega>0$ converge to zero. Order-of-convergence results will specify the rate at which $g_\text{opt}$ and $r_\text{feas}$ approach zero.

The above functions \eqref{eqn:def:F}--\eqref{eqn:def:Gamma} look similar to those encountered in well-known finite-dimensional penalty-barrier methods. However, in order to deal with the infinite-dimensional nature of the problem, note the use of \emph{integrals} in the penalty \emph{and} barrier terms. If the problem had been finite-dimensional in $x$ and if $r$ had been the squared $2$-norm of finitely many equality constraints,  then it would be given that the minimizer of $\Phi$  converges to the solution $x^\star$ under mild assumptions as  $\tau,\omega$ converge to zero~\cite{SUMT}. The infinite-dimensional case considered here, however, is more involved and requires a careful analysis relating $\tau,\omega$ to  parameters of the  discretization. This is because once we discretize on a mesh  and seek to compute an approximate solution $x_h$ on the mesh, the degrees of freedom for $x_h$ depend on the size of the finite element space. If we were to draw an analogy with the finite dimensional case, then the equivalent number of equality constraints depends on the number of quadrature points for numerically evaluating the integral in $r$. If $\omega$ is too large then~$x_h$ will not converge to satisfying the equivalent set of equality constraints. If $\omega$ is too small with respect to the mesh size, then there are not enough degrees of freedom, resulting in a potentially feasible but suboptimal  solution~\cite[p.~1078]{Hager90}.
The effects of $\tau$ are more intricate, since they relate to a local Lipschitz property of $\Phi$ that is relevant for the stability of the finite element discretization.
A balance must thus be taken between the size of the finite element space, the quadrature rule and the parameters $\omega,\tau$. This requires a non-trivial analysis, which is the core contribution of this chapter.

\subsection{Motivation from Collocation Methods}
We motivate our method from the perspective of collocation methods.

A desirable method for solving optimal control problems is Legendre-Gauss-Radau collocation because it is easy to implement (and high-order consistent): the method constructs piecewise polynomials (of high degree) using a nodal basis and solves the path constraints at a finite number of points. The nodal basis values are determined by solving a large sparse nonlinear program.

However, for solutions with singular arcs, which occur in a large number of applications, the numerical solutions can ``ring''~\cite[Sect.~4.14.1]{BettsChap2}. In particular, the solution polynomial and the residuals of the path constraints will oscillate between the collocation points --- that is, the path residuals will not converge to zero everywhere. A remedy is regularization: a small convex quadratic term is added to the objective to penalize numerical noise. Unfortunately, in most cases this remedy does not work because either the penalty is too small to remove all noise or so large that it alters the problem's solution.

The idea with the penalty method is to make ringing impossible by adding collocation points inbetween the original collocation points, where otherwise the states, controls and residuals could ring. The theoretical vehicle for this approach are integrals and penalties. Integrals, once discretized by means of numerical quadrature, can be expressed with a set of weights and abscissae, alias collocation points. Penalties, in replacement for exact constraints, will prevent any issues related to the ratio between the number of degrees of freedom and the number of constraints, such as over-determination. The resulting scheme remains easy to implement while effectively forcing global convergence of the path constraints --- as we rigorously prove in the remainder of this chapter. In particular, we prove that the feasibility residual converges to zero.

We stress that the integral penalty and log-barrier provide a useful natural scaling for the NLP. This is certainly desirable from a computational perspective, because experience shows that numerical treatment of an NLP depends significantly on scaling~\cite[Chap.~1.16.5,~4.8]{BettsChap2}. The large-scale methods in \cite{ForsgrenGill1998,ALMIPM} use a merit function that treats equality constraints with a quadratic penalty and inequality constraints with a log-barrier term. Typically, as the NLP becomes larger, caused by a finer discretization, the NLP becomes more challenging to solve, in the sense that the number of iterations to converge increases. In contrast, for the penalty-barrier method the NLP merit function matches  the infinite-dimensional merit function in the limit, which mitigates numerical issues that might otherwise arise.

\subsection{Notation}
\label{sec:assumptions_SICON}

Let $-\infty<t_0 < t_E < \infty$ and the $M \in \N$ points
$
t_k \in \overline{\Omega}$,
$\forall k\in \lbrace 1,2,\ldots,M \rbrace$. $\overline{\Omega}$ denotes the closure of $\Omega$.
The functions
$f : \R^{n_y} \times \R^{n_y} \times \R^{n_z} \times \Omega \rightarrow \R$,
$c : \R^{n_y} \times \R^{n_y} \times \R^{n_z} \times \Omega \rightarrow \R^\nc$,
$b: \R^{n_y} \times \R^{n_y} \times \ldots \times \R^{n_y} \rightarrow \R^\nb$.
The function $y:\overline{\Omega} \rightarrow \R^{n_y}, t \mapsto y(t)$  and~$z:\overline{\Omega} \rightarrow \R^{n_z}, t \mapsto z(t)$.
Given an interval $\Omega\subset \R$,  let $|\Omega|:=\int_\Omega 1\,\mathrm{d}t$.
We use Big-$\cO$ notation to analyze a function's behaviour close to zero, i.e.\ function  $\phi(\xi)=\cO(\gamma(\xi))$ if and only if $\exists C>0$ and $\xi_0>0$ such that $\phi(\xi)\leq C \gamma(\xi)$ when $0<\xi<\xi_0$. The vector $\be:=[1\,\cdots\,1]\t$ with appropriate size.

For notational convenience, we define the function
\[
x:=(y,z):\overline{\Omega} \rightarrow \R^{\nx},
\]
where $\nx:=n_y+n_z$.
The solution space of $x$ is the Hilbert space
\begin{align*}
\cX := \left(H^1\left(\Omega\right)\right)^{n_y} \times \left(L^2\left(\Omega\right)\right)^{n_z}
\end{align*}
with scalar product
\begin{align}
\langle(y,z),(v,w)\rangle_\cX := \sum_{j=1}^{n_y} \langle y_{[j]},v_{[j]} \rangle_{H^1(\Omega)} + \sum_{j=1}^{n_z} \langle z_{[j]},w_{[j]} \rangle_{L^2(\Omega)} \label{eqn:ScalarProd}
\end{align}
and induced norm $\|x\|_\cX := \sqrt{\langle x,x \rangle_\cX}$, where $\phi_{[j]}$ denotes the $j^\text{th}$ component of a function $\phi$. The Sobolev space $H^1(\Omega):=W^{1,2}(\Omega)$ and Lebesgue space $L^2(\Omega)$ with their respective scalar products are defined as in~\cite[Thm~3.6]{Adams}.
The weak derivative of $y$ is denoted by
$\dot{y} := {\mathrm{d}y}/{\mathrm{d}t}$.

Recall the embedding $H^1(\Omega) \hookrightarrow \cC^0(\overline{\Omega})$, where $\cC^0(\overline{\Omega})$ denotes the space of continuous functions over $\overline{\Omega}$~\cite[Thm~5.4, part II, eqn~10]{Adams}. Hence, by requiring that $y\in \left(H^1\left(\Omega\right)\right)^{n_y}$ it follows that $y$ is continuous.
In contrast, though $\dot{y}$ and $z$ are in $L^2(\Omega)$, they may be discontinuous.

\subsection{Assumptions}
In order to prove convergence, we make the following assumptions on \refOCP:
\begin{enumerate}[\text{(A.}1\text{)}]
	\item %Existence:
	\refOCP has at least one global minimizer $x^\star$.
	\item %Boundedness:
	$\|c(\dot{y}(t),y(t),z(t),t)\|_1,\,$ and
	$\|b(y(t_1),\dots,y(t_M))\|_1$ are bounded for all arguments $x \in \cX$ within $z\geq 0$, $t \in \Omega$. $F(x)$ is bounded below for all arguments $x \in \cX$ within $z\geq 0$.
	\item %L-continuity:
	$f,\ c,\ b$ are globally Lipschitz continuous in all arguments except $t$.
	\item %Related solutions:
	The two solutions $x^\star_\omega,x^\star_{\omega,\tau}$ related to~$x^\star$, defined in Section~\ref{sec:reform}, are  bounded in terms of $\|z\|_{L^\infty(\Omega)}$ and $\|x\|_\cX$. Also, $\|x^\star\|_\cX$ is bounded.
	\item The related solution $x^\star_{\omega,\tau}$ can be approximated to an order of at least 1/2 using piecewise polynomials; formalized in \eqref{eqn:InfBound} below.
\end{enumerate}
Similar assumptions are  implicit or explicit in most of the literature. A discussion of these assumptions is appropriate:
\begin{enumerate}[\text{(A.}1\text{)}]
	\item is just to avoid infeasible problems.
	
	\item The assumption on $b,c$ can be enforced by construction via lower and upper limits w.l.o.g.\ because they are (approximately) zero at the (numerical) solution. Boundedness below for $F$ is arguably mild when/since $\|x^\star\|_\cX,\|x_h\|_\cX$ are bounded: For minimum-time problems and positive semi-definite objectives this holds naturally. In many contexts, a lower bound can be given. The assumptions on $b,c$ have been made just to simplify the proof of a Lipschitz property and because they mean no practical restriction anyways. The boundedness assumption on $F$ is made to avoid unbounded problems.
	
	\item can be enforced. Functions that are not Lipschitz continuous, e.g.\ the square-root or Heaviside function, can be made so by replacing them with smoothed  functions, e.g.\ via a suitable mollifier. Smoothing is a common practice to ensure the derivatives used in a nonlinear optimization algorithm (e.g.\ IPOPT \cite{IPOPT}) are globally well-defined. The assumption has been made to prove a Lipschitz property of a penalty-barrier functional. Actually this property is only needed in a local neighborhood of the numerical optimal control solution, but for ease of notation we opted for global assumptions.
	
	\item can be ensured as shown in Remark~\ref{sec:ForceBoundedness} in Section~\ref{sec:PenaltyBarrierProblem}. This assumption effectively rules out the possibility of solutions with finite escape time. The assumption has been incorporated because restriction of a solution into a box means little practical restriction but significantly shortens convergence proofs due to boundedness.
	
	\item is rather mild, as discussed in Section~\ref{sec:InterpolationError} and illustrated in Appendix~\ref{app:3}. All finite-element methods based on piecewise polynomials make similar assumptions, implicitly or explicitly. The assumption is only used for the rate-of-convergence analysis. The assumption is unavoidable, since otherwise a solution $x^\star$ could exist that cannot be approximated to a certain order.
\end{enumerate}
The assumptions are not necessary but sufficient. Suppose that we have found a numerical solution. It is not of relevance to the numerical method whether the assumptions	hold outside of an open neighborhood of this solution. However, the proofs below would become considerably more lengthy with  local assumptions.  We outline in Section~\ref{sec:local} how our global analysis can be used to show local convergence under local assumptions. Hence, for the same reasons as in some of the literature, we opted for global assumptions. Our analysis is not restrictive in the sense that it imposes global requirements. No further assumptions are made for the proof.

\subsection{Outline}
Section~\ref{sec:reform} introduces a reformulation of \refOCP as an unconstrained problem. Section~\ref{sec:FEM} presents the Finite Element Method in order to formulate a finite-dimensional unconstrained optimization problem. The main result of this chapter is Theorem~\ref{thm:order}, which shows that solutions of the finite-dimensional optimization problem converge to solutions of~\refOCP with a guarantee on the order of convergence. %Section~\ref{sec:solvingNLP} discusses how one could compute a solution using NLP solvers.
%Section~\ref{sec:numerical} presents numerical results which validate that our method converges for %\todo{pathologic}
%difficult problems, whereas certain collocation methods can fail in some cases. Conclusions are drawn in Section~\ref{sec:conclusions}.

\section{Reformulation as an Unconstrained Problem}\label{sec:reform}
{\emergencystretch1em
	The reformulation of~\refOCP into an unconstrained problem is achieved in two steps. First, we introduce penalties for the %differential-algebraic and
	equality constraints. We then add logarithmic barriers for the inequality constraints.
	The resulting penalty-barrier functional will be treated numerically in Section~\ref{sec:FEM}.\par}

Before proceeding, we note that boundedness and Lipschitz-continuity of $F$ and $r$ in \eqref{eqn:def:F}--\eqref{eqn:def:r} follow from (A.2)--(A.3).
\begin{lem}[Boundedness and Lipschitz-continuity of $F$ and $r$]\label{lem:BoundLipschitz_Fr}
	$F$ is bounded below. $r$ is bounded. $F,r$ are Lipschitz continuous in $x$ with respect to
	$\|\cdot\|_\cX$. Furthermore, $F,r$ are Lipschitz continuous in $z$ with respect to the norm $\|\cdot\|_{L^1(\Omega)}$.
\end{lem}
The proof is given in Appendix~\ref{sec:Appendix_ProofLemma1}.

We  bound the Lipschitz constants (i.e., with respect to both $\|x\|_\cX$ and $\|z\|_{L^1(\Omega)}$) with $L_F\geq 2$ for $F$ and with $L_r\geq 2$ for $r$.

\subsection{Penalty Form}
We introduce the \textit{penalty problem}
\begin{equation}
\tag{PP}
\label{eqn:PP}
\text{Find } x_\omega^\star \in \operatornamewithlimits{arg\, min}_{x \in \cX}
F_\omega(x) \text{ s.t.\ } z(t)\geq 0 \text{ f.a.e.\ } t \in \Omega ,
\end{equation}
where $F_\omega(x) := F(x) + \frac{1}{2 \cdot \omega} \cdot r(x)$ and a small penalty parameter $\omega \in (0,1)$. Note that $F_\omega$ is Lipschitz continuous with constant
\begin{align}
L_\omega := \max\left\lbrace L_F + \frac{L_r}{2 \omega}\,,\,L_f+\frac{L_c}{2 \omega}\, \|c\|_1\right\rbrace\,, \label{eqn:LipLw}
\end{align}
with $L_f,L_c$ the Lipschitz-constants of $f,c$ and $\|c\|_1$ is the upper bound on the 1-norm of $c$, as asserted by (A.2). We show that $\varepsilon$-optimal solutions of \eqref{eqn:PP} solve~\refOCP in a tolerance-accurate way.

\begin{prop}[Penalty Solution]\label{thm:PenaltySolution}
	Let $\varepsilon\geq 0$. Consider an $\varepsilon$-optimal solution $x^\varepsilon_\omega$ to \eqref{eqn:PP}, i.e.
	$$ F_\omega(x^\varepsilon_\omega) \leq F_\omega(x^\star_\omega) + \varepsilon\text{ and } z^\varepsilon_\omega(t)\geq 0\ \text{f.a.e.}\ t\in\Omega\,. $$
	If we define $C_r := F(x^\star) - \operatornamewithlimits{ess\,min}_{x\in\cX, z\geq 0}F(x)$, then
	$F(x^\varepsilon_\omega) \leq F(x^\star) + \varepsilon$, $r(x^\varepsilon_\omega) \leq \omega \cdot (C_r+\varepsilon)$.
\end{prop}
\begin{proof}
	$x^\star,\,x^\star_\omega,\,x^\varepsilon_\omega$ are all feasible for \eqref{eqn:PP}, but $x^\star_\omega$ is optimal and $x^\varepsilon_\omega$ is $\varepsilon$-optimal. Thus,
	\begin{align}
	F(x^\varepsilon_\omega) \leq F_\omega(x^\varepsilon_\omega) \leq 
	F_\omega(x^\star) + \varepsilon \leq F(x^\star) + \varepsilon
	%\underbrace{F_\omega(x^\varepsilon_\omega)}_{\geq F(x^\varepsilon_\omega)} \leq \underbrace{F_\omega(x^\star)}_{= F(x^\star)} + \varepsilon\,. 
	\label{eqn:aux:prop1}
	\end{align}
	From this follows $F(x^\varepsilon_\omega)\leq F(x^\star)+\varepsilon$ because $r(x^\varepsilon_\omega)\geq 0$ and $r(x^\star)=0$ by (A.1). To show the second proposition, subtract $F(x^\varepsilon_\omega)$ from \eqref{eqn:aux:prop1}. Then it follows that  $1/(2\cdot\omega) \cdot r(x^\varepsilon_\omega) \leq F(x^\star) - F(x^\varepsilon_\omega) + \varepsilon \leq C_r + \varepsilon$\,. Multiplication of this inequality with $2\cdot\omega$ shows the result. Boundedness of $C_r$ follows from Lemma~\ref{lem:BoundLipschitz_Fr}.
\end{proof}

This result implies that for an $\varepsilon$-optimal solution  to~\eqref{eqn:PP} the optimality gap to~\refOCP is less than $\varepsilon$ and that the feasibility residual can be made arbitrarily small by choosing the  parameter~$\omega$ to be sufficiently small.

\begin{comment}
In Proposition~\ref{thm:PenaltySolution} we used (A.2) which implies $|F|$ is bounded. In fact, $F$ only needs to be bounded below. To show this, note that
$$
r(x^\varepsilon_\omega) = 2 \cdot \omega \cdot \big( \underbrace{F_\omega(x^\varepsilon_\omega)}_{\leq F(x^\star)+\varepsilon} - \underbrace{F(x^\varepsilon_\omega)}_{\geq F_{\emph{lb}}}\big).
$$
Hence,
$  r(x^\varepsilon_\omega) \leq 2 \cdot \omega \cdot \big( \underbrace{F(x^\star)-F_{\emph{lb}}}_{=\cO(1)} + \varepsilon\big).$
\end{comment}
% comment Martin: I know it hurts, but it gains us 1/8th of a page. Further, since we need the assumption |F|<bound regardless in other places, we do not gain an advantage in the grand picture.

\subsection{Penalty-Barrier Form}\label{sec:PenaltyBarrierProblem}
We reformulate \eqref{eqn:PP} once more in order to remove the inequality constraints. We do so using logarithmic barriers. Consider the \textit{penalty-barrier problem}
\begin{equation}
\tag{PBP}
\label{eqn:PBP}
\text{Find }x^\star_{\omega,\tau} \in \operatornamewithlimits{arg\, min}_{x \in \cX}
F_{\omega,\tau}(x) := F_\omega(x) + \tau \cdot \Gamma(x),
\end{equation}
where the barrier parameter $\tau \in (0,\omega]$ and $\Gamma$ is defined in~\eqref{eqn:def:Gamma}.

We have introduced $\Gamma$ in order to keep $z^\star_{\omega,\tau}$ feasible with respect to \eqref{eqn:positive}. Recall that $L^2(\Omega)$ contains functions that have poles. So the following result is to ensure that $\Gamma$ actually fulfills its purpose.
\begin{lem}[Strict Interiorness]\label{lem:StrictInteriorness}
	$$ 	z^\star_{\omega,\tau}(t) \geq \frac{\tau}{L_\omega}\cdot \qquad \text{f.a.e.\ } t \in	\Omega.$$
\end{lem}
\begin{proof}
	At the minimizer $x^\star_{\omega,\tau}$, the functional $F_{\omega,\tau}$ can be expressed in a single component $z_{[j]}$ as
	$$  \int_\Omega \Big(\,q\big( z_{[j]}(t),t \big) - \tau \cdot \log\big(\,z_{[j]}(t)\,\big)\,\Big)\,\mathrm{d}t\,,    $$
	where $q$ is Lipschitz-continuous with a constant $L_q \leq L_\omega$ (cf. right argument in the max-expression \eqref{eqn:LipLw} and compare to \eqref{eqn:LipQuadPenaltyFunc} in the proof of Lemma~\ref{lem:BoundLipschitz_Fr}).
	From the Euler-Lagrange equation it follows for $z^\star_{\omega,\tau}$ that
	$$  \frac{\partial q}{\partial z_{[j]}}q\big( z_{[j]}(t) , t \big) - \frac{\tau}{z_{[j]}(t)} = 0\qquad \text{f.a.e.\ }t \in \Omega\,.    $$
	The value of $z_{[j]}(t)$ gets closer to zero when the first term grows. However, that term is bounded by the Lipschitz constant. Hence, in the worst case
	\[  z_{[j]}(t) \geq \frac{\tau}{L_q} \geq \frac{\tau}{L_\omega}\qquad \text{f.a.e.\ }t \in \Omega\,.\]%\qed
	%TODO: the qed-box is missing after removal of this comment. UF added it manually ...
\end{proof}

We will need the following  operators:
\begin{defn}[Interior Push]
	Given $x \in \cX$,  define $\bar{x}$ and $\check{x}$ as a modified  $x$ whose components $z$ have been pushed by an amount into the interior if they are close to zero:
	\begin{align*}
	\bar{z}_{[j]}(t):=\max\left\lbrace z_{[j]}(t),{\tau}/{L_\omega}\right\rbrace\,,\qquad \check{z}_{[j]}(t):=\max\left\lbrace z_{[j]}(t),{\tau}/(2\cdot L_\omega)\right\rbrace
	\end{align*}
	for all $j \in \lbrace 1,2,\ldots,n_z\rbrace$ and $t \in \overline{\Omega}$.
\end{defn}
\noindent
Note that $\bar{x} \in \cX$ and that $x^\star_{\omega,\tau}=\bar{x}^\star_{\omega,\tau}$ from Lemma~\ref{lem:StrictInteriorness}.

Using the interior push, we show below that $x^\star_{\omega,\tau}$ is $\varepsilon$-optimal for~\eqref{eqn:PP}. Our result uses a small arbitrary fixed number $0<\zeta\ll 1$.
\begin{prop}[Penalty-Barrier Solution]\label{prop:PenaltyBarrierSolution}
	If (A.4) holds, then
	$$|F_\omega(x^\star_{\omega,\tau}) - F_\omega(x^\star_\omega)| = \cO\left(\tau^{1-\zeta}\right).$$
\end{prop}
\begin{proof}
	From the definition of the bar operator, we can use the bound
	\begin{align*}
	&\|x^\star_\omega - \bar{x}^\star_{\omega}\|_\cX = \|z^\star_{\omega}-\bar{z}^\star_{\omega}\|_{L^2(\Omega)}
	=\sqrt{\int_\Omega \|z^\star_{\omega}-\overline{z}^\star_{\omega}\|_2^2\,\mathrm{d}t}\\
	\leq& \operatornamewithlimits{max}_j\sqrt{|\Omega| \cdot n_z \cdot \|z^\star_{\omega\,[j]}-\overline{z}^\star_{\omega\,[j]}\|^2_{L^\infty(\Omega)}}\leq n_z \cdot \sqrt{|\Omega|} \cdot \frac{\tau}{L_\omega},
	\end{align*}
	together with the facts that $x^\star_{\omega,\tau}= \bar{x}^\star_{\omega,\tau}$ and $F_\omega$ is Lipschitz continuous, to get
	\begin{align*}
	0 &\leq
	F_\omega(x^\star_{\omega,\tau}) - F_\omega(x^\star_{\omega}) \leq  F_\omega(\bar{x}^\star_{\omega,\tau}) - F_\omega(\bar{x}_\omega^\star) + L_\omega \cdot \|x^\star_{\omega}-\bar{x}^\star_{\omega}\|_\cX\\
	% \end{multline*}
	% so that we can add two zeros to show that% we can bound the right-hand side further to
	% \begin{multline*}
	% F_\omega(x^\star_{\omega,\tau}) - F_\omega(x^\star_{\omega})
	& \leq \underbrace{F_\omega(\bar{x}^\star_{\omega,\tau}) - F_{\omega,\tau}(\bar{x}^\star_{\omega,\tau})}_{=-\tau\cdot\Gamma(\bar{x}^\star_{\omega,\tau})} + F_{\omega,\tau}(\bar{x}^\star_{\omega,\tau})\\
	&\phantom{ \leq }-\Big(\underbrace{F_\omega(\bar{x}^\star_{\omega}) - F_{\omega,\tau}(\bar{x}^\star_{\omega})}_{=-\tau \cdot \Gamma(\bar{x}^\star_{\omega})} + F_{\omega,\tau}(\bar{x}^\star_{\omega})\Big)+ L_\omega \cdot n_z \cdot \sqrt{|\Omega|} \cdot \frac{\tau}{L_\omega}\\
	& \leq F_{\omega,\tau}(\bar{x}^\star_{\omega,\tau})-F_{\omega,\tau}(\bar{x}^\star_{\omega}) +|\tau \cdot \Gamma(\bar{x}^\star_{\omega,\tau})| + |\tau \cdot \Gamma(\bar{x}^\star_\omega)| + n_z \cdot \sqrt{|\Omega|} \cdot \tau .
	\end{align*}
	We use $|\tau \cdot \Gamma(\bar{x}^\star_{\omega,\tau})|+|\tau \cdot \Gamma(\bar{x}^\star_\omega)|=\cO(\tau^{1-\zeta})$, as per Lemma~\ref{lem:Gamma} in Appendix~\ref{sec:Appendix_BarrierFunctioProperties}, to obtain the result from
	\begin{align*}
	\hspace{1mm}F_\omega(x^\star_{\omega,\tau}) - F_\omega(x^\star_{\omega})
	\leq \underbrace{F_{\omega,\tau}(\bar{x}^\star_{\omega,\tau})-F_{\omega,\tau}(\bar{x}^\star_{\omega})}_{\leq 0} +	\cO\left(\tau^{1-\zeta}\right)+ n_z \cdot \sqrt{|\Omega|} \cdot \tau
	\end{align*}
	The under-braced term is bounded above by zero because $\bar{x}^\star_{\omega,\tau}= x^\star_{\omega,\tau}$ is a minimizer of $F_{\omega,\tau}$.
\end{proof}

\begin{rem}	\label{sec:ForceBoundedness}
	Lemma~\ref{lem:Gamma} in the proof of Prop.~\ref{prop:PenaltyBarrierSolution} needs (A.4), i.e.
	\[\|z^\star_{\omega}\|_{L^\infty(\Omega)},\|z^\star_{\omega,\tau}\|_{L^\infty(\Omega)} = \cO(1).\]
	Note that the assumption can be enforced. For example, the path constraints
	$$ 	z_{[1]}(t)\geq 0,\quad z_{[2]}(t)\geq 0,\quad z_{[1]}(t)+z_{[2]}(t)=const$$
	lead to
	$  \|z_{[j]}\|_{L^\infty(\Omega)} \leq const$ \mbox{for $j=1,2\,.$}
	Constraints like these arise when variables have simple upper and lower  bounds before being transformed %from Bolza form 
	into \refOCP.
	
	Similarly, boundedness of $\|x\|_\cX$ can be enforced. To this end, introduce box constraints for each component of $\dot{y},y,z$, before transcribing into the form \refOCP.
\end{rem}

\section{Finite Element Method}\label{sec:FEM}
Our method constructs an approximate finite element solution $x^\epsilon_h$ by solving the unconstrained problem \eqref{eqn:PBP} computationally in a finite element space \mbox{$\cX_{h,p} \subset \cX$,} using an NLP solver.

We introduce a suitable finite element space and show a stability result in this space. Eventually, we prove convergence of the finite element solution to  solutions of \eqref{eqn:PBP} and \refOCP. %The subsequent section will discuss the NLP solver.

\subsection{Definition of the Finite Element Space}
Let the mesh parameter $h \in (0,|\Omega|]$.
The set	$\cT_h$ is called a \emph{mesh} and consists of open intervals $T \subset \Omega$ that satisfy the usual conditions~\cite[Chap.~2]{Ciarlet1978}:
\begin{enumerate}
	\item Disjunction: $T_1\cap T_2 = \emptyset$, for all  distinct  $T_1,T_2 \in \cT_h$.
	\item Coverage: $\bigcup_{T \in \cT_h} \overline{T} = \overline{\Omega}$.
	\item Resolution: $\max_{T \in \cT_h} |T| = h$.
	\item Quasi-uniformity: $\min_{T_1,T_2 \in \cT_h}\frac{|T_1|}{|T_2|} \geq \vartheta > 0$. The constant $\vartheta$ must not depend on $h$ and $1/\vartheta = \cO(1)$.
\end{enumerate}
% 	\begin{align*}
% 	    \begin{array}{ll}
%     		(i) 		&\text{Disjunction: } 	T_1\cap T_2 = \emptyset\ \forall\text{ distinct }T_1,T_2 \in \cT_h .\\
%     		(ii) 		&\text{Coverage: } 		\bigcup_{T \in \cT_h} \overline{T} = \overline{\Omega} .\\
%     		(iii) 	&\text{Resolution: } 	\max_{T \in \cT_h} |T| = h .\\
%     		(iv) 		&\text{Quasi-uniformity: } 	\min_{T_1,T_2 \in \cT_h}\frac{|T_1|}{|T_2|} \geq \vartheta > 0 .
% 	    \end{array}
% 	\end{align*}
%The mesh is uniform if $h=|T|$ for all $T\in\cT_h$.

We write $\cP_p({T})$ for the space of functions that are polynomials of degree $\leq p \in \N_0$ on interval ${T}$. Our finite element space is then given as
\begin{align*}
\cX_{h,p} := \left\{x:\overline{\Omega}\rightarrow\R^{\nx}  \mid y \in \cC^0(\overline{\Omega}), x \in \cP_p(T)^{n_x} \ \forall T \in \cT_h \right\}.
\end{align*}
$\cX_{h,p}\subset \cX$ is a Hilbert space with scalar product $\langle\cdot,\cdot\rangle_\cX$. %in~\eqref{eqn:ScalarProd}.

Note that if $(y,z)\in\cX_{h,p}$, then $y$ is continuous but~$\dot{y}$ and~$z$ can be discontinuous. Figure~\ref{fig:finiteelementsyz} illustrates two functions $(y_h,z_h) \in \cX_{h,p}$ with $\times$ and $+$ for their nodal basis, to identify them with a finite-dimensional vector. %Section~\ref{sec:solvingNLP} discusses the nodal solution and the placement of the nodes for practical computations.

\begin{figure}[tb]
	\centering
	\includegraphics[width=0.45\columnwidth]{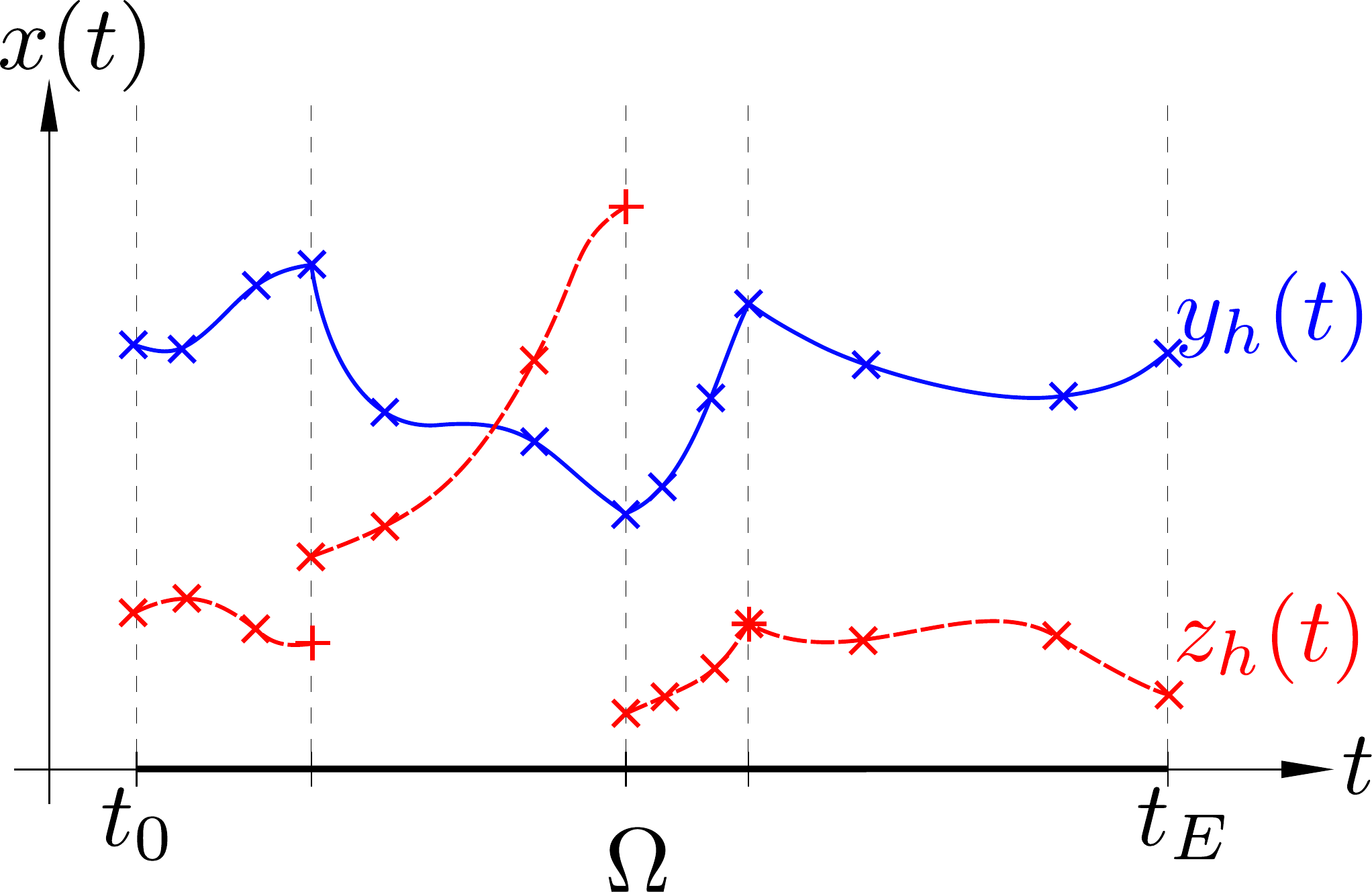}
	\caption{Continuous and discontinuous piecewise polynomial finite element functions $y_h,z_h$ on a mesh $\cT_h$ of four intervals.}
	\label{fig:finiteelementsyz}
\end{figure}

\subsection{Discrete Penalty-Barrier Problem}
We state the \textit{discrete penalty-barrier problem}
as
\begin{equation*}
\text{Find }x^\star_h \in \operatornamewithlimits{arg\,min}_{x \in \cX^{\omega,\tau}_{h,p}}\, F_{\omega,\tau}(x)
\tag{PBP\textsubscript{h}}%UF????? _h gibt Fehler
\label{eqn:PBP_h}
\end{equation*}
with the space $\cX^{\omega,\tau}_{h,p} := \left\lbrace\,x \in \cX_{h,p} \ \Big\vert \ z(t)\geq \frac{\tau}{2 \cdot L_\omega}\cdot\be \text{ f.a.e.\ }t \in\Omega\right\rbrace$.

Note that Lemma~\ref{lem:StrictInteriorness} is valid only for solutions to~\eqref{eqn:PBP}, whereas below we will consider sub-optimal solutions to~\eqref{eqn:PBP_h}. Hence, we cannot guarantee that these sub-optimal solutions will satisfy $z(t)\geq {\tau}/(1 \cdot L_\omega)\cdot\be$. The looser constraint $z(t)\geq {\tau}/(2 \cdot L_\omega)\cdot\be$ in the definition above will be used in the proof of Theorem~\ref{thm:ConvOCP}.

%In Section~\ref{sec:solvingNLP}, where we discuss a practical solution method for \eqref{eqn:PBP_h},
In a practical implementation,
we neglect these additional constraints. This is reasonable when solving the NLP with interior-point methods, since they keep the numerical solution strictly interior with a distance to zero in the order of $\tau \gg \frac{\tau}{2 \cdot L_\omega}$.

\subsection{Stability}
The following result shows that two particular Lebesgue norms are equivalent in the above finite element space.
\begin{lem}[Norm equivalence]\label{lem:NormEquivalence}
	If $x \in \cX_{h,p}$, then
	\begin{align*}
	\|x_{[j]}\|_{L^\infty(\Omega)} \leq \frac{p+1}{\sqrt{\vartheta \cdot h}} \cdot \|x\|_\cX\quad\forall j\in\lbrace 1,2,\ldots,\nx\rbrace.
	\end{align*}
\end{lem}
\begin{proof}
	We can bound
	$\|x_{[j]}\|_{L^\infty(\Omega)}
	\leq\operatornamewithlimits{max}_{T \in \cT_h} \|x_{[j]}\|_{L^\infty(T)}$. We now use \eqref{eqn:PropAppendix2} in Appendix~\ref{sec:Appendix_LebesgueIdentity}. Since $x_{[j]} \in \cP_p(T)$, it follows that
	\begin{align*}
	\operatornamewithlimits{max}_{T \in \cT_h}\|x_{[j]}\|_{L^\infty(T)}
	&\leq \operatornamewithlimits{max}_{T \in \cT_h} \frac{p+1}{\sqrt{|T|}} \cdot \|x_{[j]}\|_{L^2(T)}\leq\frac{p+1}{\sqrt{\vartheta \cdot h}} \cdot \|x_{[j]}\|_{L^2(\Omega)} \leq \frac{p+1}{\sqrt{\vartheta \cdot h}} \cdot \|x\|_\cX.
	\end{align*}
\end{proof}

Below, with the help of Lemma~\ref{lem:NormEquivalence}, we obtain a bound on the growth of $F_{\omega,\tau}$ in a neighborhood of a solution $x^\star_{\omega,\tau}$ to~\eqref{eqn:PBP} for elements in~$\cX_{h,p}$.
\begin{prop}[Lipschitz continuity]\label{prop:Lcont}
	Let
	\begin{align*}
	\delta_{\omega,\tau,h} &:= \frac{\tau}{2 \cdot L_\omega} \cdot \frac{\sqrt{\vartheta \cdot h}}{p+1}\,,\qquad
	L_{\omega,\tau,h} := L_\omega + n_z \cdot |\Omega| \cdot 2 \cdot L_\omega \cdot \frac{p+1}{\sqrt{\vartheta \cdot h}}.
	\end{align*}
	Consider the spherical neighbourhood
	\begin{align*}
	\cB := \big\lbrace\,x \in \cX\ \big\vert\ \|x^\star_{\omega,\tau}-x\|_\cX \leq \delta_{\omega,\tau,h} \big\rbrace.
	\end{align*}
	The following holds $\forall x^\text{A},x^\text{B} \in \cB\cap\cX_{h,p}$:
	\begin{align*}
	|F_{\omega,\tau}(x^\text{A}) - F_{\omega,\tau}(x^\text{B})| \leq L_{\omega,\tau,h} \cdot \|x^\text{A} - x^\text{B}\|_\cX.
	\end{align*}
\end{prop}
\begin{proof}
	From Lemma~\ref{lem:StrictInteriorness} and Lemma~\ref{lem:NormEquivalence} follows:
	\begin{align*}
	\operatornamewithlimits{ess\,inf}_{t \in \Omega} z_{[j]}(t)\geq&\underbrace{\operatornamewithlimits{ess\,inf}_{t \in \Omega} z^\star_{\omega,\tau,[j]}(t)}_{\geq \frac{\tau}{L_\omega}}
	- \underbrace{\|z^\star_{\omega,\tau,[j]}-z_{[j]}\|_{L^\infty(\Omega)}}_{\leq \frac{p+1}{\sqrt{\vartheta \cdot h}} \cdot \delta_{\omega,\tau,h}\leq \frac{\tau}{2 \cdot L_\omega}}\hspace{4mm} \forall x \in \cB \cap \cX_{h,p}
	\end{align*}
	Hence,
	\begin{align}
	\operatornamewithlimits{min}_{1 \leq j \leq n_z} \operatornamewithlimits{ess\,inf}_{t\in\Omega} z_{[j]}(t) \geq \frac{\tau}{2 \cdot L_\omega}\quad \forall x \in \cB \cap \cX_{h,p}.\label{eqn:aux1:Lcont}
	\end{align}
	From Lipschitz-continuity of $F_\omega$ we find
	\begin{align*}
	&|F_{\omega,\tau}(x^\text{A})-F_{\omega,\tau}(x^\text{B})|\\
	\leq&|F_{\omega}(x^\text{A})-F_{\omega}(x^\text{B})| + \tau \cdot \sum_{j=1}^{n_z} \int_\Omega\,\left|\log\left(z^\text{A}_{[j]}(t)\right)-\log\left(z^\text{B}_{[j]}(t)\right)\right|\,\mathrm{d}t \\
	\leq&L_\omega \cdot \|x^\text{A}-x^\text{B}\|_\cX + \tau \cdot n_z \cdot |\Omega|\cdot\operatornamewithlimits{max}_{1\leq j \leq n_z}\ \operatornamewithlimits{ess\,sup}_{t \in \Omega} \left|\log\left(z^\text{A}_{[j]}(t)\right)-\log\left(z^\text{B}_{[j]}(t)\right)\right|.
	\end{align*}
	We know a lower bound for the arguments in the logarithms from~\eqref{eqn:aux1:Lcont}. Thus, the essential supremum term can be bounded with a Lipschitz result for the logarithm:
	\begin{align*}
	&\operatornamewithlimits{max}_{1\leq j \leq n_z}\ \operatornamewithlimits{ess\,sup}_{t \in \Omega} \left|\log\left(z^\text{A}_{[j]}(t)\right)-\log\left(z^\text{B}_{[j]}(t)\right)\right|\\
	&\quad \leq\operatornamewithlimits{max}_{1\leq j \leq n_z}\ \frac{1}{\, \frac{\tau}{2 \cdot L_\omega} \,} \cdot \|z_{[j]}^\text{A}-z_{[j]}^\text{B}\|_{L^\infty(\Omega)}
	%&\quad \leq \frac{2 \cdot L_\omega}{\tau} \cdot \operatornamewithlimits{max}_{1\leq j \leq n_z}\  \|z_{[j]}^\text{A}-z_{[j]}^\text{B}\|_{L^\infty(\Omega)}\\
	\leq \frac{2 \cdot L_\omega}{\tau} \cdot\frac{p+1}{\sqrt{\vartheta \cdot h}} \cdot \|x^\text{A}-x^\text{B}\|_\cX,
	\end{align*}
	where the latter inequality is obtained using Lemma~\ref{lem:NormEquivalence}. \end{proof}

\subsection{Interpolation Error}\label{sec:InterpolationError}
In order to show high-order convergence results, it is imperative that the solution function can be represented with high accuracy in a finite element space. In the following we introduce a suitable assumption for this purpose.

Motivated by the Bramble-Hilbert Lemma \cite{BrambleHilbert}, we make the assumption (A.5) that for a fixed chosen degree $p = \cO(1)$ there exists an $ \ell \in (0,\infty)$ such that
\begin{align}
\infexpr = \cO\big(h^{\ell+1/2}\big)\,.  \label{eqn:InfBound}
\end{align}
Notice that the best approximation $x_h$ is well-defined since $\cX_{h,p}$ is a Hilbert space with induced norm $\|\cdot\|_\cX$. In Appendix~\ref{app:3} we give two examples to demonstrate the mildness of assumption~\eqref{eqn:InfBound}.

To clarify on the mildness of \eqref{eqn:InfBound}, consider the triangular inequality
\begin{align}
\infexpr \leq \|x^\star_{\omega,\tau} - x^\star\|_\cX + \operatornamewithlimits{min}_{x_h \in \cX_{h,p}}\|x^\star-x_h\|_\cX\,,
\end{align}
where $x^\star$ is the global minimizer. Clearly, the second term converges under the approximability assumption of finite elements, hence could not be milder. The first term holds under several sufficient assumptions; for instance if $x^\star$ is unique because then convergence of feasibility residual and optimality gap will determine --at a convergence rate depending on the problem instance-- that unique solution. Due to round-off errors in computations on digital computers, for numerical methods the notion of well-posedness is imperative, hence must always be assumed, relating to how fast the first term converges as optimality gap and feasibility residual converge.

For the remainder, we define $\nu:= \ell/2$, $\eta:=(1-\zeta)\cdot\nu$ with respect to $\ell,\zeta$. We choose $\tau = \cO(h^{\nu})$ and $\omega = \cO(h^{\eta}) $ with
$h>0$ suitably small such that $0 < \tau \leq \omega < 1$.

Following the assumption \eqref{eqn:InfBound}, the result below shows that the best approximation in the finite element space satisfies an approximation property.
\begin{lem}[Finite Element Approximation Property]
	\label{lem:bestApproximation}
	If~\eqref{eqn:InfBound} holds and  $h>0$ is chosen sufficiently small, then
	\begin{equation}
	\infexpr \leq \delta_{\omega,\tau,h}. \label{eqn:SphereBound}
	\end{equation}
\end{lem}
\begin{proof}
	For $h>0$ sufficiently small it follows from $\ell> \nu + \eta$, that $h^{\ell+1/2} < h^{\nu+\eta+1/2}$. Hence,
	$$	\infexpr\leq const \cdot h^{\nu+\eta+1/2}	$$
	for some constant $const$. The result follows by noting that
	\begin{align*}	
	\delta_{\omega,\tau,h}
	& %= \frac{\tau}{2 \cdot L_\omega} \cdot \frac{\sqrt{\vartheta \cdot h}}{p+1}
	\geq \frac{\tau}{\,\frac{L_r}{\omega}\,} \cdot \frac{\sqrt{\vartheta \cdot h}}{p+1}
	= \frac{\sqrt{\vartheta}}{L_r\cdot(p+1)} \cdot \tau \cdot \omega \cdot \sqrt{h} \geq const \cdot h^{\nu+\eta+1/2}.
	\end{align*}
\end{proof}
In other words, Lemma~\ref{lem:bestApproximation} says for $h>0$ sufficiently small it follows that $\cB \cap \cX_{h,p}\neq \emptyset$. This is because the minimizing argument of \eqref{eqn:SphereBound} is an element of $\cB$.

\subsection{Optimality}
We show that an $\epsilon$-optimal solution for \eqref{eqn:PBP_h} is an $\varepsilon$-optimal solution for \eqref{eqn:PBP}, where $\varepsilon \geq \epsilon$.
\begin{thm}[Optimality of Unconstrained FEM Minimizer]\label{thm:conv}
	Let $\cB$ as in Proposition~\ref{prop:Lcont}, and $x^\epsilon_h$ an $\epsilon$-optimal solution for \eqref{eqn:PBP_h}, i.e.	
	\begin{align*}
	F_{\omega,\tau}(x^\epsilon_h) \leq F_{\omega,\tau}(x^\star_h) + \epsilon.
	\end{align*}
	If $\cB \cap \cX_{h,p}\neq \emptyset$, then $x^\epsilon_h$ satisfies:
	\begin{align*}
	F_{\omega,\tau}(x^\epsilon_h)
	\leq F_{\omega,\tau}(x^\star_{\omega,\tau})+ \epsilon+ L_{\omega,\tau,h} \cdot \infexpr .
	\end{align*}
\end{thm}
\begin{proof}
	Consider the unique finite element best approximation from \eqref{eqn:SphereBound}
	$$ 	\tilde{x}_h := \operatornamewithlimits{arg\,min}_{x_h \in\cX_{h,p}}\|x^\star_{\omega,\tau}-x_h\|_\cX\,.  $$
	Since $\cB\cap \cX_{h,p} \neq \emptyset$, it follows
	$\tilde{x}_h \in \cB \cap \cX_{h,p}$. Hence,
	$$ 	\tilde{x}_h = \operatornamewithlimits{arg\,min}_{x_h \in\cB\cap\cX_{h,p}}\|x^\star_{\omega,\tau}-x_h\|_\cX .  $$
	From \eqref{eqn:aux1:Lcont} we find $\cB \cap \cX_{h,p} \subset \cX_{h,p}^{\omega,\tau}$. Thus, $\tilde{x}_h \in \cX^{\omega,\tau}_{h,p}$. Hence,
	$$ 	\tilde{x}_h = \operatornamewithlimits{arg\,min}_{x_h \in\cX^{\omega,\tau}_{h,p}}\|x^\star_{\omega,\tau}-x_h\|_\cX\,.  $$
	
	Proposition~\ref{prop:Lcont} can be used to obtain %the bound
	%\begin{align}
	$	F_{\omega,\tau}(\tilde{x}_h) \leq F_{\omega,\tau}(x^\star_{\omega,\tau}) + L_{\omega,\tau,h} \cdot \|x^\star_{\omega,\tau}-\tilde{x}_h\|_\cX$. 
	%\label{eqn:Thm1:eq1}
	%\end{align}
	Since $x^\epsilon_h$ is a global $\epsilon$-optimal minimizer of $F_{\omega,\tau}$ in $\cX^{\omega,\tau}_{h,p}$ and also $\tilde{x}_h$ lives in $\cX^{\omega,\tau}_{h,p}$, the optimalities must relate as
	%\begin{align}
	$F_{\omega,\tau}(x^\epsilon_h) \leq F_{\omega,\tau}(\tilde{x}_h)+\epsilon$. %\label{eqn:Thm1:eq2}
	%\end{align}
	The result follows. % from \eqref{eqn:Thm1:eq1}--\eqref{eqn:Thm1:eq2}.
\end{proof}

\subsection{Convergence}
We obtain a bound for the optimality gap and feasibility residual of $x^\epsilon_h$.

\begin{thm}[Convergence to~\refOCP]\label{thm:ConvOCP}
	Let $x^\epsilon_h$ be an $\epsilon$-optimal numerical solution to~\eqref{eqn:PBP_h}. If (A.4) holds, then $x^\epsilon_h$ satisfies
	\begin{align*}
	g_\mathrm{opt} = \cO\left(\tau^{1-\zeta}+\varepsilon_{h,p} \right),\
	r_\mathrm{feas} = \cO\big(\omega \cdot \left(1+ \tau^{1-\zeta} + \varepsilon_{h,p}\right)\big),
	\end{align*}
	where
	$$ 	\varepsilon_{h,p} := L_{\omega,\tau,h} \cdot \infexpr + \epsilon\,.	$$
\end{thm}
\begin{proof}
	From Theorem~\ref{thm:conv} we know \mbox{$F_{\omega,\tau}(x^\epsilon_h) \leq F_{\omega,\tau}(x^\star_{\omega,\tau}) + \varepsilon_{h,p}$}. This is equivalent to
	\begin{align*}
	&F_\omega(x^\epsilon_h) + \tau \cdot \Gamma(x^\epsilon_h) \leq F_\omega(x^\star_{\omega,\tau}) + \tau \cdot \Gamma(x^\star_{\omega,\tau}) + \varepsilon_{h,p} \\
	\Rightarrow\ \,
	&F_\omega(x^\epsilon_h) \leq F_\omega(x^\star_{\omega,\tau})+ \underbrace{|\tau \cdot \Gamma(x^\epsilon_h)| + |\tau \cdot \Gamma(x^\star_{\omega,\tau})|}_{(*)} + \varepsilon_{h,p}.
	\end{align*}
	Since $x^\epsilon_h \in \cX^{\omega,\tau}_{h,p}$, it follows that $z^\epsilon_h \geq \frac{\tau}{2 \cdot L_\omega}\cdot \be$
	and thus $x^\epsilon_h = \check{x}^\epsilon_h$.
	From Lemma~\ref{lem:StrictInteriorness} we know $x^\star_{\omega,\tau} = \bar{x}^\star_{\omega,\tau}$. Thus, we can apply Lemma~\ref{lem:Gamma} to bound $(*)$ with $\cO(\tau^{1-\zeta})$. Hence,
	%\begin{align*}
	$F_\omega(x^\epsilon_h) \leq F_\omega(x^\star_{\omega,\tau}) + \cO(\tau^{1-\zeta}) + \varepsilon_{h,p}$.
	%\end{align*}
	Since, according to Proposition~\ref{prop:PenaltyBarrierSolution}, $x^\star_{\omega,\tau}$ is $\tilde{\varepsilon}$-optimal for~\eqref{eqn:PP}, where
	$\tilde{\varepsilon} = \cO(\tau^{1-\zeta}),$
	it follows that
	\begin{align*}
	F_\omega(x^\epsilon_h) \leq F_\omega(x^\star_{\omega}) + \underbrace{\cO(\tau^{1-\zeta}) + \varepsilon_{h,p}}_{=:\varepsilon}.
	\end{align*}
	In other words, $x^\epsilon_h$ is $\varepsilon$-optimal for \eqref{eqn:PP}. The result now follows from Proposition~\ref{thm:PenaltySolution}.
\end{proof}

%If the solution to the unconstrained problem~\eqref{eqn:PBP} is sufficiently smooth in the sense of~\eqref{eqn:InfBound}, then one can obtain stronger results on the order of convergence.

Below, we translate the above theorem into an order-of-convergence result.
% For piecewise smooth optimal control solutions, $\ell$ is usually in the order of $p$.
%% comment Martin: The out-commented statement is not true. Instead: If u(t) is a sufficiently smooth function then a finite-element approximation converges in O(h^p), i.e.\ $\ell=p$. If u(t) has jumps then O(h^1), unless the jumps are captured by the mesh.
\begin{thm}[Order of Convergence to~\refOCP]\label{thm:order}
	Consider $x^\epsilon_h$ with $\epsilon = \cO(h^{\ell-\eta})$. Then
	%If the conditions of Theorem~\ref{thm:ConvOCP} and Lemma~\ref{lem:bestApproximation}   hold  with $\epsilon = \cO(h^{\ell-\eta})$,  then
	%\begin{align*}
	$g_\mathrm{opt} = \cO\left(h^\eta\right)$ and
	$r_\mathrm{feas} =\cO\left(h^\eta\right)$.
	%\end{align*}
\end{thm}
\begin{proof}
	It holds that
	\begin{align*}
	L_{\omega,\tau,h}
	&= \,\,\,\,L_\omega\,\,\, + (n_z \cdot |\Omega| \cdot 2) \cdot\,\, L_\omega\,\,\,\, \cdot \left(\frac{p+1}{\sqrt{\vartheta \cdot h}}\right)\\
	&=\frac{\cO(1)}{\omega} +\,\,\,\,\,\,\,\,\,\, \cO(1) \,\,\,\,\,\,\,\,\,\cdot \frac{\cO(1)}{\omega} \cdot \,\,\,\,\frac{\cO(1)}{\sqrt{h}}\\
	&=\cO\big(h^{-\eta} + h^{-\eta} \cdot h^{-1/2} \big) = \cO(h^{-\eta -1/2}).
	\end{align*}		
	From \eqref{eqn:InfBound}, we find
	\begin{align*}
	\varepsilon_{h,p} 	
	&= L_{\omega,\tau,h} \cdot \infexpr + \epsilon\\
	&= \cO\big(h^{-\eta -1/2}\big) \cdot \cO\big(h^{\ell+1/2}\big) + \cO\big(h^{\ell-\eta}\big)= \cO\big(h^{\ell-\eta}\big).
	\end{align*}
	Combining this with Theorem~\ref{thm:ConvOCP}, we find $x^\epsilon_h$ satisfies
	\begin{align*}
	g_\text{opt}
	=& \cO\big(\tau^{1-\zeta} + \varepsilon_{h,p}\big)
	=\cO\big(h^{\nu \cdot (1-\zeta)}+h^{\ell-\eta}\big)
	=\cO\big(h^{\min\left\lbrace \eta ,\, \ell-\eta \right\rbrace}\big),\\
	r_\text{feas}
	=& \cO\big(\omega \cdot (1+\tau^{(1-\zeta)\cdot\nu} + \varepsilon_{h,p})\big)
	= \cO\big(h^\eta + h^{\eta+(1-\zeta)\cdot\nu} + h^{\eta+\ell-\eta} \big)\\
	=& \cO\big( h^{\eta} + h^{2 \cdot\eta} + h^{\ell} \big)= \cO\big(h^{\min\lbrace \eta ,\, \ell \rbrace}\big).
	\end{align*}
	Note that $\ell > \ell-\eta>\eta$.
\end{proof}
Recall that $\eta = \ell/2 \cdot (1-\zeta)$, where $0<\zeta\ll 1$. If $\ell\approx p$ and $\eta\approx \ell/2$ it follows that $h^\eta \approx \sqrt{h^p}$.

\subsection{Numerical Quadrature}
When computing $x^\epsilon_h$, usually the integrals in $F$ and~$r$ cannot be evaluated exactly. In this case, one uses numerical quadrature and replaces $F_{\omega,\tau}$ with
$F_{\omega,\tau,h} := F_h + \frac{1}{2 \cdot \omega} \cdot r_h + \tau \cdot \Gamma$. Since $\cX_{h,p}$ is a space of piecewise polynomials, $\Gamma$ can be integrated analytically. However, the analytic integral expressions become very complicated. This is why, for a practical method, one may also wish to use quadrature for $\Gamma$.

If $F$ and $r$ have been replaced with quadrature approximations $F_h,\,r_h$, then it is sufficient that these approximations satisfy
\begin{align}
|F_{\omega,\tau,h}(x)-F_{\omega,\tau}(x)| \leq C_\text{quad} \cdot \frac{h^q}{\omega}\quad\quad \forall x \in \cX^{\omega,\tau}_{h,p},\label{eqn:QuadCond}
\end{align}
with bounded constant $C_\text{quad}$ and quadrature order $q\in\N$, to ensure that  the convergence theory holds. We discuss this further below.
%The quadrature error can be bounded independent of $x$ since $f$ and $c$ are bounded globally according to (A.3).
The constrraint \eqref{eqn:QuadCond} poses a consistency and stability condition.

\paragraph{Consistency}
There is a consistency condition in~\eqref{eqn:QuadCond} that relates to suitable values of $q$. In particular, if we want to ensure convergence of order $\cO(h^\eta)$, as presented in Theorem~\ref{thm:order}, then $q$ has to be sufficiently large.

Consider the problem
\begin{equation*}
\tilde{x}^\star_{h} \in \operatornamewithlimits{arg\,min}_{x \in \cX_{h,p}^{\omega,\tau}}\ F_{\omega,\tau,h}(x).
\end{equation*}
%with a solution $\tilde{x}^\star_{h}$.
Note that $\tilde{x}^\star_{h}$ is $\epsilon$-optimal for \eqref{eqn:PBP_h}, where from
\begin{align*}
F_{\omega,\tau}(\tilde{x}^\star_h)-C_\text{quad} \cdot \frac{h^q}{\omega} &\leq F_{\omega,\tau,h}(\tilde{x}^\star_h)
\leq F_{\omega,\tau,h}(x^\star_h) \leq F_{\omega,\tau}(x^\star_h)+C_\text{quad} \cdot \frac{h^q}{\omega}
\end{align*}
it follows that $ 	\epsilon = \cO\left({h^q}/{\omega}\right) = \cO(h^{q-\eta}).  	$
Hence, $\tilde{x}^\star_{h}$ satisfies the bounds for the optimality gap and feasibility residual presented in Theorem~\ref{thm:ConvOCP}.  We obtain the same order of convergence as in Theorem~\ref{thm:order} when maintaining $\epsilon = \cO(h^{\ell-\eta})$, i.e.\ choosing $q\geq \ell$.

\paragraph{Stability}
Beyond consistency, \eqref{eqn:QuadCond} poses a non-trivial stability condition. This is because the error bound must hold $\forall x \in \cX_{h,p}$. We show this with an example.

Consider $\Omega=(0,1)$, $n_y=0$, $n_z=1$, and \mbox{$c(x):=\sin(\pi \cdot x)$.} The constraint forces $x(t)=0$. Clearly, $c$ and $\nabla c$ are bounded globally. Consider the uniform mesh $\cT_h:=\left\lbrace\,T_j \ \vert \ T_j = \left((j-1)\cdot h, j\cdot h\right), j=1,2,\ldots,1/h \right\rbrace$
for $h \in 1/\N$, choose $p=1$ for the finite element degree, and Gauss-Legendre quadrature of order $q=3$, i.e.\ the mid-point rule quadrature scheme \cite{GaussQuad} of $n_q=1$ point per interval. Then, the finite element function $x_h$, defined as
$x(t) := -{1}/{h} + {2}/{h} \cdot (t-j \cdot h)$ for $t \in T_j$ on each interval, yields the quadrature error
\begin{align*} 	
|r_h(x)-r(x)| = \Bigg|\underbrace{ h \cdot \sum_{j=1}^{1/h} \sin^2\big(\pi \cdot x(j\cdot h - h/2)\big)}_{= 0} - \underbrace{\int_{0}^{1}\sin^2\big(\pi \cdot x(t)\big)\mathrm{d}t}_{=0.5} \Bigg|,
\end{align*}
violating \eqref{eqn:QuadCond}. In contrast, using Gauss-Legendre quadrature of order $5$ (i.e.\ using $n_q=2$ quadrature points per interval) yields satisfaction of \eqref{eqn:QuadCond} with $q=5$.

We see that in order to satisfy \eqref{eqn:QuadCond}, a suitable quadrature rule must take into account the polynomial degree $p$ of the finite element space and the nature of the nonlinearity of $c$. We clarify this using the notation $(\phi \circ \psi)(\cdot):=\phi(\psi(\cdot))$ for function compositions: If
$$ 	\left(f + \frac{1}{2\cdot\omega}\cdot \|c\|_2^2 \right) \circ x \in \cP_d({T})^{n_x},\ \forall T\in\cT_h,\forall x \in \cX_{h,p}\cap\cB,$$
for some $d \in \N$, i.e.\ the integrands of $F$ and~$r$ are polynomials in $t$, then $q \geq d$ is a sufficient order for exact quadrature. For a practical method, we propose to use Gaussian quadrature of order $q=4 \cdot p +1$, i.e.\ using~$n_q=2 \cdot p$ abscissae per interval $T \in \cT_h$. %We notice that this would not be possible with collocation, where the number $n_q$ of quadrature ($=$collocation) points cannot exceed the polynomial degree~$p$ of $\cX_{h,p}$, since this would cause overdetermination of the nonlinear program.

\subsection{On Local Minimizers}
\label{sec:local}
Above, we proved that the global NLP minimizer converges to a (or the, in case it is unique) global  minimizer of \refOCP. However, practical NLP solvers can often only compute critical points, which have a local minimality certificate at best.
For collocation methods, all critical  points of \refOCP have convergent critical NLP points if the mesh is sufficiently fine. For PBF the above global convergence result implies a more favorable assertion: For every strict local  minimizer of \refOCP there is exactly one convergent strict local NLP minimizer if the mesh is sufficiently fine. Below we explain the reason why this follows from the above global convergence property.

Consider a strict local minimizer $\tilde{x}^\star$ of \refOCP. By definition of a local minimizer, inactive box constraints $x_L\leq x\leq x_R$ could be imposed on \refOCP such that $\tilde{x}^\star$ is the unique global minimizer of a modified problem. Upon discretization we would keep the box constraints as $\bx_L \leq \bx \leq \bx_R$. From the above convergence result, since $\tilde{x}^\star$ is unique with inactive box constraints, $\bx$ must converge to $\tilde{x}^\star$, leaving the NLP box constraints inactive, as if they had been omitted as in the original problem.

\section{Numerical Experiments}
\label{sec:numerical}
The scope of this paper is the transcription method.
Practical aspects in solving the NLP \eqref{eqn:PBP_h} are discussed in \cite{Neuenhofen2020AnIP}, where we also show non-zero patterns of the sparse Jacobian and Hessian of the constraints and Lagrangian function; and show that the computational cost roughly compares to solving the NLPs from LGR collocation. Below, we present numerical results for two test problems when using our transcription method and minimizing \eqref{eqn:PBP_h} for the given instance, mesh, and finite element degree.

% experiments here

\subsection{Convex Quadratic Problem with a Singular Arc}\label{sec:numExp:CQP}
Consider a small test problem, which demonstrates convergence of PBF in a case where direct collocation methods ring:
\begin{equation}
\begin{aligned}
&\min_{y,u} \quad &&\int_0^{\pi} \left(\,y_0(t)^2 + \cos^{(1-m)}(t) \cdot  u(t)\,\right)\,\mathrm{d}t,\\
&\text{s.t.} &&\dot{y}_{k-1}(t)=y_{k}(t),\text{ for }k=1,\dots,m\,,\qquad\dot{y}_m(t)=u(t),
\end{aligned}
\end{equation}
where $\cos^{(1-m)}$ is the $(1-m)^{th}$ derivative of $\cos$, with negative derivative meaning antiderivative. Figure~\ref{fig:CQVP} shows the numerical solutions for $m=1$ for Trapezoidal (TR), Hermite-Simpson (HS), LGR collocation and PBF, where the latter two use polynomial degree $p=2$. For higher degree $p$, LGR would still ring when $m\geq p-1$. TR and HS require box constraints on $u$ for boundedness of their NLP minimizers.

\begin{figure}[tb]
	\centering
	\includegraphics[width=0.6\columnwidth]{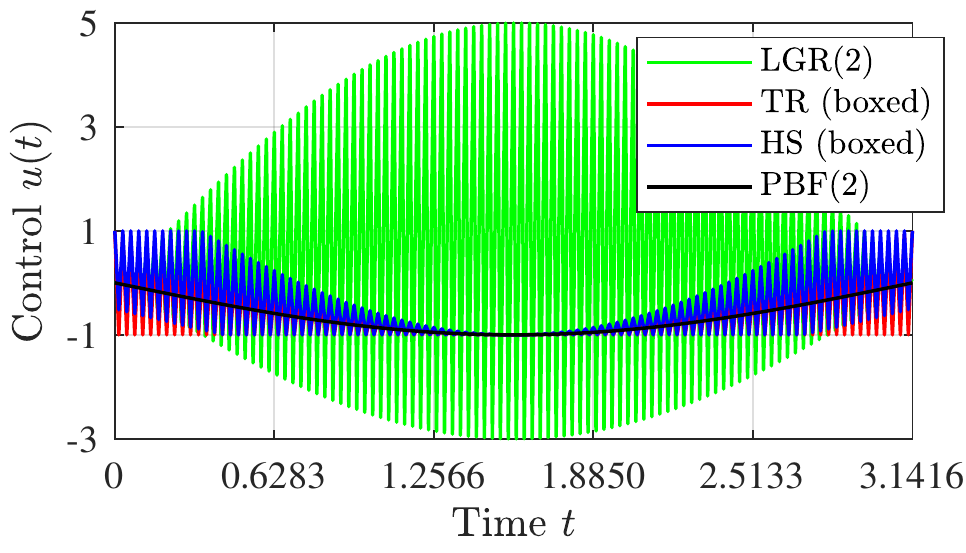}
	\caption{Comparison of control solutions for three collocation methods and PBF.}
	\label{fig:CQVP}
\end{figure}

\subsection{Second-Order Singular Regulator}
This bang-singular control problem from \cite{SOSC} with \mbox{$t_E=5,$}\,\mbox{$\eta=1$} is given as
\begin{equation}
\begin{aligned}
&\min_{y,u} &\quad &\int_0^{t_E} \left(\,y_2(t)^2 + \eta \cdot  y_1(t)^2\,\right)\,\mathrm{d}t,\\
&\text{s.t.} & y_1(0)&=0,\quad y_2(0)=1,\\
&& \dot{y}_1(t)&=y_2(t),\quad \dot{y}_2(t)=u(t),\\
&& -1\leq u(t)&\leq 1.
\end{aligned}\label{eqn:SOSR_ocp}
\end{equation}

Both LGR and PBF use $100$ elements of degree $p=5$. $\bS$~has 24507 non-zeros and bandwidth 15 for both discretizations. Forsgren-Gill solves PBF in 40 and LGR in 41 NLP iterations.

Figure~\ref{fig:SOSController} presents the control profiles of the two numerical solutions. LGR shows ringing on the time interval $[1.5,\,5]$ of the singular arc. In contrast, PBF converges with the error $\|u^\star(t)-u_h(t)\|_{L^2([1.5,\,5])} \approx 1.5 \cdot 10^{-4}$.

\begin{figure}[tb]
	\centering
	\includegraphics[width=0.6\columnwidth]{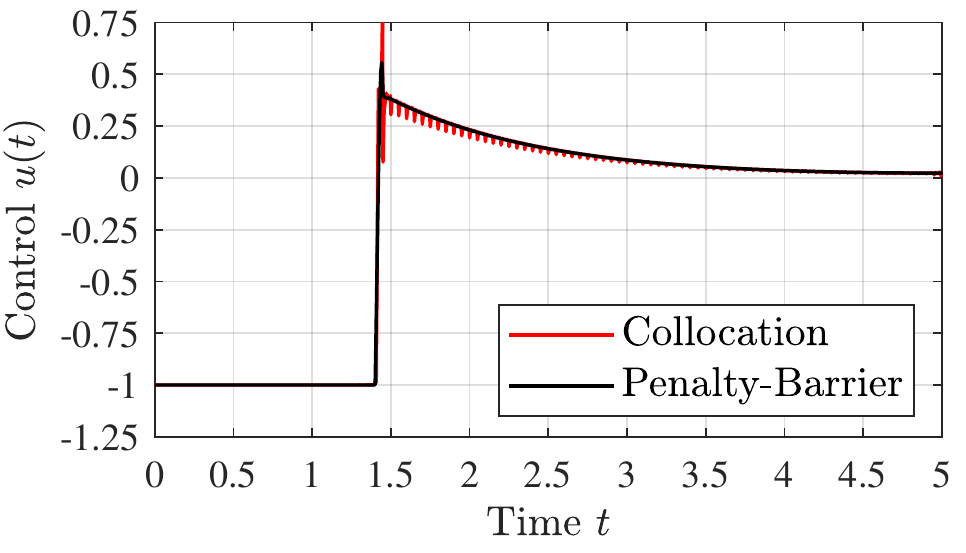}
	\caption{Comparison of control solutions from LGR and PBF for the Second-Order Singular Regulator.}
	\label{fig:SOSController}
\end{figure}

\subsection{Aly-Chan Problem}
The problem in \cite{AlyChan}, namely \eqref{eqn:SOSR_ocp} with \mbox{$t_E=\pi/2,$} \mbox{$\eta=-1,$} has a smooth totally singular control.

Both LGR and PBF use $100$ elements of degree $p=5$. $\bS$ has the same nonzero pattern as before. Forsgren-Gill for PBF/LGR converges in 48/43 iterations.
Figure~\ref{fig:NumExp_AlyChan} presents the control profiles of the two numerical solutions. PBF converges, with error
$  \|u^\star(t)-u_h(t)\|_{L^2(\Omega)} \approx 3.7 \cdot 10^{-6}     . $
LGR does not converge for this problem; cf.~\cite[Fig.~3]{Kameswaran},\cite{ChenBiegler16}. To fix the convergence for this problem, \cite{ChenBiegler16} proposes a particular mesh adaptation scheme whereas \cite{Kameswaran} employs a regularization technique.

\begin{figure}[tb]
	\centering
	\includegraphics[width=0.6\columnwidth]{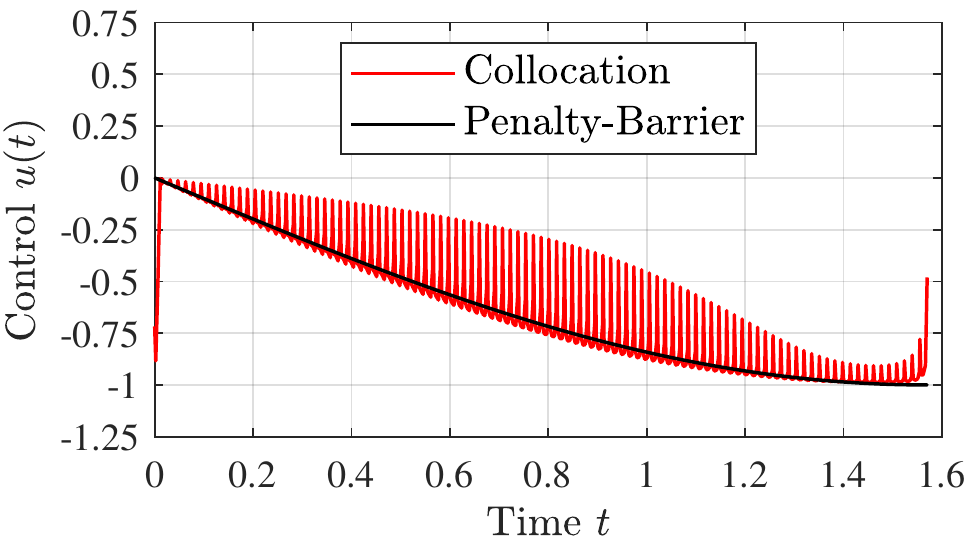}
	\caption{Comparison of control solutions from LGR and PBF for the Aly-Chan Problem.}
	\label{fig:NumExp_AlyChan}
\end{figure}

\subsection{Regular state-constrained problem}\label{sec:numExp:Sparsity}
We now consider a test problem for which both types of methods converge with success, so that we can compare conditioning, convergence, and rate of convergence to a known analytical solution:
\begin{align*}
\min_{y,u}& & J&= y_2(1),\label{eqn:ExampleOCP2}\\
\text{s.t.}& & y_1(0)&=1,\ \, \dot{y}_1(t)=\frac{u(t)}{2y_1(t)},\ \, \sqrt{0.4}\leq y_1(t),\\
& & y_2(0)&=0,\ \, \dot{y}_2(t)=4 y_1(t)^4 + u(t)^2,\ \,  -1\leq u(t)\,.
\end{align*}
The solution is shown in Figure~\ref{fig:exp2}.
\begin{figure}
	\centering
	\includegraphics[width=0.6\columnwidth]{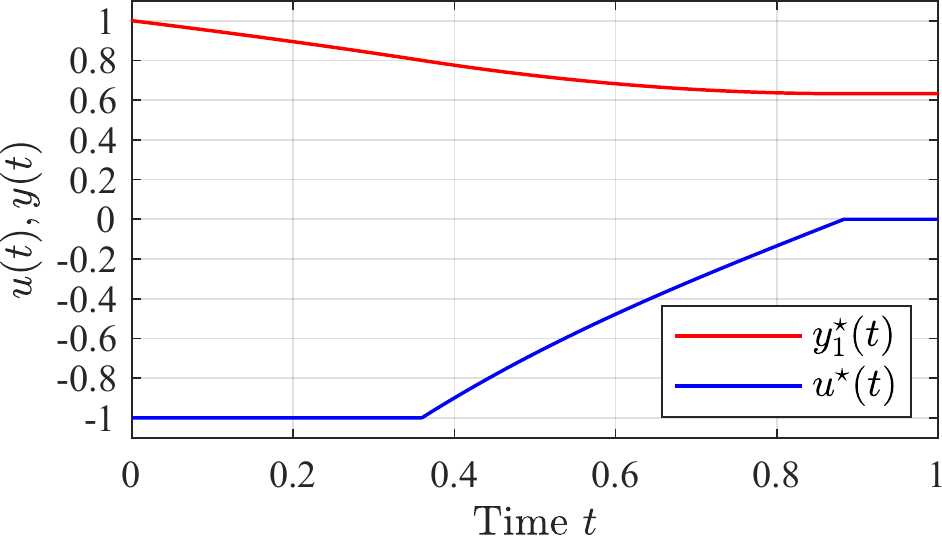}
	\caption{Analytical solution to \eqref{eqn:ExampleOCP2}.}
	\label{fig:exp2}
\end{figure}
$u^\star$ is constant outside $t_0=1-\frac{\sqrt{41}}{10}\approx 0.35$ and $t_1=t_0+\log 2-\frac{\log(\sqrt{41}-5)}{2}\approx 0.88$, between which $u^\star(t)= 0.8 \sinh\left(2 (t - t_1)\right)$, yielding $J\approx 2.0578660621682771255864272367598$. 

All methods yield accurate solutions. Figure~\ref{fig:loglogxpl2} shows the convergence of the optimality gap and feasibility residual of a respective method. Remarking on the former, we computed $J^\star-J(x_h)$ and encircled the cross when $J(x_h)<J^\star$. Note in the figure that for $\geq 40$ elements the most accurate solutions in terms of feasibility are found by PBF with $\omega=10^{-10}$. Further, we find that the collocation methods significantly underestimate the optimality value for this experiment.

\begin{figure}
	\centering
	\includegraphics[width=0.6\columnwidth]{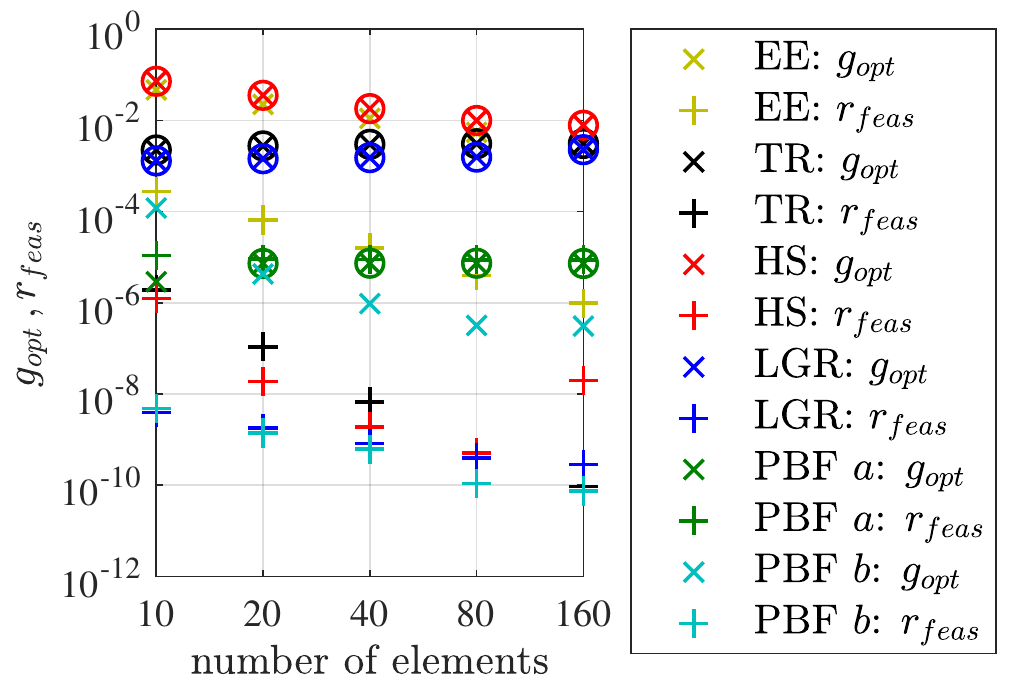}
	\caption{Convergence of optimality gap and feasibility residual%/ for collocation-type methods and PBF. 
		LGR and PBF use polynomial degree $p=5$. PBF uses two different values: a) $\omega=10^{-5}$, b) $\omega=10^{-10}$.}
	\label{fig:loglogxpl2}
\end{figure}

Now we discuss rates of convergence. Convergence of only first order is expected because higher derivatives of $y^\star$ are non-smooth and $u^\star$ has edges. Indeed, $r_\text{feas}$ converges linearly for all methods. PBF5 with $\omega=10^{-5}$ stagnates early because it converges to the optimal penalty solution, which for this instance is converged from $20$ elements onwards. $g_\text{opt},r_\text{feas}$ are then fully determined by $\omega$. The issue is resolved by choosing $\omega$ smaller. LGR5 and PBF5 with $\omega=10^{-10}$ converge similarly, and stagnate at $r_\text{feas}\approx 10^{-10}$. Due to the high exponent in the objective, small feasibility errors in the collocation methods amount to significant underestimation of the objective.

Finally, we look into computational cost. Solving the collocation methods with IPOPT and the PBF5 discretization with the interior-point method in \cite{ForsgrenGill1998}, the optimization converges in $\approx 20$ iterations for any discretization. Differences in computational cost can arise when one discretization results in much denser or larger problems than others. Here, we compare the sparsity structure of the Jacobian $\nabla_\bx C(\bx)\t$ for LGR5 in Figure~\ref{fig:sparsitylgr} and PBF5 in Figure~\ref{fig:sparsitypbf}, each using a mesh size of $h=\frac{1}{10}$.
\begin{figure}
	\centering
	\includegraphics[width=0.6\columnwidth]{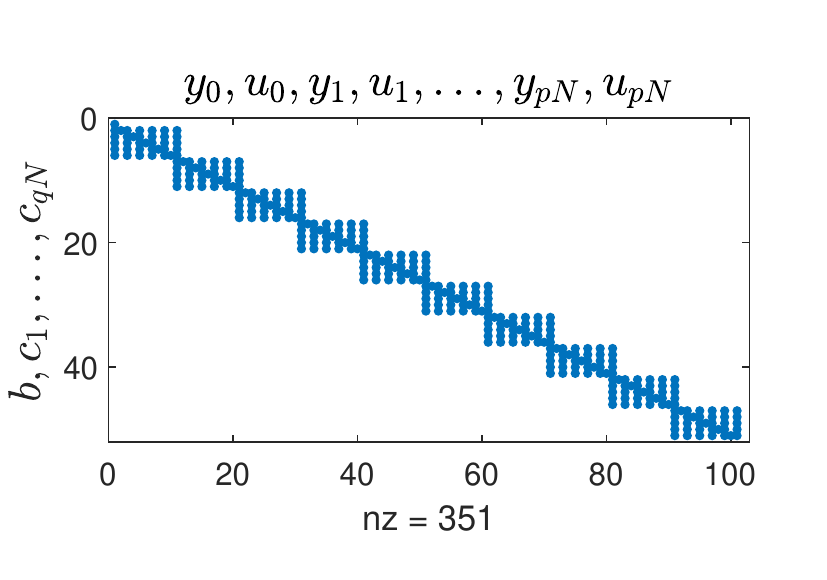}
	\caption{Sparsity of $\nabla_\bx C(\bx)\t$ for LGR5 when $h=\frac{1}{10}$, i.e.~$N=10$. For LGR, notice $q=p-1$. The discretization does not depend on $u_{pN}$.}
	\label{fig:sparsitylgr}
\end{figure}
\begin{figure}
	\centering
	\includegraphics[width=0.6\columnwidth]{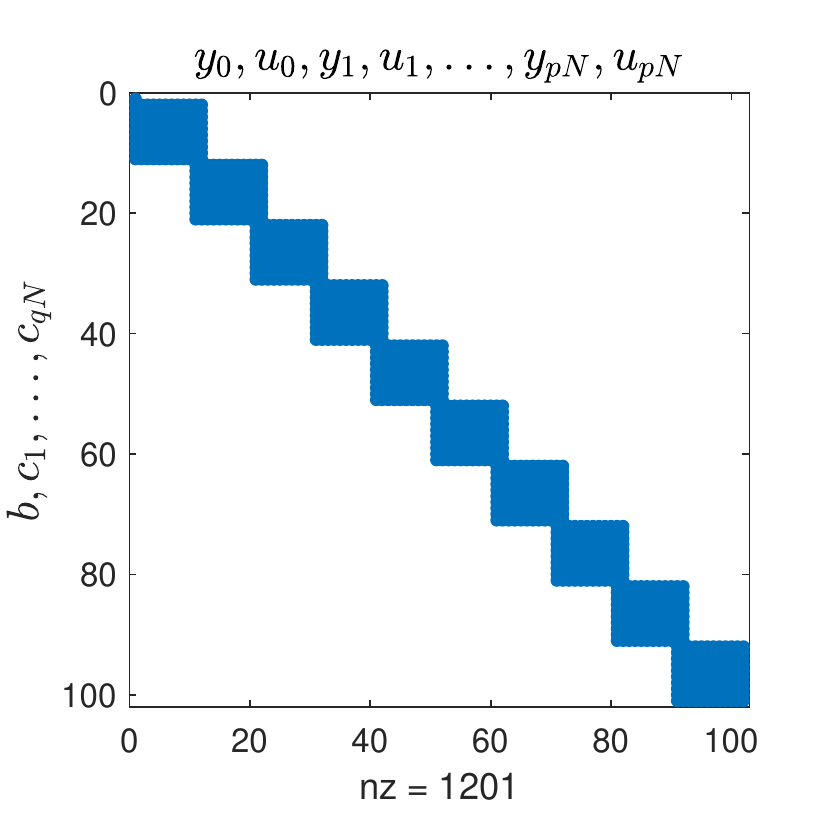}
	\caption{Sparsity of $\nabla_\bx C(\bx)\t$ for PBF5 when $h=\frac{1}{10}$, i.e.~$N=10$, with $q=2p$.}
	\label{fig:sparsitypbf}
\end{figure}
Note that for PBF5, $C(\bx)$ has more rows in the Jacobian than LGR5, thus the Jacobian has hence more non-zeros. However, critical for computations is the primal Schur complement $\mathbf{\Sigma}=\nabla_{\bx\bx}^2 \cL(\bx,\blambda)+\nabla_\bx C(\bx)\t \bD \nabla_\bx C(\bx)$, which is used when solving the KKT system via the reduced form, where $\bD$ is a diagonal matrix. $\mathbf{\Sigma}$ is a narrow-banded matrix with dense band of the same bandwidth for LGR5 and PBF5.

With regard to computational cost, it follows from Fig.~\ref{fig:loglogxpl2} that the ability to choose $\omega$ in PBF can be advantageous. In particular, on coarse meshes, one may opt for small feasibility residual by manually decreasing $\omega$, whereas with a collocation method one is stuck with the feasibility residual that one obtains for that particular mesh. The figure shows this: For $\omega=10^{-10}$, even on the coarsest mesh the PBF method achieves a solution that has a smaller feasibility residual than other methods on the same mesh. For this problem this becomes possible because the path constraint could be satisfied with zero error by choosing $y$ a polynomial of degree 3 (because here PBF uses $p=5$).

\subsection{Van der Pol Controller}

This problem uses a controller to stabilize the van der Pol differential equations on a finite-time horizon. The problem is stated as
\begin{equation*}
\begin{aligned}
&\min_{y,u} &\quad \frac{1}{2}\cdot &\int_0^4 \left(\,y_1(t)^2 + y_2(t)^2\,\right)\,\mathrm{d}t,\\
&\text{s.t.} & y_1(0)&=0,\quad\quad y_2(0)=1,\\
&& \dot{y}_1(t)&=y_2(t),\\
&& \dot{y}_2(t)&=-y_1(t)+y_2(t) \cdot \left(\,1-y_1(t)^2\,\right) + u(t),\\
&& -1 & \leq u(t)\leq 1. 
\end{aligned}
\end{equation*}
The problem features a bang-bang control with a singular arc on one sub-interval. The discontinuities in the optimal control are to five digits at $t_1 = 1.3667$ and $t_2 = 2.4601$.

We solved this problem with LGR collocation on $100$ uniform elements of order $5$. We compare this solution to the one obtained with PBF using $100$ uniform elements of order $p=5$, with $\omega=10^{-10}$ and $\tau=10^{-10}$.

Figure~\ref{fig:NumExp_VanderPol} presents the control profiles of the two numerical solutions.
LGR shows ringing on the time interval $[t_2,\,4]$ of the singular arc. In contrast,  PBF converges to the analytic solution. The solution satisfies the error bounds
$e(0)\approx 7.0\cdot 10^{-2}$, $e(t_2)\approx 1.2\cdot 10^{-2}$, $e(2.5)\approx 8.17\cdot 10^{-4}$, and $e(2.6)\approx 9.6 \cdot 10^{-5}$, where
$e(\hat{t}):=\|u^\star(t)-u_h(t)\|_{L^2([\hat{t},4])}$. The larger errors in the vicinity of the jumps occur due to the non-adaptive mesh.

\begin{figure}[tb]
	\centering
	\includegraphics[width=0.6\columnwidth]{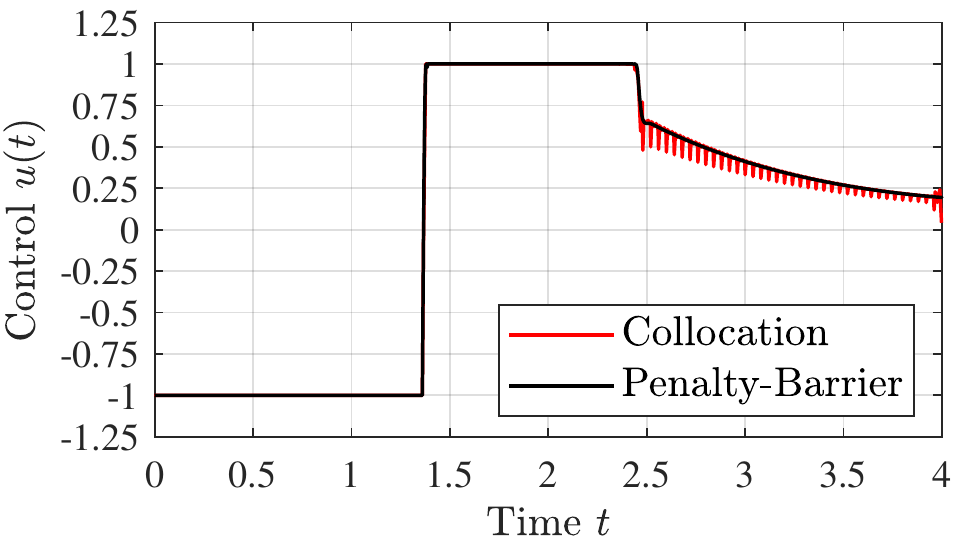}
	\caption{Comparison of control input obtained with Penalty-Barrier method against LGR collocation  for the Van der Pol problem.}
	\label{fig:NumExp_VanderPol}
\end{figure}

\subsection{Reorientation of an Asymmetric Body}
This nonlinear problem from  \cite[Ex.\ 6.12, eqn.\ 6.123]{BettsChap2} in the ODE formulation is considered numerically challenging for its minimum time objective, the control appearing only linearly in the dynamics, the problem having a bang-bang solution and multiple local minima with identical cost. Since the solution is bang-bang, rates of convergence (at least for the optimality gap) can only be linear. We compare convergence of three collocation methods and PBF, where a polynomial degree $p=3$ is used for the $hp$-methods.

Using the same initial guess from forward integration of an approximate control solution, LGR and PBF %(using $p=3$)
converge on average in 200 iterations. LGR was solved with IPOPT and PBF was solved with a penalty-barrier interior-point method presented in \cite{ForsgrenGill1998}. Both NLP solvers cost one linear system solve per iteration. For $\omega=\tau=10^{-3}$, the finite element solution $x^\star_h$ converges to the penalty-barrier minimizer $x^\star_{\omega,\tau}$ sooner, which however is not very feasible for the DOP at hand. The other collocation methods' NLPs were also solved using IPOPT, which terminated on local infeasibility for TR and HS. In contrast, LGR and PBF provide  numerical solutions that converge at similar rates, which stagnate around $10^{-6}$ for the feasibility residual and $10^{-4}$ for the optimality gap. These methods converge at similar rates. 
The size of the differential constraint violation and optimality gap for HS, LGR and PBF for different mesh sizes are given  in Table~\ref{tab:r_feas}. For HS, due to box constraints on end-time, which has been expressed as the first state, the optimality gap is negative and equal to the lower box constraint on the first state.

As with any regularization method, good values for $\omega,\tau$ can be found on the fly by saving each barrier solution and terminating when differences between subsequent barrier solutions stop decreasing. For computation of the gap, we determined $J^\star:=28.6298010321$ from PBF(3) on 2048 elements, where $r_{\textrm{feas}}\approx 2.4e-7$.

\newcommand{\linebreakcell}[2][c]{\begin{tabular}[#1]{@{}c@{}}#2\end{tabular}}
\newcommand{\lbqcell}[2]{\linebreakcell{#1/\\ \hspace{2mm}#2}\xspace}

\begin{table}[tb]
	\caption{$L^2(\Omega)$-norm for differential constraints / optimality gap of the Asymmetric Body Reorientation Problem. All methods use consistency order $p=3$. $N_\text{el}$ is the number of elements.}
	\label{tab:r_feas}
	\centering
	\begin{tabular}{|c||c|c|c|c|c|c|c|}
		\hline
		$N_\text{el}$ 	& HS 						& LGR 						& \linebreakcell{PBF\\$\omega=\tau=10^{-3}$} 	& \linebreakcell{PBF\\ $\omega=\tau=10^{-7}$} 	& \linebreakcell{PBF\\ $\omega=\tau=10^{-10}$} 	\\ \hline\hline
		8   			& \lbqcell{3.2e-2}{-1.3e-1} & \lbqcell{6.1e-3}{1.2e+0} 	& \lbqcell{4.5e-3}{-4.4e-1} 				& \lbqcell{1.9e-4}{1.4e+0}  				& \lbqcell{1.8e-4}{1.4e+0} 					\\ \hline
		32  			& \lbqcell{3.6e-3}{-1.3e-1} & \lbqcell{6.6e-5}{2.2e-2} 	& \lbqcell{4.2e-3}{-4.5e-1}	 				& \lbqcell{2.5e-5}{3.6e-2}  				& \lbqcell{7.7e-6}{4.2e-1} 					\\ \hline
		128 			& \lbqcell{1.4e-3}{-1.3e-1} & \lbqcell{1.0e-6}{5.7e-4} 	& \lbqcell{4.1e-3}{-4.5e-1} 				& \lbqcell{1.7e-6}{6.9e-4}  				& \lbqcell{3.4e-7}{9.5e-3} 					\\ \hline
		512 			& \lbqcell{7.0e-4}{-1.3e-1} & \lbqcell{2.1e-6}{9.2e-5} 	& \lbqcell{4.1e-3}{-4.5e-1} 				& \lbqcell{1.4e-6}{7.4e-5}  				& \lbqcell{1.1e-8}{6.1e-4} 					\\ \hline
		\hline
	\end{tabular}
	
\end{table}

\subsection{{Obstacle Avoidance Problem}}
Since we limited our presentation to a convergence analysis for global minimizers, we give this example to demonstrate PBF's practical capability to also converge to non-global minimizers.

Consider the minimum-time trajectory from $\vec{\chi}_0=[-10\ 10]\t$ to $\vec{\chi}_E=[10\ 10]\t$ around an obstacle at $\vec{\chi}_C=[0\ 8]\t$ of radius $R=3$:
\begin{equation*}
\begin{aligned}
\min_{\vec{\chi},u,t_E} \quad t_E&,\\
\text{s.t.} \quad \vec{\chi}(0)&=\vec{\chi}_0,\quad\vec{\chi}(t_E)=\vec{\chi}_E, \quad\|\vec{\chi}(t)-\vec{\chi}_C\|_2^2\geq R^2\\
\dot{\vec{\chi}}(t)&=\big[\cos\big(u(t)\big)\ \sin\big(u(t)\big)\big]\t,
\end{aligned}
\end{equation*}
Passing the obstacle above or below results in two locally optimal trajectories. Both are found by PBF, depicted in Figure~\ref{fig:NumExp_Obst}, using the dashed curves as initial guesses (with $t_E$ and $u$ computed feasible from $\vec{\chi}$ via integration and differentiation, respectively) on $100$ finite elements of degree 5.

The computed times as in the figure are accurate except to the last digit. The red/black trajectory converges in 52/58 NLP iterations. For comparison, LGR of the same degree and mesh converges in 51/51 iterations. $\bS$ has 73521 nonzeros and bandwidth 25 for both PBF and LGR.
\begin{figure}[tb]
	\centering
	\includegraphics[width=0.6\columnwidth]{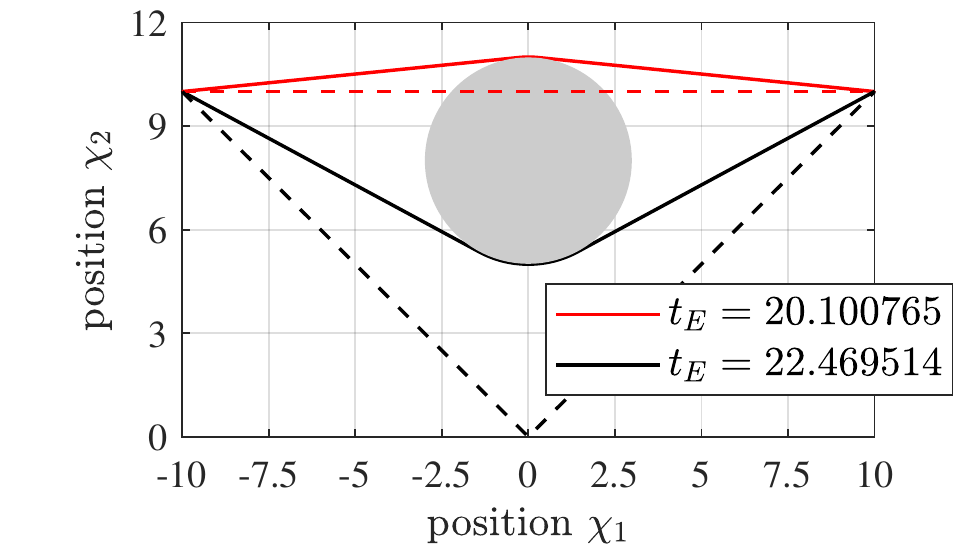}
	\caption{Optimal trajectories from PBF for the Obstacle Avoidance Problem.}
	\label{fig:NumExp_Obst}
\end{figure}

\subsection{Pendulum in Differential-Algebraic Form}
In this example from~\cite[Chap.~55]{betts2015collection}, a control force decelerates a frictionless pendulum to rest. The objective is to minimize the integral of the square of the control:
\begin{equation*}
\begin{aligned}
&\min_{\vec{\chi},\xi,u} & &\int_0^{3} u(t)^2\,\mathrm{d}t,\\
&\text{s.t.} & \vec{\chi}(0)&=[1\ 0]\t,\quad \dot{\vec{\chi}}(0)=\vec{0},\\
&&\vec{\chi}(3)&=[0\ -1]\t,\quad \dot{\vec{\chi}}(3)=\vec{0},\\
&& \ddot{\vec{\chi}}(t)&=[0\ -9.81]\t +2 \cdot \vec{\chi}(t) \cdot \xi(t) + \vec{\chi}^\perp(t) \cdot u(t),
\end{aligned}
\end{equation*}
with an additional DAE constraint introduced below. The ODE for $\vec{\chi}$ is a force balance in the pendulum mass. $u(t)$ is the control force acting in the direction $\vec{\chi}^\perp := [-\chi_2\ \chi_1]\t$.%, $\vec{\chi}$ rotated by $90$ degrees.

The DAE constraint determines the beam force $\xi(t)$ in the pendulum arm in an implicit way, such that the length remains 1 for all time;
\cite[Chap.~55]{betts2015collection} uses
\begin{align}
0&=\|\dot{\vec{\chi}}(t)\|_2^2 -2 \cdot \xi(t) -g \cdot \chi_2(t).\label{eqn:BeamDAE1}
\end{align}
The following alternative constraint achieves the same:
\begin{align}
\tag{\ref{eqn:BeamDAE1}'}
0&=\|\vec{\chi}(t)\|_2^2 - 1\,.\label{eqn:BeamDAE3}
\end{align}
\eqref{eqn:BeamDAE1} is a DAE of index 1, whereas \eqref{eqn:BeamDAE3} is of index~3. %Typically, DAEs of higher index are more difficult to solve numerically \cite[Chap.~VII.2]{HairerII}.

In the following we study the convergence of TR, HS, LGR ($p=5$) and PBF ($p=5$) on meshes of increasing size. Here, the collocation methods are solved with IPOPT in ICLOCS2, whereas PBF is solved with Forsgren-Gill as before.

TR \& HS are likely to converge at a slower rate than PBF \& LGR. However, our focus is primarily on determining whether a given method converges, and only secondarily on rates of convergence. To find out where solvers struggle, we consider three variants of the pendulum problem,
\begin{enumerate}[\text{\ Case\,}A]
	\item where we consider the original problem with \eqref{eqn:BeamDAE1} as given in \cite{betts2015collection}.
	\item where we add the path constraint $\xi(t)\leq 8$.
	\item where we exchange \eqref{eqn:BeamDAE1} with \eqref{eqn:BeamDAE3}.
\end{enumerate}

All methods converge  for case A. Figure~\ref{fig:pendulumcaseAlog} shows that TR converges slowly, while HS, LGR and PBF converge fast. At small magnitudes of $g_{\text{opt}},r_{\text{feas}}$, further decrease of LGR and PBF deteriorates, presumably due to limits in solving the NLP accurately under rounding errors.
\begin{figure}[tb]
	\centering
	\includegraphics[width=0.6\columnwidth]{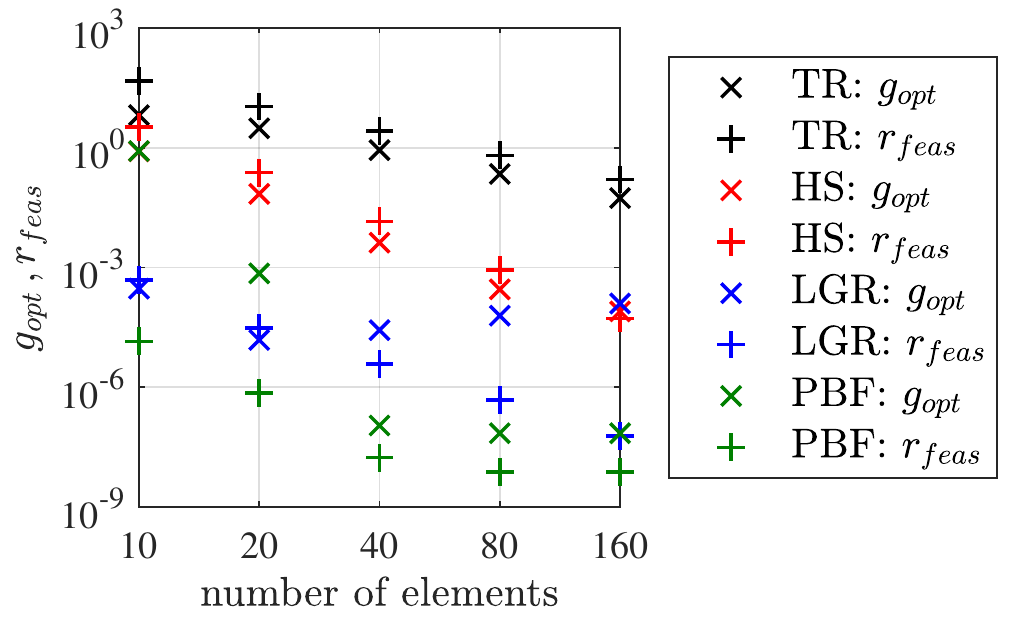}
	\caption{Convergence of optimality gap and feasibility residual for Pendulum example, case A.}
	\label{fig:pendulumcaseAlog}
\end{figure}

Case B is shown in Figure~\ref{fig:pendulumcaseBarcs}. The control force decelerates the pendulum more aggressively before the pendulum mass surpasses the lowest point, such that the beam force obeys the imposed upper bound. Figure~\ref{fig:pendulumcaseBlog}  confirms convergence for all methods. The rate of convergence is  slower compared to case A, as expected, because  the solution of $u$ is locally non-smooth.
%While the collocation methods tend to yield more accurate values for $g_\text{opt}$ than $r_\text{feas}$, the feasibility residual of PBF is orders of magnitude smaller.

\begin{figure}[tb]
	\centering
	\includegraphics[width=0.6\columnwidth]{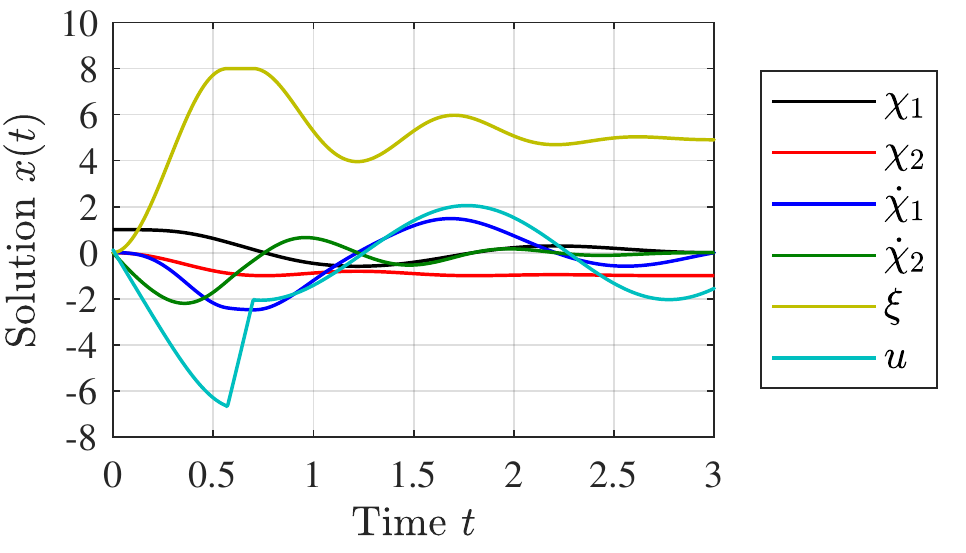}
	\caption{Numerical solution of PBF on $80$ elements for Pendulum example, case B.}
	\label{fig:pendulumcaseBarcs}
\end{figure}	
\begin{figure}[tb]
	\centering
	\includegraphics[width=0.6\columnwidth]{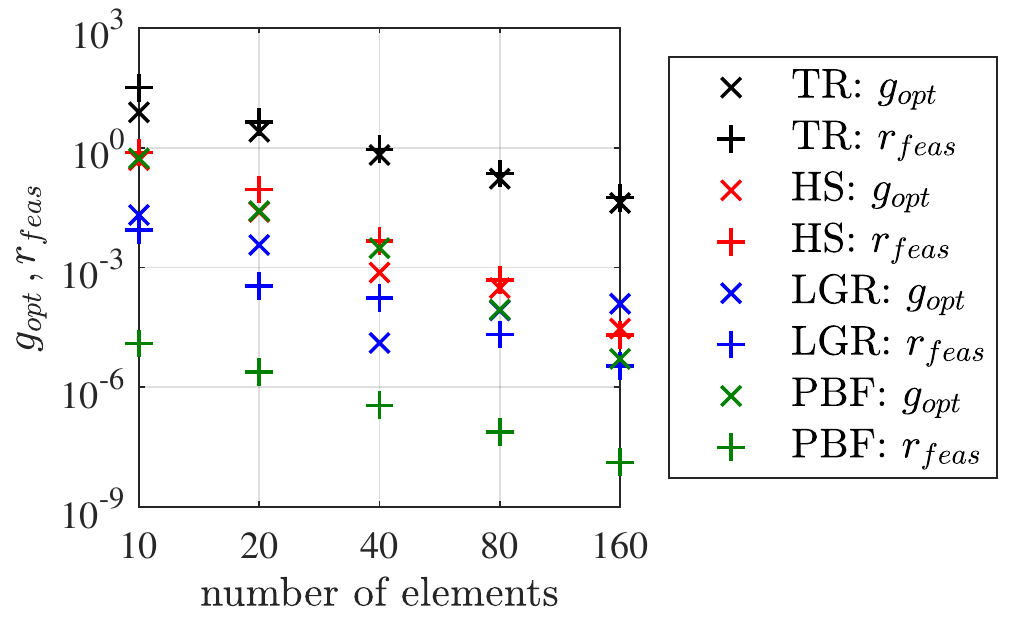}
	\caption{Convergence of optimality gap and feasibility residual for Pendulum experiment, case B.}
	\label{fig:pendulumcaseBlog}
\end{figure}

For case~C, some collocation methods struggle: For HS on all meshes, the restoration phase in IPOPT converged to an infeasible point, indicating infeasibility of \eqref{eqn:NLP} for this scheme \cite[Sec.~3.3]{IPOPT}. For TR, the feasibility residual does not converge, as shown in Figure~\ref{fig:pendulumcaseClog}.
\begin{figure}[tb]
	\centering
	\includegraphics[width=0.6\columnwidth]{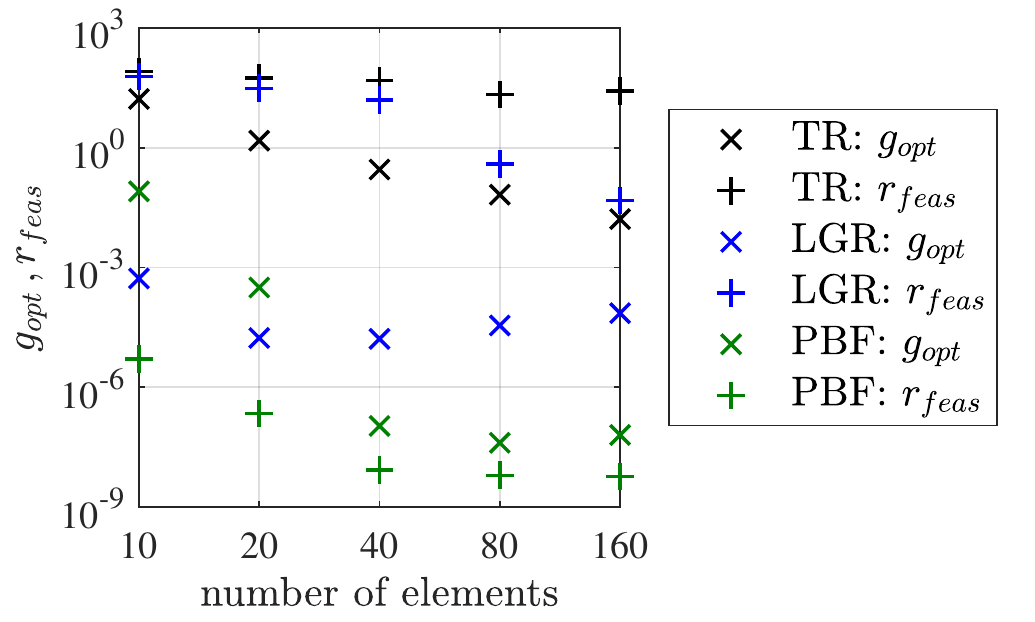}
	\caption{Convergence of optimality gap and feasibility residual for Pendulum example, case C.}
	\label{fig:pendulumcaseClog}
\end{figure}	
Figure~\ref{fig:pendulumcaseCarcs} shows that this is due to ringing in the numerical solution for the beam force.
\begin{figure}[tb]
	\centering
	\includegraphics[width=0.6\columnwidth]{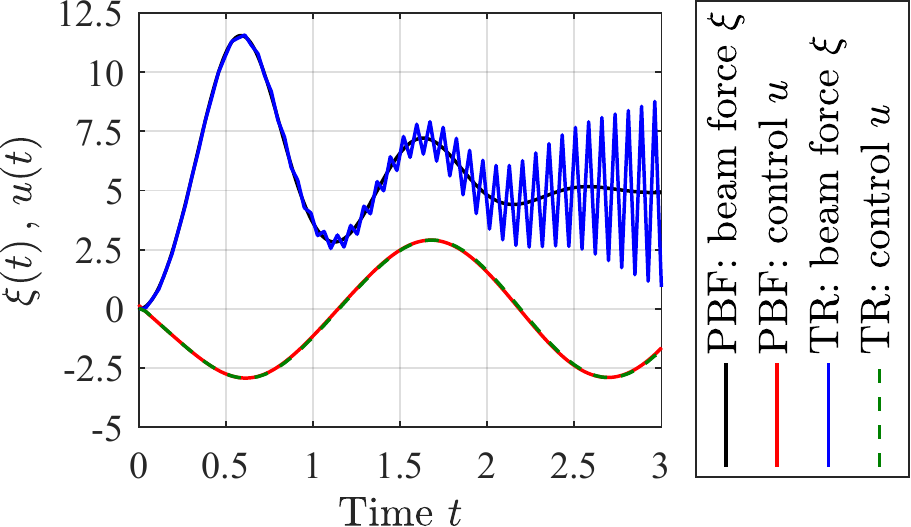}
	\caption{Numerical solutions of PBF and TR on $80$ elements for Pendulum example, case C. The optimal control is identical to case A.}
	\label{fig:pendulumcaseCarcs}
\end{figure}
Regarding LGR, Figure~\ref{fig:pendulumcaseClog} shows that the feasibility residual converges only for relatively fine meshes. In contrast to the collocation methods, PBF converges as fast as for case~A.

Finally, we discuss the computational cost: Using 160 elements, the PBF discretization results in $\bS$ of bandwidth~30, with 44964 nonzeros for cases~\mbox{A,\,C} and 42951 nonzeros for case~B; requiring 66, 51, and 66 NLP iterations for cases A--C. LGR yields the same sparsity pattern for $\bS$ as PBF, solving on average in $30$ IPOPT iterations (with second-order corrections).

\section{Conclusions}
\label{sec:conclusions}
We presented PBF and proved convergence under mild and easily-enforced assumptions. Key to the convergence proof is the formulation of a suitable unconstrained penalty-barrier problem, which is discretized using finite elements and solved with primal-dual penalty-barrier NLP solvers.

Theorem~\ref{thm:order} provides high-order convergence guarantees even if the component $z$ has discontinuities, provided that the trajectory can be approximated accurately in the finite element space; see~\eqref{eqn:InfBound} and the discontinuous elements in Figure~\ref{fig:finiteelementsyz}. It is  a practical matter to employ an adaptive meshing technique for achieving this in an efficient manner.

The practicality of our novel transcription has been illustrated in numerical examples. The scheme converged for challenging problems, which included solutions with singular arcs and discontinuities. These problems caused issues for  three commonly used direct transcription methods based on collocation, namely TR, HS and LGR.

\part{Convergence Analysis of Quadrature Penalty Methods}
\label{part:convproof}

\chapter{Introduction}
This part presents a proof of convergence for quadrature penalty methods.

We suppose that there exists a local\footnote{A global minimizer is a just a special local minimizer} minimizer $(y^\star,u^\star) \in \cX$ of \eqref{eqn:OCP} and we wish to approximate this minimizer numerically.
We consider an approximation $(y^\star_h,u^\star_h) \in \cX_{h,p}$ that is a \emph{suitable}\footnote{In the sense described in Section~\ref{sec:partconv:notation}} local minimizer of \eqref{eqn:POCPh}.
We prove convergence rates of the measures $\delta,\rho,\gamma$ from Section~\ref{sec:ConvMeasures} as $h \rightarrow 0$. The main result is Theorem~\ref{thm:main}. We state some notations and a few prerequisites in advance.

\section{Notation}\label{sec:partconv:notation}

\subsubsection{Feasible Candidates}
We define the space $\cB_{h,p}\subset\cX_{h,p}$ of candidates $(y_h,u_h)$ that satisfy the bound constraints in \eqref{eqn:POCPh}.

\subsubsection{Suitable Minimizer}
NLPs do often have several different local minimizers. We use an NLP minimizer $\bx^\star$ to construct a numerical optimal control solution $y^\star_h,u^\star_h$ to a demanded exact minimizer $y^\star,u^\star$. However, if our NLP solver picked the wrong $\bx^\star$ then $y^\star_h,u^\star_h$ will not converge with respect to $y^\star,u^\star$ but potentially with respect to a different minimizer $\tilde{y}^\star,\tilde{u}^\star$.

Whenever minimizing locally, one needs a local property to identify the exact local minimizer that is considered.
This is why in the following we specify a simple and sufficient characterization of which numerical NLP minimizer we seek. We call such a minimizer a \emph{suitable} local minimizer. For us, this local property is the objective value.
\begin{mdframed}
	For each $h \in \R_{>0}$, define an arbitrary fixed pair $(\hat{y}_h,\hat{u}_h) \in \cB_{h,p} \subset \cX_{h,p}$ that satisfies~\eqref{eqn:infxh}, i.e., such that there is a constant $C_\eta \in \R$ such that:
	\begin{align}
	\big\|(y^\star,u^\star)-(\hat{y}_h,\hat{u}_h)\big\|_{\cX} \leq C_\eta \cdot h^\eta\,. \label{eqn:DefHatyu}
	\end{align}
\end{mdframed}
Then a suitable numerical local minimizer $(y^\star_h,u^\star_h)$ is any local minimizer of \eqref{eqn:POCPh} with smaller objective than $(\hat{y}_h,\hat{u}_h)$, i.e.:
\begin{align}
\begin{split}
&M\big(y^\star_h(0),y^\star_h(T)\big) + \frac{1}{2 \cdot \omega} \cdot \left( Q_{h,q}\left[ \left\|f\big(\dot{y}^\star_h(\cdot),y^\star_h(\cdot),u^\star_h(\cdot),\cdot\big)\right\|_2^2\,\right] + \left\|b\big(\,y^\star_h(0),y^\star_h(T)\,\big)\right\|_2^2 \right)\\
\leq &M\big(\hyh(0),\hyh(T)\big) + \frac{1}{2 \cdot \omega} \cdot \left( Q_{h,q}\left[ \left\|f\big(\dhyh(\cdot),\hyh(\cdot),\huh(\cdot),\cdot\big)\right\|_2^2\,\right] + \left\|b\big(\,\hyh(0),\hyh(T)\,\big)\right\|_2^2 \right)\label{eqn:suitableMinimizer}
\end{split}
\end{align}
Problem~\eqref{eqn:POCPh} certainly has a solution satisfying this bound, because $(\hat{y}_h,\hat{u}_h)$ is feasible for \eqref{eqn:POCPh} and in the extreme case $(\hat{y}_h,\hat{u}_h)$ is coincidentally the global minimizer of \eqref{eqn:POCPh}.

\section{Prerequisites}
We presuppose three properties. We do not pose them as assumptions because two things can be easily agreed upon for each of them: First, the assumptions are inherently necessary in the sense that any numerical method which is based on finite elements must use these assumptions. Second, the assumptions pose an insignificant restriction on the generality of problem instances in the sense that they either hold naturally or can be forced practically with ease.

\subsection{Approximability of Exact Minimizer}\label{sec:approximability}

We shall require that $(y^\star,u^\star)$ satisfies the following approximability condition:
\begin{mdframedwithfoot}
	There exist $h_0 \in \R_{>0}$ finitely above zero, $\eta \in \R_{>0}$ finitely above zero, and $C_\eta \in \R_{>0}$ finite, such that
	\begin{align}
	\forall\,h\leq h_0\ \exists (y_h,u_h) \in \cB_{h,p}\,:\ \big\|(y^\star,u^\star)-({y}_h,{u}_h)\big\|_{\cX} \leq C_\eta \cdot h^\eta\,. 	\label{eqn:infxh}
	\end{align}
\end{mdframedwithfoot}
\noindent
For short, this can be stated as
\begin{align*}
\operatornamewithlimits{inf}_{({y}_h,{u}_h) \in \cB_{h,p}} \big\|(y^\star,u^\star)-({y}_h,{u}_h)\big\|_{\cX} \in \cO(h^\eta)\,.
\end{align*}

Approximability means nothing more than that $y^\star,u^\star$ can be approximated to some order $\eta$. Sufficient criteria on $(y^\star,u^\star)$ to satisfy the condition are widely known: For example, \cite[Thm.~8.8]{Brezis} bounds the $L^\infty$-norm by the $H^1$-norm. In turn, \cite[Prop~1.12]{GuermondInterp} provides orders for the interpolation error in the $H^1$-norm. Very high orders $\eta \gg 1$ can be attained for very smooth functions $y^\star,u^\star$. The following corollary helps for the contrary case when $y^\star,u^\star$ are not very smooth.

\begin{cor}\label{cor:interpolationerror}
	Let $\dot{y}^\star$ and $u^\star$ be $\eta$-H\"older continuous on each mesh-interval for some $0\leq \eta\leq 1$. I.e. there exists some finite constant $L \in \R_{>0}$ such that for $i=1,\dots,N$:
	\begin{align}
	\|\dot{y}^\star(\hat{t})-\dot{y}^\star(\check{t})\|_2 \leq& L \cdot |t-\tilde{t}|^\eta &&\forall\ \check{t},\hat{t} \in I_i\,, \label{eqn:cor:interpolationerror:y}\\
	\|u^\star(\hat{t})-u^\star(\check{t})\|_2 \leq& L \cdot |t-\tilde{t}|^\eta &&\forall\ \check{t},\hat{t} \in I_i\,. \label{eqn:cor:interpolationerror:u}
	\end{align}
	Then \eqref{eqn:infxh} is satisfied.
\end{cor}
\noindent
\underline{Proof}: In Section~\ref{sec:proof:cor:interpolationerror}.\qed

\begin{remark}
	Some optimal control problems feature solutions with many edges and discontinuities. In this case, by virtue of Corollary~\ref{cor:interpolationerror}, convergence with a rate $\eta>0$ can still be asserted if the mesh is adapted appropriately.
	
	But even when the mesh is not adaptive, the rate is likely $\eta\geq 0.5$. For instance, the convergence of a piecewise linear continuous function to a piecewise constant discontinuous function in the $L^2$-norm is in $\cO(h^{0.5})$.
\end{remark}

\subsubsection{Insignificance}
The prerequisite of approximability is an insignificant restriction because the requirement of Corollary~\ref{cor:interpolationerror} is satisfied for all but the most obscure functions $y^\star,u^\star$.

\subsubsection{Necessity}
Like any direct transcription method, we approximate $y^\star,u^\star$ with piecewise polynomials $y_h^\star,u^\star_h$. This approximation is entirely in vain when $y^\star,u^\star$ are not by any means approximatable via piecewise polynomials. Thus, approximability is inherently necessary.

\subsection{Shape of Bound Constraints}
To simplify technicalities, we shall require $(\yL,\uL) \in \cX_{h,p}$ and $(\yR,\uR) \in \cX_{h,p}$. 

\subsubsection{Insignificance}
The prerequisite is an insignificant restriction because it may be forced by simply fitting $\yL,\yR,\uL,\uR$ with piecewise polynomials. This can also be done with sum-of-squares techniques \cite{sos1,sos2} in order to satisfy bound constraints truly everywhere.

\subsubsection{Necessity}
Without exact knowledge of their entire shape it is impossible to make assertions on the violation of $\yL,\yR,\uL,\uR$ at any point $t$ at which these functions have not been evaluated. Hence, it would be impossible to make assertions on the convergence of $\gamma$ from~\eqref{eqn:meas:gamma}.

\subsection{Consistency Order of Quadrature}\label{sec:prelims:quad}
We shall require that the quadrature rule $Q_{h,q}$ is of some consistency order $\ell>0$ according to our definition from~\eqref{eqn:quadcond} in Section~\ref{sec:QuadOrder}.

\subsubsection{Insignificance}
The property can be forced easily in practice by selecting a sufficiently large number $q$ of quadrature points.

\subsubsection{Necessity}
Consistency of quadrature is inherently necessary for convergence whenever integrals are discretized with quadrature. A typical and famous example for this is in the second Strang Lemma \cite{SecondStrangLemma}, where
$$\operatornamewithlimits{sup}_{w_h \in \cW_h} |F_h(w_h)-F(w_h)| \in \cO(h^\ell)$$
is necessary to guarantee convergence of a Galerkin method. In Strang's bound: $\cW_h$ is a piecewise polynomials space like $\cX_{h,p}$; $F$ is a functional like $r^2$; and $F_h$ is the quadrature approximation of $F$ like $Q_{h,q}$ is the quadrature approximation of $r^2$. Strang's requirement on the quadrature's consistency is virtually identical to ours.

\section{Main Result}
In Section~\ref{sec:ConvMeasures} we introduced the three measures $\delta,\rho,\gamma$. These measures quantify, in this order, the convergence of optimality for the objective, the convergence of feasibility for the equality constraints, and the convergence of feasibility for the inequality constraints. The following main theorem asserts that these measures indeed converge to zero at a certain order.

\begin{theorem}\label{thm:main}
	If the below assumptions \AssumpI, \AssumpII and \AssumpIII from Section~\ref{sec:assumptions} hold then the accuracy measures $\delta,\rho,\gamma$ from Section~\ref{sec:ConvMeasures} between the exact minimizer $(y^\star,u^\star)$ and the numerical minimizer $(y^\star_h,u^\star_h)$ of problem~\eqref{eqn:POCPh} satisfy the following bounds:fdsa
	\begin{align*}
	\delta 	& \in 	\cO\big(\,h^{\eta \cdot \lambda} + 
	h^{\min\lbrace 2\cdot\eta\cdot\lambda,\ell\rbrace}/\omega\ \big)\,,\\ 
	\rho  	& \in 	\cO\big(\,\sqrt{\omega}+h^{\min\lbrace \eta\cdot\lambda,\ell/2\rbrace}\,\big)\,,\\
	\gamma 	& \in 	\cO\big(\,\sqrt{p} \cdot (p/m)^2\ \big)\,.
	\end{align*}
	The parameter $\lambda$ is a H\"older exponent that is introduced in assumption \AssumpIII; $\eta$ reflects the smoothness of the exact minimizer $(y^\star,u^\star)$; and $\ell$ is the quadrature order of $Q_{h,q}$.
\end{theorem}
\noindent
\underline{Proof:} in Section~\ref{sec:proof:thm:main}. \qed

\noindent
In the experiments we saw that all three measures converge rapidly.

Convergence of $\delta,\rho$ can be forced by selecting $\omega$ sufficiently small and then converging~\mbox{$h \rightarrow 0$}. The measure $\gamma$ can be decreased by selecting $m$ sufficiently large, where $m$ is the Chebyshev-Gauss-Lobatto degree in \eqref{eqn:POCPh}, where $\Tcl{i,m}$ was defined in~\eqref{eqn:cheblob}.

\paragraph{Optimality of the Rates of Convergence}
As is known from piecewise polynomial interpolation, a sufficiently smooth function $f$ may be at best interpolated with a piecewise polynomial function $g$ of degree $p \in \N$ on a mesh of size $h > 0$ with a convergence in $\cO(h^p)$. A faster rate cannot be attained. Thus, the rate $\cO(h^p)$ is called \emph{optimal}.

Likewise, when an optimal control solution $y^\star,u^\star$ and the problem-defining functions are sufficiently smooth then the convergence measures $\delta,\rho$ converge at an optimal rate. In particular, let $\eta=p$ (i.e., the solution is sufficiently piecewise smooth over the mesh such that the finite elements best-approximation converges in $\cO(h^p)$) and $\lambda=1$ (i.e., the problem-defining functions are at least local Lipschitz continuous). Select a quadrature with $\ell\geq2 \cdot p$ such as, e.g., Gauss-Legendre quadrature of sufficient degree. Then $\delta,\rho \in \cO(h^p)$.

For $\gamma$, our analysis has focused solely on the most difficult case, i.e., when bound constraint violations may occur due to discontinuities in the exact solution. Due to the discontinuities and hence lack of any smoothness, an order-of-approximation result is unavailable for the rate of convergence of $\gamma$ within the scope of our analysis. However, we demonstrate a practical approach below Theorem~\ref{thm:convgamma} for the adaptive placement of the number of sampling points $m$ in order to obtain high accuracy for satisfaction of inequality constraints.

\chapter{All Assumptions}\label{sec:assumptions}
In order to prove Theorem~\ref{thm:main}, we make use of three assumptions on~\eqref{eqn:OCP}. These three assumptions are labeled \AssumpI, \AssumpII, \AssumpIII and are stated below.

\section*{(A.1) Boundedness of Bound Constraints}\label{assumption:A1}
We assume that there is some finite constant $\Cbox \in \R_{>0}$ such that
\begin{align*}
\left\|\begin{bmatrix}
\yR(t)-\yL(t)\\
\uR(t)-\uL(t)
\end{bmatrix}\right\|_{\infty} \leq \Cbox\qquad \forall t \in [0,T]\,.
\end{align*}

\section*{(A.2) Lower Boundedness of Objective}\label{assumption:A2}
We assume that there is some finite constant $\Cobj \in \R$ such that
\begin{align*}
M(\tmpyO,\tmpyT)\geq \Cobj \quad \forall\ \tmpyO,\tmpyT \in \R^{n_y}\text{ that satisfy }\yL(0)\leq \tmpyO \leq \yR(0)\ \land\ \yL(T)\leq \tmpyT \leq \yR(T)\,.
\end{align*}
To enhance readability, the symbols $\tmpyO,\tmpyT$ are templates for arbitrary vectors in $\R^{n_y}$.

\section*{(A.3) Local H\"older Continuity of $M,b,f_1,f_2$}\label{assumption:A3}
Recall from~\eqref{eqn:OCP} the problem-defining objective function $M$, the boundary conditions function $b$, and the differential and algrabic equality constraints functions $f_1$ and $f_2$.

We assume that $M,b,f_1,f_2$ are H\"older continuous locally around $(y^\star,u^\star)$. In particular, we assume that there exist constants $\lambda,\Clam,\epsilon$, where $0<\lambda\leq 1$ finitely above zero, $\Clam \in \R_{>0}$ finite, and $\epsilon \in \R_{>0}$ finitely above zero, such that:
\begin{subequations}
	\begin{align}
	\Big|M\big(y^\star(0),y^\star(T)\big)-M(\tmpyO,\tmpyT)\Big| &\leq \Clam \cdot \left\| \begin{bmatrix}
	y^\star(0)-\tmpyO\\
	y^\star(T)-\tmpyT
	\end{bmatrix} \right\|_2^\lambda \tageq\label{eqn:A3:M:bound}\\
	&\forall \tmpyO,\tmpyT \in \R^{n_y} \text{ that satisfy } \left\| \begin{bmatrix}
	y^\star(0)-\tmpyO\\
	y^\star(T)-\tmpyT
	\end{bmatrix} \right\|_2 \leq \epsilon\,, \tageq\label{eqn:A3:M:cond}
	\end{align}
\end{subequations}
and
\begin{subequations}
	\begin{align}
	\Big\|b\big(y^\star(0),y^\star(T)\big)-b(\tmpyO,\tmpyT)\Big\|_2 &\leq \Clam \cdot \left\| \begin{bmatrix}
	y^\star(0)-\tmpyO\\
	y^\star(T)-\tmpyT
	\end{bmatrix} \right\|_2^\lambda 	\tageq\label{eqn:A3:b:bound}\\
	&\forall \tmpyO,\tmpyT \in \R^{n_y} \text{ that satisfy } \left\| \begin{bmatrix}
	y^\star(0)-\tmpyO\\
	y^\star(T)-\tmpyT
	\end{bmatrix} \right\|_2 \leq \epsilon\,,\tageq\label{eqn:A3:b:cond}
	\end{align}
\end{subequations}
and that at each $t \in [0,T]$
\begin{subequations}
	\begin{align}
	\left\|\begin{bmatrix}
	f_1\big(y^\star(t),u^\star(t),t\big)-f_1(\tmpyt,\tmput,t)\\
	f_2\big(y^\star(t),u^\star(t),t\big)-f_2(\tmpyt,\tmput,t)
	\end{bmatrix}\right\|_2 &\leq \Clam \cdot \left\| \begin{bmatrix}
	y^\star(t)-\tmpyt\\
	u^\star(t)-\tmput
	\end{bmatrix} \right\|_2^\lambda\tageq\label{eqn:A3:f:bound}\\
	\forall &\tmpyt \in \R^{n_y},\tmput \in \R^{n_u} \text{ that satisfy }\left\|\begin{bmatrix}
	y^\star(t)-\tmpyt\\
	u^\star(t)-\tmput
	\end{bmatrix} \right\|_2 \leq \epsilon\,.\tageq\label{eqn:A3:f:cond}
	\end{align}
\end{subequations}
The symbols $\tmpyO,\tmpyT,\tmpyt$ are templates for vectors in $\R^{n_y}$ and $\tmput$ is a template for vectors in $\R^{n_u}$. These symbols are used only in Section~\ref{sec:assumptions} and Section~\ref{sec:proof:fundamental}. Local H\"older continuity is a milder assumption than local Lipschitz continuity.

\chapter{Convergence Analysis}\label{sec:convAnalysis}
The present section provides a convergence analysis, comprising of three theorems. The Main Theorem~\ref{thm:main} just summarizes these three theorems.

The separate Section~\ref{sec:proofs} provides illustrated detailed proofs of all theorems and intermediate results. One subsection is dedicated to each proof. Subsection~\ref{sec:proof:thm:main} proves Theorem~\ref{thm:main}.

Since we prove convergence orders, our proofs will use several exponents of $\ell,\eta,\lambda$ and constants $C$. The latter have individual footnotes. The exact formulas of the constants do not matter for the conceptual ideas. We opted to include the formulas anyways, mainly for rigor and for avoidance of excessive big-O notation.

Reminder: The measures $\rho,\gamma,\delta$ measure the convergence of equality feasibility, inequality feasibility, and optimality gap of the numerical optimal control solution. For convergence, they are supposed to converge to zero as $h$ decreases.

\section{Convergence of $\gamma$}
Recall the set of Chebyshev-Gauss-Lobatto points $\Tcl{i,m}$ of sampling degree $m \in \N$ on interval $I_i$ from~\eqref{eqn:cheblob}. Further, recall from Section~\ref{def:spaceXhp} the space $\cP_p(I)$ of functions that are polynomials of degree $\leq p$ on the interval $I$, also depicted in Figure~\ref{fig:spacepiecewisepolynomials}~(a).

\begin{theorem}\label{thm:convgamma}
	Let three scalar functions $\phiL,\phi,\phiR \in \cP_p(I_i)$ for some $p \in \N$, and let $m \in \N \cdot p$. If $\phiL,\phi,\phiR$ satisfy
	\begin{subequations}
		\begin{align}
		&\phiL(t)\leq \phi(t) \leq \phiR(t)\quad \forall t \in \Tcl{i,m}\,, \label{eqn:convgamma:1}\\
		&\sup_{t \in I_i} \big\lbrace\,\phiR(t)-\phiL(t)\,\big\rbrace \leq \Cbox\,, 					\label{eqn:convgamma:2}
		\end{align}
	\end{subequations}
	then the following bounds hold $\forall t \in I_i$:
	\begin{subequations}
		\begin{align}
		\phi(t) & \geq \phiL(t) - \frac{\pi^2 \cdot \Cbox}{8} \cdot \sqrt{p} \cdot \left(\frac{p}{m}\right)^2\,, \label{eqn:boxthmpropL}\\
		\phi(t) & \leq \phiR(t) + \frac{\pi^2 \cdot \Cbox}{8} \cdot \sqrt{p} \cdot \left(\frac{p}{m}\right)^2\,.\label{eqn:boxthmpropR}
		\end{align}
	\end{subequations}
\end{theorem}
\noindent
\underline{Proof:} in Section~\ref{sec:thm:convgamma}. \qed

The theorem provides a template for each component $y_{h,[\upsilon]}$ and $u_{h,[\upsilon]}$ on each interval $I_i$, explained via the following example.

\subsubsection{Example for interval-wise applicability of Theorem~\ref{thm:convgamma}}
Consider problem~\eqref{eqn:box_counter}: We saw in Figure~\ref{fig:counter2col} that $u_{h}$ of degree $p=4,h=2/3$ suffers from Gibbs' phenomenon on $I_2$. To avoid the overshoot, we may choose $m=16 \cdot p = 64$ sampling points for the bound constraints of $u_h$ on the particular interval $I_2$ where the overshoot occurs. Thereby, Theorem~\ref{thm:convgamma} asserts that the overshoot will be no more than
\begin{align*}
\frac{\pi^2 \cdot \Cbox}{8} \cdot \sqrt{p} \cdot \left(\frac{p}{m}\right)^2 = \frac{\pi^2 \cdot 2}{8} \cdot 2 \cdot \left(\frac{4}{64}\right)^2 \approx 0.019276\,.
\end{align*}
In general, the term $\left(\frac{p}{m}\right)^2$ converges to zero as $m$ is increased.

\section{Convergence of $\delta$ and $\rho$}
The next theorem proves convergence of $\delta,\rho$ subject to the condition that the numerical minimizer achieves a certain optimality gap to the exact minimizer, denoted with $\chi$. The subsequent theorem shows that this gap indeed converges.
\begin{thm}\label{thm:convdelta}
	Let
	\begin{align}
	M(y^\star_h,u^\star_h) + \frac{r^2(y^\star_h,u^\star_h)}{2 \cdot \omega} \leq M(y^\star,u^\star)+\chi \label{eqn:def:chi}
	\end{align}
	for some $\chi \in \R_{\geq 0}$. If \AssumpII holds then:
	\begin{align*}
	\delta 	& \leq \chi\,,\\
	\rho  		& \leq \sqrt{2} \cdot \sqrt{M\big(y^\star(0),y^\star(T)\big)-\Cobj+\chi} \cdot \sqrt{\omega}\,.
	\end{align*}
\end{thm}
\noindent
\underline{Proof:} in Section~\ref{sec:thm:convdelta}. \qed

As the next theorem shows, for $h\rightarrow 0$ the measure $\chi$ converges. For the theorem, recall $C_\eta,\eta$ from \eqref{eqn:infxh}, the piecewise polynomials quadrature order $\ell$ from Section~\ref{sec:prelims:quad}, and the H\"older exponent $\lambda$ from \AssumpIII.
\begin{thm}\label{thm:eps}
	\textcolor{orange}{\dunderline[-1pt]{2pt}{\textcolor{black}{If \AssumpIII holds and $h$ is sufficiently small}}} then there exists some finite \mbox{constant $C_\chi \in \R_{>0}$} such that the following bound holds:
	\begin{align}
	M\big(y^\star_h(0),y^\star_h(T)\big) + \frac{r^2(y^\star_h,u^\star_h)}{2 \cdot \omega} \leq M\big(y^\star(0),y^\star(T)\big) + \underbrace{ C_\chi \cdot \left( h^{\eta\cdot\lambda} + \frac{h^{2\cdot\eta\cdot\lambda} + h^\ell}{\omega}\right)}_{\equiv \chi}\label{eqn:thm:eps:prop}
	\end{align}
\end{thm}
\noindent
\underline{Proof:} in Section~\ref{sec:thm:eps}. \qed\\
Some remarks are in order:
\begin{itemize}
	\item $h$ is sufficiently small when it satisfies \eqref{eqn:hSuffSmall}.
	\item The if-clause is \textcolor{orange}{\dunderline[-1pt]{2pt}{\textcolor{black}{underlined in orange}}} to indicate that said if-clause merely propagates itself from an interiorly used \emph{fundamental lemma}. This fundamental lemma is introduced below in Section~\ref{sec:fundamentalLemma}.
\end{itemize}

\section{Fundamental Lemma}\label{sec:fundamentalLemma}
In preparation for the next section, we introduce a fundamental lemma, that is critical for the proof of many intermediate results. We now explain the purpose of the lemma. Afterwards, we state the lemma.

Some of our results make use of assumption \AssumpIII. Usually, the bounds from \AssumpIII cannot be used directly because they hold only locally. For example, \eqref{eqn:A3:M:bound} gives a bound that only holds when the locality condition \eqref{eqn:A3:M:cond} is satisfied.

Thus, before showing any further results, we will prove up-front that the bounds from \AssumpIII can be used directly for $\hyh,\huh$ from Section~\ref{sec:partconv:notation}.

\begin{lem}[Fundamental Lemma]\label{lem:fundamental}
	Recall the H\"older exponent $\lambda$ from \AssumpIII and the approximation order $\eta$ from~Section~\ref{sec:approximability}. \textcolor{orange}{\dunderline[-1pt]{2pt}{\textcolor{black}{If \AssumpIII holds and $h$ is sufficiently small}}} then there exist finite constants $C_M,C_b,C_f \in \R_{>0}$ such that:
	\begin{subequations}
		\begin{align}
		\pig|M\big(y^\star(0),y^\star(T)\big)-M\big(\hyh(0),\hyh(T)\big)\pig| &\leq C_M \cdot h^{\lambda \cdot \eta} \tageq\label{eqn:A3:M:bound:h}\\
		\pig\|b\big(\hyh(0),\hyh(T)\big)\pig\|_2^2 &\leq C_b \cdot h^{2 \cdot \lambda \cdot \eta} 	\tageq\label{eqn:A3:b:bound:h}\\
		\left\|\begin{bmatrix}
		f_1\big(y^\star(t),u^\star(t),t\big)-f_1\big(\hyh(t),\huh(t),t\big)\\
		f_2\big(y^\star(t),u^\star(t),t\big)-f_2\big(\hyh(t),\huh(t),t\big)
		\end{bmatrix}\right\|_2^2 &\leq C_f \cdot h^{2 \cdot \lambda \cdot \eta} \tageq\label{eqn:A3:f:bound:h}
		\end{align}
	\end{subequations}
\end{lem}
\noindent
\underline{Proof:} in Section~\ref{sec:proof:fundamental}.\qed

Theorem~\ref{thm:eps} and several other results below require that \AssumpIII holds and $h$ be sufficiently small. This is because they make use of Lemma~\ref{lem:fundamental}. To indicate that these requirements originate from Lemma~\ref{lem:fundamental}, we underlined the requirement in orange each time.

\chapter{Proofs of all Theorems}\label{sec:proofs}
Theorem~\ref{thm:main} asserts convergence of several measures and a rate-of-convergence result for each of these measures. By itself, this necessitates a complex proof. A complex proof should not incentivize one to opt for proving a weaker result instead or to forego any helpful details.

Given the result and the volume of its proof, we decided to encapsulate as many intermediate results as possible. This shortens each individual proof but leads to a multitude of small lemmas. Figure~\ref{fig:proofstructure} gives a dependency graph of all proofs and assumptions. As shown in the figure, the majority of the proofs are on an intermediate result that is used for the proof of Theorem~\ref{thm:eps}.

\begin{figure}
	\centering
	\includegraphics[width=0.9\linewidth]{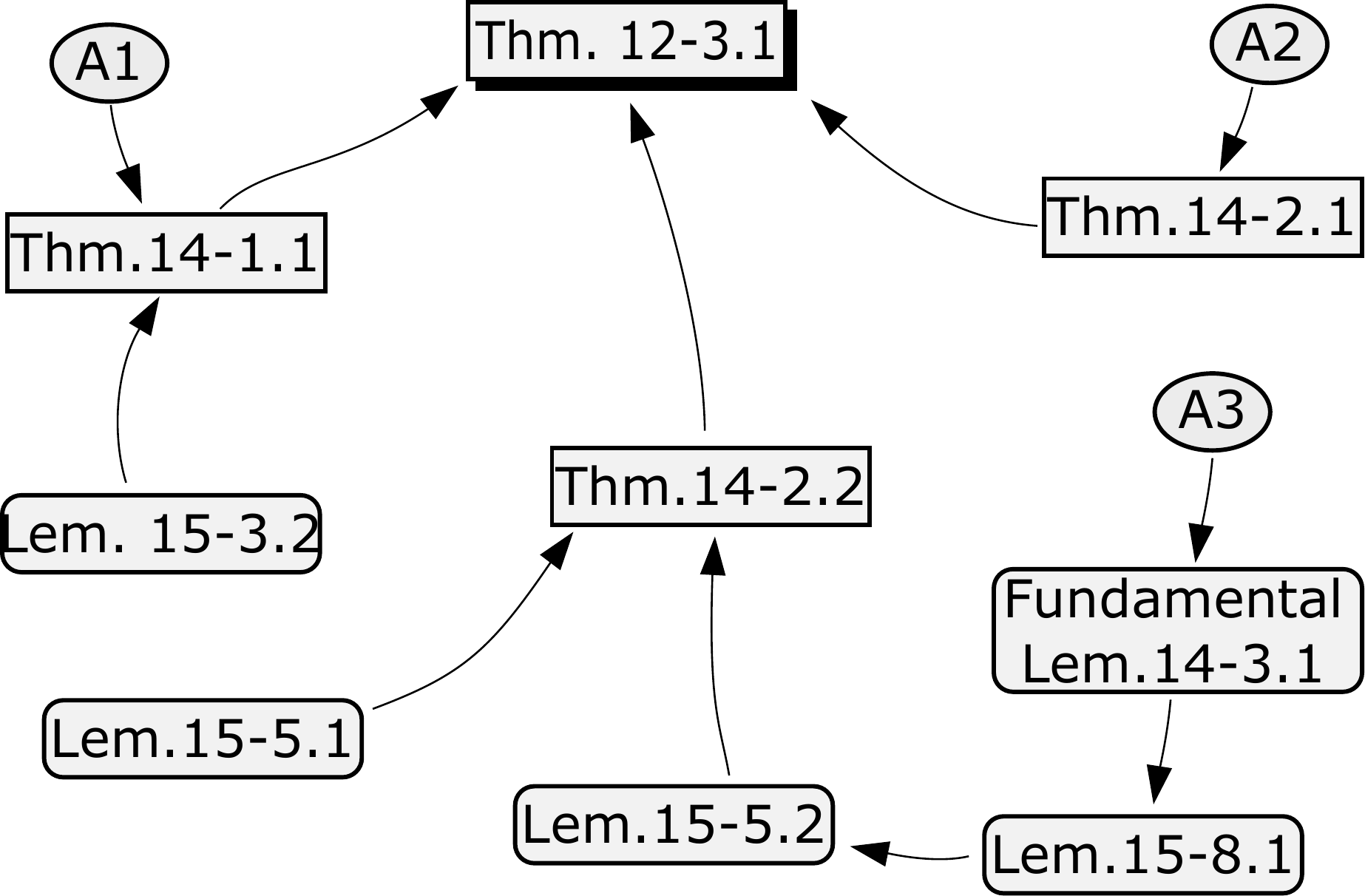}
	\caption{Dependency Graph: The diagram shows all theorems and lemmas that are proven in relation to Theorem~\ref{thm:main} together with the three assumptions and in which proofs they are used.}
	\label{fig:proofstructure}
\end{figure}

\section{Proof of the Fundamental Lemma~\ref{lem:fundamental}}\label{sec:proof:fundamental}

\subsubsection{Structure of the Proof}
Recall the parameter $\epsilon$ from \AssumpIII. The proof depends on the following two norms:
\begin{subequations}
	\begin{align}
	\left\|\begin{bmatrix}
	y^\star(0)-\hyh(0)\\
	y^\star(T)-\hyh(T)
	\end{bmatrix}\right\|_2\,, \label{eqn:norm_y0T}\\
	\left\|\begin{bmatrix}
	y^\star(t)-\hyh(t)\\
	u^\star(t)-\huh(t)
	\end{bmatrix}\right\|_2 \,.\label{eqn:norm_yu}
	\end{align}
\end{subequations}
If $h$ is sufficiently small then both norms are smaller than $\epsilon$. A sufficiently small value for $h$ is developed in Section~\ref{sec:proof:fundamental:1}. Afterwards, Sections~\ref{sec:proof:fundamental:2} and 
\ref{sec:proof:fundamental:3} prove bounds for each of the above norms. Finally, Sections~\ref{sec:proof:fundamental:4},~\ref{sec:proof:fundamental:5},~\ref{sec:proof:fundamental:6} in turn use these bounds to show the propositions \eqref{eqn:A3:M:bound:h},~\eqref{eqn:A3:b:bound:h},~\eqref{eqn:A3:f:bound:h}, respectively.

\subsection{A Sufficiently Small Value for $h$}\label{sec:proof:fundamental:1}
By requirement, $h$ is sufficiently small. Thus, we may consider
\begin{align}
h \leq \left(\frac{\epsilon}{C_\eta \cdot \sqrt{2 \cdot n_y+n_u}}\right)^{\frac{1}{\eta}} \label{eqn:hSuffSmall}
\end{align}
for the purpose that
\begin{align}
\sqrt{2 \cdot n_y + n_u} \cdot C_\eta \cdot h^\eta \leq \epsilon\,. \label{eqn:h_bound}
\end{align}

\subsection{Bound of Norm~\eqref{eqn:norm_y0T}}\label{sec:proof:fundamental:2}
We apply the following estimates:
\medmuskip=0.1mu
\thickmuskip=0.1mu
\begin{align*}
\left\|\begin{bmatrix}
y^\star(0)-\hyh(0)\\
y^\star(T)-\hyh(T)
\end{bmatrix}\right\|_2
\leq& \sqrt{n_y + n_y} \cdot \left\|\begin{bmatrix}
y^\star(0)-\hyh(0)\\
y^\star(T)-\hyh(T)
\end{bmatrix}\right\|_\infty && \text{$\|\cdot\|_2 \leq \sqrt{n} \,\|\cdot\|_\infty$ in $\R^n$}\\
= &\sqrt{n_y + n_y} \cdot \operatorname{max}\left\lbrace\,\|y^\star(0)-\hyh(0)\|_\infty\,,\,\|y^\star(T)-\hyh(T)\|_\infty\,\right\rbrace &&\text{split $\infty$-norm}\\
= &\sqrt{n_y + n_y} \cdot \operatornamewithlimits{max}_{t \in \lbrace 0,T\rbrace}\left\lbrace\,\|y^\star(t)-\hyh(t)\|_\infty\,\right\rbrace &&\text{rewrite max}\\
\leq &\sqrt{n_y + n_y} \cdot \operatornamewithlimits{ess\,sup}_{t \in [0,T]}\left\lbrace\,\|y^\star(t)-\hyh(t)\|_\infty\,\right\rbrace&&\text{use $\lbrace 0,T \rbrace \subset [0,T]$}\\
\leq &\sqrt{n_y + n_y} \cdot \operatornamewithlimits{ess\,sup}_{t \in [0,T]}\left\lbrace\,\left\|\begin{bmatrix}
y^\star(t)-\hyh(t)\\
u^\star(t)-\huh(t)
\end{bmatrix}\right\|_\infty\,\right\rbrace&&\text{upper bound}\\
\leq& \sqrt{n_y + n_y} \cdot \|(y^\star,u^\star)-(\hyh,\huh)\|_{L^\infty}\,.&& \text{use \eqref{eqn:normX:inf}}
\end{align*}
This can be simplified and further bounded:
\begin{subequations}
	\begin{align*}
	&&\left\|\begin{bmatrix}
	y^\star(0)-\hyh(0)\\
	y^\star(T)-\hyh(T)
	\end{bmatrix}\right\|_2 &\leq \sqrt{n_y + n_y} \cdot \|(y^\star,u^\star)-(\hyh,\huh)\|_{L^\infty} && \text{summarized from above}\\
	\Rightarrow &&\left\|\begin{bmatrix}
	y^\star(0)-\hyh(0)\\
	y^\star(T)-\hyh(T)
	\end{bmatrix}\right\|_2 &\leq \sqrt{n_y + n_y} \cdot C_\eta \cdot h^\eta &&  \text{bound with \eqref{eqn:DefHatyu}} \tageq\label{eqn:bound:hyh0T}\\
	\Rightarrow &&\left\|\begin{bmatrix}
	y^\star(0)-\hyh(0)\\
	y^\star(T)-\hyh(T)
	\end{bmatrix}\right\|_2 &\leq \epsilon\,.&&  \text{bound with \eqref{eqn:h_bound}} \tageq\label{eqn:bound:hyh0Teps}
	\end{align*}
\end{subequations}

\subsection{Bound of Norm~\eqref{eqn:norm_yu}}\label{sec:proof:fundamental:3}
We apply the following estimates:
\begin{align*}
\left\|\begin{bmatrix}
y^\star(t)-\hyh(t)\\
u^\star(t)-\huh(t)
\end{bmatrix}\right\|_2 
\leq &\sqrt{n_y+n_u} \cdot \left\|\begin{bmatrix}
y^\star(t)-\hyh(t)\\
u^\star(t)-\huh(t)
\end{bmatrix}\right\|_\infty 	&&\text{use $\|\cdot\|_2 \leq \sqrt{n} \cdot \|\cdot\|_\infty$ in $\R^n$}\\
\leq &\sqrt{n_y+n_u} \cdot \operatornamewithlimits{ess\,sup}_{t \in [0,T]} \left\lbrace\,\left\|\begin{bmatrix}
y^\star(t)-\hyh(t)\\
u^\star(t)-\huh(t)
\end{bmatrix}\right\|_\infty\,\right\rbrace 	&&\text{bound with worst-case for $t$}\\
= &\sqrt{n_y+n_u} \cdot \|(y^\star,u^\star)-(\hyh,\huh)\|_{L^\infty} 	&&\text{use \eqref{eqn:normX:inf}}
\end{align*}
This can be simplified and further bounded:
\begin{subequations}
	\begin{align*}
	&&\left\|\begin{bmatrix}
	y^\star(t)-\hyh(t)\\
	u^\star(t)-\huh(t)
	\end{bmatrix}\right\|_2 &\leq \sqrt{n_y+n_u} \cdot \|(y^\star,u^\star)-(\hyh,\huh)\|_{L^\infty} &&\text{summarized from above}\\
	\Rightarrow&&\left\|\begin{bmatrix}
	y^\star(t)-\hyh(t)\\
	u^\star(t)-\huh(t)
	\end{bmatrix}\right\|_2	& \leq \sqrt{n_y+n_u} \cdot C_\eta \cdot h^\eta &&\text{bound with \eqref{eqn:DefHatyu}} \tageq\label{eqn:bound:hyhu}\\
	\Rightarrow&&\left\|\begin{bmatrix}
	y^\star(t)-\hyh(t)\\
	u^\star(t)-\huh(t)
	\end{bmatrix}\right\|_2 &\leq \epsilon\,. &&\text{bound with \eqref{eqn:h_bound}}\tageq\label{eqn:bound:hyhueps}
	\end{align*}
\end{subequations}
\medmuskip=3mu
\thickmuskip=3mu

\subsection{Proof of~\eqref{eqn:A3:M:bound:h}}\label{sec:proof:fundamental:4}
From \eqref{eqn:bound:hyh0Teps} follows that \eqref{eqn:A3:M:cond} is satisfied for $\tmpyO:=\hyh(0),\,\tmpyT:=\hyh(T)$. Insertion of $\tmpyO:=\hyh(0),\,\tmpyT:=\hyh(T)$ into \eqref{eqn:A3:M:bound} yields:
\begin{align*}
\pig|M\big(y^\star(0),y^\star(T)\big)-M\big(\hyh(0),\hyh(T)\big)\pig| &\leq \Clam \cdot \left\|\begin{bmatrix}
y^\star(0)-\hyh(0)\\
y^\star(T)-\hyh(T)
\end{bmatrix}\right\|_2^\lambda \leq \Clam \cdot \Big( \sqrt{2\cdot n_y} \cdot C_\eta \cdot h^\eta\Big)^\lambda\,,
\end{align*}
where for the last bound we used~\eqref{eqn:bound:hyh0T}. This proves proposition \eqref{eqn:A3:M:bound:h} for $C_M=\Clam \cdot \big(\sqrt{2\cdot n_y} \cdot C_\eta \big)^\lambda$.

\subsection{Proof of~\eqref{eqn:A3:b:bound:h}}\label{sec:proof:fundamental:5}
Notice that $y^\star,u^\star$ satisfies the boundary conditions
\begin{align}
b\big(y^\star(0),y^\star(T)\big)=\bO \label{eqn:ystarb0}
\end{align}
exactly because $y^\star,u^\star$ is an exact solution.

From \eqref{eqn:bound:hyh0Teps} follows that \eqref{eqn:A3:b:cond} is satisfied for $\tmpyO:=\hyh(0),\,\tmpyT:=\hyh(T)$. Insertion of $\tmpyO:=\hyh(0)$, $\tmpyT:=\hyh(T)$ into \eqref{eqn:A3:b:bound} yields:
\begin{align*}
&&\pig\|\underbrace{b\big(y^\star(0),y^\star(T)\big)}_{\equiv \bO}-b\big(\hyh(0),\hyh(T)\big)\pig\|_2 &\leq \Clam \cdot \left\|\begin{bmatrix} 
y^\star(0)-\hyh(0)\\
y^\star(T)-\hyh(T)
\end{bmatrix}\right\|_2^\lambda &&\text{remove \eqref{eqn:ystarb0}}\\ 
\Rightarrow&&\pig\|b\big(\hyh(0),\hyh(T)\big)\pig\|_2 &\leq \Clam \cdot \Big( \sqrt{2 \cdot n_y} \cdot C_\eta \cdot h^\eta\Big)^\lambda &&\text{insert \eqref{eqn:bound:hyh0T}}\\
\Rightarrow&&\pig\|b\big(\hyh(0),\hyh(T)\big)\pig\|_2^2 &\leq \Clam^2 \cdot \Big( \sqrt{2 \cdot n_y} \cdot C_\eta \cdot h^\eta\Big)^{2\cdot\lambda} &&\text{take square}
\end{align*}
This proves proposition \eqref{eqn:A3:b:bound:h} for $C_b=\Clam^2 \cdot \big(\sqrt{2 \cdot n_y} \cdot C_\eta \big)^{2 \cdot \lambda}$.

\subsection{Proof of~\eqref{eqn:A3:f:bound:h}}\label{sec:proof:fundamental:6}
From \eqref{eqn:bound:hyhueps} follows that \eqref{eqn:A3:f:cond} is satisfied for $\tmpyt:=\hyh(t),\,\tmput:=\huh(t)$. Insertion of $\tmpyt:=\hyh(t),\,\tmput:=\huh(t)$ into the square of \eqref{eqn:A3:f:bound} yields:
\medmuskip=0mu
\thickmuskip=0mu
\begin{align*}
\left\|\begin{bmatrix}
f_1\big(y^\star(t),u^\star(t),t\big)-f_1\big(\hyh(t),\huh(t),t\big)\\
f_2\big(y^\star(t),u^\star(t),t\big)-f_2\big(\hyh(t),\huh(t),t\big)
\end{bmatrix}\right\|_2^2 &\leq \Clam^2 \cdot \Bigg(\left\|\begin{bmatrix}
y^\star(t)-\hyh(t)\\
u^\star(t)-\huh(t)
\end{bmatrix}\right\|_2^\lambda\Bigg)^2\\
&\leq \Clam^2 \cdot \Big( \sqrt{n_y+n_u} \cdot C_\eta \cdot h^\eta \Big)^{2 \cdot \lambda}\,,
\end{align*}
\medmuskip=3mu
\thickmuskip=3mu
where in the last bound we inserted \eqref{eqn:bound:hyhu}. This proves proposition \eqref{eqn:A3:f:bound:h} for $C_f=\Clam^2 \cdot \big(\sqrt{n_y+n_u} \cdot C_\eta \big)^{2 \cdot \lambda}$.

\section{Proof of Main Theorem~\ref{thm:main}}\label{sec:proof:thm:main}
From here on, we dedicate one subsection into the proof of one theorem or lemma. This subsection is the proof of Theorem~\ref{thm:main}. Recall that you can go back and forth via the embedded hyperlinks.

\subsubsection{Structure of the Proof}
We first introduce a bound for $\chi$. This helps us bound $\delta,\rho$. Afterwards, we prove convergence of $\gamma$ via straightforward application of Theorem~\ref{thm:convgamma}.

\subsubsection{Convergence Order of $\delta$ and $\rho$}
We use Theorem~\ref{thm:convdelta} and Theorem~\ref{thm:eps}. The latter asserts:
\begin{align*}
\delta \leq \chi \in \cO\left( h^{\eta \cdot \lambda} + \frac{h^{\min\lbrace 2 \cdot \eta \cdot \lambda,\ell \rbrace}}{\omega} \right)\tageq\label{eqn:boundXi}
\end{align*}
From Theorem~\ref{thm:convdelta} we obtain:
\begin{align*}
& 			&\rho^{\phantom{2}}   	&\leq \sqrt{2} \cdot \sqrt{M\big(y^\star(0),y^\star(T)\big)-\Cobj + \chi\,} \cdot \sqrt{\omega} & &\\
&\Rightarrow&\rho^2 	&\leq 2 \cdot \Big(M\big(y^\star(0),y^\star(T)\big)-\Cobj\Big) \cdot \omega + 2 \cdot \chi \cdot \omega & &\text{take the square and multiply out}\\
&\Rightarrow&\rho^2 	&\in \cO\Bigg( \omega + \left( h^{\eta \cdot \lambda} + \frac{h^{\min\lbrace 2 \cdot \eta \cdot \lambda,\ell \rbrace}}{\omega} \right) \cdot \omega \Bigg)& &\text{insert the bound \eqref{eqn:boundXi}}\\
&\Rightarrow &\rho^2 &\in\cO\left( \omega + h^{\min\lbrace 2\cdot\eta\cdot\lambda,\ell \rbrace} \right)&& \text{bound by lowest order}\tageq\label{eqn:proof:main}
\end{align*}
Taking the square-root of~\eqref{eqn:proof:main} yields\footnote{because $\sqrt{\theta+\kappa}\leq \sqrt{\theta}+\sqrt{\kappa}\qquad \forall \theta,\kappa \in \R_{>0}$.}:
\begin{align*}
\rho \in \cO\left( \sqrt{\omega} + h^{\min\lbrace \eta\cdot\lambda,\ell/2 \rbrace} \right)
\end{align*}
Thus the bounds on $\delta,\rho$ hold.

\subsubsection{Convergence Order of $\gamma$}
Finally, with the help of Theorem~\ref{thm:convgamma}, we show the bound on $\gamma$. Each component $y_{h,[\upsilon]}$ for $\upsilon=1,\dots,n_y$ and each component $u_{h,[\upsilon]}$ for $\upsilon=1,\dots,n_u$ satisfies \eqref{eqn:convgamma:1} due to the bound constraints in \eqref{eqn:POCPh}. Further, all components $y_{\texttt{L},[\upsilon]},y_{\texttt{R},[\upsilon]}$ and $u_{\texttt{L},[\upsilon]},u_{\texttt{R},[\upsilon]}$ satisfy \eqref{eqn:convgamma:2} due to \AssumpI. Theorem~\ref{thm:convgamma} thus asserts the bounds \eqref{eqn:boxthmpropL} and \eqref{eqn:boxthmpropR} for each $y_{h,[\upsilon]}$ and each $u_{h,[\upsilon]}$ on each interval $I_i$.

\section{Proof of Theorem~\ref{thm:convgamma}}\label{sec:thm:convgamma}
The proof makes use of the following corollary.
\begin{cor}\label{cor:j4}
	Let $n \in \N$. Then
	\begin{align*}
	\sum_{j=0}^n j^4 = \frac{1}{30} \cdot  n \cdot (1 + n) \cdot (2 \cdot n + 1) \cdot (3 \cdot  n^2+ 3 \cdot n-1) \leq n^5\,.
	\end{align*}
\end{cor}
\noindent
\underline{Proof:} in \cite[Summae Potestatum]{Bernoulli1713Ars}. \qed

\subsubsection{Strategy of the Proof}
The sampling conditions~\eqref{eqn:convgamma:1} guarantee satisfaction of the constraints only at the sampling points $t \in \Tcl{i,m}$, whereas we want to bound the maximum violation of $\phi$ in between any two neighboring sampling points. This is achieved by bounding the curvature of a transformed version of $\phi$.

We only prove the left bound~\eqref{eqn:boxthmpropL} because the right bound~\eqref{eqn:boxthmpropR} follows by analogy. Wlog., we replace $I_i$ with $\Iref$ and $\Tcl{i,m}$ with $\Tclref{m}$. Figure~\ref{fig:proofboxerror}~(a) depicts a possible scenario for $\phiL,\phi,\phiR$. The dashed black curve is obtained by taking $\Cbox + \phiL(t)$. The sampling points are marked with crosses and partly with circles as well.

\subsubsection{Structure of the Proof}
In Figure~\ref{fig:proofboxerror}~(b), we consider a rescaled version of $\phiL,\phi,\phiR$ such that the lower bound (blue curve) is a constant of~$-1$. We then perform a cosine transformation from Figure~\ref{fig:proofboxerror}~(b) into a function $g$ in Figure~\ref{fig:proofboxerror}~(c). Both transformations are given in Section~\ref{sec:proof:boxerror:Trafo}. 

We can now bound the lowest value of $g$ from a worst-case interpolation. This interpolation is formed by using a bound on the second derivative of $g$, introduced in Section~\ref{sec:proof:boxerror:Curvature}. The bound for the second derivative is provided by the Lemma~\ref{thm:Parseval}, introduced just before Section~\ref{sec:proof:boxerror:Curvature}.

The worst-case interpolation is constructed in Section~\ref{sec:proof:boxerror:Estimate}. The worst-case interpolation is illustrated as the violet parabola $w$ in the lower left of Figure~\ref{fig:proofboxerror}~(c). The proof is completed by back-transforming the worst-case violation bound on $g$ into a worst-case violation bound on $\phi$.

\subsection{Transformation into a Finite Cosine Series}\label{sec:proof:boxerror:Trafo}
We uniquely identify $\phiL,\phi,\phiR$ with functions $\psiL,\psi,\psiR$ such that $\psi_L=-1$ and \mbox{$\psiR\leq 1$}:
\begin{subequations}
	\label{eqn:phiToPsi}%
	\begin{align}
	\phiL(t) &= \phiL(t) + \frac{\Cbox}{2} + \frac{\Cbox}{2} \cdot \psiL(t)\,, 	\\
	\phi (t) &= \phiL(t) + \frac{\Cbox}{2} + \frac{\Cbox}{2} \cdot \psi (t)\,, 	\\
	\phiR(t) &= \phiL(t) + \frac{\Cbox}{2} + \frac{\Cbox}{2} \cdot \psiR(t)\,.
	\end{align}
\end{subequations}
The functions $\psiL,\psi,\psiR$ are plotted in Figure~\ref{fig:proofboxerror}~(b).

Using a change of variable for $t \in \Iref$ into $\xi \in [0,\pi]$ via the relation\footnote{the negative sign is to be consistent with the literature convention in \cite{ChebLob} for Chebyshev-Gauss-Lobatto points $\tau_k$ from~\eqref{eqn:cheblob}. This is so $\tau_k$ are increasing as $k$ increases.}~\mbox{$\xi=-\arccos(t)$}, we define the function $g(\xi):=\psi(t)$. Because $\psi \in \cP_p(\Iref)$, we may express $\psi$ as a series of Chebyshev polynomials of the first kind\footnote{The $j^\text{th}$ Chebyshev polynomial on $[-1,1]$ can be expressed as $\cos\big(-j\cdot\arccos(t)\big)$}, thus:
\begin{align}
\psi(t)&= \sum_{j=0}^p \beta_j \cdot \cos\big(-j\cdot\arccos(t)\big)\\
\Leftrightarrow\quad g(\xi) &= \sum_{j=0}^p \beta_j \cdot \cos(j \cdot \xi)\,, \label{eqn:cosine_g}
\end{align}
with coefficients $\beta_0,\dots,\beta_p \in \R$.

Figure~\ref{fig:proofboxerror} shows an example of $\phi,\psi,g$ for $p=3$, $\Cbox=2.8$, and~\mbox{$m=6 \in \N \cdot p$}. Notice that the Chebyshev-Gauss-Lobatto points $\tau_k$ from~\eqref{eqn:cheblob} become equidistant as the variable $t$ changes into \mbox{$\xi=-\arccos(t)$}. The points $\xi_k$ are equidistant due to the $\arccos$-transformation of the points $\tau_{k}$ from~\eqref{eqn:cheblob}. This can be observed in Figure~\ref{fig:proofboxerror}~(c), where all crosses and circles are equidistant. Recall that the circles in Figure~\ref{fig:proofboxerror} indicate each $(\frac{m}{p})^{\text{th}}$ sampling point. There are $p+1$ encircled points. We can check that the transformed points $\xi_k=-\arccos(\tau_k)$ from~\eqref{eqn:cheblob}  satisfy
\begin{align}
\xi_k \equiv \frac{k}{p}\cdot \pi\qquad \text{for } k=0,\dots,p\,. \label{eqn:def:xi}
\end{align}
As apparent from~\eqref{eqn:def:xi}, the points $\xi_k$ only depend on $p$ and not on $m$. Figure~\ref{fig:proofboxerror}~(d) illustrates the sampling points $\tau_k$ (with symbol $\times$) over each $(m/p)^\text{th}$ sampling point (with symbol $\circ$).

\begin{lem}\label{thm:Parseval}
	Assign the vectors $\bg=\big(g(\xi_0),\dots,g(\xi_k),\dots,g(\xi_p)\big)$, $\bbeta=(\beta_0,\dots,\beta_j,\dots,\beta_p) \in \R^{p+1}$ for the quantities in \eqref{eqn:cosine_g}. The following relation holds in general:
	\begin{align*}
	\|\bbeta\|_2 \leq \frac{2}{\sqrt{p+1}} \cdot \|\bg\|_2\,.
	\end{align*}
\end{lem}
\noindent
\underline{Proof:} in Section~\ref{sec:thm:Parseval}. \qed

\subsection{Bounding the Curvature}\label{sec:proof:boxerror:Curvature}
We can bound $\|\bg\|_\infty = 1$ because \mbox{$-1=\gL(\xi_k)\leq g(\xi_k)\leq\gR(\xi_k)\leq 1$}, cf.~the encircled blue and black points in Figure~\ref{fig:proofboxerror}~(c). Hence $\|\bg\|_2 \leq \sqrt{p+1}$. Using Lemma~\ref{thm:Parseval}, this means $\|\bbeta\|_2 \leq 2$.

We now consider second derivatives of $g$ after $\xi$: We can bound $g'':=\frac{\mathrm{d}^2 g}{\mathrm{d}\xi^2}$ from~\eqref{eqn:cosine_g} as follows:
\begin{align*}
\operatornamewithlimits{max}_{\xi \in [0,\pi]} |g''(\xi)| \leq \sum_{j=0}^p |\beta_j| \cdot \underbrace{\max_{\xi \in [0,\pi]}\Big|\big(\cos(j \cdot \xi)\big)''\Big|}_{=j^2} = \tbbeta\t \cdot \bj =:\Psi
\end{align*}
with $\bj:=\big(0,1,2^2,\dots,j^2,\dots,p^2\big) \in \R^{p+1}$ and $\tbbeta=\big(|\beta_0|,\,|\beta_1|,\,\dots,|\beta_p|\,\big) \in \R^{p+1}$. Use~\mbox{$\|\bj\|_2= \sqrt{\sum_{j=0}^p j^4} \leq p^{2.5}$} from Corollary~\ref{cor:j4}. Using Cauchy-Schwarz for $|\bj\t\cdot\tbbeta|$ yields~\mbox{$\Psi \leq \|\tbbeta\|_2 \cdot \|\bj\|_2= \|\bbeta\|_2 \cdot \|\bj\|_2 \leq 2 \cdot p^{2.5}$}.

\subsection{Construction of a Worst-Case Estimate}\label{sec:proof:boxerror:Estimate}
Thus far, we obtained the bound $\Psi$ for $|g''|$ from using only the encircled black and blue sampling points in Figure~\ref{fig:proofboxerror}~(c). These points are invariant to $m$, hence why the bound $\Psi$ holds $\forall m \in \N \cdot p$. Now, we use the blue crossed sampling points. They are spaced with distance $\pi/m$.

Construct the violet parabola $w$ with $w(\xi_k)=w(\xi_{k+1})=-1$ and $w''=\Psi$, i.e., the parabola that interpolates two neighboring blue crossed sampling points and has curvature $\Psi$, as depicted in the lower left of Figure~\ref{fig:proofboxerror}~(c). Parabolas like $w$ can be used to bound the value of $g$ everywhere in-between the sampling points. This is depicted with the dashed violet graph in Figure~\ref{fig:proofboxerror}~(c). The dashed graph consists of successive replicas of $w$.

For $k=0$, polynomial interpolation yields the parabola of the solid violet graph as
\begin{align}
w(\xi) = \frac{\Psi}{2} \cdot \xi^2 - \frac{\Psi \cdot \pi}{2 \cdot m} \cdot \xi -1\,.
\end{align}
We can check that $w(0)=-1$, $w(\frac{\pi}{m})=-1$, and $w''=\Psi$. Logically, the minimum of $w$ is at $\xi^\star=\frac{\pi}{2\cdot m}$, thus $w(\xi^\star)=-1-\frac{\pi^2 \cdot \Psi}{8 \cdot m^2}$. Hence, $g-\gL$ is bounded below by $-\frac{\pi^2 \cdot p^{2.5}}{4 \cdot m^2}$. Back-transformation from $g-\gL$ into $\phi-\phiL$ via \eqref{eqn:phiToPsi} introduces the factor $\Cbox/2$, completing the proof.

\begin{figure}
	\centering
	\includegraphics[width=1\linewidth]{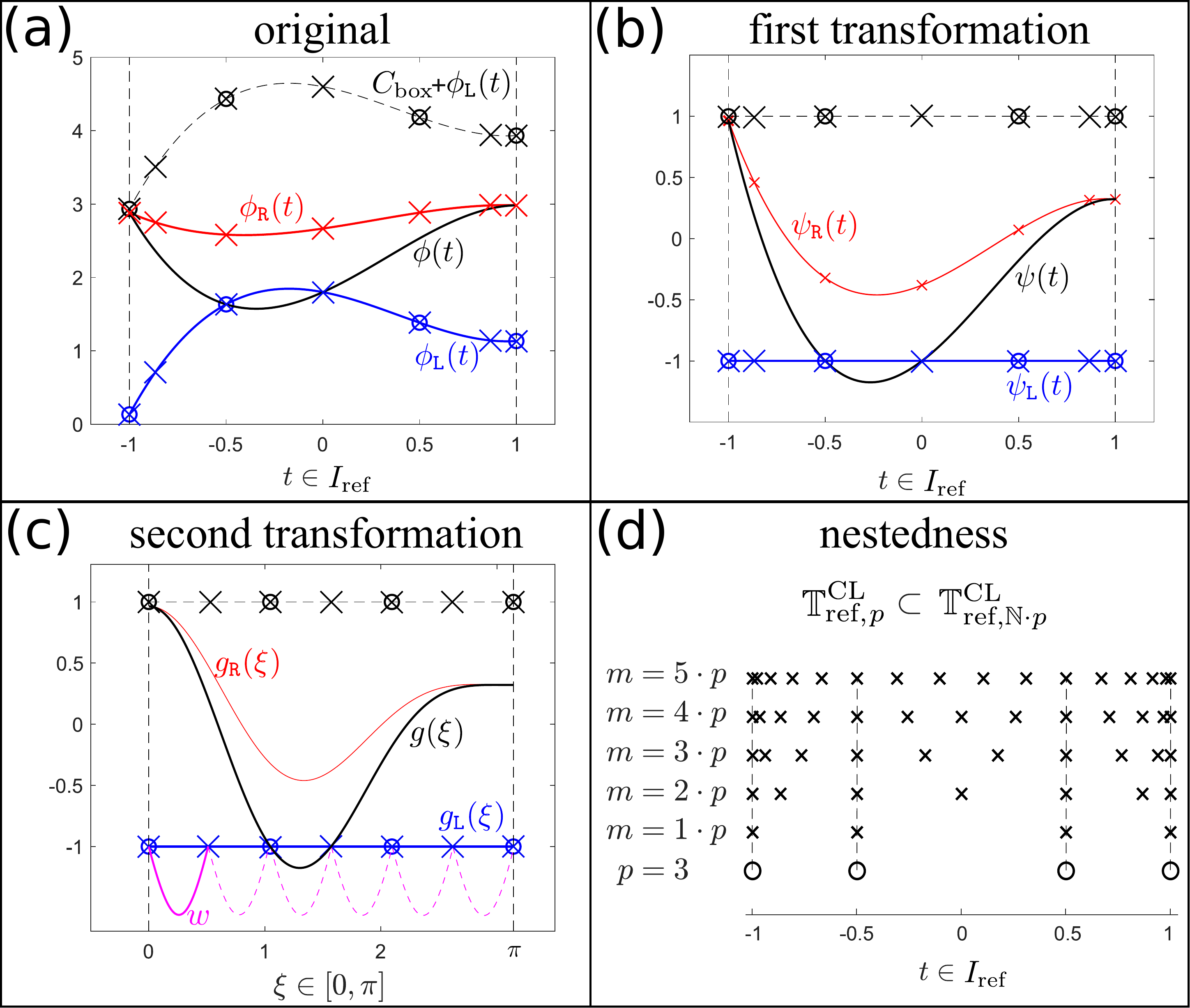}
	\caption{(a)--(c): Construction of polynomials and trigonometric polynomials on reference intervals $\Iref$ and $[0,\pi]$. (d): Nestedness of the spaces $\Tclref{\N\cdot p}$.}
	\label{fig:proofboxerror}
\end{figure}

\section{Proof of Theorem~\ref{thm:convdelta}}\label{sec:thm:convdelta}
From \AssumpII follows $M\big(y^\star_h(0),y^\star_h(T)\big)\geq \Cobj$. Since $\omega>0$, from~\eqref{eqn:def:chi} follows
\begin{align*}
\underbrace{M\big(y^\star_h(0),y^\star_h(T)\big)}_{\geq \Cobj} + \underbrace{\frac{r^2(y^\star_h,u^\star_h)}{2 \cdot \omega}}_{\geq 0} \leq M\big(y^\star(0),y^\star(T)\big) + \chi\,.\tageq\label{eqn:proof:delta:1}
\end{align*}

\subsubsection{First Proposition}
In~\eqref{eqn:proof:delta:1}, dropping the second term in the left-hand side results in:
\begin{align*}
M\big(y^\star_h(0),y^\star_h(T)\big) \leq M\big(y^\star(0),y^\star(T)\big) + \chi\,.
\end{align*}
This yields satisfaction of~\eqref{eqn:meas:delta} with $\delta=\chi$.

\subsubsection{Second Proposition}
In~\eqref{eqn:proof:delta:1}, bounding the first term below and subtracting $\Cobj$ from both sides results in:
\begin{align*}
\frac{r^2(y^\star_h,u^\star_h)}{2 \cdot \omega} \leq  \underbrace{M\big(y^\star(0),y^\star(T)\big)}_{\geq \Cbox}- \Cobj + \chi
\end{align*}
Multiplying both sides with $2\cdot\omega$ yields:
\begin{align*}
r^2(y^\star_h,u^\star_h) \leq 2 \cdot \omega \cdot \Big( M\big(y^\star(0),y^\star(T)\big)-\Cobj + \chi \Big)
\end{align*}
This shows satisfaction of~\eqref{eqn:meas:rho} for~\mbox{$\rho=\sqrt{2} \cdot \sqrt{\omega} \cdot \sqrt{M\big(y^\star(0),y^\star(T)\big)-\Cobj+\chi\,}$}.

\section{Proof of Theorem~\ref{thm:eps}}\label{sec:thm:eps}
For the scope of this proof, we make use of the following two functionals:
\begin{align}
J(y,u)&:=M\big(y(0),y(T)\big) + \frac{1}{2 \cdot \omega} \cdot r^2(y,u)\,, \label{eqn:def:J}\\
J_{h}(y,u)&:=M\big(y(0),y(T)\big) + \frac{1}{2 \cdot \omega} \cdot \left( Q_{h,q}[y,u] + \left\|b\big(\,y(0),y(T)\,\big)\right\|_2^2 \right)\,.\label{eqn:def:Jh}
\end{align}
The second functional is a quadrature-approximation of the first. The following lemma asserts an error bound on the quadrature.
\begin{lem}\label{lem:errJ}
	Let $(y,u)\in\cX_{h,p}$ and $(\tilde{y},\tilde{u}) \in \cX_{h,p}$. If
	\begin{align*}
	J_h(y,u)\leq J_h(\tilde{y},\tilde{u})\tageq\label{eqn:lem:errJ:0}
	\end{align*}
	then
	\begin{align*}
	J(y,u) \leq J(\tilde{y},\tilde{u}) + \frac{\Cquad \cdot h^\ell}{\omega}\,.
	\end{align*}
\end{lem}
\noindent
\underline{Proof:} in Section~\ref{sec:proof:lem:errJ}. \qed\\[6pt]

We also make use of the following lemma, which uses $\hyh,\huh$ from~\eqref{eqn:suitableMinimizer} and the constant $\lambda$ from \AssumpIII.\\[-15pt]
\begin{lem}\label{thm:forwardstab}
	\textcolor{orange}{\dunderline[-1pt]{2pt}{\textcolor{black}{If \AssumpIII holds and $h$ is sufficiently small}}} then there exists a finite constant $C_r \in \R_{>0}$ such that:
	\begin{align}
	r^2(\hyh,\huh) &\leq C_r \cdot h^{2 \cdot \lambda\cdot\eta}\,.\label{eqn:forwardstab2}
	\end{align}
\end{lem}
\noindent
\underline{Proof:} in Section~\ref{sec:thm:forwardstab}. \qed\\[3pt]

We now present the actual proof of Theorem~\ref{thm:eps}.\\[-25pt]
\subsubsection{Structure of the Proof}
The proof consists of two parts: i) In Section~\ref{sec:thm:eps:1} we use Lemma~\ref{lem:fundamental} and Lemma~\ref{thm:forwardstab} to develop a bound for $J(\hyh,\huh)$. ii) In Section~\ref{sec:thm:eps:2} we use Lemma~\ref{lem:errJ} to bound $J(y^\star_h,u^\star_h)$.

\subsection{Bound for $J(\hyh,\huh)$}\label{sec:thm:eps:1}
Summation of \eqref{eqn:A3:M:bound:h} and \eqref{eqn:forwardstab2} yields:
\medmuskip=-0.1mu
\thickmuskip=-0.1mu
\begin{align*}
&&\pig| M\big(\hyh(0),\hyh(T)\big) - M\big(y^\star(0),y^\star(T)\big) \pig| + \frac{r^2(\hyh,\huh)}{2 \cdot \omega} &\leq C_M \cdot h^{\lambda \cdot \eta} + \frac{C_r \cdot h^{2 \cdot \lambda \cdot \eta}}{2 \cdot \omega}\\
\Rightarrow&& M\big(\hyh(0),\hyh(T)\big) + \frac{r^2(\hyh,\huh)}{2 \cdot \omega} &\leq M\big(y^\star(0),y^\star(T)\big) + C_M \cdot h^{\lambda \cdot \eta} + \frac{C_r \cdot h^{2 \cdot \lambda \cdot \eta}}{2 \cdot \omega}\\[10pt]
\Leftrightarrow&& \textcolor{blue}{J(\hyh,\huh)}& \leq M\big(y^\star(0),y^\star(T)\big) + C_M \cdot h^{\lambda \cdot \eta} + \frac{C_r \cdot h^{2 \cdot \lambda \cdot \eta}}{2 \cdot \omega}\tageq\label{eqn:bound:Jhyh}
\end{align*}
\medmuskip=3mu
\thickmuskip=3mu

\subsection{Bound for $J(y^\star_h,u_h^\star)$}\label{sec:thm:eps:2}
The bound \eqref{eqn:suitableMinimizer} can be written compactly by using $J_h$ from \eqref{eqn:def:Jh}:
\begin{align*}
J_h(y^\star_h,u^\star_h) \leq J_h(\hyh,\huh)\,.
\end{align*}
Application of Lemma~\ref{lem:errJ} thereon yields:
\begin{align*}
J(y^\star_h,u^\star_h) \leq \textcolor{blue}{J(\hyh,\huh)} + \frac{C_\ell \cdot h^\ell}{\omega}\,.\tageq\label{eqn:bound:Jhstar}
\end{align*}
Inserting \eqref{eqn:bound:Jhyh} into \eqref{eqn:bound:Jhstar} for $\textcolor{blue}{J(\hyh,\huh)}$ yields:
\begin{align*}
&&J(y^\star_h,u^\star_h) &\leq M\big(y^\star(0),y^\star(T)\big) + C_M \cdot h^{\lambda \cdot \eta} + \frac{C_r \cdot h^{2 \cdot \lambda \cdot \eta}}{2 \cdot \omega} + \frac{C_\ell \cdot h^\ell}{\omega}\\
\Leftrightarrow&& M\big(y^\star_h(0),y^\star_h(T)\big) +\frac{r^2(y^\star_h,u^\star_h)}{2 \cdot \omega} &\leq M\big(y^\star(0),y^\star(T)\big) + C_M \cdot h^{\lambda \cdot \eta} + \frac{C_r \cdot h^{2 \cdot \lambda \cdot \eta}}{2 \cdot \omega} + \frac{C_\ell \cdot h^\ell}{\omega}\,.
\end{align*}
This shows proposition~\eqref{eqn:thm:eps:prop} for $C_\chi = C_M + C_r + C_\ell$.

\section{Proof of Lemma~\ref{thm:Parseval}}\label{sec:thm:Parseval}
We may write~\eqref{eqn:cosine_g} as
\begin{align}
\bF \cdot \bbeta = \bg\,, \label{eqn:gFc}
\end{align}
where $\bF \in \R^{(p+1)\times (p+1)}$ is the symmetric matrix with entries $F_{k,j} = \cos(j \cdot \xi_k)$ for $j,k=0,\dots,p$, where $\xi_k$ from~\eqref{eqn:def:xi}. Also, define $\bD = \opdiag(\frac{1}{\sqrt{2}},1,1,\dots,1,1,\frac{1}{\sqrt{2}}) \in \R^{(p+1)\times (p+1)}$. As known from discrete cosine transform of first kind \cite{dct}, the matrix $\bG := \sqrt{\frac{2}{p}} \cdot \bD \cdot \bF \cdot \bD$ is orthogonal. Multiplying $\textcolor{blue}{\sqrt{\frac{2}{p}} \cdot \bD}$ from the left onto \eqref{eqn:gFc} and inserting an \textcolor{red}{identity matrix} yields:
\begin{align*}
\underbrace{\textcolor{blue}{\sqrt{\frac{2}{p}} {}\cdot{}  \bD {}\cdot{}} \bF \cdot \textcolor{red}{\bD}}_{\bG} \textcolor{red}{{}\cdot{} \bD\inv {}\cdot{}} \bbeta =	\textcolor{blue}{\sqrt{\frac{2}{p}} {}\cdot{} \bD {}\cdot{}} \bg
\end{align*}
Using orthogonality of $\bG$, the following relation is implied for the norms:
\begin{align*}
\|\bD\inv {}\cdot{} \bbeta\|_2 = \left\|\sqrt{\frac{2}{p}} {}\cdot{} \bD {}\cdot{} \bg\right\|_2\tageq\label{eqn:Orthonorm}
\end{align*}
Use submultiplicativity, i.e.:
\begin{subequations}
	\begin{align*}
	&&\left\|\sqrt{\frac{2}{p}} {}\cdot{} \bD {}\cdot{} \bg\right\|_2 &\leq \sqrt{\frac{2}{p}} \cdot \|\bD\|_2 \cdot \|\bg\|_2\,,\tageq\label{eqn:submul:1}\\[10pt]
	&&\|\bbeta\|_2 = \|\bD\cdot\bD\inv\cdot\bbeta\|_2&\leq \|\bD\|_2 \cdot \|\bD\inv\cdot\bbeta\|_2\\
	\Leftrightarrow&& \frac{1}{\|\bD\|_2}\cdot{\|\bbeta\|_2} &\leq \|\bD\inv\cdot\bbeta\|_2\,.\tageq\label{eqn:submul:2}
	\end{align*}
\end{subequations}
Using \eqref{eqn:submul:2} for the left-hand side of \eqref{eqn:Orthonorm} and \eqref{eqn:submul:1} for the right-hand side of \eqref{eqn:Orthonorm}, we obtain:
\begin{align*}
\frac{1}{\|\bD\|_2} {}\cdot{} \|\bbeta\|_2 \leq \sqrt{\frac{2}{p}} {}\cdot{} \|\bD\|_2 {}\cdot{} \|\bg\|_2 
\end{align*}
The proposition follows from $\|\bD\|_2=1$ and bounding $\sqrt{2/p} \leq \frac{2}{\sqrt{p+1}} \quad \forall\ p \in \N$. A variant of Lemma~\ref{thm:Parseval} is known as Parseval's identity. This variant is subject to different scalings, notations, and not for the cosine but for the Fourier transformation.

\section{Proof of Lemma~\ref{lem:errJ}}\label{sec:proof:lem:errJ}
From~\eqref{eqn:quadcond} follows:
\begin{subequations}
	\begin{align}
	|J_h(y,u) - J(y,u) | &\leq \frac{\Cquad \cdot h^\ell}{2 \cdot \omega}\label{eqn:lem:errJ:1}\\
	|J_h(\tilde{y},\tilde{u}) - J(\tilde{y},\tilde{u}) | &\leq \frac{\Cquad \cdot h^\ell}{2 \cdot \omega}\label{eqn:lem:errJ:2}\,.
	\end{align}
\end{subequations}
Below: in \eqref{eqn:lem:errJ:3} we use \eqref{eqn:lem:errJ:1} to bound $J(y,u)$ from above; in \eqref{eqn:lem:errJ:4} we restate \eqref{eqn:lem:errJ:0}; and in \eqref{eqn:lem:errJ:5} we use \eqref{eqn:lem:errJ:2} to bound $J_h(\tilde{y},\tilde{u})$ from above:
\newcommand{\Jyt}{J(\tilde{y},\tilde{u})\xspace}
\newcommand{\Jhyt}{J_h(\tilde{y},\tilde{u})\xspace}
\newcommand{\Jy}{J({y},{u})\xspace}
\newcommand{\Jhy}{J_h({y},{u})\xspace}
\newcommand{\Coquad}{\frac{\Cquad\cdot h^\ell}{2 \cdot \omega}}
\begin{subequations}
	\begin{align*}
	&      			&&\Jy 	&& \leq  &&\Jhy && 		&& 		 &&    &&      &&        &&{}+\quad{}{}\Coquad 	\tageq\label{eqn:lem:errJ:3}\\
	&\ \land 		&& 		&& 		 &&\Jhy && \leq && \Jhyt &&    &&	   &&        &&        	\tageq\label{eqn:lem:errJ:4}\\
	&\ \land 		&&      &&       &&     &&      && \Jhyt &&\leq&& \Jyt &&+\quad\Coquad&&	\tageq\label{eqn:lem:errJ:5} \\[10pt] \hline\hline\\[-1.5ex]
	&\Rightarrow 	&&\Jy   &&       && 	&&      &&       &&\leq&& \Jyt &&+\quad 2\quad \cdot &&\phantom{{}+\quad{}}{}\Coquad\qquad
	\tageq\label{eqn:lem:errJ:6}
	\end{align*}
\end{subequations}
The proposition~\eqref{eqn:lem:errJ:6} is attained by applying \eqref{eqn:lem:errJ:3}, \eqref{eqn:lem:errJ:4}, \eqref{eqn:lem:errJ:5} in sequence.

\section{Proof of Lemma~\ref{thm:forwardstab}}\label{sec:thm:forwardstab}
Our proof makes use of the following lemma. The function~$f$ therein is defined in~\eqref{eqn:def:f}.\\[-15pt]
\begin{lem}\label{thm:w2bound}
	\textcolor{orange}{\dunderline[-1pt]{2pt}{\textcolor{black}{If \AssumpIII holds and $h$ is sufficiently small}}}
	then there exists a constat $C_{T,\lambda,\epsilon}\in \R_{>0}$ such that:
	\begin{align*}
	\int_0^T \left\|f\big(\dhyh(t),\hyh(t),\huh(t),t\big)\right\|_2^2\,\mathrm{d}t \leq C_{T,\lambda,\epsilon} \cdot h^{2 \cdot \lambda \cdot \eta} \tageq\label{eqn:lem:w2bound}
	\end{align*}
\end{lem}
\noindent
\underline{Proof:} in Section~\ref{sec:thm:w2bound}. \qed

We now prove Lemma~\ref{thm:forwardstab}:
\begin{align*}
r^2(y,u) 
&= \int_0^T \pig\|f\big(\dot{y}(t),y(t),u(t),t\big)\pig\|_2^2\,\mathrm{d}t + \pig\|b\big(y(0),y(T)\big)\pig\|_2^2 &&\text{from definition~\eqref{eqn:IntRes}}\\
&\leq C_{T,\lambda,\epsilon} \cdot h^{2 \cdot \lambda \cdot \eta} + C_b \cdot h^{2 \cdot \lambda \cdot \eta} &&\text{use \eqref{eqn:lem:w2bound} and \eqref{eqn:A3:b:bound:h}}
\end{align*}
This shows the proposition~\eqref{eqn:forwardstab2} for $C_r = C_{T,\lambda,\epsilon} + C_b$.

\section{Proof of Lemma~\ref{thm:w2bound}}\label{sec:thm:w2bound}
We use the following intermediate result.
\begin{cor}\label{thm:Cauch}
	Let ${v_1},{v_2} \in \R^{m}$ and $v_3 \in \R^{n}$ for some $m,n \in \N$. Then:
	\begin{align*}
	\left\|\begin{bmatrix} \textcolor{red}{v_1}+{}{}\textcolor{violet}{v_2}\\ \phantom{v_1+{}}{}\textcolor{blue}{v_3} \end{bmatrix}\right\|_2^2 \leq 3 \cdot \|v_1\|_2^2 + 3 \cdot \left\|\begin{bmatrix}
	v_2\\
	v_3
	\end{bmatrix}\right\|_2^2
	\end{align*}
\end{cor}
\noindent
\underline{Proof:}
\begin{align*}
&\left\|\begin{bmatrix} v_1+v_2\\ \phantom{v_1+{}}{}w \end{bmatrix}\right\|_2^2 	&&\\
= 	& \begin{bmatrix} v_1+v_2\\ \phantom{v_1+{}}{}w \end{bmatrix}\t \cdot \begin{bmatrix} v_1+v_2\\ \phantom{v_1+{}}{}w \end{bmatrix} &&\text{write out definition}\\
= 	& \|v_1\|_2^2 + \|v_2\|_2^2 + \|v_3\|_2^2 + v_1\t \cdot v_2 + v_2\t \cdot v_1&&\text{multiply out}\\
\leq& \|v_1\|_2^2 + \|v_2\|_2^2 + \|v_3\|_2^2 + 2 \cdot |v_1\t \cdot v_2|&&\text{bound cross term by abs}\\
\leq& \|v_1\|_2^2 + \|v_2\|_2^2 + \|v_3\|_2^2 + 2 \cdot \|v_1\|_2 \cdot \|v_2\|_2&&\text{apply Cauchy-Schwarz}\\
\leq& \|v_1\|_2^2 + \|v_2\|_2^2 + \|v_3\|_2^2 + 2 \cdot \big( \|v_1\|_2^2 + \|v_2\|_2^2 \big)\\
\leq& 3 \cdot \|v_1\|_2^2 + 3 \cdot \|v_2\|_2^2 + 3 \cdot \|v_3\|_2^2 	&&\text{bound generously} 	\tag*{$\square$}
\end{align*}

Now that we have finished the proof of the intermediate result, we come to the actual proof of Lemma~\ref{thm:w2bound}. It works by first bounding the integrand and then integrating over the bound.

\subsubsection{Bounding the Integrand}\label{sec:proof:thm:w2bound:AlgebraicBound}
Recall $f$ from~\eqref{eqn:def:f}. The exact solution satisfies the constraints exactly, hence
\begin{align}
f\big(\dot{y}^\star(t),y^\star(t),u^\star(t),t\big)=\bO\quad \tforall\ t \in [0,T]\,.\label{eqn:proof:w2bound:f0}
\end{align}
Thus, it follows $\tforall\, t \in [0,T]$:
\begin{align*}
& \left\|f\big(\dhyh(t),\hyh(t),\huh(t),t\big)\right\|_2^2\\[15pt]
= 	& \left\|f\big(\dhyh(t),\hyh(t),\huh(t),t\big)-f\big(\dot{y}^\star(t),y^\star(t),u^\star(t),t\big)\right\|_2^2 &&\text{subtract \eqref{eqn:proof:w2bound:f0}}\\[15pt]
= 	&\Bigg\| 
\begin{bmatrix}
f_1\big(\,\hyh(t),\huh(t),t\,\big){}{}-{}\dhyh(t)\\[4pt]
f_2\big(\,\hyh(t),\huh(t),t\,\big){}\phantom{{}-{}\dhyh(t)}\\
\end{bmatrix}
-
\begin{bmatrix}
f_1\big(\,y^\star(t),u^\star(t),t\,\big){}{}-{}\dot{y}^\star(t)\\[4pt]
f_2\big(\,y^\star(t),u^\star(t),t\,\big){}\phantom{{}-{}\dot{y}^\star(t)}\\
\end{bmatrix} 
\Bigg\|_2^2 &&\text{write out}\\[15pt]
=	&\Bigg\| 
\begin{bmatrix}
{}{}\textcolor{red}{\Big(\dot{y}^\star(t) - \dhyh(t) \Big)}+{}{}\textcolor{violet}{\Big(f_1\big(\,\hyh(t),\huh(t),t\,\big) - f_1\big(\,y^\star(t),u^\star(t),t\,\big) \Big)}\\[5pt]
\phantom{\Big( \dot{y}^\star(t) - \dhyh(t) \Big) + {}}{}\textcolor{blue}{\Big(f_2\big(\,\hyh(t),\huh(t),t\,\big) - f_2\big(\,y^\star(t),u^\star(t),t\,\big) \Big)}\\
\end{bmatrix}
\Bigg\|_2^2 &&\text{rearrange}\\[15pt]
\leq 	&3 \cdot \|\dot{y}^\star(t)-\dhyh(t)\|_2^2 + 3 \cdot \left\|\begin{bmatrix}
f_1\big(\hyh(t),\huh(t),t\big) - f_1\big(y^\star(t),u^\star(t),t\big)\\[4pt]
f_2\big(\hyh(t),\huh(t),t\big) - f_2\big(y^\star(t),u^\star(t),t\big)
\end{bmatrix}\right\|_2^2 && \text{use Corollary~\ref{thm:Cauch}}\\[15pt]
\leq & 3 \cdot \|\dot{y}^\star(t)-\dhyh(t)\|_2^2 + 3 \cdot C_f \cdot h^{2 \cdot \lambda \cdot \eta}&& \text{use Lemma~\ref{lem:fundamental}}  \tageq\label{eqn:proof:thm:w2bound:AlgebraicBound2}
\end{align*}

\subsubsection{Integrating over the Bound}
Taking the integral of \eqref{eqn:proof:thm:w2bound:AlgebraicBound2} for $t$ over $[0,T]$ yields:
\medmuskip=0.4mu
\thickmuskip=0.2mu
\begin{align*}
&\int_0^T \pig\|f\big(\dhyh(t),\hyh(t),\huh(t),t\big)\pig\|_2^2 \,\mathrm{d}t\\
&\leq 3 \cdot \int_0^T \|\dot{y}^\star(t)-\dhyh(t)\|_2^2\,\mathrm{d}t + 3 \cdot C_f \cdot \int_0^T h^{2 \cdot \lambda \cdot \eta}\,\mathrm{d}t&&\\
&= 3 \cdot \|\dot{y}^\star-\dhyh\|_{L^2}^2 + 3 \cdot C_f \cdot T \cdot h^{2 \cdot \lambda \cdot \eta} &&\text{use~\eqref{eqn:LpNorm}}\\
&\leq 3 \cdot \|(y^\star,u^\star)-(\hyh,\huh)\|_\cX^2 + 3 \cdot C_f \cdot T \cdot h^{2 \cdot \lambda \cdot \eta}\,. &&\text{use \eqref{eqn:normX}}\\
&\leq 3 \cdot C_\eta^2 \cdot h^{2 \cdot \eta} + 3 \cdot C_f \cdot T \cdot h^{2 \cdot \lambda \cdot \eta}\,. &&\text{use \eqref{eqn:DefHatyu}}\\
&\leq 3 \cdot C_\eta^2 \cdot h^{2 \cdot (1-\lambda) \cdot \eta} \cdot h^{2 \cdot \lambda \cdot \eta} + 3 \cdot C_f \cdot T \cdot h^{2 \cdot \lambda \cdot \eta}\,. &&\text{split power}\\
&\leq 3 \cdot C_\eta^2 \cdot \epsilon \cdot \bigg(\frac{\epsilon}{\sqrt{2\cdot n_y+n_u}}\bigg)^{2 \cdot (1-\lambda)} \cdot h^{2 \cdot \lambda \cdot \eta} + 3 \cdot C_f \cdot T \cdot h^{2 \cdot \lambda \cdot \eta}\,. &&\text{use \eqref{eqn:hSuffSmall}}
\end{align*}
\thinmuskip=3mu
\medmuskip=3mu
\thickmuskip=3mu
This shows proposition~\eqref{eqn:lem:w2bound} for $C_{T,\lambda,\epsilon}= 3 \cdot C_\eta^2 \cdot \Big(\frac{\epsilon}{\sqrt{2\cdot n_y+n_u}}\Big)^{2 \cdot (1-\lambda)} + 3 \cdot C_f \cdot T$.

\section{Proof of Corollary~\ref{cor:interpolationerror}}\label{sec:proof:cor:interpolationerror}

\subsubsection{Structure}
We first show pointwise interpolation errors for $y^\star,\dot{y}^\star,u^\star$. We then show \eqref{eqn:infxh} by integrating over these pointwise error-bounds.

\subsection{Pointwise Errors}
Pick an arbitrary $t \in [0,T]$. Wlog., assume $t \in I_i$. Choose $j \in \lbrace 1,\dots,n_y\rbrace$ arbitrarily.

On $I_i$, interpolate $y_{[j]}^\star$ piecewise linear:
\begin{align}
y_{[j],h}(t) := y_{[j]}^\star(t_i) + \frac{y_{[j]}^\star(t_{i+1})-y_{[j]}^\star(t_i)}{t_{i+1}-t_i} \cdot (t-t_i) \label{eqn:interpolation_y}
\end{align}
Since the mean-value theorem does not work when $y_{[j]}^\star$ has edges\footnote{Example: $y_{[j]}^\star(t)=|t|$, then $y_{[j]}^\star(1)-y_{[j]}^\star(-1)\neq 2 \cdot \dot{y}_{[j]}^\star(\xi)$ for any $\xi \in (-1,1)$.}, we use the following trick with the convex hull:
\begin{align}
\exists\ \check{Y}_{[j]} \in \operatornamewithlimits{conv}_{t \in I_i}\lbrace \dot{y}_{[j]}^\star(t)\rbrace\ :
&& y_{[j]}^\star(t_{i+1}) &= y_{[j]}^\star(t_i) + \check{Y}_{[j]} \cdot (t_{i+1}-t_i) \label{eqn:mvt1}\\
\exists\ \hat{Y}_{[j]} \in \operatornamewithlimits{conv}_{t \in I_i}\lbrace \dot{y}_{[j]}^\star(t)\rbrace\ :
&& y_{[j]}^\star(t) &= y_{[j]}^\star(t_i) + \hat{Y}_{[j]} \cdot (t-t_i) \label{eqn:mvt2}
\end{align}
Hence:
\begin{align*}
|y_{[j]}^\star(t) \textcolor{red}{-y_{[j],h}(t)}|=& \left|y_{[j]}^\star(t) \textcolor{red}{- y_{[j]}^\star(t_i) - \frac{y_{[j]}^\star(t_{i+1})-y_{[j]}^\star(t_i)}{t_{i+1}-t_i} \cdot (t-t_i)}\right| &&\text{insert \eqref{eqn:interpolation_y} (red)}\\
=& \left|y_{[j]}^\star(t) -y_{[j]}^\star(t_i) - \check{Y}_{[j]} \cdot (t-t_i) \right\|_2 && \text{use \eqref{eqn:mvt1}}\\
=& \Big|\underbrace{y_{[j]}^\star(t) -y_{[j]}^\star(t_i) \textcolor{blue}{- \hat{Y}_{[j]}} \cdot (t-t_i)}_{\equiv 0 \text{, due to \eqref{eqn:mvt2}}} + \big(\textcolor{blue}{\hat{Y}_{[j]}} - \check{Y}_{[j]}\big) \cdot (t-t_i) \Big| && \text{add zero (blue)}\\
=& \Big|\big(\hat{Y}_{[j]}-\check{Y}_{[j]}\big) \cdot (t-t_i)\Big| = |t-t_i| \cdot \left|\hat{Y}_{[j]}-\check{Y}_{[j]}\right| 	&&\text{drop zero term}\\
\leq& |t-t_i| \cdot \operatornamewithlimits{max}_{\hat{t},\check{t} \in I_i}|\dot{y}_{[j],h}^\star(\hat{t})-\dot{y}_{[j],h}^\star(\check{t})| && \text{bound diameter}\\
\leq & |t-t_i| \cdot L \cdot |\hat{t}-\check{t}|^\eta  	&&\text{use \eqref{eqn:cor:interpolationerror:y}}\\
\leq &L \cdot h \cdot h^\eta \leq L \cdot T \cdot h^\eta &&\text{use $\check{t},\hat{t},t \in I_i$}
\end{align*}
Likewise, the derivative can be bounded:
\begin{align*}
|\dot{y}_{[j],h}(t)-\dot{y}_{[j]}^\star(t)|=&\left|\frac{y_{[j]}^\star(t_{i+1})-y_{[j]}^\star(t_i)}{t_{i+1}-t_i} - \dot{y}_{[j]}^\star(t)\right| &&\text{differentiate \eqref{eqn:interpolation_y}}\\
=&\left|\check{Y}_{[j]}-\dot{y}_{[j]}^\star(t)\right| \leq \operatornamewithlimits{max}_{\hat{t},\check{t} \in I_i}|\dot{y}_{[j],h}^\star(\hat{t})-\dot{y}^\star_{[j],h}(\check{t})| &&\text{bound distance}\\
\leq & L \cdot |\hat{t}-\check{t}|^\eta \leq L \cdot h^\eta 	&&\text{use \eqref{eqn:cor:interpolationerror:y}}
\end{align*}

Now, choose $j \in \lbrace 1,\dots,n_u\rbrace$ arbitrarily. On $I_i$, interpolate $u_{[j]}^\star$ piecewise constant:
\begin{align}
u_{[j],h}(t) := u^\star_{[j]}(t_i) \label{eqn:interpolation_u}
\end{align}
Hence:
\begin{align*}
|u_{[j]}^\star(t) {-u_{[j],h}(t)}|=&\left|u_{[j]}^\star(t)-u_{[j],h}^\star(t_i)\right|\leq L\cdot |t-t_i|^\eta \leq L \cdot h^\eta\,.
\end{align*}

Finally, we can use $\|y^\star(t)-y_{h}(t)\|_2 \leq \sqrt{n_y} \cdot \max_{1\leq j\leq n_y}|y^\star_{[j]}(t)-y_{[j],h}(t)|$ and likewise for $u^\star-u_h$. In conclusion, the following bounds hold $\forall t \in [0,T]$:
\begin{align*}
\|\dot{y}^\star(t) -\dot{y}_h(t)\|_2 	& \leq \sqrt{n_y} \cdot (1+T) \cdot L \cdot h^\eta\,,\\
\|y^\star(t) -y_h(t)\|_2  				& \leq \sqrt{n_y} \cdot (1+T) \cdot L \cdot h^\eta\,,\\
\|u^\star(t) -u_h(t)\|_2 				& \leq \sqrt{n_u} \cdot (1+T) \cdot L \cdot h^\eta\,.
\end{align*}

\subsection{Integral Errors}
We write out the norms in \eqref{eqn:infxh} and insert the above bounds:
\begin{align*}
&\|(y^\star,u^\star)-(y_h,u_h)\|_\cX + \|(y^\star,u^\star)-(y_h,u_h)\|_{L^\infty}\\
= 	&\sqrt{\int_0^T \left(\|\dot{y}^\star(t) -\dot{y}_h(t)\|_2^2 + \|{y}^\star(t) -y_h(t)\|_2^2 + \|\dot{u}^\star(t) -u_h(t)\|_2^2\right)\,\mathrm{d}t} + \operatornamewithlimits{ess\,sup}_{t \in [0,T]} \left\lbrace \left\|\begin{bmatrix}
y^\star(t)-y_h(t)\\
u^\star(t)-u_h(t)
\end{bmatrix}\right\|_\infty \right\rbrace\\
\leq & \sqrt{\int_0^T 3 \cdot (1+T) \cdot L \cdot h^\eta\,\mathrm{d}t} + \operatornamewithlimits{sup}_{t \in [0,T]}\left\lbrace \|y^\star(t) -y_h(t)\|_2 \right\rbrace + \operatornamewithlimits{sup}_{t \in [0,T]}\left\lbrace \|u^\star(t) -u_h(t)\|_2 \right\rbrace\\
\leq& 3 \cdot (1+T) \cdot T \cdot L \cdot h^\eta + 2 \cdot (1+T) \cdot L \cdot h^\eta
\end{align*}
This yields the proposition \eqref{eqn:infxh} with $C_\eta = (3 \cdot T + 2) \cdot (1+T) \cdot L$.

\chapter{Conclusions on QPM}\label{sec:Conclusions}
We have seen that QPM and DCM are very similar in their computational structure and expense, but very different in terms of robustness: QPM converges in practically relevant instances where DCM does not converge. The theory presented here explains why QPM converges for a large set of possible circumstances.

In contrast to prevailing opinion, QPM does in fact converge regardless of singular arcs or high-index differential-algebraic constraints. This has been confirmed here with both a proof and with numerical experiments. In contrast, DCM does sometimes not converge for singular arc problems, as is known in theory and as has been confirmed here by our numerical experiments.

When faced with having to choose a method for solving optimal control problems, and if you cannot foresee whether any of these problems may have singular arcs or any of the other features where DCM fails, then QPM is a suitable first resort.

\part{Conclusions and Recommendations}
\chapter{Conclusions}\label{sec:ConclusionsTotal}
We have presented three numerical methods: For the mathematical problem class of dynamic optimization we presented a direct quadratic penalty transcription method (QPM) and a direct penalty-barrier transcription scheme. These transcription schemes result in NLPs with large quadratic penalty terms. We presented MALM as a tailored solution algorithm for treating these NLPs more efficiently and reliably. Conclusions on each method have been given in the respective chapters.

\section{Possible Impacts}
In the following we will forecast impacts that the presented research could have on developments in mathematical methods for the numerical solution of optimization problems and of optimal control problems.

\paragraph{Replacement of Collocation Methods}
Penalty-based direct transcription methods are structurally very similar to collocation methods; in part because they generalize them. Yet, they can offer significant improvements to the robustness (i.e., reliability of convergence) and the rate of convergence over collocation methods. Both classes of methods are of the same computational cost, judged from identical sparsity patterns in their reduced Newton equations and similar iteration counts in the NLP solvers. The NLPs of QPM and DCM are constructed in a very similar way, as evident from comparision of Section~\ref{sec:NLP_of_DCM} and Section~\ref{sec:NLP_of_QPM}. QPM's overhead, discussed along Figure~\ref{fig:costdiagram}, results from the fact that QPM uses some additional quadrature points in comparison to the number of collocation points that DCM uses.

\paragraph{Paradigm Shift in NLP Solver Templates}
Current software for the solution of NLP in form~\eqref{eqn:NLP} assumes that the equality-constraints describe a subdomain of feasible points in the search space $\cX$. For example, SNOPT and IPOPT assume that the number of equality constraints do not exceed the dimension of $\cX$. In concrete terms, when $\bx \in \R^{\nx}$ then $c(\bx) \in \R^{\nc}$ where $\nc<\nx$ is assumed by SNOPT. In the light of QPM and in the light of NLP instances that originate from problems where the search space and constraint space is infinite, those concepts seem unjustified.

We believe that there should be three classes of NLP solver templates:
\begin{enumerate}
	\item Small dense: These problems may indeed assume that $\nc<\nx$. Solvers can use active sets, null-space projection, or -- in case the constraints are linear -- orthogonalize the rows of the constraint Jacobian.
	\item Large sparse: These problems may still assume that $\nc<\nx$. Due to practicality, solvers must make compromises when attempting to leverage the techniques from small dense problems.
	\item From discretization: These problems are finite-dimensional discretizations of infinite-dimensional optimization problems. For discretized problems, constraints and search spaces are to be treated as approximate projections from infinite-dimensional spaces; hence concepts such as exactly feasible points and all the geometric ideas of small dense problems must be abandoned, such as, e.g.,: orthogonalizability of constraint Jacobian rows, second-order corrections, exact Lagrange multipliers, dual boundedness, active sets, reduced Hessians, not to forget ``presolve'', and everything relating to non-regularized KKT conditions. Another aspect that is important for problems from discretization is an interface for the supply of quadrature weights to the NLP solver. In particular, we wish to supply barrier weights for the integral log-barrier that is inherited in primal-dual IPM, as this can enable mesh-independent rates of convergence.
\end{enumerate}
For the last class there virtually exists no solid open-source software project. In \cite{neuenhofen2017stable} we had already developed a numerically stable interior-point method that does not consider equations $\|\bA \cdot \bx - \bb\|_2$ as equality constraints as a description of a geometric subdomain but instead as a measure that shall be bounded by another value that is implicitly given from the problem instance. This is more generic and in alignment with what we need for discretized infinite-dimensional optimization problems.

A last aspect, that will also be critical for the practical utility of NLP solvers for the third category, is a suitable standard interface for continuous-level information. Particularly for direct transcription applications for the optimal control of PDE, it is necessary to define a suitable preconditioner interface. This interface should forsee NLP iterations on various levels of the discretization (nonlinear multi-grid method) and different strategies of preconditioning (null-space preconditioners, constraint-space preconditioners, domain decomposition, etc.). Each different strategy requires a different interweaving between the preconditioner and the NLP solver iteration.

\section{Recommendations for Future Research}
We suggests the following possible directions.

\paragraph{Mesh Refinement}
We have not presented mesh refinement strategies for QPM. This should work in an identical fashion to mesh refinement for DCM. Notably, in QPM we do not only have the mesh that we can adapt but we can also choose $\omega$ of different magnitude on different segments of the mesh. For instance, when solving on $[0,T]=[0,2]$ and finding larger errors over the sub-domain $[1,2]$ then it makes sense to choose $\omega(t=1.5)<\omega(t=0.5)$.

\paragraph{NLP Software for Problems from Discretization}
In the present landscape, there is a jungle of proprietary software for LP, QP, convex QP, and NLP. We are not aware of solvers for problems from discretization, according to our standards. Such a method would offer a bias-parameter like $\omega$ and interfaces for matrices $\bS$, such that $\bx\t\cdot\bS\cdot\bx=\|x\|_\cX^2$), $\bW$ of constraint quadrature weights, and $\bG$, such that $-\be\t\cdot\bG\cdot\log(\bx) = -\int_0^T \log x(t)\,\mathrm{d}t$. This results in different formulas for the central path, the primal-dual equations, the convergence complexity, solution accuracy, numerical stability, and practical performance.

\paragraph{GPU Parallelism}
We have not yet implemented optimal control solvers on GPU. Because the computational step (N.2) in our presented NLP solver is massively parallel, there is huge potential for GPU-accelerated versions of QPM. The challenge is with suitable GPU hardware. Since the matrices involved are sparse, the compute intensity is relatively low. This implies that GPU-internal RAM and PCIe bus form a computational bottleneck.

%========================== Appendices ============%===========================

\part{Appendices}
\addcontentsline{toc}{part}{Appendices}%
\appendix

	\section{Lebesgue Equivalence for Polynomials}\label{sec:Appendix_LebesgueIdentity}
Let $T:=(a,b) \in \cT_h$ and $p \in \N_0$. We show that
$$ 	\|\beta\cdot u\|_{L^\infty(T)} \leq \frac{p+1}{\sqrt{|T|}} \cdot \|\beta \cdot u\|_{L^2(T)} \quad \forall u \in \cP_p(T),\forall \beta \in \R. $$

Choose $u \in \cP_p(T)$ arbitrary. Since $\|\beta\cdot u \|_{L^k(T)} = |\beta| \cdot \|u\|_{L^k(T)}$ holds for both $k \in \lbrace 2,\infty\rbrace$, and for all $\beta\in\R$, w.l.o.g.\ let $\|u\|_{L^\infty(T)}=1$.
Since $\operatorname{sgn}(\beta)$ is arbitrary, w.l.o.g.\ let $u(\hat{t})=1$ for some $\hat{t} \in \overline{T}$.
Define $T_L:=[a,\hat{t}]$, $T_R:=[\hat{t},b]$, $\hat{\cP}_p:=\cP_p(T_L)\cap\cP_p(T_R)\cap\cC^0(T)$,
and
$ 	\hat{u} := \operatornamewithlimits{arg\,min}_{v \in \hat{\cP}_p}\big\lbrace \|v\|_{L^2(T)}\ \big\vert\ v(\hat{t})=1 \big\rbrace. 	$
Since $\cP_p \subset \hat{\cP}_p$, it holds $\|u\|_{L^2(T)}\geq \|\hat{u}\|_{L^2(T)}$.
Figure~\ref{fig:polyplothur} illustrates $u,\hat{u}$ for $p=8$.

\begin{figure}[tb]
	\centering
	\includegraphics[width=0.6\columnwidth]{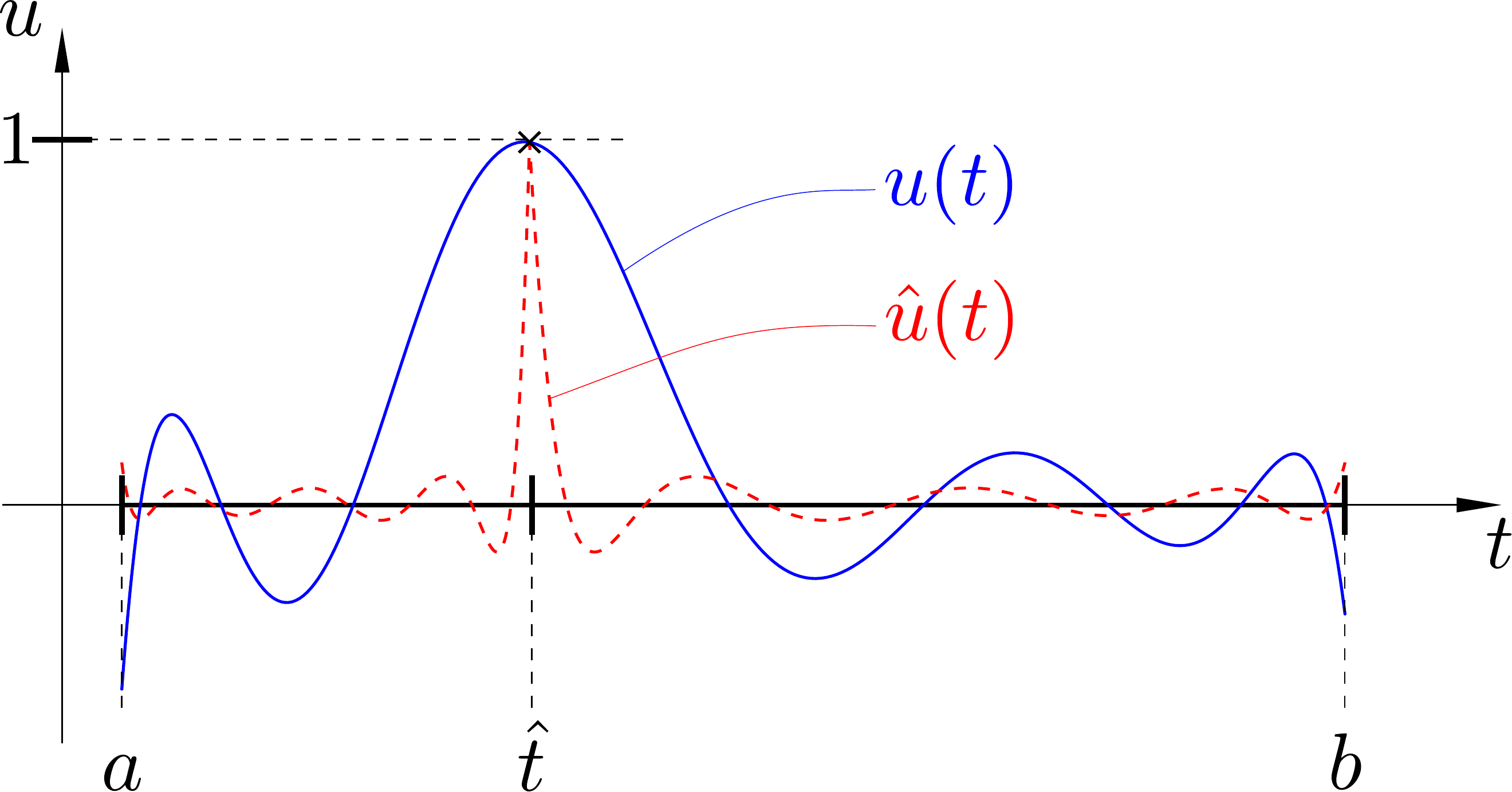}
	\caption{Polynomial $u$ and piecewise polynomial $\hat{u}$ over $T$.}
	\label{fig:polyplothur}
\end{figure}

\newcommand{\hur}{{\hat{u}_\text{ref}}\xspace}
\newcommand{\Tr}{{T_\text{ref}}\xspace}
Use $\|\hat{u}\|^2_{L^2(T)}=\int_a^b \hat{u}(t)^2\mathrm{d}t=(b-a)/2\cdot\int_{-1}^1 \hur(t)^2\mathrm{d}t=\frac{|T|}{2} \cdot \|\hur\|^2_{L^2(\Tr)}$, where $\hur$ is $\hat{u}$ linearly transformed from $T$ onto $\Tr:=(-1,1)$.
Since $\|\hat{u}\|_{L^2(T)}$ is invariant under changes of $\hat{t}$ because $\hat{u}(\hat{t}+(b-\hat{t})\cdot\xi)=\hat{u}(\hat{t}+(\hat{t}-a)\cdot\xi)$ $\forall \xi \in [0,1]$, w.l.o.g.\ we can assume for $\hat{u}$ that $\hat{t}=b$ and hence $\hur(1)=1$.
Since minimizing the $L^2(\Tr)$-norm, $\hur$ solves
\begin{equation}
\label{eqn:CQP_hur}
\operatornamewithlimits{min}_{v \in \cP_p(\Tr)}\ 1/2 \cdot \int_\Tr v(t)^2\,\mathrm{d}t\text{ subject to } v(1)=1.
\end{equation}

We represent $\hur=\sum_{j=0}^p \alpha_j \cdot \phi_j$, where $\phi_j$ is the $j^\text{th}$ Legendre polynomial. These satisfy \cite{refLegendre}: $\phi_j(1)=1\quad\forall j\in \N_0,\quad \int_\Tr \phi_j(t)\cdot\phi_k(t)\mathrm{d}t=\delta_{j,k} \cdot \gamma_j\quad\forall j,k\in\N_0,$
where $\gamma_j:=2/(2\cdot j +1)$ and $\delta_{j,k}$ the Kronecker delta. We write $\bx=(\alpha_0,\alpha_1,\dots,\alpha_p)\t\in\R^{p+1}$, ${D}=\opdiag(\gamma_0,\gamma_1,\dots,\gamma_p)\in\R^{(p+1)\times (p+1)}$ and $\be\in\R^{p+1}$.
Then \eqref{eqn:CQP_hur} can be written in~$\bx$:
$$ 	\min_{\bx\in\R^{p+1}}\ \psi(\bx):=1/2\cdot\bx\t\cdot{D}\cdot\bx \text{ subject to }\be\t\cdot\bx=1. 	$$
From the optimality conditions \cite[p.~451]{Nocedal}
$
\left[\begin{array}{c|c}
{D} & \be\\
\hline
\phantom{{}^{{}^A}}\be\t\phantom{{}^{{}^A}} & 0
\end{array}\right]
\cdot
\left(\begin{array}{c}
\bx\\
\hline
-\lambda
\end{array}\right) =
\left(\begin{array}{c}
\bO\\
\hline
1
\end{array}\right)$ follows $\bx={D}\inv\cdot\be\cdot\lambda$ and thus $\be\t\cdot{D}^{-1}\cdot\be\cdot\lambda=1$.
Using $\be\t\cdot{D}^{-1}\cdot\be=\sum_{j=0}^p 1/\gamma_j=\frac{(p+1)^2}{2}$ yields $\lambda=2/(p+1)^2$ and
$\psi(\bx)=\frac{1}{2} \cdot ({D}\inv\cdot\be\cdot\lambda)\t\cdot{D}\cdot({D}\inv\cdot\be\cdot\lambda)=\frac{\lambda}{2}=\frac{1}{(p+1)^2}$.
Hence, $\frac{1}{2}\cdot\|\hur\|^2_{L^2(\Tr)}=1/(p+1)^2$.
Hence, $\frac{1}{2}\cdot\|\hat{u}\|^2_{L^2(T)}=\frac{|T|}{2}\cdot 1/(p+1)^2$.
Hence, $\|u\|_{L^2(T)}\geq\|\hat{u}\|_{L^2(T)}=\sqrt{|T|}/(p+1)\cdot \underbrace{\|u\|_{L^\infty(T)}}_{=1}$,
or, \mbox{$\|\beta \cdot u\|_{L^2(T)}\geq \sqrt{|T|}/(p+1)\cdot \|\underbrace{\beta \cdot u}_{\tilde{u}}\|_{L^\infty(T)}$}.
In conclusion:
\begin{align}
\|\tilde{u}\|_{L^\infty(T)}\leq \frac{p+1}{\sqrt{|T|}}\cdot\|\tilde{u}\|_{L^2(T)}\ \, \forall \tilde{u} \in \cP_p(T)\, \forall T \in \cT_h. \label{eqn:PropAppendix2}
\end{align}

\section{Order of Approximation for Non-smooth and Continuous Non-differentiable Functions}\label{app:3}

In the following we illustrate that the assumption $\ell>0$ in \eqref{eqn:InfBound} is rather mild. To this end, we consider two pathological functions for $g:=x^\star_{\omega,\tau}$. In our setting, $n_y=0,\,n_z=1$, and we  interpolate a given pathological function $g$ with $x_h \in \cX_{h,p}$ over $\Omega=(-1,1)$. We use $p=0$.

\paragraph{A function with infinitely many discontinuities}
The first example is a non-smooth function that has infinitely many discontinuities. Similar functions can arise as optimal control solutions; cf.\ Fuller's problem \cite{Fuller}.

Consider the limit $g_\infty$ of the following series:
\newcommand{\sign}{\operatorname{sign}\xspace}
\begin{align*}
g_0(t) := -1\,,\qquad	g_{k+1}(t) := \left\lbrace \begin{matrix}
g_{k}(t) & \text{if }t \leq 1-2^{-k}\\
-g_k(t) & \text{otherwise}
\end{matrix} \right.& &k=0,1,2,\dots.
\end{align*}
$g_\infty$ switches between $-1$ and $1$ whenever~$t$ halves its distance to $1$. Figure~\ref{fig:NestedStep} shows $g_k$ for $k=4,\,5$\,.

\begin{figure}[tb]
	\centering
	\includegraphics[width=0.6\columnwidth]{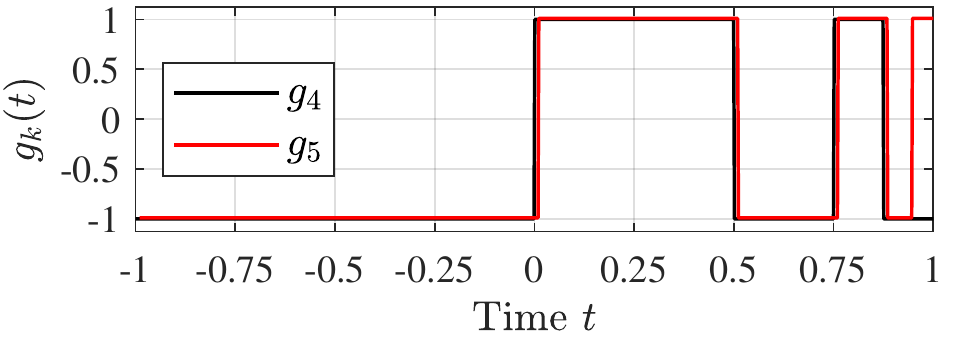}
	\caption{Nested step-function $g_k$ for $k=4,\,5$\,.}
	\label{fig:NestedStep}
\end{figure}

Using mesh-size $h=2^{-k}$ for some $k \in \N$,  define $u(t):=g_k(t)\in \cX_{h,p}$. Hence,
$	\operatornamewithlimits{inf}_{x_h \in \cX_{h,p}}%\left\lbrace
\|g_\infty-x_h\|_\cX %\right\rbrace 
\leq \|g_\infty-u\|_{L^2(\Omega)} $.
It follows that
$$ 	|u(t) - g_\infty(t)| \leq \left\lbrace \begin{matrix}
0 & \text{if }t \leq 1-2^{-k}\\
2 & \text{otherwise}
\end{matrix} \right. $$
Hence, $ \|g_\infty-u\|_{L^2(\Omega)}\leq \|g_\infty-u\|_{L^1(\Omega)} \leq 2/2^k = \cO(h^1).$
Therefore, all $\ell \in (0,0.5]$ satisfy~\eqref{eqn:InfBound}.

\paragraph{A continuous but nowhere differentiable function}
Consider the following Weierstrass function, which is continuous but non-differentiable:
$$ 	g(t) := \frac{1}{2} \cdot \sum_{k=0}^\infty a^{k} \cdot \cos(7^k \cdot \pi \cdot t)$$
for $0<a\leq 0.5$. This function with range $\subset[-1,1]$ satisfies the H\"older property
$ |g(t)-g(s)| \leq C \cdot |t-s|^\alpha $
with some $C \in \R_+$ for $\alpha = -\log(a)/\log(7)$\, \cite{Zygmund}.
For $a\leq 0.375$ we have $\alpha\geq 0.504$\,.

According to this property, a piecewise constant interpolation $u \in \cX_{h,p}$ of $g$ satisfies
$ |g(t)-u(t)| \leq |g(t)-g(s)| \leq C \cdot |t-s|^\alpha\leq |h|^\alpha $. 
In conclusion, $\operatornamewithlimits{inf}_{x_h \in \cX_{h,p}}\left\lbrace \|g-x_h\|_\cX \right\rbrace \leq \|g-u\|_{L^2(\Omega)} \leq \|g-u\|_{L^1(\Omega)} = \cO(h^\alpha)$. Therefore, all $\ell \in (0,\alpha-0.5]$ satisfy~\eqref{eqn:InfBound}.

\section{Proof of Lemma~\ref{lem:BoundLipschitz_Fr}}\label{sec:Appendix_ProofLemma1}

%\begin{proof}
The boundedness follows from (A.2).

Lipschitz continuity of  $r$ is not as straightforward. We will make use of the following trace theorem \cite{TraceDing}: For an open interval $I \subseteq \Omega$ it holds that $\|u\|_{L^2(\partial I)} \leq K \cdot \|u\|_{H^1(I)}$	with a constant $K$ independent of $u$. Assume $|u|$ attains its essential supremum on $\overline{\Omega}$ at $t=t^\star$. Choosing $I=(t^\star,t_E)\subset\Omega$, then
$  \|u\|_{L^\infty(\Omega)} = |u(t^\star)| \leq \|u\|_{L^2(\partial I)}.  $
Using this together with the above bound and
$  \|u\|_{H^1(I)}\leq\|u\|_{H^1(\Omega)}  $
results in
\begin{align}
\|u\|_{L^\infty(\Omega)} \leq K \cdot \|u\|_{H^1(\Omega)}.  \label{eqn:TraceResult}
\end{align}

Below, for a generic Lipschitz continuous function $g:\R^k\rightarrow \R^{n_g}$ with Lipschitz-constant $L_g$ and $\|\cdot\|_1$-bound $|g|_\text{max}$, we use the relation
\begin{align*}
\begin{split}
&\left|\|g(\xi_2)\|_2^2\,-\,\|g(\xi_1)\|_2^2\right|
=\big|\|g(\xi_2)\|_2+\|g(\xi_1)\|_2\big| \cdot \underbrace{\big|\|g(\xi_2)\|_2-\|g(\xi_1)\|_2\big|}_{\leq\|g(\xi_2)-g(\xi_1)\|_2}\\
&\leq n_g \cdot \big|\|g(\xi_2)\|_1+\|g(\xi_1)\|_1\big| \cdot \|g(\xi_2)-g(\xi_1)\|_1
\leq n_g \cdot 2 \cdot |g|_{\text{max}} \cdot L_g \cdot \|\xi_2 - \xi_1\|_1\,,
\end{split}\tageq\label{eqn:LipQuadPenaltyFunc}
\end{align*}
where we used $|\alpha^2-\beta^2|=|\alpha+\beta|\cdot|\alpha-\beta|$ in the first line and the triangular inequality in the second line.
Using the above bound, we can show Lipschitz continuity of~$r$:
\begin{align*}
&|r(x_2)-r(x_1)|\leq \int_\Omega \Big| \left\|c\left(\dot{y}_2(t),{y}_2(t),z_2(t),t\right)\right\|_2^2 - \left\|c\left(\dot{y}_1(t),{y}_1(t),z_1(t),t\right)\right\|_2^2\Big| \mathrm{d}t\\
&\quad + \Big|\left\|b\left(y_2(t_1),\ldots,y_2(t_M)\right)\right\|_2^2
-\left\|b\left(y_1(t_1),\ldots,y_1(t_M)\right)\right\|_2^2\Big| \\
&\leq \int_\Omega 2 \cdot \nc \cdot |c|_\text{max} \cdot L_c \cdot \left\|\begin{pmatrix}
\dot{y}_2(t)-\dot{y}_1(t)\\
y_2(t)-y_1(t)\\
z_2(t)-z_1(t)
\end{pmatrix}\right\|_1 \mathrm{d}t\\
&\quad + 2 \cdot \nb \cdot |b|_\text{max} \cdot L_b \cdot \underbrace{\left\|\begin{pmatrix}
	y_2(t_1)-y_1(t_1)\\
	\vdots\\
	y_2(t_M)-y_1(t_M)
	\end{pmatrix}\right\|_1}_{\leq M \cdot \|y_2-y_1\|_{L^\infty(\Omega)}} \\
&\leq 2  \cdot \nc \cdot |c|_\text{max} \cdot L_c \cdot \left(\|\dot{y}_2-\dot{y}_1\|_{L^1(\Omega)}+\|y_2-y_1\|_{L^1(\Omega)}+\|z_2-z_1\|_{L^1(\Omega)}\right)\\
&\quad + 2  \cdot \nb \cdot |b|_\text{max} \cdot L_b \cdot M \cdot K \cdot \|y_2-y_1\|_{H^1(\Omega)},
\end{align*}
where \eqref{eqn:TraceResult} has been used %in the last line
to bound $\|y_2-y_1\|_{L^\infty(\Omega)}$.

If $y_2=y_1$ then we see the result shows Lipschitz continuity of $r$ with respect to $\|z\|_{L^1(\Omega)}$. Using
$$  \|u\|_{L^1(\Omega)} \leq \sqrt{|\Omega|} \cdot \|u\|_{L^2(\Omega)}\quad \forall u\in L^1(\Omega)$$
according to \cite[Thm.~2.8, eqn.~8]{Adams}, and the definition of $\|\cdot\|_\cX$, we arrive at
\begin{align*}
&\|\dot{y}_2-\dot{y}_1\|_{L^1(\Omega)} +
\|y_2-y_1\|_{L^1(\Omega)} +
\|z_2-z_1\|_{L^1(\Omega)} \\
\leq& \sqrt{|\Omega|}\cdot\left(\|\dot{y}_2-\dot{y}_1\|_{L^2(\Omega)} +
\|y_2-y_1\|_{L^2(\Omega)} +
\|z_2-z_1\|_{L^2(\Omega)}\right)
\leq 3 \cdot \sqrt{|\Omega|}\cdot\|x_2-x_1\|_\cX,
\end{align*}
which shows Lipschitz continuity of $r$ with respect to~$\|x\|_\cX$.

Lipschitz continuity of $F$ follows from Lipschitz continuity of $f$:
\begin{align*}
|F(x_2)-F(x_1)|
&\leq \int_\Omega |f(\dot{y}_2(t),y_2(t),z_2(t))-f(\dot{y}_1(t),y_1(t),z_1(t))|\,\mathrm{d}t\\
&\leq \int_\Omega L_f \cdot \left\|\begin{pmatrix}
\dot{y}_2(t)-\dot{y}_1(t)\\
{y}_2(t)-{y}_1(t)\\
{z}_2(t)-{z}_1(t)
\end{pmatrix}\right\|_{1}\mathrm{d}t
\leq L_f \cdot \left\|\begin{pmatrix}
\dot{y}_2-\dot{y}_1\\
{y}_2-{y}_1\\
{z}_2-{z}_1
\end{pmatrix}\right\|_{L^1(\Omega)}
\end{align*}
%	\end{proof}

\section{Properties of the $\log$-Barrier Function}\label{sec:Appendix_BarrierFunctioProperties}
Let $0<\zeta\ll 1$ be a fixed small arbitrary number.
\begin{lem}[Order of the $\log$ Term]\label{lem:tauLw}It holds:
	$ 	\left|\tau \cdot \log\left({\tau}/{L_\omega}\right)\right| = \cO\left(\tau^{1-\zeta}\right).	$
\end{lem}
\begin{proof}
	We use $L_\omega$ from \eqref{eqn:LipLw}, where $L_F\geq 2,\ L_r\geq 2$ and $0<\tau\leq\omega\leq 1$.  We get
	\begin{align*}
	\left|\tau\cdot\log\left({\tau}/{L_\omega}\right)\right|
	&= \tau \cdot \left(|\log(\tau)-\log(L_\omega)|\right) \leq \tau \cdot \left(|\log(\tau)|+|\log(L_\omega)|\right)\\
	&= \tau \cdot \left(\left|\log\left(L_F+\frac{L_r}{2\cdot\omega}\right)\right|+|\log(\tau)|\right)\\
	&\leq \tau \cdot \left( 1+|\log(L_F)|+\left|\log\left(\frac{L_r}{2\cdot\omega}\right)\right|+|\log(\tau)| \right)\\
	&\leq \tau \cdot\Big( \underbrace{1+|\log(L_F)| +|\log(L_r/2)|}_{= \cO(1)}+\underbrace{|\log(\omega)|}_{\leq|\log(\tau)|}+|\log(\tau)|\Big)\\
	&= \cO(\tau) + \cO(\tau \cdot |\log(\tau)|).
	\end{align*}
	In the third line above, we used the fact that for $\alpha,\beta\geq 2$,
	follows $\log(\alpha+\beta)\leq\log(\alpha)+\log(\beta).$
	The result follows from $\tau \cdot |\log(\tau)| = \cO(\tau^{1-\zeta})$, as we show using L'H\^opital's rule:
	\begin{small}
		\begin{align*}
		\lim\limits_{\tau \rightarrow 0}\frac{\tau \cdot \log(\tau)}{\tau^{1-\zeta}}
		=\lim\limits_{\tau \rightarrow 0} \frac{\log(\tau)}{\tau^{-\zeta}}
		{\stackrel{\mathclap{\normalfont\mbox{L'H}}}{=}} \lim\limits_{\tau \rightarrow 0} \frac{\frac{\mathrm{d}}{\mathrm{d}\tau} \Big(\log(\tau)\Big)}{\frac{\mathrm{d}}{\mathrm{d}\tau}\Big(\tau^-\zeta\Big)} = \lim\limits_{\tau \rightarrow 0}\frac{\frac{1}{\tau}}{-\zeta \cdot \tau^{-\zeta-1}}
		= \lim\limits_{\tau \rightarrow 0} \frac{\tau^\zeta}{-\zeta}
		=0
		\end{align*}
	\end{small}
\end{proof}

\begin{lem}[Bound for $\Gamma$]\label{lem:Gamma}
	If $x \in \cX$ with $\|z\|_{L^\infty(\Omega)} = \cO(1)$, then
	\begin{align*}
	\left|\tau \cdot \Gamma\left(\bar{x}\right)\right| &= \cO\left(\tau^{1-\zeta}\right)\,,&
	\left|\tau \cdot \Gamma\left(\check{x}\right)\right| &= \cO\left(\tau^{1-\zeta}\right)\,.
	\end{align*}
\end{lem}
\begin{proof}
	Since the definitions are similar, we only show the proof for $\bar{x}$:
	\begin{align*}
	|\tau \cdot \Gamma(\bar{x})| &\leq\left|\tau \cdot \sum_{j=1}^{n_z}\int_\Omega\,\log\left(\bar{z}_{[j]}(t)\right)\mathrm{d}t \right| \leq n_z \cdot |\Omega| \cdot \operatornamewithlimits{max}_{1\leq j\leq n_z} \|\tau \cdot \log(\bar{z}_{[j]})\|_{L^{\infty}(\Omega)}\\
	&\leq n_z \cdot |\Omega| \cdot \Big(\underbrace{\cO\left(\tau^{1-\zeta}\right)}_{\text{bound for }\bar{z}_{[j]}< 1} + \underbrace{\cO(\tau)}_{\text{bound for }\bar{z}_{[j]}\geq 1}\Big)= \cO\left(\tau^{1-\zeta}\right).\tageq\label{eqn:aux:PBP_OptGap}
	\end{align*}
	In the third line, we  distinguished two cases, namely
	$ 	\left|\log\left(\bar{z}_{[j]}(t)\right)\right| 	$
	attains its essential supremum at a $t\in \overline{\Omega}$ where either $\bar{z}_{[j]}(t)<1$ (case 1) or where $\bar{z}_{[j]}(t)\geq 1$ (case 2). In the first case, we can use  Lemma~\ref{lem:StrictInteriorness}~\&~\ref{lem:tauLw}. In the second case, we simply bound the logarithm using
	$	\|\bar{z}_{[j]}\|_{L^\infty(\Omega)}\leq\|z\|_{L^\infty(\Omega)}= \cO(1)$ to arrive at the term $\cO(\tau)$ in the brackets. The final line follows from the fact that $|\log(\tau)+\log(\omega)|\geq 1$.
\end{proof}

\section*{List of Symbols}\label{sec:Symbols} 
\begin{longtable}{p{0.1\textwidth}p{0.9\textwidth}}
	\phantom{-}					& \textbf{Optimal Control Functions}\\
	$t$ 						& time \\
	$y$ 						& function of $t$; states\\
	$u$ 						& function of $t$; controls\\
	$\dot{y}$					& time-derivative of $y$\\
	$y_{[\upsilon]}$			& the $\upsilon^\text{th}$ vectorial component of $y$\\
	$u_{[\upsilon]}$			& the $\upsilon^\text{th}$ vectorial component of $u$\\
	$n_y$ 						& number of states\\
	$n_u$ 						& number of controls\\[10pt]
	\phantom{-}					& \textbf{Optimal Control Problem}\\
	$T$ 						& final time \\
	$M$ 						& objective\\
	$b$	 						& function of $y(0)$ and $y(T)$; boundary conditions\\
	$f_1$ 						& function of $y(t)$, $u(t)$, and $t$; differential equations right-hand side\\
	$f_2$ 						& function of $y(t)$, $u(t)$, and $t$; algebraic equations\\
	$f$ 						& function of $\dot{y}(t)$, $y(t)$, $u(t)$, and $t$; short-hand for differential and algebraic equations\\
	$\yL,\uL$ 					& function of $t$; left bounds on $y,u$ \\
	$\yR,\uR$ 					& function of $t$; right bounds on $y,u$ \\[10pt]
	\phantom{-}					& \textbf{Mesh}\\
	$i$ 						& mesh point index\\
	$N$ 						& number of mesh intervals\\
	$t_i$ 						& mesh points, from $1$ to $N+1$\\
	$I_i$ 						& mesh interval $[t_i,t_{i+1}]$\\
	$h$ 						& mesh size\\
	$y_i,u_i$					& Explicit Euler approximations to $y(t_i),u(t_i)$\\[10pt]
	\phantom{-}					& \textbf{Local Minimizers}\\
	$y^\star,u^\star$ 			& an exact local minimizer of the optimal control problem\\
	$y^\star_h,u^\star_h$		& a numerical local minimizer of the optimal control problem\\
	$\hat{y}_h,\hat{u}_h$ 		& an approximation of $y^\star,u^\star$ on the mesh\\
	$\bx^\star$ 				& local minimizer of an optimization problem\\[10pt]
	\phantom{-}					& \textbf{Solution Measures}\\
	$\delta$ 					& optimality gap\\
	$r$							& constraint violation measure; functional of $y,u$\\
	$\rho$						& equality feasibility residual\\
	$\gamma$ 					& inequality feasibility residual\\[10pt]
	\phantom{-}					& \textbf{Piecewise Polynomials}\\
	$p$ 						& piecewise polynomials degree\\
	$y_h$ 						& continuous piecewise polynomials function of degree $p$\\
	$u_h$ 						& discontinuous piecewise polynomials function of degree $p-1$\\
	$\cX_{h,p}$					& space of all $y_h$ of degree $p$ and $u_h$ of degree $p-1$\\[10pt]
	\phantom{-}					& \textbf{Collocation}\\
	$\cT_p$						& a set of $p$ collocation points in $\Iref$\\
	$\cT_{h,p}$ 	 			& a set of $p$ collocation points on each mesh interval $I_i$\\
	$\Tclref{m}$				& set of all Chebyshev-Gauss-Lobatto points $\tau_{j}$ of degree $m$\\
	$\Tcl{i,m}$ 				& set of all Chebyshev-Gauss-Lobatto points $t_{i,j}$ of degree $m$\\[10pt]
	\phantom{-}					& \textbf{Reference Conventions}\\
	$n$ 						& natural number\\
	$\theta,\kappa$				& real numbers\\
	$v,w$						& real vectors\\
	$\Iref$ 					& reference interval $[-1,1]$\\
	$\tau$ 						& a number on $\Iref$\\
	$\xi$ 						& the $\arccos$-transformation of $\tau$ into $[0,\pi]$\\
	$z$						 	& vectorial functions in $t$ space\\
	$\phi,\zeta$				& scalar functions of $t$\\
	$\psi$						& scalar functions of $\tau$\\
	$g$							& scalar function in $\xi$ space\\
	$g'$ 						& derivative with respect to $\xi$\\
	$\texttt{L},\texttt{R}$		& foot-indices for left and right bounds on quantities\\
	$\bF$	 					& Fourier-type matrices\\
	$\bG$	 					& unitary matrices\\
	$\bD$ 						& diagonal matrices\\[10pt]
	\phantom{-}					& \textbf{NLP}\\
	$\bx$ 						& local minimizer\\
	$n_\bx$						& dimension of minimizer\\
	$n_\bc$						& number of equality constraints\\
	$\bf$ 						& objective function, mapping $\bx \in \R^{n_\bx}$ into $\R$\\
	$\bc$ 						& equality constraints function, mapping $\bx \in \R^{n_\bx}$ into $\R^{n_\bc}$\\
	$\be$ 						& vector of ones\\
	$\bO$ 						& vector of zeros\\
	$\bI$ 						& identity matrix\\
	$\bA$ 						& matrix of affine inequality constraints\\
	$\bbL,\bbR$ 				& left/right inequality constraints vectors\\
	$\by$ 						& Lagrange multipliers vector for equality constraints\\
	$\bzL,\bzR$					& Lagrange multipliers vector for left/right inequality constraints\\
	$n_\bb$ 					& dimension of $\bbL,\bbR$\\[10pt]
	\phantom{-}					& \textbf{Constrained Optimization Algorithm}\\
	$\bH$ 						& Hessian of the Lagrangian\\
	$\bJ$ 						& Jacobian of the equality constraints\\
	$\bS$ 						& reduced Newton Matrix\\
	$\Delta\bx$					& Newton step for $\bx$\\
	$\Delta\by$					& Newton step for $\by$\\
	$\Delta\bzL$				& Newton step for $\bzL$\\
	$\Delta\bzR$				& Newton step for $\bzR$\\
	$\bw$ 						& stacked vector of $\bx,\by,\bzL,\bzR$\\
	$\Phi_\omega$				& merit function, mapping $\bw$ into $\R$\\
	$\omega$					& penalty parameter of the interior-point method\\
	$\mu$ 						& barrier parameter of the interior-point method\\
	$s_\text{max}$ 				& maximum step size for the line-search\\
	$s$ 						& line-search step size\\
	$\Delta\bw$					& stacked vector of $\Delta\bx,\Delta\by,\Delta\bzL,\Delta\bzR$\\[10pt]
	\phantom{-}					& \textbf{Quadrature}\\
	$\cQ_{h,q}$					& set of $q$ quadrature points $t$ with their respective quadrature weights $\alpha$ per mesh interval $I_i$\\
	$q$ 						& quadrature degree\\
	$Q_{h,q}[f;y,u]$			& short-hand writing for the quadrature formula $\sum_{(t,\alpha) \in \cQ_{h,q}} \alpha \cdot \|f(\dot{y}(t),y(t),u(t),t)\|_2^2$\\[10pt]
	\phantom{-}					& \textbf{Orders}\\
	$\ell$						& piecewise polynomials quadrature order\\
	$\eta$ 	 					& mesh interpolation order\\
	$\lambda$ 					& H\"older exponent\\[10pt]
	\phantom{-}					& \textbf{Factors}\\
	$\Cquad$ 					& piecewise polynomials quadrature factor\\
	$C_\eta$ 					& interpolation factor\\
	$\Clam$ 					& H\"older factor\\[10pt]
	\phantom{-}					& \textbf{Bounds}\\
	$\Cbox$ 					& in~\AssumpI; maximum diameter between $\yL,\uL$ and $\yR,\uR$\\
	$\Cobj$ 					& in~\AssumpII; lower bound on $M$\\
	$\epsilon$ 					& in~\AssumpIII; neighborhood margin of H\"older continuity\\
	$\tmpyO,\,\tmpyT,\,\tmpyt$	& in~\AssumpIII; template for vectors in $R^{n_y}$\\
	$\tmput$ 					& in~\AssumpIII; template for vector in $R^{n_u}$\\[10pt]
	\phantom{-}					& \textbf{Convergence Analysis}\\[2pt]
	$C_\chi$ 					& equals $\max\left\lbrace\,L_M \cdot C_\eta^\lambda\ , \ 0.5 \cdot C_r \cdot C_\eta^{2\cdot\lambda}\ ,\ \Cquad\,\right\rbrace$\\[5pt]
	$C_M$ 						& in Fundamental Lemma~\ref{lem:fundamental}; equals $\Clam \cdot \big(\sqrt{2\cdot n_y} \cdot C_\eta \big)^\lambda$\\
	$C_b$ 						& in Fundamental Lemma~\ref{lem:fundamental}; equals $\Clam^2 \cdot \big(\sqrt{2\cdot n_y} \cdot C_\eta \big)^{2 \cdot \lambda}$\\
	$C_f$ 						& in Fundamental Lemma~\ref{lem:fundamental}; equals $\Clam^2 \cdot \big(\sqrt{n_y+n_u} \cdot C_\eta \big)^{2 \cdot \lambda}$\\
	$C_{r}$ 					& equals $3 \cdot C_{\epsilon,\lambda} + 4 \cdot \Clam^2 \cdot C_{T,\lambda} + \Clam^2$\\
	$C_{T,\lambda,\epsilon}$	& equals $3 \cdot C_\eta^2 \cdot \Big(\frac{\epsilon}{\sqrt{2\cdot n_y+n_u}}\Big)^{2 \cdot (1-\lambda)} + 3 \cdot C_f \cdot T$ \\[10pt]
	\phantom{-}					& \textbf{Spaces}\\
	$\cP_p(I)$					& all functions that are polynomials of degree $\leq p$ over $I$\\
	$\cC^{k,\lambda}$			& H\"older space\\
	$L^d$			            & Lebesgue space\\
	$W^{k,d}$			        & Sobolev space\\
	$\cX$			            & Candidate space\\
	$\cB$			            & Space of feasible candidates\\
	$\cP_{p}(I)$	            & Space of functions that are polynomials on an interval\\
	$\cX_{h,p}$		            & Piecewise polynomials space\\
	$\cB_{h,p}$					& Space of feasible candidates to \eqref{eqn:POCPh}\\
	\phantom{-}					& \textbf{Norms}\\
	$\|\cdot\|_2$ 				& Euclidean norm\\
	$\|\cdot\|_\infty$			& maximum norm\\
	$\|\cdot\|_{L^d}$		    & Lebesgue $L^d$ norm\\
	$\|\cdot\|_{W^{k,d}}$ 	    & Sobolev $W^{k,d}$ norm\\
	$\|\cdot\|_{\cX}$			& Sobolev $\cX$ norm\\[10pt]
	\phantom{-}					& \textbf{Operators}\\
	$\bA\t$			 			& transpose of $\bA$\\
	$\opdiag(\bx)$	 			& diagonal matrix with entries of $\bx$
\end{longtable}

\end{onehalfspacing}

%========================== Back matter ======================================

\backmatter

\bibliography{bibliography}

\newpage
\thispagestyle{empty}
\mbox{}

\end{document}